\documentclass [12pt]{amsart}
\pdfoutput=1

\usepackage[utf8]{inputenc}
\pdfoutput=1
\usepackage{amsmath,amssymb,amsthm,amsfonts}
\usepackage{mathtools}

\usepackage[english]{babel}
\usepackage{comment}
\usepackage{bbm}
\usepackage{mathrsfs}
\usepackage{xspace}
\usepackage{tikz,graphicx,color}
\usepackage{tikz-cd}
\usepackage{tikz-3dplot}
\usetikzlibrary{calc}
\usetikzlibrary{arrows}
\usetikzlibrary{shapes}
\usetikzlibrary{patterns}
\usetikzlibrary{positioning}
\usetikzlibrary{arrows.meta}
\usetikzlibrary{decorations.markings}
\usepackage{epstopdf}
\usepackage{accents}

\usetikzlibrary{external}

\usepackage{enumerate}
\usepackage[arrow]{xy}
\usepackage{diagbox}
\usepackage{subfig}
\usepackage{arcs}
\usepackage{xcolor}
\usepackage{tabu}
\usepackage{booktabs}

\usepackage[outline]{contour}
\contourlength{1.2pt}

\usepackage[margin=1.0in]{geometry}

\usepackage{enumitem}
\usepackage{letltxmacro}
\usepackage{thmtools,etoolbox}

\def\myarabic#1{\normalfont(\roman{#1})}
\newlist{theoremlist}{enumerate}{1}
\setlist[theoremlist]{label=\myarabic{theoremlisti},ref={\myarabic{theoremlisti}},itemindent=0pt,labelindent=0pt,
leftmargin=*,noitemsep}

\makeatletter
\renewcommand{\p@theoremlisti}{\perh@ps{\thetheorem}}
\protected\def\perh@ps#1#2{\textup{#1#2}}
\newcommand{\itemrefperh@ps}[2]{\textup{#2}}
\newcommand{\itemref}[1]{\begingroup\let\perh@ps\itemrefperh@ps\ref{#1}\endgroup}
\makeatother

\usepackage{nameref,hyperref}
\usepackage[capitalize]{cleveref}

\newtheorem{theorem}{Theorem}[section]

\newtheorem{lemma}[theorem]{Lemma}
\newtheorem{proposition}[theorem]{Proposition}
\newtheorem{corollary}[theorem]{Corollary}
\theoremstyle{definition}
\newtheorem{remark}[theorem]{Remark}
\theoremstyle{definition}
\newtheorem{definition}[theorem]{Definition}
\newtheorem{conjecture}[theorem]{Conjecture}
\newtheorem{question}[theorem]{Question}
\theoremstyle{definition}
\newtheorem{problem}[theorem]{Problem}
\theoremstyle{definition}
\newtheorem{example}[theorem]{Example}

\crefname{figure}{Figure}{Figures}

\addtotheorempostheadhook[theorem]{\crefalias{theoremlisti}{theorem}}
\addtotheorempostheadhook[lemma]{\crefalias{theoremlisti}{lemma}}
\addtotheorempostheadhook[proposition]{\crefalias{theoremlisti}{proposition}}
\addtotheorempostheadhook[corollary]{\crefalias{theoremlisti}{corollary}}

\def\Acal{\mathcal{A}}\def\Ccal{\mathcal{C}}\def\Hcal{\mathcal{H}}\def\Zcal{\mathcal{Z}}

\def\zbf{\mathbf{z}}

% ...

\def\one{{\mathbbm{1}}}
\def\C{\mathbb{C}}
\def\R{\mathbb{R}}

\def\Z{\mathbb{Z}}

% ...

% ...

\def\<{{\langle}}
\def\>{{\rangle}}

\def\Span{ \operatorname{Span}}

\def\wt{\operatorname{wt}}

\def\ALPHA{\alpha}
\def\DELTA{\delta}

\def\zbz{0.2}
\def\MU{{
\begin{tikzpicture}[baseline=(ZUZU.base),scale=0.3]
  \coordinate(ZUZU) at (0,\zbz);
\draw[white] (0,0.5)--(1,1);
\draw[white] (0,0.5)--(1,0);
\draw[->] (0,0.5)--(1,1);
\node[anchor=east,inner sep=0pt](A) at (0,0.5){$\ALPHA$};
\end{tikzpicture}
}}

\def\MMa{{
\begin{tikzpicture}[baseline=(ZUZU.base),scale=0.3]
  \coordinate(ZUZU) at (0,\zbz);
\draw[white] (0,0.5)--(1,1);
\draw[white] (0,0.5)--(1,0);
\draw[->] (0,0.5)--(1,0.5);
\node[anchor=east,inner sep=0pt](A) at (0,0.5){$\ALPHA$};
\end{tikzpicture}
}}
\def\MD{{
\begin{tikzpicture}[baseline=(ZUZU.base),scale=0.3]
  \coordinate(ZUZU) at (0,\zbz);
\draw[white] (0,0.5)--(1,1);
\draw[white] (0,0.5)--(1,0);
\draw[->] (0,0.5)--(1,0);
\node[anchor=east,inner sep=0pt](A) at (0,0.5){$\ALPHA$};
\end{tikzpicture}
}}
\def\MDb{{
\begin{tikzpicture}[baseline=(ZUZU.base),scale=0.3]
  \coordinate(ZUZU) at (0,\zbz);
\draw[white] (0,0.5)--(1,1);
\draw[white] (0,0.5)--(1,0);
\draw[->,blue] (0,0.5)--(1,0);
\node[anchor=east,inner sep=0pt,blue](A) at (0,0.5){$\ALPHA$};
\end{tikzpicture}
}}

\def\UMb{{
\begin{tikzpicture}[baseline=(ZUZU.base),scale=0.3]
  \coordinate(ZUZU) at (0,\zbz);
\draw[white] (0,0.5)--(1,1);
\draw[white] (0,0.5)--(1,0);
\draw[->,blue] (0,1)--(1,0.5);
\node[anchor=west,inner sep=0pt,blue](A) at (1,0.7){$\DELTA$};
\end{tikzpicture}
}}

\def\UM{{
\begin{tikzpicture}[baseline=(ZUZU.base),scale=0.3]
  \coordinate(ZUZU) at (0,\zbz);
\draw[white] (0,0.5)--(1,1);
\draw[white] (0,0.5)--(1,0);
\draw[->] (0,1)--(1,0.5);
\node[anchor=west,inner sep=0pt](A) at (1,0.7){$\DELTA$};
\end{tikzpicture}
}}
\def\MMd{{
\begin{tikzpicture}[baseline=(ZUZU.base),scale=0.3]
  \coordinate(ZUZU) at (0,\zbz);
\draw[white] (0,0.5)--(1,1);
\draw[white] (0,0.5)--(1,0);
\draw[->] (0,0.5)--(1,0.5);
\node[anchor=west,inner sep=0pt](A) at (1,0.7){$\DELTA$};
\end{tikzpicture}
}}
\def\DM{{
\begin{tikzpicture}[baseline=(ZUZU.base),scale=0.3]
  \coordinate(ZUZU) at (0,\zbz);
\draw[white] (0,0.5)--(1,1);
\draw[white] (0,0.5)--(1,0);
\draw[->] (0,0)--(1,0.5);
\node[anchor=west,inner sep=0pt](A) at (1,0.7){$\DELTA$};
\end{tikzpicture}
}}

\def\XX{{
\begin{tikzpicture}[baseline=(ZUZU.base),scale=0.3]
  \coordinate(ZUZU) at (0,\zbz);
\def\ep{0.2}
\draw[fill=black] (\ep,\ep) rectangle (1-\ep,1-\ep);
\node[anchor=east,inner sep=0pt](A) at (0,0.5){$\ALPHA$};
\node[anchor=west,inner sep=0pt](A) at (1,0.7){$\DELTA$};
\end{tikzpicture}
}}

\def\MM{{
\begin{tikzpicture}[baseline=(ZUZU.base),scale=0.3]
  \coordinate(ZUZU) at (0,\zbz);
\draw[white] (0,0.5)--(1,1);
\draw[white] (0,0.5)--(1,0);
\draw[->] (0,0.5)--(1,0.5);
\node[anchor=east,inner sep=0pt](A) at (0,0.5){$\ALPHA$};
\node[anchor=west,inner sep=0pt](A) at (1,0.7){$\DELTA$};
\end{tikzpicture}
}}

\def\Yov{\overline{Y}}
\def\pip{\Pi}
\def\SKDZ{\SKD_\Z}
\def\SKDpZ{\SKDp_\Z}
\def\Prob(#1){\PF(#1)}
\def\wtij#1#2{\operatorname{wt}_{#2}(#1)}
\def\wt{\operatorname{wt}}

\def\cperm_#1{\pi_{#1}}
\def\lazy(#1,#2){\Pi_{#2}(#1)}
\def\irrel(#1,#2){\operatorname{Irrel}(#1,#2)}
\def\rel(#1,#2){\operatorname{Rel}(#1,#2)}

\let\subset\subseteq
\let\supset\supseteq

\def\bi{{\mathbf{i}}}
\def\bx{{\mathbf{x}}}
\def\by{{\mathbf{y}}}
\def\bu{{\mathbf{u}}}
\def\bv{{\mathbf{v}}}

\def\PF{{\mathbf{P}}}
\def\Zpi{\PF^{\Ppath,\Qpath}_\pi(\bx,\by)}

\def\H{\mathbb{H}}
\def\V{\mathbb{V}}
\def\h{h}
\def\v{v}

\def\ci{{\mathfrak{i}}}
\def\cj{{\mathfrak{j}}}

\def\ZHV{\PF^{\H,\V}(\bx,\by)}

\def\Sn{S_n}

\def\flip{\operatorname{180}^\circ}
\def\rev{\operatorname{rev}}
\def\ZHFV{\PF^{\flip(\H),\V}(\bx,\rev(\by))}

\def\midd{\Big|}

\def\HTop{\operatorname{Ht}}
\def\HTBop{{\mathbf{Ht}}}
\def\HT(#1){\HTop^{\Ppath,\Qpath}(#1;\bx,\by)}
\def\HTp(#1){\HTop^{\Ppath',\Qpath'}(#1;\bx',\by')}
\def\Hij{\HTop_\pi(\ci,\cj)}

\def\HTB(#1){\HTBop^{\Ppath,\Qpath}(#1;\bx,\by)}
\def\HTBp(#1){\HTBop^{\Ppath',\Qpath'}(#1;\bx',\by')}

\def\HTpi_#1(#2){\HTop^{\Ppath,\Qpath}_{#1}(#2)}

\def\SATop{\operatorname{SAT}}
\def\SATHV{\operatorname{SAT}(\H,\V)}

\def\ZP{\Z_{\geq1}}
\def\ZNN{\Z_{\geq0}}

\def\hpp(#1,#2){\left(#1+\frac12,#2+\frac12\right)}
\def\hmm(#1,#2){\left(#1-\frac12,#2-\frac12\right)}

\def\eqd{\stackrel{d}{=}}

\def\YBC{\kappa}
\def\YBCxx#1#2{\YBC^{#1}_{#2}(\bzz)}

\def\rv_#1{\rev_{[#1]}}

\def\lr#1{[l_{#1},r_{#1}]}

\def\du#1{[d_{#1},u_{#1}]}

\def\lrp#1{[l'_{#1},r'_{#1}]}

\def\Hecke{\Hcal_q(\Sn;\bz)}

\def\zz{z}
\def\bzz{\zbf}
\def\bz{\bzz}

\def\YB{Y}

\def\spx{\sp}

\def\SCSV{SC6V\xspace}

\def\ein#1{\Ppath_{#1}}
\def\fout#1{\Qpath_{#1}}

\def\eout#1{\fout{#1}}

\def\sp{{\mathfrak{p}}}

\def\Ppath{P}
\def\Qpath{Q}
\def\SKD{(\Ppath,\Qpath)}
\def\SKDp{(\Ppath',\Qpath')}
\def\skew{skew\xspace}

\def\cut{C}
\def\cutL{l}
\def\cutD{d}
\def\cutR{r}
\def\cutU{u}
\def\CUT{\Ccal}

\def\supp{\operatorname{supp}}
\def\suppH(#1){\supp_H(#1;\bx)}
\def\suppV(#1){\supp_V(#1;\by)}
\def\suppHp(#1){\supp_H(#1;\bx')}
\def\suppVp(#1){\supp_V(#1;\by')}

\def\HL{l}
\def\HR{r}
\def\VD{d}
\def\VU{u}

\def\Hpi{\H^{M,N}_\pi}
\def\Vpi{\V^{M,N}_\pi}

\def\IMop{\operatorname{IM}}
\def\IM{\IMop}
\def\IMP{\IMop^+}
\def\bd{{\mathbf{d}}}

\def\Zcal{{\mathcal{Z}}}
\def\KZx#1{\Zcal^{(#1)}}
\def\KZxxx#1#2#3{\Zcal^{(#1)}(#2,#3)}

\def\zu{{\mathfrak{Z}}}
\def\zux#1#2#3#4{\zu_{(#1,#2)\to(#3,#4)}}

\def\id{{\operatorname{id}}}

\def\cper{\sigma}

\def\ZBiber{\PF^{\bi,\bro,\cper}_\pi(\bzz)}

\def\ZBibx#1#2#3{\PF^{#1,#2,\cper}_{#3}(\bzz)}

\newcommand\SmallMatrix[1]{{%
    \small\arraycolsep=0.8\arraycolsep\ensuremath{\begin{pmatrix}#1\end{pmatrix}}}}

\def\YBarb#1#2{\YB^{#1,#2}}

\def\YBiber{\YB^{\bi,\bro,\cper}}

\def\bro{{\mathfrak{P}}}
\def\bro{{\pmb{\sp}}}

\def\concat#1#2{(#1,#2)}
\def\bj{\mathbf{j}}
\def\revv(#1){\rev(#1)}

\def\Daut{D}
\def\wPQ{w^{\Ppath, \Qpath}}
\def\bzzPQ{\bzz}
\def\zzPQ{\zz}
\def\wMN{w^{M, N}}

\def\HVpx#1#2{{#1}^{#2}}
\def\HVpart#1{\HVpx{#1}{\H}}

\def\HLx(#1){\mathcal{L}(#1)}
\def\HRx(#1){\mathcal{R}(#1)}

\def\HRS{\mathcal{R}(\H)}

\def\ZHY[#1,#2]{\PF[#1,#2]}

\def\sper{\sigma}

\def\PA{A}

\def\sz_#1{s_{#1;\bzz}}

\def\leqI{\preceq}
\def\lessI{\prec}

\def\xing{\operatorname{xing}}

\def\lbar{\bar l}
\def\rbar{\bar r}

\def\Yij{\YB^{\ci,\cj}_{M\times N}}
\def\Yijx{\overline{\YB}^{\ci,\cj}_{M\times N}}

\def\HVpx#1#2{{#1}^{#2,\V}}
\def\HVpart#1{\HVpx{#1}{\H}}

\def\Spar#1#2{S_{[#1,#2]}}
\def\Hpar#1#2{\Hcal_q(\Spar{#1}{#2};\bz)}
\def\kmin{k_{\operatorname{min}}}
\def\kmax{k_{\operatorname{max}}}

\def\al{l}
\def\ar{r}

\def\pih{\phi_H}
\def\piv{\phi_V}
\def\PHI{\Phi}

\def\Minf{M_{\infty}}
\def\Ninf{N_{\infty}}

\def\Vd{$\V$}
\def\Hd{$\H$}

\def\cl[#1]{[#1]}
\def\suppcut{\supp(\CUT)}
\def\suppcutp{\supp(\CUT')}

\def\cont#1#2{{\operatorname{cont}}_{#1,#2}}

\def\PD{\operatorname{PD}}

\def\PD{\operatorname{PD}}

\def\Gov(#1){G^{\operatorname{ov}}(#1)}

\def\NC_#1{\mathcal{P}_{#1}}
\def\PC{\NC_{\CUT}}
\def\PCP{\NC_{\CUT'}}

\def\trd{{\scalebox{0.5}{\textregistered}}}
\def\ler{\stackrel{{\trd}}{<}}

 \def\dist{\operatorname{dist}}

\def\bplr{\mathcal{B}(l,r)}

\def\ab#1{$\{a,b\}$}
\def\cutk_#1{\cut_{k_{#1}}}

\def\zero{0}
\def\one{1}

\def\br{{\mathbf{r}}}

\def\onebb{\mathbbm{1}_{m\times m}}
\def\Rvw{R^{\pi,w}}

\def\sb{s^{\bv}}
\def\bw{{\mathbf{w}}}
\def\vp_#1{\pi_{(#1)}}
\def\wp_#1{w_{(#1)}}
\def\Jo{J_{\bv}^\circ}
\def\Jm{J_{\bv}^-}
\def\Rich{\mathcal{R}^{\pi,w}}
\def\SL{\operatorname{SL}}
\def\Fbb{\mathbb{F}}

\def\Gr{\operatorname{Gr}}
\def\GRMN{\Gr(M,n)}
\def\Bound{\operatorname{Bound}(M,n)}
\def\Pio{\Pi^\circ}
\def\MPN{n}

\def\tom{t_{\omega_M}}
\def\Haff{\widetilde \H}
\def\Vaff{\widetilde \V}

\def\fop{\operatorname{rev}_{[1,N]}}
\def\xrasim{\xrightarrow{\sim}}
\def\xing{\operatorname{xing}}

\def\cntcol{white}

\def\cntcol{white}

\def\sclbx{1}

\def\gridop{0.5}
\def\nodescl{0.7}

\def\drawbox(#1,#2){
\draw[black!50,dashed] (#1-0.5,#2-0.5) rectangle (#1+0.5,#2+0.5);
}

\def\drawgrid#1#2#3#4{
\foreach\i in {#1,...,#3}{
\foreach \j in {#2,...,#4}{
\drawbox(\i,\j)
}
}
}

\def\cntcol{white}

\def\usefig#1#2{}

\def\FIGhtpq{
\def\PQscl{1.2}
\noindent\resizebox{0.8\textwidth}{!}{\begin{tabular}{ccc}\\\scalebox{0.67}{\begin{tikzpicture}[baseline=(ZUZU.base)]
\coordinate(ZUZU) at (1,1);\drawgrid{-1}{-1}{6}{7}\draw[line width=1pt,->,black!50] (-0.6,-1)--(-0.6,7.00);
\draw[line width=1pt,->,black!50] (-1,-0.6)--(6.00,-0.6);
\draw[line width=2pt,draw=black!50] (5.50,0.50)--(4.50,0.50)--(3.50,0.50)--(3.50,1.50)--(2.50,1.50)--(1.50,1.50)--(1.50,2.50)--(1.50,3.50)--(0.50,3.50)--(0.50,4.50)--(0.50,5.50)--(0.50,6.50);

\node[scale=1.10,black,anchor=north,inner sep=1pt](PP1) at (4.83,0.50) {$\ein{1}$};
\node[scale=1.10,black,anchor=north,inner sep=1pt](PP2) at (3.83,0.50) {$\ein{2}$};
\node[scale=1.10,black,anchor=east,inner sep=1pt](PP3) at (3.50,0.83) {$\ein{3}$};
\node[scale=1.10,black,anchor=north,inner sep=1pt](PP4) at (2.83,1.50) {$\ein{4}$};
\node[scale=1.10,black,anchor=north,inner sep=1pt](PP5) at (1.83,1.50) {$\ein{5}$};
\node[scale=1.10,black,anchor=east,inner sep=1pt](PP6) at (1.50,1.83) {$\ein{6}$};
\node[scale=1.10,black,anchor=east,inner sep=1pt](PP7) at (1.50,2.83) {$\ein{7}$};
\node[scale=1.10,black,anchor=north,inner sep=1pt](PP8) at (0.83,3.50) {$\ein{8}$};
\node[scale=1.10,black,anchor=east,inner sep=1pt](PP9) at (0.50,3.83) {$\ein{9}$};
\node[scale=1.10,black,anchor=east,inner sep=1pt](PP10) at (0.50,4.83) {$\ein{10}$};
\node[scale=1.10,black,anchor=east,inner sep=1pt](PP11) at (0.50,5.83) {$\ein{11}$};
\draw[line width=2pt,draw=black!50] (5.50,0.50)--(5.50,1.50)--(5.50,2.50)--(5.50,3.50)--(5.50,4.50)--(4.50,4.50)--(3.50,4.50)--(3.50,5.50)--(3.50,6.50)--(2.50,6.50)--(1.50,6.50)--(0.50,6.50);

\node[scale=1.10,black,anchor=west,inner sep=1pt](PP1) at (5.50,1.17) {$\eout{1}$};
\node[scale=1.10,black,anchor=west,inner sep=1pt](PP2) at (5.50,2.17) {$\eout{2}$};
\node[scale=1.10,black,anchor=west,inner sep=1pt](PP3) at (5.50,3.17) {$\eout{3}$};
\node[scale=1.10,black,anchor=west,inner sep=1pt](PP4) at (5.50,4.17) {$\eout{4}$};
\node[scale=1.10,black,anchor=south,inner sep=1pt](QQ5) at (5.17,4.50) {$\eout{5}$};
\node[scale=1.10,black,anchor=south,inner sep=1pt](QQ6) at (4.17,4.50) {$\eout{6}$};
\node[scale=1.10,black,anchor=west,inner sep=1pt](PP7) at (3.50,5.17) {$\eout{7}$};
\node[scale=1.10,black,anchor=west,inner sep=1pt](PP8) at (3.50,6.17) {$\eout{8}$};
\node[scale=1.10,black,anchor=south,inner sep=1pt](QQ9) at (3.17,6.50) {$\eout{9}$};
\node[scale=1.10,black,anchor=south,inner sep=1pt](QQ10) at (2.17,6.50) {$\eout{10}$};
\node[scale=1.10,black,anchor=south,inner sep=1pt](QQ11) at (1.17,6.50) {$\eout{11}$};
\node[scale=2.20,black](A) at (1,-1) {$x_{1}$};
\node[scale=2.20,black](A) at (2,-1) {$x_{2}$};
\node[scale=2.20,black](A) at (3,-1) {$x_{3}$};
\node[scale=2.20,black](A) at (4,-1) {$x_{4}$};
\node[scale=2.20,black](A) at (5,-1) {$x_{5}$};
\node[scale=2.20,black](A) at (-1,1) {$y_{1}$};
\node[scale=2.20,black](A) at (-1,2) {$y_{2}$};
\node[scale=2.20,black](A) at (-1,3) {$y_{3}$};
\node[scale=2.20,black](A) at (-1,4) {$y_{4}$};
\node[scale=2.20,black](A) at (-1,5) {$y_{5}$};
\node[scale=2.20,black](A) at (-1,6) {$y_{6}$};
\end{tikzpicture}}
&\scalebox{0.67}{\begin{tikzpicture}[baseline=(ZUZU.base)]
\coordinate(ZUZU) at (1,1);\drawgrid{-1}{-1}{6}{7}\draw[line width=1pt,->,black!50] (-0.6,-1)--(-0.6,7.00);
\draw[line width=1pt,->,black!50] (-1,-0.6)--(6.00,-0.6);
\draw[line width=2pt,draw=black!50] (5.50,0.50)--(4.50,0.50)--(3.50,0.50)--(3.50,1.50)--(2.50,1.50)--(1.50,1.50)--(1.50,2.50)--(1.50,3.50)--(0.50,3.50)--(0.50,4.50)--(0.50,5.50)--(0.50,6.50);

\node[scale=\PQscl,blue,anchor=north,inner sep=1pt](PP4) at (3.00,1.50) {$\ein{\ci}$};
\node[scale=\PQscl,blue,anchor=north,inner sep=1pt](PP5) at (2.00,1.50) {$\ein{\ci+1}$};
\draw[line width=2pt,draw=black!50] (5.50,0.50)--(5.50,1.50)--(5.50,2.50)--(5.50,3.50)--(5.50,4.50)--(4.50,4.50)--(3.50,4.50)--(3.50,5.50)--(3.50,6.50)--(2.50,6.50)--(1.50,6.50)--(0.50,6.50);

\node[scale=\PQscl,blue,anchor=south,inner sep=1pt](QQ5) at (5.00,4.50) {$\eout{\cj}$};
\node[scale=\PQscl,blue,anchor=south,inner sep=1pt](QQ6) at (4.00,4.50) {$\eout{\cj+1}$};
\node[scale=2.20,blue](A) at (3,-1) {$x_{l}$};
\node[scale=2.20,blue](A) at (4,-1) {$x_{r}$};
\node[scale=2.20,blue](A) at (-1,2) {$y_{d}$};
\node[scale=2.20,blue](A) at (-1,4) {$y_{u}$};
\draw[line width=2pt,blue](2.50,1.50)--(4.50,4.50);
\draw[line width=1.5pt,blue,dashed](2.60,1.60)rectangle(4.40,4.40);
\node[scale=2] (A) at (4.00,3.00) {$\textcolor{blue}{\cut}$};
\end{tikzpicture}}
&\scalebox{0.67}{\begin{tikzpicture}[baseline=(ZUZU.base)]
\coordinate(ZUZU) at (1,1);\drawgrid{0}{0}{6}{7}\draw[line width=2pt,draw=black!50] (5.50,0.50)--(4.50,0.50)--(3.50,0.50)--(3.50,1.50)--(2.50,1.50)--(1.50,1.50)--(1.50,2.50)--(1.50,3.50)--(0.50,3.50)--(0.50,4.50)--(0.50,5.50)--(0.50,6.50);

\draw[line width=2pt,draw=black!50] (5.50,0.50)--(5.50,1.50)--(5.50,2.50)--(5.50,3.50)--(5.50,4.50)--(4.50,4.50)--(3.50,4.50)--(3.50,5.50)--(3.50,6.50)--(2.50,6.50)--(1.50,6.50)--(0.50,6.50);

\draw[line width=1.5pt,blue,dashed](2.60,1.60)rectangle(4.40,4.40);
\draw[black!70,line width=1.5pt] (0.50,6.00)--(1.50,6.00);\draw[black!70,line width=1.5pt] (1.00,5.50)--(1.00,6.50);\draw[black!70,line width=1.5pt] (2.50,6.00)--(3.50,6.00);\draw[black!70,line width=1.5pt] (3.00,5.50)--(3.00,6.50);\draw[black!70,line width=1.5pt] (1.50,5.00)--(2.50,5.00);\draw[black!70,line width=1.5pt] (2.00,4.50)--(2.00,5.50);\draw[black!70,line width=1.5pt] (2.50,5.00)--(3.50,5.00);\draw[black!70,line width=1.5pt] (3.00,4.50)--(3.00,5.50);\draw[black!70,line width=1.5pt] (3.50,4.00)--(4.50,4.00);\draw[black!70,line width=1.5pt] (4.00,3.50)--(4.00,4.50);\draw[black!70,line width=1.5pt] (2.50,3.00)--(3.50,3.00);\draw[black!70,line width=1.5pt] (3.00,2.50)--(3.00,3.50);\draw[black!70,line width=1.5pt] (3.50,3.00)--(4.50,3.00);\draw[black!70,line width=1.5pt] (4.00,2.50)--(4.00,3.50);\draw[black!70,line width=1.5pt] (1.50,2.00)--(2.50,2.00);\draw[black!70,line width=1.5pt] (2.00,1.50)--(2.00,2.50);\draw[black!70,line width=1.5pt] (2.50,2.00)--(3.50,2.00);\draw[black!70,line width=1.5pt] (3.00,1.50)--(3.00,2.50);\draw[black!70,line width=1.5pt] (3.50,2.00)--(4.50,2.00);\draw[black!70,line width=1.5pt] (4.00,1.50)--(4.00,2.50);\draw[black!70,line width=1.5pt] (4.50,1.00)--(5.50,1.00);\draw[black!70,line width=1.5pt] (5.00,0.50)--(5.00,1.50);\draw[black!70,line width=1.5pt, rounded corners=12] (1.50,6.00)--(2.00,6.00)--(2.00,6.50);\draw[black!70,line width=1.5pt, rounded corners=12] (2.00,5.50)--(2.00,6.00)--(2.50,6.00);\draw[black!70,line width=1.5pt, rounded corners=12] (0.50,5.00)--(1.00,5.00)--(1.00,5.50);\draw[black!70,line width=1.5pt, rounded corners=12] (1.00,4.50)--(1.00,5.00)--(1.50,5.00);\draw[black!70,line width=1.5pt, rounded corners=12] (0.50,4.00)--(1.00,4.00)--(1.00,4.50);\draw[black!70,line width=1.5pt, rounded corners=12] (1.00,3.50)--(1.00,4.00)--(1.50,4.00);\draw[black!70,line width=1.5pt, rounded corners=12] (1.50,4.00)--(2.00,4.00)--(2.00,4.50);\draw[black!70,line width=1.5pt, rounded corners=12] (2.00,3.50)--(2.00,4.00)--(2.50,4.00);\draw[black!70,line width=1.5pt, rounded corners=12] (2.50,4.00)--(3.00,4.00)--(3.00,4.50);\draw[black!70,line width=1.5pt, rounded corners=12] (3.00,3.50)--(3.00,4.00)--(3.50,4.00);\draw[black!70,line width=1.5pt, rounded corners=12] (4.50,4.00)--(5.00,4.00)--(5.00,4.50);\draw[black!70,line width=1.5pt, rounded corners=12] (5.00,3.50)--(5.00,4.00)--(5.50,4.00);\draw[black!70,line width=1.5pt, rounded corners=12] (1.50,3.00)--(2.00,3.00)--(2.00,3.50);\draw[black!70,line width=1.5pt, rounded corners=12] (2.00,2.50)--(2.00,3.00)--(2.50,3.00);\draw[black!70,line width=1.5pt, rounded corners=12] (4.50,3.00)--(5.00,3.00)--(5.00,3.50);\draw[black!70,line width=1.5pt, rounded corners=12] (5.00,2.50)--(5.00,3.00)--(5.50,3.00);\draw[black!70,line width=1.5pt, rounded corners=12] (4.50,2.00)--(5.00,2.00)--(5.00,2.50);\draw[black!70,line width=1.5pt, rounded corners=12] (5.00,1.50)--(5.00,2.00)--(5.50,2.00);\draw[black!70,line width=1.5pt, rounded corners=12] (3.50,1.00)--(4.00,1.00)--(4.00,1.50);\draw[black!70,line width=1.5pt, rounded corners=12] (4.00,0.50)--(4.00,1.00)--(4.50,1.00);\draw[red,line width=3pt, rounded corners=12] (1.50,2.00)--(5.00,2.00)--(5.00,3.00)--(5.50,3.00);\draw[red,line width=3pt, rounded corners=12] (2.00,1.50)--(2.00,3.00)--(5.00,3.00)--(5.00,4.00)--(5.50,4.00);\node[scale=2](A) at (3,-1.2) {$\HT(\textcolor{blue}{\cut})=\textcolor{red}{2}$};
\end{tikzpicture}}
\end{tabular}}}
\usefig\FIGhtpq{FIGhtpq}
\def\FIGflip{\noindent\resizebox{\textwidth}{!}{\scalebox{1.00}{\begin{tikzpicture}[baseline=(ZUZU.base)]
\coordinate(ZUZU) at (1,1);\node(A) at (0,0){\begin{tabular}{ccccc}
\scalebox{0.50}{\begin{tikzpicture}[baseline=(ZUZU.base)]
\coordinate(ZUZU) at (1,1);\drawgrid{0}{0}{5}{8}\draw[line width=2pt,draw=black!50] (4.50,0.50)--(3.50,0.50)--(2.50,0.50)--(1.50,0.50)--(0.50,0.50)--(0.50,1.50)--(0.50,2.50)--(0.50,3.50)--(0.50,4.50)--(0.50,5.50)--(0.50,6.50)--(0.50,7.50);

\node[scale=1.30,black!50,anchor=north,inner sep=1pt](PP1) at (4.00,0.50) {$\ein{1}$};
\node[scale=1.30,black!50,anchor=north,inner sep=1pt](PP2) at (3.00,0.50) {$\ein{2}$};
\node[scale=1.30,black!50,anchor=north,inner sep=1pt](PP3) at (2.00,0.50) {$\ein{3}$};
\node[scale=1.30,blue,anchor=north,inner sep=1pt](PP4) at (1.00,0.50) {$\ein{4}$};
\node[scale=1.30,black!50,anchor=east,inner sep=1pt](PP5) at (0.50,1.00) {$\ein{5}$};
\node[scale=1.30,red,anchor=east,inner sep=1pt](PP6) at (0.50,2.00) {$\ein{6}$};
\node[scale=1.30,red,anchor=east,inner sep=1pt](PP7) at (0.50,3.00) {$\ein{7}$};
\node[scale=1.30,black!50,anchor=east,inner sep=1pt](PP8) at (0.50,4.00) {$\ein{8}$};
\node[scale=1.30,red,anchor=east,inner sep=1pt](PP9) at (0.50,5.00) {$\ein{9}$};
\node[scale=1.30,red,anchor=east,inner sep=1pt](PP10) at (0.50,6.00) {$\ein{10}$};
\node[scale=1.30,black!50,anchor=east,inner sep=1pt](PP11) at (0.50,7.00) {$\ein{11}$};
\draw[line width=2pt,draw=black!50] (4.50,0.50)--(4.50,1.50)--(4.50,2.50)--(4.50,3.50)--(4.50,4.50)--(4.50,5.50)--(4.50,6.50)--(4.50,7.50)--(3.50,7.50)--(2.50,7.50)--(1.50,7.50)--(0.50,7.50);

\node[scale=1.30,black!50,anchor=west,inner sep=1pt](PP1) at (4.50,1.00) {$\eout{1}$};
\node[scale=1.30,black!50,anchor=west,inner sep=1pt](PP2) at (4.50,2.00) {$\eout{2}$};
\node[scale=1.30,black!50,anchor=west,inner sep=1pt](PP3) at (4.50,3.00) {$\eout{3}$};
\node[scale=1.30,red,anchor=west,inner sep=1pt](PP4) at (4.50,4.00) {$\eout{4}$};
\node[scale=1.30,red,anchor=west,inner sep=1pt](PP5) at (4.50,5.00) {$\eout{5}$};
\node[scale=1.30,red,anchor=west,inner sep=1pt](PP6) at (4.50,6.00) {$\eout{6}$};
\node[scale=1.30,red,anchor=west,inner sep=1pt](PP7) at (4.50,7.00) {$\eout{7}$};
\node[scale=1.30,black!50,anchor=south,inner sep=1pt](QQ8) at (4.00,7.50) {$\eout{8}$};
\node[scale=1.30,black!50,anchor=south,inner sep=1pt](QQ9) at (3.00,7.50) {$\eout{9}$};
\node[scale=1.30,blue,anchor=south,inner sep=1pt](QQ10) at (2.00,7.50) {$\eout{10}$};
\node[scale=1.30,black!50,anchor=south,inner sep=1pt](QQ11) at (1.00,7.50) {$\eout{11}$};
\draw[black!50,line width=1.5pt] (0.50,7.00)--(1.50,7.00);\draw[black!50,line width=1.5pt] (1.00,6.50)--(1.00,7.50);\draw[black!50,line width=1.5pt] (1.50,7.00)--(2.50,7.00);\draw[black!50,line width=1.5pt] (2.00,6.50)--(2.00,7.50);\draw[black!50,line width=1.5pt] (3.50,7.00)--(4.50,7.00);\draw[black!50,line width=1.5pt] (4.00,6.50)--(4.00,7.50);\draw[black!50,line width=1.5pt] (0.50,6.00)--(1.50,6.00);\draw[black!50,line width=1.5pt] (1.00,5.50)--(1.00,6.50);\draw[black!50,line width=1.5pt] (1.50,6.00)--(2.50,6.00);\draw[black!50,line width=1.5pt] (2.00,5.50)--(2.00,6.50);\draw[black!50,line width=1.5pt] (3.50,6.00)--(4.50,6.00);\draw[black!50,line width=1.5pt] (4.00,5.50)--(4.00,6.50);\draw[black!50,line width=1.5pt] (0.50,5.00)--(1.50,5.00);\draw[black!50,line width=1.5pt] (1.00,4.50)--(1.00,5.50);\draw[black!50,line width=1.5pt] (1.50,5.00)--(2.50,5.00);\draw[black!50,line width=1.5pt] (2.00,4.50)--(2.00,5.50);\draw[black!50,line width=1.5pt] (2.50,5.00)--(3.50,5.00);\draw[black!50,line width=1.5pt] (3.00,4.50)--(3.00,5.50);\draw[black!50,line width=1.5pt] (3.50,5.00)--(4.50,5.00);\draw[black!50,line width=1.5pt] (4.00,4.50)--(4.00,5.50);\draw[black!50,line width=1.5pt] (1.50,4.00)--(2.50,4.00);\draw[black!50,line width=1.5pt] (2.00,3.50)--(2.00,4.50);\draw[black!50,line width=1.5pt] (2.50,4.00)--(3.50,4.00);\draw[black!50,line width=1.5pt] (3.00,3.50)--(3.00,4.50);\draw[black!50,line width=1.5pt] (3.50,4.00)--(4.50,4.00);\draw[black!50,line width=1.5pt] (4.00,3.50)--(4.00,4.50);\draw[black!50,line width=1.5pt] (3.50,3.00)--(4.50,3.00);\draw[black!50,line width=1.5pt] (4.00,2.50)--(4.00,3.50);\draw[black!50,line width=1.5pt] (2.50,3.00)--(3.50,3.00);\draw[black!50,line width=1.5pt] (3.00,2.50)--(3.00,3.50);\draw[black!50,line width=1.5pt] (1.50,2.00)--(2.50,2.00);\draw[black!50,line width=1.5pt] (2.00,1.50)--(2.00,2.50);\draw[black!50,line width=1.5pt] (0.50,2.00)--(1.50,2.00);\draw[black!50,line width=1.5pt] (1.00,1.50)--(1.00,2.50);\draw[black!50,line width=1.5pt] (0.50,1.00)--(1.50,1.00);\draw[black!50,line width=1.5pt] (1.00,0.50)--(1.00,1.50);\draw[black!50,line width=1.5pt] (1.50,1.00)--(2.50,1.00);\draw[black!50,line width=1.5pt] (2.00,0.50)--(2.00,1.50);\draw[black!50,line width=1.5pt] (3.50,1.00)--(4.50,1.00);\draw[black!50,line width=1.5pt] (4.00,0.50)--(4.00,1.50);\draw[black!50,line width=1.5pt, rounded corners=12] (0.50,3.00)--(1.00,3.00)--(1.00,3.50);\draw[black!50,line width=1.5pt, rounded corners=12] (1.00,2.50)--(1.00,3.00)--(1.50,3.00);\draw[black!50,line width=1.5pt, rounded corners=12] (0.50,4.00)--(1.00,4.00)--(1.00,4.50);\draw[black!50,line width=1.5pt, rounded corners=12] (1.00,3.50)--(1.00,4.00)--(1.50,4.00);\draw[black!50,line width=1.5pt, rounded corners=12] (1.50,3.00)--(2.00,3.00)--(2.00,3.50);\draw[black!50,line width=1.5pt, rounded corners=12] (2.00,2.50)--(2.00,3.00)--(2.50,3.00);\draw[black!50,line width=1.5pt, rounded corners=12] (2.50,1.00)--(3.00,1.00)--(3.00,1.50);\draw[black!50,line width=1.5pt, rounded corners=12] (3.00,0.50)--(3.00,1.00)--(3.50,1.00);\draw[black!50,line width=1.5pt, rounded corners=12] (2.50,2.00)--(3.00,2.00)--(3.00,2.50);\draw[black!50,line width=1.5pt, rounded corners=12] (3.00,1.50)--(3.00,2.00)--(3.50,2.00);\draw[black!50,line width=1.5pt, rounded corners=12] (2.50,6.00)--(3.00,6.00)--(3.00,6.50);\draw[black!50,line width=1.5pt, rounded corners=12] (3.00,5.50)--(3.00,6.00)--(3.50,6.00);\draw[black!50,line width=1.5pt, rounded corners=12] (2.50,7.00)--(3.00,7.00)--(3.00,7.50);\draw[black!50,line width=1.5pt, rounded corners=12] (3.00,6.50)--(3.00,7.00)--(3.50,7.00);\draw[black!50,line width=1.5pt, rounded corners=12] (3.50,2.00)--(4.00,2.00)--(4.00,2.50);\draw[black!50,line width=1.5pt, rounded corners=12] (4.00,1.50)--(4.00,2.00)--(4.50,2.00);\draw[blue,line width=3pt, rounded corners=12] (1.00,0.50)--(1.00,3.00)--(2.00,3.00)--(2.00,7.50);\draw[red,line width=3pt, rounded corners=12] (0.50,2.00)--(3.00,2.00)--(3.00,6.00)--(4.50,6.00);\draw[red,line width=3pt, rounded corners=12] (0.50,3.00)--(1.00,3.00)--(1.00,4.00)--(4.50,4.00);\draw[red,line width=3pt, rounded corners=12] (0.50,5.00)--(4.50,5.00);\draw[red,line width=3pt, rounded corners=12] (0.50,6.00)--(3.00,6.00)--(3.00,7.00)--(4.50,7.00);\end{tikzpicture}}
&\scalebox{0.50}{\begin{tikzpicture}[baseline=(ZUZU.base)]
\coordinate(ZUZU) at (1,1);\drawgrid{-1}{-1}{5}{8}\draw[line width=1pt,->,black!50] (-0.6,-1)--(-0.6,8.00);
\draw[line width=1pt,->,black!50] (-1,-0.6)--(5.00,-0.6);
\draw[line width=2pt,draw=black!50] (4.50,0.50)--(3.50,0.50)--(2.50,0.50)--(1.50,0.50)--(0.50,0.50)--(0.50,1.50)--(0.50,2.50)--(0.50,3.50)--(0.50,4.50)--(0.50,5.50)--(0.50,6.50)--(0.50,7.50);

\node[scale=1.30,black!50,anchor=north,inner sep=1pt](PP1) at (4.00,0.50) {$\ein{1}$};
\node[scale=1.30,black!50,anchor=north,inner sep=1pt](PP2) at (3.00,0.50) {$\ein{2}$};
\node[scale=1.30,black!50,anchor=north,inner sep=1pt](PP3) at (2.00,0.50) {$\ein{3}$};
\node[scale=1.30,blue,anchor=north,inner sep=1pt](PP4) at (1.00,0.50) {$\ein{4}$};
\node[scale=1.30,black!50,anchor=east,inner sep=1pt](PP5) at (0.50,1.00) {$\ein{5}$};
\node[scale=1.30,red,anchor=east,inner sep=1pt](PP6) at (0.50,2.00) {$\ein{6}$};
\node[scale=1.30,red,anchor=east,inner sep=1pt](PP7) at (0.50,3.00) {$\ein{7}$};
\node[scale=1.30,black!50,anchor=east,inner sep=1pt](PP8) at (0.50,4.00) {$\ein{8}$};
\node[scale=1.30,red,anchor=east,inner sep=1pt](PP9) at (0.50,5.00) {$\ein{9}$};
\node[scale=1.30,red,anchor=east,inner sep=1pt](PP10) at (0.50,6.00) {$\ein{10}$};
\node[scale=1.30,black!50,anchor=east,inner sep=1pt](PP11) at (0.50,7.00) {$\ein{11}$};
\draw[line width=2pt,draw=black!50] (4.50,0.50)--(4.50,1.50)--(4.50,2.50)--(4.50,3.50)--(4.50,4.50)--(4.50,5.50)--(4.50,6.50)--(4.50,7.50)--(3.50,7.50)--(2.50,7.50)--(1.50,7.50)--(0.50,7.50);

\node[scale=1.30,black!50,anchor=west,inner sep=1pt](PP1) at (4.50,1.00) {$\eout{1}$};
\node[scale=1.30,black!50,anchor=west,inner sep=1pt](PP2) at (4.50,2.00) {$\eout{2}$};
\node[scale=1.30,black!50,anchor=west,inner sep=1pt](PP3) at (4.50,3.00) {$\eout{3}$};
\node[scale=1.30,red,anchor=west,inner sep=1pt](PP4) at (4.50,4.00) {$\eout{4}$};
\node[scale=1.30,red,anchor=west,inner sep=1pt](PP5) at (4.50,5.00) {$\eout{5}$};
\node[scale=1.30,red,anchor=west,inner sep=1pt](PP6) at (4.50,6.00) {$\eout{6}$};
\node[scale=1.30,red,anchor=west,inner sep=1pt](PP7) at (4.50,7.00) {$\eout{7}$};
\node[scale=1.30,black!50,anchor=south,inner sep=1pt](QQ8) at (4.00,7.50) {$\eout{8}$};
\node[scale=1.30,black!50,anchor=south,inner sep=1pt](QQ9) at (3.00,7.50) {$\eout{9}$};
\node[scale=1.30,blue,anchor=south,inner sep=1pt](QQ10) at (2.00,7.50) {$\eout{10}$};
\node[scale=1.30,black!50,anchor=south,inner sep=1pt](QQ11) at (1.00,7.50) {$\eout{11}$};
\node[scale=1.80,black](A) at (1,-1) {$x_{1}$};
\node[scale=1.80,black](A) at (2,-1) {$x_{2}$};
\node[scale=1.80,black](A) at (3,-1) {$x_{3}$};
\node[scale=1.80,black](A) at (4,-1) {$x_{4}$};
\node[scale=1.80,black](A) at (-1,1) {$y_{1}$};
\node[scale=1.80,black](A) at (-1,2) {$y_{2}$};
\node[scale=1.80,black](A) at (-1,3) {$y_{3}$};
\node[scale=1.80,black](A) at (-1,4) {$y_{4}$};
\node[scale=1.80,black](A) at (-1,5) {$y_{5}$};
\node[scale=1.80,black](A) at (-1,6) {$y_{6}$};
\node[scale=1.80,black](A) at (-1,7) {$y_{7}$};
\draw[blue,line width=3pt] (1.00,0.50) to[in=-90,out=90] (2.00,7.50);
\draw[red,line width=3pt] (0.50,2.00) to[in=180,out=0] (4.50,6.00);
\draw[red,line width=3pt] (0.50,3.00) to[in=180,out=0] (4.50,4.00);
\draw[red,line width=3pt] (0.50,5.00) to[in=180,out=0] (4.50,5.00);
\draw[red,line width=3pt] (0.50,6.00) to[in=180,out=0] (4.50,7.00);
\end{tikzpicture}}
&\ \ &\scalebox{0.50}{\begin{tikzpicture}[baseline=(ZUZU.base)]
\coordinate(ZUZU) at (1,1);\drawgrid{-1}{-1}{5}{8}\draw[line width=1pt,->,black!50] (-0.6,-1)--(-0.6,8.00);
\draw[line width=1pt,->,black!50] (-1,-0.6)--(5.00,-0.6);
\draw[line width=2pt,draw=black!50] (4.50,0.50)--(3.50,0.50)--(2.50,0.50)--(1.50,0.50)--(0.50,0.50)--(0.50,1.50)--(0.50,2.50)--(0.50,3.50)--(0.50,4.50)--(0.50,5.50)--(0.50,6.50)--(0.50,7.50);

\node[scale=1.30,black!50,anchor=north,inner sep=1pt](PP1) at (4.00,0.50) {$\ein{1}$};
\node[scale=1.30,black!50,anchor=north,inner sep=1pt](PP2) at (3.00,0.50) {$\ein{2}$};
\node[scale=1.30,black!50,anchor=north,inner sep=1pt](PP3) at (2.00,0.50) {$\ein{3}$};
\node[scale=1.30,blue,anchor=north,inner sep=1pt](PP4) at (1.00,0.50) {$\ein{4}$};
\node[scale=1.30,red,anchor=east,inner sep=1pt](PP5) at (0.50,1.00) {$\ein{5}$};
\node[scale=1.30,red,anchor=east,inner sep=1pt](PP6) at (0.50,2.00) {$\ein{6}$};
\node[scale=1.30,red,anchor=east,inner sep=1pt](PP7) at (0.50,3.00) {$\ein{7}$};
\node[scale=1.30,red,anchor=east,inner sep=1pt](PP8) at (0.50,4.00) {$\ein{8}$};
\node[scale=1.30,black!50,anchor=east,inner sep=1pt](PP9) at (0.50,5.00) {$\ein{9}$};
\node[scale=1.30,black!50,anchor=east,inner sep=1pt](PP10) at (0.50,6.00) {$\ein{10}$};
\node[scale=1.30,black!50,anchor=east,inner sep=1pt](PP11) at (0.50,7.00) {$\ein{11}$};
\draw[line width=2pt,draw=black!50] (4.50,0.50)--(4.50,1.50)--(4.50,2.50)--(4.50,3.50)--(4.50,4.50)--(4.50,5.50)--(4.50,6.50)--(4.50,7.50)--(3.50,7.50)--(2.50,7.50)--(1.50,7.50)--(0.50,7.50);

\node[scale=1.30,black!50,anchor=west,inner sep=1pt](PP1) at (4.50,1.00) {$\eout{1}$};
\node[scale=1.30,red,anchor=west,inner sep=1pt](PP2) at (4.50,2.00) {$\eout{2}$};
\node[scale=1.30,red,anchor=west,inner sep=1pt](PP3) at (4.50,3.00) {$\eout{3}$};
\node[scale=1.30,black!50,anchor=west,inner sep=1pt](PP4) at (4.50,4.00) {$\eout{4}$};
\node[scale=1.30,red,anchor=west,inner sep=1pt](PP5) at (4.50,5.00) {$\eout{5}$};
\node[scale=1.30,red,anchor=west,inner sep=1pt](PP6) at (4.50,6.00) {$\eout{6}$};
\node[scale=1.30,black!50,anchor=west,inner sep=1pt](PP7) at (4.50,7.00) {$\eout{7}$};
\node[scale=1.30,black!50,anchor=south,inner sep=1pt](QQ8) at (4.00,7.50) {$\eout{8}$};
\node[scale=1.30,black!50,anchor=south,inner sep=1pt](QQ9) at (3.00,7.50) {$\eout{9}$};
\node[scale=1.30,blue,anchor=south,inner sep=1pt](QQ10) at (2.00,7.50) {$\eout{10}$};
\node[scale=1.30,black!50,anchor=south,inner sep=1pt](QQ11) at (1.00,7.50) {$\eout{11}$};
\node[scale=1.80,black](A) at (1,-1) {$x_{1}$};
\node[scale=1.80,black](A) at (2,-1) {$x_{2}$};
\node[scale=1.80,black](A) at (3,-1) {$x_{3}$};
\node[scale=1.80,black](A) at (4,-1) {$x_{4}$};
\node[scale=1.80,green!70!black](A) at (-1,1) {$y_{7}$};
\node[scale=1.80,green!70!black](A) at (-1,2) {$y_{6}$};
\node[scale=1.80,green!70!black](A) at (-1,3) {$y_{5}$};
\node[scale=1.80,green!70!black](A) at (-1,4) {$y_{4}$};
\node[scale=1.80,green!70!black](A) at (-1,5) {$y_{3}$};
\node[scale=1.80,green!70!black](A) at (-1,6) {$y_{2}$};
\node[scale=1.80,green!70!black](A) at (-1,7) {$y_{1}$};
\draw[blue,line width=3pt] (1.00,0.50) to[in=-90,out=90] (2.00,7.50);
\draw[red,line width=3pt] (0.50,1.00) to[in=180,out=0] (4.50,2.00);
\draw[red,line width=3pt] (0.50,2.00) to[in=180,out=0] (4.50,6.00);
\draw[red,line width=3pt] (0.50,3.00) to[in=180,out=0] (4.50,3.00);
\draw[red,line width=3pt] (0.50,4.00) to[in=180,out=0] (4.50,5.00);
\end{tikzpicture}}
&\scalebox{0.50}{\begin{tikzpicture}[baseline=(ZUZU.base)]
\coordinate(ZUZU) at (1,1);\drawgrid{0}{0}{5}{8}\draw[line width=2pt,draw=black!50] (4.50,0.50)--(3.50,0.50)--(2.50,0.50)--(1.50,0.50)--(0.50,0.50)--(0.50,1.50)--(0.50,2.50)--(0.50,3.50)--(0.50,4.50)--(0.50,5.50)--(0.50,6.50)--(0.50,7.50);

\node[scale=1.30,black!50,anchor=north,inner sep=1pt](PP1) at (4.00,0.50) {$\ein{1}$};
\node[scale=1.30,black!50,anchor=north,inner sep=1pt](PP2) at (3.00,0.50) {$\ein{2}$};
\node[scale=1.30,black!50,anchor=north,inner sep=1pt](PP3) at (2.00,0.50) {$\ein{3}$};
\node[scale=1.30,blue,anchor=north,inner sep=1pt](PP4) at (1.00,0.50) {$\ein{4}$};
\node[scale=1.30,red,anchor=east,inner sep=1pt](PP5) at (0.50,1.00) {$\ein{5}$};
\node[scale=1.30,red,anchor=east,inner sep=1pt](PP6) at (0.50,2.00) {$\ein{6}$};
\node[scale=1.30,red,anchor=east,inner sep=1pt](PP7) at (0.50,3.00) {$\ein{7}$};
\node[scale=1.30,red,anchor=east,inner sep=1pt](PP8) at (0.50,4.00) {$\ein{8}$};
\node[scale=1.30,black!50,anchor=east,inner sep=1pt](PP9) at (0.50,5.00) {$\ein{9}$};
\node[scale=1.30,black!50,anchor=east,inner sep=1pt](PP10) at (0.50,6.00) {$\ein{10}$};
\node[scale=1.30,black!50,anchor=east,inner sep=1pt](PP11) at (0.50,7.00) {$\ein{11}$};
\draw[line width=2pt,draw=black!50] (4.50,0.50)--(4.50,1.50)--(4.50,2.50)--(4.50,3.50)--(4.50,4.50)--(4.50,5.50)--(4.50,6.50)--(4.50,7.50)--(3.50,7.50)--(2.50,7.50)--(1.50,7.50)--(0.50,7.50);

\node[scale=1.30,black!50,anchor=west,inner sep=1pt](PP1) at (4.50,1.00) {$\eout{1}$};
\node[scale=1.30,red,anchor=west,inner sep=1pt](PP2) at (4.50,2.00) {$\eout{2}$};
\node[scale=1.30,red,anchor=west,inner sep=1pt](PP3) at (4.50,3.00) {$\eout{3}$};
\node[scale=1.30,black!50,anchor=west,inner sep=1pt](PP4) at (4.50,4.00) {$\eout{4}$};
\node[scale=1.30,red,anchor=west,inner sep=1pt](PP5) at (4.50,5.00) {$\eout{5}$};
\node[scale=1.30,red,anchor=west,inner sep=1pt](PP6) at (4.50,6.00) {$\eout{6}$};
\node[scale=1.30,black!50,anchor=west,inner sep=1pt](PP7) at (4.50,7.00) {$\eout{7}$};
\node[scale=1.30,black!50,anchor=south,inner sep=1pt](QQ8) at (4.00,7.50) {$\eout{8}$};
\node[scale=1.30,black!50,anchor=south,inner sep=1pt](QQ9) at (3.00,7.50) {$\eout{9}$};
\node[scale=1.30,blue,anchor=south,inner sep=1pt](QQ10) at (2.00,7.50) {$\eout{10}$};
\node[scale=1.30,black!50,anchor=south,inner sep=1pt](QQ11) at (1.00,7.50) {$\eout{11}$};
\draw[black!50,line width=1.5pt] (1.50,7.00)--(2.50,7.00);\draw[black!50,line width=1.5pt] (2.00,6.50)--(2.00,7.50);\draw[black!50,line width=1.5pt] (2.50,7.00)--(3.50,7.00);\draw[black!50,line width=1.5pt] (3.00,6.50)--(3.00,7.50);\draw[black!50,line width=1.5pt] (3.50,6.00)--(4.50,6.00);\draw[black!50,line width=1.5pt] (4.00,5.50)--(4.00,6.50);\draw[black!50,line width=1.5pt] (0.50,5.00)--(1.50,5.00);\draw[black!50,line width=1.5pt] (1.00,4.50)--(1.00,5.50);\draw[black!50,line width=1.5pt] (2.50,5.00)--(3.50,5.00);\draw[black!50,line width=1.5pt] (3.00,4.50)--(3.00,5.50);\draw[black!50,line width=1.5pt] (3.50,5.00)--(4.50,5.00);\draw[black!50,line width=1.5pt] (4.00,4.50)--(4.00,5.50);\draw[black!50,line width=1.5pt] (0.50,4.00)--(1.50,4.00);\draw[black!50,line width=1.5pt] (1.00,3.50)--(1.00,4.50);\draw[black!50,line width=1.5pt] (2.50,4.00)--(3.50,4.00);\draw[black!50,line width=1.5pt] (3.00,3.50)--(3.00,4.50);\draw[black!50,line width=1.5pt] (3.50,4.00)--(4.50,4.00);\draw[black!50,line width=1.5pt] (4.00,3.50)--(4.00,4.50);\draw[black!50,line width=1.5pt] (0.50,3.00)--(1.50,3.00);\draw[black!50,line width=1.5pt] (1.00,2.50)--(1.00,3.50);\draw[black!50,line width=1.5pt] (1.50,3.00)--(2.50,3.00);\draw[black!50,line width=1.5pt] (2.00,2.50)--(2.00,3.50);\draw[black!50,line width=1.5pt] (2.50,3.00)--(3.50,3.00);\draw[black!50,line width=1.5pt] (3.00,2.50)--(3.00,3.50);\draw[black!50,line width=1.5pt] (0.50,2.00)--(1.50,2.00);\draw[black!50,line width=1.5pt] (1.00,1.50)--(1.00,2.50);\draw[black!50,line width=1.5pt] (1.50,2.00)--(2.50,2.00);\draw[black!50,line width=1.5pt] (2.00,1.50)--(2.00,2.50);\draw[black!50,line width=1.5pt] (3.50,3.00)--(4.50,3.00);\draw[black!50,line width=1.5pt] (4.00,2.50)--(4.00,3.50);\draw[black!50,line width=1.5pt] (3.50,2.00)--(4.50,2.00);\draw[black!50,line width=1.5pt] (4.00,1.50)--(4.00,2.50);\draw[black!50,line width=1.5pt] (3.50,1.00)--(4.50,1.00);\draw[black!50,line width=1.5pt] (4.00,0.50)--(4.00,1.50);\draw[black!50,line width=1.5pt] (1.50,1.00)--(2.50,1.00);\draw[black!50,line width=1.5pt] (2.00,0.50)--(2.00,1.50);\draw[black!50,line width=1.5pt] (0.50,1.00)--(1.50,1.00);\draw[black!50,line width=1.5pt] (1.00,0.50)--(1.00,1.50);\draw[black!50,line width=1.5pt, rounded corners=12] (0.50,6.00)--(1.00,6.00)--(1.00,6.50);\draw[black!50,line width=1.5pt, rounded corners=12] (1.00,5.50)--(1.00,6.00)--(1.50,6.00);\draw[black!50,line width=1.5pt, rounded corners=12] (0.50,7.00)--(1.00,7.00)--(1.00,7.50);\draw[black!50,line width=1.5pt, rounded corners=12] (1.00,6.50)--(1.00,7.00)--(1.50,7.00);\draw[black!50,line width=1.5pt, rounded corners=12] (1.50,4.00)--(2.00,4.00)--(2.00,4.50);\draw[black!50,line width=1.5pt, rounded corners=12] (2.00,3.50)--(2.00,4.00)--(2.50,4.00);\draw[black!50,line width=1.5pt, rounded corners=12] (1.50,5.00)--(2.00,5.00)--(2.00,5.50);\draw[black!50,line width=1.5pt, rounded corners=12] (2.00,4.50)--(2.00,5.00)--(2.50,5.00);\draw[black!50,line width=1.5pt, rounded corners=12] (1.50,6.00)--(2.00,6.00)--(2.00,6.50);\draw[black!50,line width=1.5pt, rounded corners=12] (2.00,5.50)--(2.00,6.00)--(2.50,6.00);\draw[black!50,line width=1.5pt, rounded corners=12] (2.50,1.00)--(3.00,1.00)--(3.00,1.50);\draw[black!50,line width=1.5pt, rounded corners=12] (3.00,0.50)--(3.00,1.00)--(3.50,1.00);\draw[black!50,line width=1.5pt, rounded corners=12] (2.50,2.00)--(3.00,2.00)--(3.00,2.50);\draw[black!50,line width=1.5pt, rounded corners=12] (3.00,1.50)--(3.00,2.00)--(3.50,2.00);\draw[black!50,line width=1.5pt, rounded corners=12] (2.50,6.00)--(3.00,6.00)--(3.00,6.50);\draw[black!50,line width=1.5pt, rounded corners=12] (3.00,5.50)--(3.00,6.00)--(3.50,6.00);\draw[black!50,line width=1.5pt, rounded corners=12] (3.50,7.00)--(4.00,7.00)--(4.00,7.50);\draw[black!50,line width=1.5pt, rounded corners=12] (4.00,6.50)--(4.00,7.00)--(4.50,7.00);\draw[blue,line width=3pt, rounded corners=12] (1.00,0.50)--(1.00,6.00)--(2.00,6.00)--(2.00,7.50);\draw[red,line width=3pt, rounded corners=12] (0.50,1.00)--(3.00,1.00)--(3.00,2.00)--(4.50,2.00);\draw[red,line width=3pt, rounded corners=12] (0.50,2.00)--(3.00,2.00)--(3.00,6.00)--(4.50,6.00);\draw[red,line width=3pt, rounded corners=12] (0.50,3.00)--(4.50,3.00);\draw[red,line width=3pt, rounded corners=12] (0.50,4.00)--(2.00,4.00)--(2.00,5.00)--(4.50,5.00);\end{tikzpicture}}
\end{tabular}
};
\node[scale=0.70](A) at (-5.70,-2.40) {$\Hpi=\textcolor{red}{\H}$};
\node[scale=0.70](A) at (-5.70,-2.90) {$\Vpi=\textcolor{blue}{\V}$};
\node[scale=1.10](RAVNO) at (0,-2.90) {$=$};
\node[scale=1.10,anchor=east](A) at (RAVNO.west) {$\PF^{\textcolor{red}{\H},\textcolor{blue}{\V}}(\bx,\by)$};
\node[scale=1.10,anchor=west](A) at (RAVNO.east) {$\PF^{\textcolor{red}{\flip(\H)},\textcolor{blue}{\V}}(\bx,\textcolor{green!70!black}{\rev(\by)})$};
\node[scale=0.70,anchor=west](A) at (5.00,-2.40) {$\H^{M,N}_{\pi'}=\textcolor{red}{\flip(\H)}$};
\node[scale=0.70,anchor=west](A) at (5.00,-2.90) {$\V^{M,N}_{\pi'}=\textcolor{blue}{\V}$};
\end{tikzpicture}}
}}
\usefig\FIGflip{FIGflip}
\def\FIGexdt{
\def\tabscl{0.7}
\def\hsps{0.02in}
\noindent\resizebox{0.7\textwidth}{!}{\begin{tabular}{ccc|ccc}
\scalebox{0.50}{\begin{tikzpicture}[baseline=(ZUZU.base)]
\coordinate(ZUZU) at (1,1);\drawgrid{-1}{-1}{3}{4}\draw[line width=1pt,->,black!50] (-0.6,-1)--(-0.6,4);
\draw[line width=1pt,->,black!50] (-1,-0.6)--(3,-0.6);
\draw[line width=2pt,draw=black!50] (2.50,0.50)--(1.50,0.50)--(0.50,0.50)--(0.50,1.50)--(0.50,2.50)--(0.50,3.50);

\node[scale=1.30,black!50,anchor=north,inner sep=1pt](PP1) at (2.00,0.50) {$\ein{1}$};
\node[scale=1.30,blue,anchor=north,inner sep=1pt](PP2) at (1.00,0.50) {$\ein{2}$};
\node[scale=1.30,red,anchor=east,inner sep=1pt](PP3) at (0.50,1.00) {$\ein{3}$};
\node[scale=1.30,black!50,anchor=east,inner sep=1pt](PP4) at (0.50,2.00) {$\ein{4}$};
\node[scale=1.30,red,anchor=east,inner sep=1pt](PP5) at (0.50,3.00) {$\ein{5}$};
\draw[line width=2pt,draw=black!50] (2.50,0.50)--(2.50,1.50)--(2.50,2.50)--(2.50,3.50)--(1.50,3.50)--(0.50,3.50);

\node[scale=1.30,black!50,anchor=west,inner sep=1pt](PP1) at (2.50,1.00) {$\eout{1}$};
\node[scale=1.30,red,anchor=west,inner sep=1pt](PP2) at (2.50,2.00) {$\eout{2}$};
\node[scale=1.30,red,anchor=west,inner sep=1pt](PP3) at (2.50,3.00) {$\eout{3}$};
\node[scale=1.30,blue,anchor=south,inner sep=1pt](QQ4) at (2.00,3.50) {$\eout{4}$};
\node[scale=1.30,black!50,anchor=south,inner sep=1pt](QQ5) at (1.00,3.50) {$\eout{5}$};
\node[scale=1.80,black](A) at (1,-1) {$x_{1}$};
\node[scale=1.80,black](A) at (2,-1) {$x_{2}$};
\node[scale=1.80,black](A) at (-1,1) {$y_{1}$};
\node[scale=1.80,black](A) at (-1,2) {$y_{2}$};
\node[scale=1.80,black](A) at (-1,3) {$y_{3}$};
\draw[black!50,line width=1.5pt] (0.50,3.00)--(1.50,3.00);\draw[black!50,line width=1.5pt] (1.00,2.50)--(1.00,3.50);\draw[black!50,line width=1.5pt] (1.50,3.00)--(2.50,3.00);\draw[black!50,line width=1.5pt] (2.00,2.50)--(2.00,3.50);\draw[black!50,line width=1.5pt] (0.50,1.00)--(1.50,1.00);\draw[black!50,line width=1.5pt] (1.00,0.50)--(1.00,1.50);\draw[black!50,line width=1.5pt, rounded corners=12] (0.50,2.00)--(1.00,2.00)--(1.00,2.50);\draw[black!50,line width=1.5pt, rounded corners=12] (1.00,1.50)--(1.00,2.00)--(1.50,2.00);\draw[black!50,line width=1.5pt, rounded corners=12] (1.50,1.00)--(2.00,1.00)--(2.00,1.50);\draw[black!50,line width=1.5pt, rounded corners=12] (2.00,0.50)--(2.00,1.00)--(2.50,1.00);\draw[black!50,line width=1.5pt, rounded corners=12] (1.50,2.00)--(2.00,2.00)--(2.00,2.50);\draw[black!50,line width=1.5pt, rounded corners=12] (2.00,1.50)--(2.00,2.00)--(2.50,2.00);\draw[blue,line width=3pt, rounded corners=12] (1.00,0.50)--(1.00,2.00)--(2.00,2.00)--(2.00,3.50);\draw[red,line width=3pt, rounded corners=12] (0.50,1.00)--(2.00,1.00)--(2.00,2.00)--(2.50,2.00);\draw[red,line width=3pt, rounded corners=12] (0.50,3.00)--(2.50,3.00);\end{tikzpicture}}
&\scalebox{0.50}{\begin{tikzpicture}[baseline=(ZUZU.base)]
\coordinate(ZUZU) at (1,1);\drawgrid{-1}{-1}{3}{4}\draw[line width=1pt,->,black!50] (-0.6,-1)--(-0.6,4);
\draw[line width=1pt,->,black!50] (-1,-0.6)--(3,-0.6);
\draw[line width=2pt,draw=black!50] (2.50,0.50)--(1.50,0.50)--(0.50,0.50)--(0.50,1.50)--(0.50,2.50)--(0.50,3.50);

\node[scale=1.30,black!50,anchor=north,inner sep=1pt](PP1) at (2.00,0.50) {$\ein{1}$};
\node[scale=1.30,blue,anchor=north,inner sep=1pt](PP2) at (1.00,0.50) {$\ein{2}$};
\node[scale=1.30,red,anchor=east,inner sep=1pt](PP3) at (0.50,1.00) {$\ein{3}$};
\node[scale=1.30,black!50,anchor=east,inner sep=1pt](PP4) at (0.50,2.00) {$\ein{4}$};
\node[scale=1.30,red,anchor=east,inner sep=1pt](PP5) at (0.50,3.00) {$\ein{5}$};
\draw[line width=2pt,draw=black!50] (2.50,0.50)--(2.50,1.50)--(2.50,2.50)--(2.50,3.50)--(1.50,3.50)--(0.50,3.50);

\node[scale=1.30,black!50,anchor=west,inner sep=1pt](PP1) at (2.50,1.00) {$\eout{1}$};
\node[scale=1.30,red,anchor=west,inner sep=1pt](PP2) at (2.50,2.00) {$\eout{2}$};
\node[scale=1.30,red,anchor=west,inner sep=1pt](PP3) at (2.50,3.00) {$\eout{3}$};
\node[scale=1.30,blue,anchor=south,inner sep=1pt](QQ4) at (2.00,3.50) {$\eout{4}$};
\node[scale=1.30,black!50,anchor=south,inner sep=1pt](QQ5) at (1.00,3.50) {$\eout{5}$};
\node[scale=1.80,black](A) at (1,-1) {$x_{1}$};
\node[scale=1.80,black](A) at (2,-1) {$x_{2}$};
\node[scale=1.80,black](A) at (-1,1) {$y_{1}$};
\node[scale=1.80,black](A) at (-1,2) {$y_{2}$};
\node[scale=1.80,black](A) at (-1,3) {$y_{3}$};
\draw[black!50,line width=1.5pt] (0.50,3.00)--(1.50,3.00);\draw[black!50,line width=1.5pt] (1.00,2.50)--(1.00,3.50);\draw[black!50,line width=1.5pt] (1.50,3.00)--(2.50,3.00);\draw[black!50,line width=1.5pt] (2.00,2.50)--(2.00,3.50);\draw[black!50,line width=1.5pt] (1.50,2.00)--(2.50,2.00);\draw[black!50,line width=1.5pt] (2.00,1.50)--(2.00,2.50);\draw[black!50,line width=1.5pt, rounded corners=12] (0.50,1.00)--(1.00,1.00)--(1.00,1.50);\draw[black!50,line width=1.5pt, rounded corners=12] (1.00,0.50)--(1.00,1.00)--(1.50,1.00);\draw[black!50,line width=1.5pt, rounded corners=12] (0.50,2.00)--(1.00,2.00)--(1.00,2.50);\draw[black!50,line width=1.5pt, rounded corners=12] (1.00,1.50)--(1.00,2.00)--(1.50,2.00);\draw[black!50,line width=1.5pt, rounded corners=12] (1.50,1.00)--(2.00,1.00)--(2.00,1.50);\draw[black!50,line width=1.5pt, rounded corners=12] (2.00,0.50)--(2.00,1.00)--(2.50,1.00);\draw[blue,line width=3pt, rounded corners=12] (1.00,0.50)--(1.00,1.00)--(2.00,1.00)--(2.00,3.50);\draw[red,line width=3pt, rounded corners=12] (0.50,1.00)--(1.00,1.00)--(1.00,2.00)--(2.50,2.00);\draw[red,line width=3pt, rounded corners=12] (0.50,3.00)--(2.50,3.00);\end{tikzpicture}}
&&
&\scalebox{0.50}{\begin{tikzpicture}[baseline=(ZUZU.base)]
\coordinate(ZUZU) at (1,1);\drawgrid{-1}{-1}{3}{4}\draw[line width=1pt,->,black!50] (-0.6,-1)--(-0.6,4);
\draw[line width=1pt,->,black!50] (-1,-0.6)--(3,-0.6);
\draw[line width=2pt,draw=black!50] (2.50,0.50)--(1.50,0.50)--(0.50,0.50)--(0.50,1.50)--(0.50,2.50)--(0.50,3.50);

\node[scale=1.30,black!50,anchor=north,inner sep=1pt](PP1) at (2.00,0.50) {$\ein{1}$};
\node[scale=1.30,blue,anchor=north,inner sep=1pt](PP2) at (1.00,0.50) {$\ein{2}$};
\node[scale=1.30,red,anchor=east,inner sep=1pt](PP3) at (0.50,1.00) {$\ein{3}$};
\node[scale=1.30,red,anchor=east,inner sep=1pt](PP4) at (0.50,2.00) {$\ein{4}$};
\node[scale=1.30,black!50,anchor=east,inner sep=1pt](PP5) at (0.50,3.00) {$\ein{5}$};
\draw[line width=2pt,draw=black!50] (2.50,0.50)--(2.50,1.50)--(2.50,2.50)--(2.50,3.50)--(1.50,3.50)--(0.50,3.50);

\node[scale=1.30,red,anchor=west,inner sep=1pt](PP1) at (2.50,1.00) {$\eout{1}$};
\node[scale=1.30,black!50,anchor=west,inner sep=1pt](PP2) at (2.50,2.00) {$\eout{2}$};
\node[scale=1.30,red,anchor=west,inner sep=1pt](PP3) at (2.50,3.00) {$\eout{3}$};
\node[scale=1.30,blue,anchor=south,inner sep=1pt](QQ4) at (2.00,3.50) {$\eout{4}$};
\node[scale=1.30,black!50,anchor=south,inner sep=1pt](QQ5) at (1.00,3.50) {$\eout{5}$};
\node[scale=1.80,black](A) at (1,-1) {$x_{1}$};
\node[scale=1.80,black](A) at (2,-1) {$x_{2}$};
\node[scale=1.80,green!70!black](A) at (-1,1) {$y_{3}$};
\node[scale=1.80,green!70!black](A) at (-1,2) {$y_{2}$};
\node[scale=1.80,green!70!black](A) at (-1,3) {$y_{1}$};
\draw[black!50,line width=1.5pt] (0.50,1.00)--(1.50,1.00);\draw[black!50,line width=1.5pt] (1.00,0.50)--(1.00,1.50);\draw[black!50,line width=1.5pt] (1.50,1.00)--(2.50,1.00);\draw[black!50,line width=1.5pt] (2.00,0.50)--(2.00,1.50);\draw[black!50,line width=1.5pt] (0.50,2.00)--(1.50,2.00);\draw[black!50,line width=1.5pt] (1.00,1.50)--(1.00,2.50);\draw[black!50,line width=1.5pt, rounded corners=12] (0.50,3.00)--(1.00,3.00)--(1.00,3.50);\draw[black!50,line width=1.5pt, rounded corners=12] (1.00,2.50)--(1.00,3.00)--(1.50,3.00);\draw[black!50,line width=1.5pt, rounded corners=12] (1.50,2.00)--(2.00,2.00)--(2.00,2.50);\draw[black!50,line width=1.5pt, rounded corners=12] (2.00,1.50)--(2.00,2.00)--(2.50,2.00);\draw[black!50,line width=1.5pt, rounded corners=12] (1.50,3.00)--(2.00,3.00)--(2.00,3.50);\draw[black!50,line width=1.5pt, rounded corners=12] (2.00,2.50)--(2.00,3.00)--(2.50,3.00);\draw[blue,line width=3pt, rounded corners=12] (1.00,0.50)--(1.00,3.00)--(2.00,3.00)--(2.00,3.50);\draw[red,line width=3pt, rounded corners=12] (0.50,1.00)--(2.50,1.00);\draw[red,line width=3pt, rounded corners=12] (0.50,2.00)--(2.00,2.00)--(2.00,3.00)--(2.50,3.00);\end{tikzpicture}}
&\scalebox{0.50}{\begin{tikzpicture}[baseline=(ZUZU.base)]
\coordinate(ZUZU) at (1,1);\drawgrid{-1}{-1}{3}{4}\draw[line width=1pt,->,black!50] (-0.6,-1)--(-0.6,4);
\draw[line width=1pt,->,black!50] (-1,-0.6)--(3,-0.6);
\draw[line width=2pt,draw=black!50] (2.50,0.50)--(1.50,0.50)--(0.50,0.50)--(0.50,1.50)--(0.50,2.50)--(0.50,3.50);

\node[scale=1.30,black!50,anchor=north,inner sep=1pt](PP1) at (2.00,0.50) {$\ein{1}$};
\node[scale=1.30,blue,anchor=north,inner sep=1pt](PP2) at (1.00,0.50) {$\ein{2}$};
\node[scale=1.30,red,anchor=east,inner sep=1pt](PP3) at (0.50,1.00) {$\ein{3}$};
\node[scale=1.30,red,anchor=east,inner sep=1pt](PP4) at (0.50,2.00) {$\ein{4}$};
\node[scale=1.30,black!50,anchor=east,inner sep=1pt](PP5) at (0.50,3.00) {$\ein{5}$};
\draw[line width=2pt,draw=black!50] (2.50,0.50)--(2.50,1.50)--(2.50,2.50)--(2.50,3.50)--(1.50,3.50)--(0.50,3.50);

\node[scale=1.30,red,anchor=west,inner sep=1pt](PP1) at (2.50,1.00) {$\eout{1}$};
\node[scale=1.30,black!50,anchor=west,inner sep=1pt](PP2) at (2.50,2.00) {$\eout{2}$};
\node[scale=1.30,red,anchor=west,inner sep=1pt](PP3) at (2.50,3.00) {$\eout{3}$};
\node[scale=1.30,blue,anchor=south,inner sep=1pt](QQ4) at (2.00,3.50) {$\eout{4}$};
\node[scale=1.30,black!50,anchor=south,inner sep=1pt](QQ5) at (1.00,3.50) {$\eout{5}$};
\node[scale=1.80,black](A) at (1,-1) {$x_{1}$};
\node[scale=1.80,black](A) at (2,-1) {$x_{2}$};
\node[scale=1.80,green!70!black](A) at (-1,1) {$y_{3}$};
\node[scale=1.80,green!70!black](A) at (-1,2) {$y_{2}$};
\node[scale=1.80,green!70!black](A) at (-1,3) {$y_{1}$};
\draw[black!50,line width=1.5pt] (0.50,1.00)--(1.50,1.00);\draw[black!50,line width=1.5pt] (1.00,0.50)--(1.00,1.50);\draw[black!50,line width=1.5pt] (1.50,1.00)--(2.50,1.00);\draw[black!50,line width=1.5pt] (2.00,0.50)--(2.00,1.50);\draw[black!50,line width=1.5pt] (1.50,3.00)--(2.50,3.00);\draw[black!50,line width=1.5pt] (2.00,2.50)--(2.00,3.50);\draw[black!50,line width=1.5pt, rounded corners=12] (0.50,2.00)--(1.00,2.00)--(1.00,2.50);\draw[black!50,line width=1.5pt, rounded corners=12] (1.00,1.50)--(1.00,2.00)--(1.50,2.00);\draw[black!50,line width=1.5pt, rounded corners=12] (0.50,3.00)--(1.00,3.00)--(1.00,3.50);\draw[black!50,line width=1.5pt, rounded corners=12] (1.00,2.50)--(1.00,3.00)--(1.50,3.00);\draw[black!50,line width=1.5pt, rounded corners=12] (1.50,2.00)--(2.00,2.00)--(2.00,2.50);\draw[black!50,line width=1.5pt, rounded corners=12] (2.00,1.50)--(2.00,2.00)--(2.50,2.00);\draw[blue,line width=3pt, rounded corners=12] (1.00,0.50)--(1.00,2.00)--(2.00,2.00)--(2.00,3.50);\draw[red,line width=3pt, rounded corners=12] (0.50,1.00)--(2.50,1.00);\draw[red,line width=3pt, rounded corners=12] (0.50,2.00)--(1.00,2.00)--(1.00,3.00)--(2.50,3.00);\end{tikzpicture}}
\\\hspace{\hsps}\scalebox{\tabscl}{\begin{tabular}{|c|c|}\hline $\sp_{1,3}$ & $\sp_{2,3}$\\\hline $1-\sp_{1,2}$ & $1-q\sp_{2,2}$\\\hline $\sp_{1,1}$ & $1-\sp_{2,1}$\\\hline\end{tabular}}
 & 
\hspace{\hsps}\scalebox{\tabscl}{\begin{tabular}{|c|c|}\hline $\sp_{1,3}$ & $\sp_{2,3}$\\\hline $1-\sp_{1,2}$ & $\sp_{2,2}$\\\hline $1-\sp_{1,1}$ & $1-\sp_{2,1}$\\\hline\end{tabular}}
 & &&
\hspace{\hsps}\scalebox{\tabscl}{\begin{tabular}{|c|c|}\hline $1-\sp_{1,1}$ & $1-q\sp_{2,1}$\\\hline $\sp_{1,2}$ & $1-\sp_{2,2}$\\\hline $\sp_{1,3}$ & $\sp_{2,3}$\\\hline \end{tabular}}
 & 
\hspace{\hsps}\scalebox{\tabscl}{\begin{tabular}{|c|c|}\hline $1-\sp_{1,1}$ & $\sp_{2,1}$\\\hline $1-\sp_{1,2}$ & $1-\sp_{2,2}$\\\hline $\sp_{1,3}$ & $\sp_{2,3}$\\\hline \end{tabular}}
 \end{tabular}
}}
\usefig\FIGexdt{FIGexdt}
\def\FIGspec{\noindent\resizebox{0.85\textwidth}{!}{\begin{tabular}{ccccc|ccccc}\scalebox{1.50}{\begin{tikzpicture}[baseline=(ZUZU.base)]
\coordinate(ZUZU) at (0,0);\drawgrid{0}{0}{0}{0}\draw[blue,line width=2pt, rounded corners=12] (0.00,-0.80)--(0.00,-0.50);\draw[red,line width=2pt, rounded corners=12] (-0.80,0.00)--(-0.50,0.00);\node[scale=1.00,blue,anchor=north] (A) at (0,-0.80) {$c_1$};
\node[scale=1.00,red,anchor=east] (A) at (-0.80,0) {$c_2$};
\node[scale=1.00] (A) at (0,0) {$\sp$};
\end{tikzpicture}}
& \scalebox{1.70}{\hspace{0.1in} \begin{tikzpicture}[baseline=(ZUZU.base)]\coordinate(ZUZU) at (0,0); \node(A) (0,0) {$\to$};\end{tikzpicture}} &\scalebox{1.50}{\begin{tikzpicture}[baseline=(ZUZU.base)]
\coordinate(ZUZU) at (0,0);\drawgrid{0}{0}{0}{0}\draw[blue,line width=2pt, rounded corners=12] (0.00,-0.80)--(0.00,0.50);\draw[red,line width=2pt, rounded corners=12] (-0.80,0.00)--(0.50,0.00);\node[scale=1.00,blue,anchor=north] (A) at (0,-0.80) {$c_1$};
\node[scale=1.00,red,anchor=east] (A) at (-0.80,0) {$c_2$};
\end{tikzpicture}}
& \scalebox{1.70}{\hspace{0.1in} \begin{tikzpicture}[baseline=(ZUZU.base)]\coordinate(ZUZU) at (0,0); \node(A) at (0,0) {or};\end{tikzpicture}} &\scalebox{1.50}{\begin{tikzpicture}[baseline=(ZUZU.base)]
\coordinate(ZUZU) at (0,0);\drawgrid{0}{0}{0}{0}\draw[blue,line width=2pt, rounded corners=12] (0.00,-0.80)--(0.00,0.00)--(0.50,0.00);\draw[red,line width=2pt, rounded corners=12] (-0.80,0.00)--(0.00,0.00)--(0.00,0.50);\node[scale=1.00,blue,anchor=north] (A) at (0,-0.80) {$c_1$};
\node[scale=1.00,red,anchor=east] (A) at (-0.80,0) {$c_2$};
\end{tikzpicture}}
&\scalebox{1.50}{\begin{tikzpicture}[baseline=(ZUZU.base)]
\coordinate(ZUZU) at (0,0);\drawgrid{0}{0}{0}{0}\draw[red,line width=2pt, rounded corners=12] (0.00,-0.80)--(0.00,-0.50);\draw[blue,line width=2pt, rounded corners=12] (-0.80,0.00)--(-0.50,0.00);\node[scale=1.00,red,anchor=north] (A) at (0,-0.80) {$c_2$};
\node[scale=1.00,blue,anchor=east] (A) at (-0.80,0) {$c_1$};
\node[scale=1.00] (A) at (0,0) {$\sp$};
\end{tikzpicture}}
& \scalebox{1.70}{\hspace{0.1in} \begin{tikzpicture}[baseline=(ZUZU.base)]\coordinate(ZUZU) at (0,0); \node(A) (0,0) {$\to$};\end{tikzpicture}} &\scalebox{1.50}{\begin{tikzpicture}[baseline=(ZUZU.base)]
\coordinate(ZUZU) at (0,0);\drawgrid{0}{0}{0}{0}\draw[blue,line width=2pt, rounded corners=12] (-0.80,0.00)--(0.50,0.00);\draw[red,line width=2pt, rounded corners=12] (0.00,-0.80)--(0.00,0.50);\node[scale=1.00,red,anchor=north] (A) at (0,-0.80) {$c_2$};
\node[scale=1.00,blue,anchor=east] (A) at (-0.80,0) {$c_1$};
\end{tikzpicture}}
& \scalebox{1.70}{\hspace{0.1in} \begin{tikzpicture}[baseline=(ZUZU.base)]\coordinate(ZUZU) at (0,0); \node(A) at (0,0) {or};\end{tikzpicture}} &\scalebox{1.50}{\begin{tikzpicture}[baseline=(ZUZU.base)]
\coordinate(ZUZU) at (0,0);\drawgrid{0}{0}{0}{0}\draw[blue,line width=2pt, rounded corners=12] (-0.80,0.00)--(0.00,0.00)--(0.00,0.50);\draw[red,line width=2pt, rounded corners=12] (0.00,-0.80)--(0.00,0.00)--(0.50,0.00);\node[scale=1.00,red,anchor=north] (A) at (0,-0.80) {$c_2$};
\node[scale=1.00,blue,anchor=east] (A) at (-0.80,0) {$c_1$};
\end{tikzpicture}}
\\ &&&&&\\
\multicolumn{2}{c}{\scalebox{2.00}{Probability:}}& \hspace{0.2in} \scalebox{2.00}{$\sp$} && \hspace{0.2in} \scalebox{2.00}{$1-\sp$} &
\multicolumn{2}{c}{\scalebox{2.00}{Probability:}}& \hspace{0.2in} \scalebox{2.00}{$q\sp$} && \hspace{0.2in} \scalebox{2.00}{$1-q\sp$} \end{tabular}}}
\usefig\FIGspec{FIGspec}
\def\FIGflipbase{\noindent\resizebox{0.8\textwidth}{!}{\scalebox{1.00}{\begin{tikzpicture}[baseline=(ZUZU.base)]
\coordinate(ZUZU) at (1,1);\node(A) at (0,0){\begin{tabular}{ccc}
\scalebox{0.50}{\begin{tikzpicture}[baseline=(ZUZU.base)]
\coordinate(ZUZU) at (1,1);\drawgrid{-1}{-1}{5}{6}\draw[line width=1pt,->,black!50] (-0.6,-1)--(-0.6,6.00);
\draw[line width=1pt,->,black!50] (-1,-0.6)--(5.00,-0.6);
\draw[line width=2pt,draw=black!50] (4.50,0.50)--(3.50,0.50)--(2.50,0.50)--(1.50,0.50)--(0.50,0.50)--(0.50,1.50)--(0.50,2.50)--(0.50,3.50)--(0.50,4.50)--(0.50,5.50);

\node[scale=1.30,black!50,anchor=north,inner sep=1pt](PP1) at (4.00,0.50) {$\ein{1}$};
\node[scale=1.30,black!50,anchor=north,inner sep=1pt](PP2) at (3.00,0.50) {$\ein{2}$};
\node[scale=1.30,black!50,anchor=north,inner sep=1pt](PP3) at (2.00,0.50) {$\ein{3}$};
\node[scale=1.30,blue,anchor=north,inner sep=1pt](PP4) at (1.00,0.50) {$\ein{4}$};
\node[scale=1.30,red,anchor=east,inner sep=1pt](PP5) at (0.50,1.00) {$\ein{5}$};
\node[scale=1.30,red,anchor=east,inner sep=1pt](PP6) at (0.50,2.00) {$\ein{6}$};
\node[scale=1.30,black!50,anchor=east,inner sep=1pt](PP7) at (0.50,3.00) {$\ein{7}$};
\node[scale=1.30,black!50,anchor=east,inner sep=1pt](PP8) at (0.50,4.00) {$\ein{8}$};
\node[scale=1.30,black!50,anchor=east,inner sep=1pt](PP9) at (0.50,5.00) {$\ein{9}$};
\draw[line width=2pt,draw=black!50] (4.50,0.50)--(4.50,1.50)--(4.50,2.50)--(4.50,3.50)--(4.50,4.50)--(4.50,5.50)--(3.50,5.50)--(2.50,5.50)--(1.50,5.50)--(0.50,5.50);

\node[scale=1.30,red,anchor=west,inner sep=1pt](PP1) at (4.50,1.00) {$\eout{1}$};
\node[scale=1.30,red,anchor=west,inner sep=1pt](PP2) at (4.50,2.00) {$\eout{2}$};
\node[scale=1.30,black!50,anchor=west,inner sep=1pt](PP3) at (4.50,3.00) {$\eout{3}$};
\node[scale=1.30,black!50,anchor=west,inner sep=1pt](PP4) at (4.50,4.00) {$\eout{4}$};
\node[scale=1.30,black!50,anchor=west,inner sep=1pt](PP5) at (4.50,5.00) {$\eout{5}$};
\node[scale=1.30,black!50,anchor=south,inner sep=1pt](QQ6) at (4.00,5.50) {$\eout{6}$};
\node[scale=1.30,black!50,anchor=south,inner sep=1pt](QQ7) at (3.00,5.50) {$\eout{7}$};
\node[scale=1.30,blue,anchor=south,inner sep=1pt](QQ8) at (2.00,5.50) {$\eout{8}$};
\node[scale=1.30,black!50,anchor=south,inner sep=1pt](QQ9) at (1.00,5.50) {$\eout{9}$};
\node[scale=1.80,black](A) at (1,-1) {$x_{1}$};
\node[scale=1.80,black](A) at (2,-1) {$x_{2}$};
\node[scale=1.80,black](A) at (3,-1) {$x_{3}$};
\node[scale=1.80,black](A) at (4,-1) {$x_{4}$};
\node[scale=1.80,black](A) at (-1,1) {$y_{1}$};
\node[scale=1.80,black](A) at (-1,2) {$y_{2}$};
\node[scale=1.80,black](A) at (-1,3) {$y_{3}$};
\node[scale=1.80,black](A) at (-1,4) {$y_{4}$};
\node[scale=1.80,black](A) at (-1,5) {$y_{5}$};
\draw[black!50,line width=1.5pt] (0.50,1.00)--(1.50,1.00);\draw[black!50,line width=1.5pt] (1.00,0.50)--(1.00,1.50);\draw[black!50,line width=1.5pt] (0.50,2.00)--(1.50,2.00);\draw[black!50,line width=1.5pt] (1.00,1.50)--(1.00,2.50);\draw[black!50,line width=1.5pt] (1.50,1.00)--(2.50,1.00);\draw[black!50,line width=1.5pt] (2.00,0.50)--(2.00,1.50);\draw[black!50,line width=1.5pt] (1.50,2.00)--(2.50,2.00);\draw[black!50,line width=1.5pt] (2.00,1.50)--(2.00,2.50);\draw[black!50,line width=1.5pt] (2.50,1.00)--(3.50,1.00);\draw[black!50,line width=1.5pt] (3.00,0.50)--(3.00,1.50);\draw[black!50,line width=1.5pt] (2.50,2.00)--(3.50,2.00);\draw[black!50,line width=1.5pt] (3.00,1.50)--(3.00,2.50);\draw[black!50,line width=1.5pt] (3.50,1.00)--(4.50,1.00);\draw[black!50,line width=1.5pt] (4.00,0.50)--(4.00,1.50);\draw[black!50,line width=1.5pt] (3.50,2.00)--(4.50,2.00);\draw[black!50,line width=1.5pt] (4.00,1.50)--(4.00,2.50);\draw[blue,line width=3pt, rounded corners=12] (1.00,0.50)--(1.00,2.50);\draw[red,line width=3pt, rounded corners=12] (0.50,1.00)--(4.50,1.00);\draw[red,line width=3pt, rounded corners=12] (0.50,2.00)--(4.50,2.00);\draw[blue,line width=3pt] (1.00,2.50) to[in=-90,out=90] (2.00,5.50);
\end{tikzpicture}}
&\scalebox{0.50}{\begin{tikzpicture}[baseline=(ZUZU.base)]
\coordinate(ZUZU) at (1,1);\drawgrid{-1}{-1}{5}{6}\draw[line width=1pt,->,black!50] (-0.6,-1)--(-0.6,6.00);
\draw[line width=1pt,->,black!50] (-1,-0.6)--(5.00,-0.6);
\draw[line width=2pt,draw=black!50] (4.50,0.50)--(3.50,0.50)--(2.50,0.50)--(1.50,0.50)--(0.50,0.50)--(0.50,1.50)--(0.50,2.50)--(0.50,3.50)--(0.50,4.50)--(0.50,5.50);

\node[scale=1.30,black!50,anchor=north,inner sep=1pt](PP1) at (4.00,0.50) {$\ein{1}$};
\node[scale=1.30,black!50,anchor=north,inner sep=1pt](PP2) at (3.00,0.50) {$\ein{2}$};
\node[scale=1.30,black!50,anchor=north,inner sep=1pt](PP3) at (2.00,0.50) {$\ein{3}$};
\node[scale=1.30,blue,anchor=north,inner sep=1pt](PP4) at (1.00,0.50) {$\ein{4}$};
\node[scale=1.30,black!50,anchor=east,inner sep=1pt](PP5) at (0.50,1.00) {$\ein{5}$};
\node[scale=1.30,black!50,anchor=east,inner sep=1pt](PP6) at (0.50,2.00) {$\ein{6}$};
\node[scale=1.30,black!50,anchor=east,inner sep=1pt](PP7) at (0.50,3.00) {$\ein{7}$};
\node[scale=1.30,red,anchor=east,inner sep=1pt](PP8) at (0.50,4.00) {$\ein{8}$};
\node[scale=1.30,red,anchor=east,inner sep=1pt](PP9) at (0.50,5.00) {$\ein{9}$};
\draw[line width=2pt,draw=black!50] (4.50,0.50)--(4.50,1.50)--(4.50,2.50)--(4.50,3.50)--(4.50,4.50)--(4.50,5.50)--(3.50,5.50)--(2.50,5.50)--(1.50,5.50)--(0.50,5.50);

\node[scale=1.30,black!50,anchor=west,inner sep=1pt](PP1) at (4.50,1.00) {$\eout{1}$};
\node[scale=1.30,black!50,anchor=west,inner sep=1pt](PP2) at (4.50,2.00) {$\eout{2}$};
\node[scale=1.30,black!50,anchor=west,inner sep=1pt](PP3) at (4.50,3.00) {$\eout{3}$};
\node[scale=1.30,red,anchor=west,inner sep=1pt](PP4) at (4.50,4.00) {$\eout{4}$};
\node[scale=1.30,red,anchor=west,inner sep=1pt](PP5) at (4.50,5.00) {$\eout{5}$};
\node[scale=1.30,black!50,anchor=south,inner sep=1pt](QQ6) at (4.00,5.50) {$\eout{6}$};
\node[scale=1.30,black!50,anchor=south,inner sep=1pt](QQ7) at (3.00,5.50) {$\eout{7}$};
\node[scale=1.30,blue,anchor=south,inner sep=1pt](QQ8) at (2.00,5.50) {$\eout{8}$};
\node[scale=1.30,black!50,anchor=south,inner sep=1pt](QQ9) at (1.00,5.50) {$\eout{9}$};
\draw[black!50,line width=1.5pt] (0.50,4.00)--(1.50,4.00);\draw[black!50,line width=1.5pt] (1.00,3.50)--(1.00,4.50);\draw[black!50,line width=1.5pt] (0.50,5.00)--(1.50,5.00);\draw[black!50,line width=1.5pt] (1.00,4.50)--(1.00,5.50);\draw[black!50,line width=1.5pt] (1.50,4.00)--(2.50,4.00);\draw[black!50,line width=1.5pt] (2.00,3.50)--(2.00,4.50);\draw[black!50,line width=1.5pt] (1.50,5.00)--(2.50,5.00);\draw[black!50,line width=1.5pt] (2.00,4.50)--(2.00,5.50);\draw[black!50,line width=1.5pt] (2.50,4.00)--(3.50,4.00);\draw[black!50,line width=1.5pt] (3.00,3.50)--(3.00,4.50);\draw[black!50,line width=1.5pt] (2.50,5.00)--(3.50,5.00);\draw[black!50,line width=1.5pt] (3.00,4.50)--(3.00,5.50);\draw[black!50,line width=1.5pt] (3.50,4.00)--(4.50,4.00);\draw[black!50,line width=1.5pt] (4.00,3.50)--(4.00,4.50);\draw[black!50,line width=1.5pt] (3.50,5.00)--(4.50,5.00);\draw[black!50,line width=1.5pt] (4.00,4.50)--(4.00,5.50);\node[scale=1.80,black](A) at (1,-1) {$x_{1}$};
\node[scale=1.80,black](A) at (2,-1) {$x_{2}$};
\node[scale=1.80,black](A) at (3,-1) {$x_{3}$};
\node[scale=1.80,black](A) at (4,-1) {$x_{4}$};
\node[scale=1.80,black](A) at (-1,1) {$y_{3}$};
\node[scale=1.80,black](A) at (-1,2) {$y_{4}$};
\node[scale=1.80,black](A) at (-1,3) {$y_{5}$};
\node[scale=1.80,green!70!black](A) at (-1,4) {$y_{2}$};
\node[scale=1.80,green!70!black](A) at (-1,5) {$y_{1}$};
\draw[blue,line width=3pt, rounded corners=12] (2.00,3.50)--(2.00,5.50);\draw[red,line width=3pt, rounded corners=12] (0.50,4.00)--(4.50,4.00);\draw[red,line width=3pt, rounded corners=12] (0.50,5.00)--(4.50,5.00);\draw[blue,line width=3pt] (1.00,0.50) to[in=-90,out=90] (2.00,3.50);
\end{tikzpicture}}
&\scalebox{0.50}{\begin{tikzpicture}[baseline=(ZUZU.base)]
\coordinate(ZUZU) at (1,1);\drawgrid{-1}{-1}{5}{6}\draw[line width=1pt,->,black!50] (-0.6,-1)--(-0.6,6.00);
\draw[line width=1pt,->,black!50] (-1,-0.6)--(5.00,-0.6);
\draw[line width=2pt,draw=black!50] (4.50,0.50)--(3.50,0.50)--(2.50,0.50)--(1.50,0.50)--(0.50,0.50)--(0.50,1.50)--(0.50,2.50)--(0.50,3.50)--(0.50,4.50)--(0.50,5.50);

\node[scale=1.30,black!50,anchor=north,inner sep=1pt](PP1) at (4.00,0.50) {$\ein{1}$};
\node[scale=1.30,black!50,anchor=north,inner sep=1pt](PP2) at (3.00,0.50) {$\ein{2}$};
\node[scale=1.30,black!50,anchor=north,inner sep=1pt](PP3) at (2.00,0.50) {$\ein{3}$};
\node[scale=1.30,blue,anchor=north,inner sep=1pt](PP4) at (1.00,0.50) {$\ein{4}$};
\node[scale=1.30,black!50,anchor=east,inner sep=1pt](PP5) at (0.50,1.00) {$\ein{5}$};
\node[scale=1.30,black!50,anchor=east,inner sep=1pt](PP6) at (0.50,2.00) {$\ein{6}$};
\node[scale=1.30,black!50,anchor=east,inner sep=1pt](PP7) at (0.50,3.00) {$\ein{7}$};
\node[scale=1.30,red,anchor=east,inner sep=1pt](PP8) at (0.50,4.00) {$\ein{8}$};
\node[scale=1.30,red,anchor=east,inner sep=1pt](PP9) at (0.50,5.00) {$\ein{9}$};
\draw[line width=2pt,draw=black!50] (4.50,0.50)--(4.50,1.50)--(4.50,2.50)--(4.50,3.50)--(4.50,4.50)--(4.50,5.50)--(3.50,5.50)--(2.50,5.50)--(1.50,5.50)--(0.50,5.50);

\node[scale=1.30,black!50,anchor=west,inner sep=1pt](PP1) at (4.50,1.00) {$\eout{1}$};
\node[scale=1.30,black!50,anchor=west,inner sep=1pt](PP2) at (4.50,2.00) {$\eout{2}$};
\node[scale=1.30,black!50,anchor=west,inner sep=1pt](PP3) at (4.50,3.00) {$\eout{3}$};
\node[scale=1.30,red,anchor=west,inner sep=1pt](PP4) at (4.50,4.00) {$\eout{4}$};
\node[scale=1.30,red,anchor=west,inner sep=1pt](PP5) at (4.50,5.00) {$\eout{5}$};
\node[scale=1.30,black!50,anchor=south,inner sep=1pt](QQ6) at (4.00,5.50) {$\eout{6}$};
\node[scale=1.30,black!50,anchor=south,inner sep=1pt](QQ7) at (3.00,5.50) {$\eout{7}$};
\node[scale=1.30,blue,anchor=south,inner sep=1pt](QQ8) at (2.00,5.50) {$\eout{8}$};
\node[scale=1.30,black!50,anchor=south,inner sep=1pt](QQ9) at (1.00,5.50) {$\eout{9}$};
\draw[black!50,line width=1.5pt] (0.50,4.00)--(1.50,4.00);\draw[black!50,line width=1.5pt] (1.00,3.50)--(1.00,4.50);\draw[black!50,line width=1.5pt] (0.50,5.00)--(1.50,5.00);\draw[black!50,line width=1.5pt] (1.00,4.50)--(1.00,5.50);\draw[black!50,line width=1.5pt] (1.50,4.00)--(2.50,4.00);\draw[black!50,line width=1.5pt] (2.00,3.50)--(2.00,4.50);\draw[black!50,line width=1.5pt] (1.50,5.00)--(2.50,5.00);\draw[black!50,line width=1.5pt] (2.00,4.50)--(2.00,5.50);\draw[black!50,line width=1.5pt] (2.50,4.00)--(3.50,4.00);\draw[black!50,line width=1.5pt] (3.00,3.50)--(3.00,4.50);\draw[black!50,line width=1.5pt] (2.50,5.00)--(3.50,5.00);\draw[black!50,line width=1.5pt] (3.00,4.50)--(3.00,5.50);\draw[black!50,line width=1.5pt] (3.50,4.00)--(4.50,4.00);\draw[black!50,line width=1.5pt] (4.00,3.50)--(4.00,4.50);\draw[black!50,line width=1.5pt] (3.50,5.00)--(4.50,5.00);\draw[black!50,line width=1.5pt] (4.00,4.50)--(4.00,5.50);\node[scale=1.80,black](A) at (1,-1) {$x_{1}$};
\node[scale=1.80,black](A) at (2,-1) {$x_{2}$};
\node[scale=1.80,black](A) at (3,-1) {$x_{3}$};
\node[scale=1.80,black](A) at (4,-1) {$x_{4}$};
\node[scale=1.80,green!70!black](A) at (-1,1) {$y_{5}$};
\node[scale=1.80,green!70!black](A) at (-1,2) {$y_{4}$};
\node[scale=1.80,green!70!black](A) at (-1,3) {$y_{3}$};
\node[scale=1.80,green!70!black](A) at (-1,4) {$y_{2}$};
\node[scale=1.80,green!70!black](A) at (-1,5) {$y_{1}$};
\draw[blue,line width=3pt, rounded corners=12] (2.00,3.50)--(2.00,5.50);\draw[red,line width=3pt, rounded corners=12] (0.50,4.00)--(4.50,4.00);\draw[red,line width=3pt, rounded corners=12] (0.50,5.00)--(4.50,5.00);\draw[blue,line width=3pt] (1.00,0.50) to[in=-90,out=90] (2.00,3.50);
\end{tikzpicture}}
\end{tabular}
};
\node[scale=0.85](RAVNO1) at (2,-2.35) {$=$};
\node[scale=0.85](RAVNO2) at (-2,-2.35) {$=$};
\node[scale=0.85](Z1) at (-3.75,-2.35) {$\PF^{\textcolor{red}{\H_0},\textcolor{blue}{\V}}(\bx,\by)$};
\node[scale=0.85](Z1) at (0.25,-2.35) {$\PF^{\textcolor{red}{\flip(\H_0)},\textcolor{blue}{\V}}(\bx,\textcolor{green!70!black}{\by'})$};
\node[scale=0.85](Z1) at (4.25,-2.35) {$\PF^{\textcolor{red}{\flip(\H_0)},\textcolor{blue}{\V}}(\bx,\textcolor{green!70!black}{\rev(\by)})$};
\end{tikzpicture}}
}}
\usefig\FIGflipbase{FIGflipbase}
\def\FIGhsr{\noindent\resizebox{0.65\textwidth}{!}{\begin{tabular}{ccc}
\scalebox{0.50}{\begin{tikzpicture}[baseline=(ZUZU.base)]
\coordinate(ZUZU) at (1,1);\drawgrid{1}{1}{4}{6}\draw[line width=2pt,draw=black!50] (4.50,0.50)--(3.50,0.50)--(2.50,0.50)--(1.50,0.50)--(0.50,0.50)--(0.50,1.50)--(0.50,2.50)--(0.50,3.50)--(0.50,4.50)--(0.50,5.50)--(0.50,6.50);

\node[scale=1.30,red,anchor=east,inner sep=1pt](PP7) at (0.50,3.00) {$\ein{l}$};
\node[scale=1.30,red,anchor=east,inner sep=1pt](PP8) at (0.50,4.00) {$\ein{l+1}$};
\draw[line width=2pt,draw=black!50] (4.50,0.50)--(4.50,1.50)--(4.50,2.50)--(4.50,3.50)--(4.50,4.50)--(4.50,5.50)--(4.50,6.50)--(3.50,6.50)--(2.50,6.50)--(1.50,6.50)--(0.50,6.50);

\node[scale=1.30,black!50,anchor=west,inner sep=1pt](PP3) at (4.50,3.00) {$\eout{r}$};
\node[scale=1.30,red,anchor=west,inner sep=1pt](PP4) at (4.50,4.00) {$\eout{r+1}$};
\draw[red,line width=3pt] (0.50,1.00) to[in=180,out=0] (4.50,4.00);
\draw[red,line width=3pt] (0.50,3.00) to[in=180,out=0] (4.50,6.00);
\draw[red,line width=3pt] (0.50,4.00) to[in=180,out=0] (4.50,5.00);
\node[xscale=3.2,yscale=2,anchor=south,black!50] (B) at (2.5,0.2) {\rotatebox{-90}{$\Bigg\{$}};
\draw[black!50,line width=3pt,dashed] (2.5,1) to[in=180,out=90] (4.50,3.00);
\end{tikzpicture}}
&\scalebox{0.50}{\begin{tikzpicture}[baseline=(ZUZU.base)]
\coordinate(ZUZU) at (1,1);\drawgrid{1}{1}{4}{6}\draw[line width=2pt,draw=black!50] (4.50,0.50)--(3.50,0.50)--(2.50,0.50)--(1.50,0.50)--(0.50,0.50)--(0.50,1.50)--(0.50,2.50)--(0.50,3.50)--(0.50,4.50)--(0.50,5.50)--(0.50,6.50);

\node[scale=1.30,red,anchor=east,inner sep=1pt](PP7) at (0.50,3.00) {$\ein{l}$};
\node[scale=1.30,red,anchor=east,inner sep=1pt](PP8) at (0.50,4.00) {$\ein{l+1}$};
\draw[line width=2pt,draw=black!50] (4.50,0.50)--(4.50,1.50)--(4.50,2.50)--(4.50,3.50)--(4.50,4.50)--(4.50,5.50)--(4.50,6.50)--(3.50,6.50)--(2.50,6.50)--(1.50,6.50)--(0.50,6.50);

\node[scale=1.30,red,anchor=west,inner sep=1pt](PP3) at (4.50,3.00) {$\eout{r}$};
\node[scale=1.30,black!50,anchor=west,inner sep=1pt](PP4) at (4.50,4.00) {$\eout{r+1}$};
\draw[red,line width=3pt] (0.50,1.00) to[in=180,out=0] (4.50,3.00);
\draw[red,line width=3pt] (0.50,3.00) to[in=180,out=0] (4.50,6.00);
\draw[red,line width=3pt] (0.50,4.00) to[in=180,out=0] (4.50,5.00);
\node[xscale=3.2,yscale=2,anchor=south,black!50] (B) at (2.5,0.2) {\rotatebox{-90}{$\Bigg\{$}};
\draw[black!50,line width=3pt,dashed] (2.5,1) to[in=180,out=90] (4.50,4.00);
\end{tikzpicture}}
&\scalebox{0.50}{\begin{tikzpicture}[baseline=(ZUZU.base)]
\coordinate(ZUZU) at (1,1);\drawgrid{1}{1}{4}{6}\draw[line width=2pt,draw=black!50] (4.50,0.50)--(3.50,0.50)--(2.50,0.50)--(1.50,0.50)--(0.50,0.50)--(0.50,1.50)--(0.50,2.50)--(0.50,3.50)--(0.50,4.50)--(0.50,5.50)--(0.50,6.50);

\node[scale=1.30,red,anchor=east,inner sep=1pt](PP7) at (0.50,3.00) {$\ein{l}$};
\node[scale=1.30,red,anchor=east,inner sep=1pt](PP8) at (0.50,4.00) {$\ein{l+1}$};
\draw[line width=2pt,draw=black!50] (4.50,0.50)--(4.50,1.50)--(4.50,2.50)--(4.50,3.50)--(4.50,4.50)--(4.50,5.50)--(4.50,6.50)--(3.50,6.50)--(2.50,6.50)--(1.50,6.50)--(0.50,6.50);

\node[scale=1.30,red,anchor=west,inner sep=1pt](PP3) at (4.50,3.00) {$\eout{r}$};
\node[scale=1.30,black!50,anchor=west,inner sep=1pt](PP4) at (4.50,4.00) {$\eout{r+1}$};
\draw[red,line width=3pt] (0.50,1.00) to[in=180,out=0] (4.50,3.00);
\draw[red,line width=3pt] (0.50,4.00) to[in=180,out=0] (4.50,6.00);
\draw[red,line width=3pt] (0.50,3.00) to[in=180,out=0] (4.50,5.00);
\node[xscale=3.2,yscale=2,anchor=south,black!50] (B) at (2.5,0.2) {\rotatebox{-90}{$\Bigg\{$}};
\draw[black!50,line width=3pt,dashed] (2.5,1) to[in=180,out=90] (4.50,4.00);
\end{tikzpicture}}
\\ $\textcolor{red}{\H}$ & $\textcolor{red}{\H'}=\H\cdot s_r$ & $\textcolor{red}{\H''}=s_l\cdot \H\cdot s_r$\end{tabular}
}}
\usefig\FIGhsr{FIGhsr}
\def\FIGdoubleh{\noindent\resizebox{0.8\textwidth}{!}{\setlength{\tabcolsep}{0pt}\begin{tabular}{ccc}
\scalebox{0.50}{\begin{tikzpicture}[baseline=(ZUZU.base)]
\coordinate(ZUZU) at (5,6);\drawgrid{-1}{-1}{9}{11}\draw[line width=1pt,->,black!50] (-0.6,-1)--(-0.6,11.50);
\draw[line width=1pt,->,black!50] (-1,-0.6)--(9.50,-0.6);
\draw[line width=2pt,draw=black!50] (9.50,0.50)--(8.50,0.50)--(7.50,0.50)--(6.50,0.50)--(5.50,0.50)--(4.50,0.50)--(4.50,1.50)--(3.50,1.50)--(3.50,2.50)--(3.50,3.50)--(3.50,4.50)--(2.50,4.50)--(2.50,5.50)--(2.50,6.50)--(2.50,7.50)--(1.50,7.50)--(1.50,8.50)--(1.50,9.50)--(0.50,9.50)--(0.50,10.50)--(0.50,11.50);

\draw[line width=2pt,draw=black!50] (9.50,0.50)--(9.50,1.50)--(9.50,2.50)--(9.50,3.50)--(9.50,4.50)--(9.50,5.50)--(8.50,5.50)--(7.50,5.50)--(7.50,6.50)--(7.50,7.50)--(7.50,8.50)--(7.50,9.50)--(7.50,10.50)--(7.50,11.50)--(6.50,11.50)--(5.50,11.50)--(4.50,11.50)--(3.50,11.50)--(2.50,11.50)--(1.50,11.50)--(0.50,11.50);

\node[scale=1.65,black](A) at (1,-1) {$x_{1}$};
\node[scale=1.65,black](A) at (2,-1) {$x_{2}$};
\node[scale=1.65,black](A) at (3,-1) {$x_{3}$};
\node[scale=1.65,black](A) at (4,-1) {$x_{4}$};
\node[scale=1.65,black](A) at (5,-1) {$x_{5}$};
\node[scale=1.65,black](A) at (6,-1) {$x_{6}$};
\node[scale=1.65,black](A) at (7,-1) {$x_{7}$};
\node[scale=1.65,black](A) at (8,-1) {$x_{8}$};
\node[scale=1.65,black](A) at (9,-1) {$x_{9}$};
\node[scale=1.65,black](A) at (-1,1) {$y_{1}$};
\node[scale=1.65,black](A) at (-1,2) {$y_{2}$};
\node[scale=1.65,black](A) at (-1,3) {$y_{3}$};
\node[scale=1.65,black](A) at (-1,4) {$y_{4}$};
\node[scale=1.65,black](A) at (-1,5) {$y_{5}$};
\node[scale=1.65,black](A) at (-1,6) {$y_{6}$};
\node[scale=1.65,black](A) at (-1,7) {$y_{7}$};
\node[scale=1.65,black](A) at (-1,8) {$y_{8}$};
\node[scale=1.65,black](A) at (-1,9) {$y_{9}$};
\node[scale=1.65,black](A) at (-1,10) {$y_{10}$};
\node[scale=1.65,black](A) at (-1,11) {$y_{11}$};
\draw[line width=1.5pt,black!50] (4.50,0.50)--(6.50,11.50);
\node[scale=2.00] (A) at (4.84,2.37) {\contour{\cntcol}{\textcolor{black!50}{$\cut_{6}$}}};\draw[line width=1.5pt,black!50] (3.50,1.50)--(9.50,2.50);
\node[scale=2.00] (A) at (6.50,2.00) {\contour{\cntcol}{\textcolor{black!50}{$\cut_{7}$}}};\draw[line width=1.5pt,black!50] (6.50,0.50)--(9.50,1.50);
\node[scale=2.00] (A) at (7.10,0.90) {\contour{\cntcol}{\textcolor{black!50}{$\cut_{8}$}}};\draw[line width=1.5pt,black!50] (0.50,9.50)--(7.50,11.50);
\node[scale=2.00] (A) at (4.70,10.70) {\contour{\cntcol}{\textcolor{black!50}{$\cut_{9}$}}};\draw[line width=1.5pt,black!50] (1.50,9.50)--(3.50,11.50);
\node[scale=2.00] (A) at (2.90,10.90) {\contour{\cntcol}{\textcolor{black!50}{$\cut_{10}$}}};\draw[line width=2pt,blue] (2.50,4.50)--(9.50,5.50);
\node[scale=2.00] (A) at (8.45,5.10) {\contour{\cntcol}{\textcolor{blue}{$\cut_{4}$}}};\draw[line width=2pt,blue] (2.50,5.50)--(7.50,6.50);
\node[scale=2.00] (A) at (3.00,5.60) {\contour{\cntcol}{\textcolor{blue}{$\cut_{5}$}}};\draw[line width=2pt,red] (3.50,2.50)--(7.50,8.50);
\node[scale=2.00] (A) at (6.70,7.30) {\contour{\cntcol}{\textcolor{red}{$\cut_{2}$}}};\draw[line width=2pt,red] (3.50,3.50)--(7.50,6.50);
\node[scale=2.00] (A) at (6.30,5.60) {\contour{\cntcol}{\textcolor{red}{$\cut_{3}$}}};\draw[line width=2pt,red] (3.50,4.50)--(7.50,9.50);
\node[scale=2.00] (A) at (6.50,8.25) {\contour{\cntcol}{\textcolor{red}{$\cut_{1}$}}};\end{tikzpicture}}
&\scalebox{0.60}{\begin{tikzpicture}[baseline=(ZUZU.base)]
\coordinate(ZUZU) at (1,1);\node[scale=1.1,red] (A) at (0,2){$\cut_1,\cut_2,\cut_3:\flip$};
\node[scale=1.1,blue] (A) at (0,1){$\cut_4,\cut_5:$ shift $\uparrow$};
\node[scale=3,black] (A) at (0,0){$\longrightarrow$};
\end{tikzpicture}}
&\scalebox{0.50}{\begin{tikzpicture}[baseline=(ZUZU.base)]
\coordinate(ZUZU) at (5,6);\drawgrid{-1}{-1}{9}{11}\draw[line width=1pt,->,black!50] (-0.6,-1)--(-0.6,11.50);
\draw[line width=1pt,->,black!50] (-1,-0.6)--(9.50,-0.6);
\draw[line width=2pt,draw=black!50] (9.50,0.50)--(8.50,0.50)--(7.50,0.50)--(6.50,0.50)--(5.50,0.50)--(4.50,0.50)--(4.50,1.50)--(3.50,1.50)--(3.50,2.50)--(3.50,3.50)--(3.50,4.50)--(3.50,5.50)--(2.50,5.50)--(2.50,6.50)--(2.50,7.50)--(2.50,8.50)--(2.50,9.50)--(1.50,9.50)--(0.50,9.50)--(0.50,10.50)--(0.50,11.50);

\draw[line width=2pt,draw=black!50] (9.50,0.50)--(9.50,1.50)--(9.50,2.50)--(9.50,3.50)--(9.50,4.50)--(9.50,5.50)--(9.50,6.50)--(8.50,6.50)--(7.50,6.50)--(7.50,7.50)--(7.50,8.50)--(7.50,9.50)--(7.50,10.50)--(7.50,11.50)--(6.50,11.50)--(5.50,11.50)--(4.50,11.50)--(3.50,11.50)--(2.50,11.50)--(1.50,11.50)--(0.50,11.50);

\node[scale=1.65,black](A) at (1,-1) {$x_{1}$};
\node[scale=1.65,black](A) at (2,-1) {$x_{2}$};
\node[scale=1.65,black](A) at (3,-1) {$x_{3}$};
\node[scale=1.65,black](A) at (4,-1) {$x_{4}$};
\node[scale=1.65,black](A) at (5,-1) {$x_{5}$};
\node[scale=1.65,black](A) at (6,-1) {$x_{6}$};
\node[scale=1.65,black](A) at (7,-1) {$x_{7}$};
\node[scale=1.65,black](A) at (8,-1) {$x_{8}$};
\node[scale=1.65,black](A) at (9,-1) {$x_{9}$};
\node[scale=1.65,black](A) at (-1,1) {$y_{1}$};
\node[scale=1.65,black](A) at (-1,2) {$y_{2}$};
\node[scale=1.65,green!70!black](A) at (-1,3) {$y_{9}$};
\node[scale=1.65,green!70!black](A) at (-1,4) {$y_{8}$};
\node[scale=1.65,green!70!black](A) at (-1,5) {$y_{7}$};
\node[scale=1.65,blue](A) at (-1,6) {$y_{5}$};
\node[scale=1.65,blue](A) at (-1,7) {$y_{6}$};
\node[scale=1.65,green!70!black](A) at (-1,8) {$y_{4}$};
\node[scale=1.65,green!70!black](A) at (-1,9) {$y_{3}$};
\node[scale=1.65,black](A) at (-1,10) {$y_{10}$};
\node[scale=1.65,black](A) at (-1,11) {$y_{11}$};
\draw[line width=1.5pt,black!50] (4.50,0.50)--(6.50,11.50);
\node[scale=2.00] (A) at (4.84,2.37) {\contour{\cntcol}{\textcolor{black!50}{$\cut'_{6}$}}};\draw[line width=1.5pt,black!50] (3.50,1.50)--(9.50,2.50);
\node[scale=2.00] (A) at (6.50,2.00) {\contour{\cntcol}{\textcolor{black!50}{$\cut'_{7}$}}};\draw[line width=1.5pt,black!50] (6.50,0.50)--(9.50,1.50);
\node[scale=2.00] (A) at (7.10,0.90) {\contour{\cntcol}{\textcolor{black!50}{$\cut'_{8}$}}};\draw[line width=1.5pt,black!50] (0.50,9.50)--(7.50,11.50);
\node[scale=2.00] (A) at (4.70,10.70) {\contour{\cntcol}{\textcolor{black!50}{$\cut'_{9}$}}};\draw[line width=1.5pt,black!50] (1.50,9.50)--(3.50,11.50);
\node[scale=2.00] (A) at (2.90,10.90) {\contour{\cntcol}{\textcolor{black!50}{$\cut'_{10}$}}};\draw[line width=2pt,blue] (2.50,5.50)--(9.50,6.50);
\node[scale=2.00] (A) at (8.45,6.10) {\contour{\cntcol}{\textcolor{blue}{$\cut'_{4}$}}};\draw[line width=2pt,blue] (2.50,6.50)--(7.50,7.50);
\node[scale=2.00] (A) at (3.00,6.60) {\contour{\cntcol}{\textcolor{blue}{$\cut'_{5}$}}};\draw[line width=2pt,red] (3.50,2.50)--(7.50,7.50);
\node[scale=2.00] (A) at (4.30,3.50) {\contour{\cntcol}{\textcolor{red}{$\cut'_{1}$}}};\draw[line width=2pt,red] (3.50,3.50)--(7.50,9.50);
\node[scale=2.00] (A) at (4.50,5.00) {\contour{\cntcol}{\textcolor{red}{$\cut'_{2}$}}};\draw[line width=2pt,red] (3.50,5.50)--(7.50,8.50);
\node[scale=2.00] (A) at (4.50,6.25) {\contour{\cntcol}{\textcolor{red}{$\cut'_{3}$}}};\end{tikzpicture}}
\end{tabular}
}}
\usefig\FIGdoubleh{FIGdoubleh}
\def\FIGlocalh{\noindent\scalebox{1}{\setlength{\tabcolsep}{0pt}\begin{tabular}{ccc}
\scalebox{0.50}{\begin{tikzpicture}[baseline=(ZUZU.base)]
\coordinate(ZUZU) at (5,4);\drawgrid{-1}{-1}{5}{8}\draw[line width=1pt,->,black!50] (-0.6,-1)--(-0.6,8.50);
\draw[line width=1pt,->,black!50] (-1,-0.6)--(5.50,-0.6);
\draw[line width=2pt,draw=black!50] (5.50,0.50)--(4.50,0.50)--(3.50,0.50)--(2.50,0.50)--(1.50,0.50)--(1.50,1.50)--(1.50,2.50)--(1.50,3.50)--(1.50,4.50)--(1.50,5.50)--(1.50,6.50)--(1.50,7.50)--(0.50,7.50)--(0.50,8.50);

\draw[line width=2pt,draw=black!50] (5.50,0.50)--(5.50,1.50)--(5.50,2.50)--(4.50,2.50)--(4.50,3.50)--(4.50,4.50)--(4.50,5.50)--(4.50,6.50)--(4.50,7.50)--(3.50,7.50)--(3.50,8.50)--(2.50,8.50)--(1.50,8.50)--(0.50,8.50);

\node[scale=1.65,black](A) at (1,-1) {$x_{1}$};
\node[scale=1.65,black](A) at (2,-1) {$x_{2}$};
\node[scale=1.65,black](A) at (3,-1) {$x_{3}$};
\node[scale=1.65,black](A) at (4,-1) {$x_{4}$};
\node[scale=1.65,black](A) at (5,-1) {$x_{5}$};
\node[scale=1.65,black](A) at (-1,1) {$y_{1}$};
\node[scale=1.65,black](A) at (-1,2) {$y_{2}$};
\node[scale=1.65,black](A) at (-1,3) {$y_{3}$};
\node[scale=1.65,black](A) at (-1,4) {$y_{4}$};
\node[scale=1.65,black](A) at (-1,5) {$y_{5}$};
\node[scale=1.65,black](A) at (-1,6) {$y_{6}$};
\node[scale=1.65,black](A) at (-1,7) {$y_{7}$};
\node[scale=1.65,black](A) at (-1,8) {$y_{8}$};
\draw[line width=1.5pt,green!70!black,dashed](1.60,2.60)rectangle(4.40,7.40);
\draw[line width=1.5pt,black!50] (1.50,0.50)--(3.50,8.50);
\node[scale=1.50] (A) at (3.10,6.90) {\contour{\cntcol}{\textcolor{black!50}{$\cut_{4}$}}};\draw[line width=1.5pt,black!50] (1.50,1.50)--(5.50,2.50);
\node[scale=1.50] (A) at (2.70,1.80) {\contour{\cntcol}{\textcolor{black!50}{$\cut_{5}$}}};\draw[line width=1.5pt,black!50] (2.50,0.50)--(5.50,2.50);
\node[scale=1.50] (A) at (3.10,1.10) {\contour{\cntcol}{\textcolor{black!50}{$\cut_{6}$}}};\draw[line width=1.5pt,black!50] (0.50,7.50)--(3.50,8.50);
\node[scale=1.50] (A) at (2.00,8.00) {\contour{\cntcol}{\textcolor{black!50}{$\cut_{7}$}}};\draw[line width=2pt,red] (1.50,3.50)--(4.50,5.50);
\node[scale=1.50] (A) at (3.60,4.90) {\contour{\cntcol}{\textcolor{red}{$\cut_{1}$}}};\draw[line width=2pt,red] (1.50,4.50)--(4.50,7.50);
\node[scale=1.50] (A) at (2.10,5.10) {\contour{\cntcol}{\textcolor{red}{$\cut_{2}$}}};\draw[line width=2pt,red] (1.50,5.50)--(4.50,6.50);
\node[scale=1.50] (A) at (4.05,6.35) {\contour{\cntcol}{\textcolor{red}{$\cut_{3}$}}};\end{tikzpicture}}
&\scalebox{0.50}{\begin{tikzpicture}[baseline=(ZUZU.base)]
\coordinate(ZUZU) at (1,1);\node[scale=1.1,red] (A) at (0,1){$\cut_1,\cut_2,\cut_3:\flip$};
\node[scale=3,black] (A) at (0,0){$\longrightarrow$};
\end{tikzpicture}}
&\scalebox{0.50}{\begin{tikzpicture}[baseline=(ZUZU.base)]
\coordinate(ZUZU) at (5,4);\drawgrid{-1}{-1}{5}{8}\draw[line width=1pt,->,black!50] (-0.6,-1)--(-0.6,8.50);
\draw[line width=1pt,->,black!50] (-1,-0.6)--(5.50,-0.6);
\draw[line width=2pt,draw=black!50] (5.50,0.50)--(4.50,0.50)--(3.50,0.50)--(2.50,0.50)--(1.50,0.50)--(1.50,1.50)--(1.50,2.50)--(1.50,3.50)--(1.50,4.50)--(1.50,5.50)--(1.50,6.50)--(1.50,7.50)--(0.50,7.50)--(0.50,8.50);

\draw[line width=2pt,draw=black!50] (5.50,0.50)--(5.50,1.50)--(5.50,2.50)--(4.50,2.50)--(4.50,3.50)--(4.50,4.50)--(4.50,5.50)--(4.50,6.50)--(4.50,7.50)--(3.50,7.50)--(3.50,8.50)--(2.50,8.50)--(1.50,8.50)--(0.50,8.50);

\node[scale=1.65,black](A) at (1,-1) {$x_{1}$};
\node[scale=1.65,black](A) at (2,-1) {$x_{2}$};
\node[scale=1.65,black](A) at (3,-1) {$x_{3}$};
\node[scale=1.65,black](A) at (4,-1) {$x_{4}$};
\node[scale=1.65,black](A) at (5,-1) {$x_{5}$};
\node[scale=1.65,black](A) at (-1,1) {$y_{1}$};
\node[scale=1.65,black](A) at (-1,2) {$y_{2}$};
\node[scale=1.65,green!70!black](A) at (-1,3) {$y_{7}$};
\node[scale=1.65,green!70!black](A) at (-1,4) {$y_{6}$};
\node[scale=1.65,green!70!black](A) at (-1,5) {$y_{5}$};
\node[scale=1.65,green!70!black](A) at (-1,6) {$y_{4}$};
\node[scale=1.65,green!70!black](A) at (-1,7) {$y_{3}$};
\node[scale=1.65,black](A) at (-1,8) {$y_{8}$};
\draw[line width=1.5pt,green!70!black,dashed](1.60,2.60)rectangle(4.40,7.40);
\draw[line width=1.5pt,black!50] (1.50,0.50)--(3.50,8.50);
\node[scale=1.50] (A) at (3.10,6.90) {\contour{\cntcol}{\textcolor{black!50}{$\cut'_{4}$}}};\draw[line width=1.5pt,black!50] (1.50,1.50)--(5.50,2.50);
\node[scale=1.50] (A) at (2.70,1.80) {\contour{\cntcol}{\textcolor{black!50}{$\cut'_{5}$}}};\draw[line width=1.5pt,black!50] (2.50,0.50)--(5.50,2.50);
\node[scale=1.50] (A) at (3.10,1.10) {\contour{\cntcol}{\textcolor{black!50}{$\cut'_{6}$}}};\draw[line width=1.5pt,black!50] (0.50,7.50)--(3.50,8.50);
\node[scale=1.50] (A) at (2.00,8.00) {\contour{\cntcol}{\textcolor{black!50}{$\cut'_{7}$}}};\draw[line width=2pt,red] (1.50,4.50)--(4.50,6.50);
\node[scale=1.50] (A) at (2.16,4.94) {\contour{\cntcol}{\textcolor{red}{$\cut'_{1}$}}};\draw[line width=2pt,red] (1.50,2.50)--(4.50,5.50);
\node[scale=1.50] (A) at (3.90,4.90) {\contour{\cntcol}{\textcolor{red}{$\cut'_{2}$}}};\draw[line width=2pt,red] (1.50,3.50)--(4.50,4.50);
\node[scale=1.50] (A) at (1.95,3.65) {\contour{\cntcol}{\textcolor{red}{$\cut'_{3}$}}};\end{tikzpicture}}
\end{tabular}
}
}
\usefig\FIGlocalh{FIGlocalh}
\def\FIGglobalh{\noindent\scalebox{1}{\setlength{\tabcolsep}{0pt}\begin{tabular}{ccc}
\scalebox{0.50}{\begin{tikzpicture}[baseline=(ZUZU.base)]
\coordinate(ZUZU) at (5,4);\drawgrid{-1}{-1}{6}{8}\draw[line width=1pt,->,black!50] (-0.6,-1)--(-0.6,8.50);
\draw[line width=1pt,->,black!50] (-1,-0.6)--(6.50,-0.6);
\draw[line width=2pt,draw=black!50] (6.50,0.50)--(5.50,0.50)--(4.50,0.50)--(3.50,0.50)--(2.50,0.50)--(1.50,0.50)--(1.50,1.50)--(1.50,2.50)--(0.50,2.50)--(0.50,3.50)--(0.50,4.50)--(0.50,5.50)--(0.50,6.50)--(0.50,7.50)--(0.50,8.50);

\draw[line width=2pt,draw=black!50] (6.50,0.50)--(6.50,1.50)--(6.50,2.50)--(6.50,3.50)--(6.50,4.50)--(5.50,4.50)--(5.50,5.50)--(5.50,6.50)--(4.50,6.50)--(4.50,7.50)--(4.50,8.50)--(3.50,8.50)--(2.50,8.50)--(1.50,8.50)--(0.50,8.50);

\node[scale=1.65,black](A) at (1,-1) {$x_{1}$};
\node[scale=1.65,black](A) at (2,-1) {$x_{2}$};
\node[scale=1.65,black](A) at (3,-1) {$x_{3}$};
\node[scale=1.65,black](A) at (4,-1) {$x_{4}$};
\node[scale=1.65,black](A) at (5,-1) {$x_{5}$};
\node[scale=1.65,black](A) at (6,-1) {$x_{6}$};
\node[scale=1.65,black](A) at (-1,1) {$y_{1}$};
\node[scale=1.65,black](A) at (-1,2) {$y_{2}$};
\node[scale=1.65,black](A) at (-1,3) {$y_{3}$};
\node[scale=1.65,black](A) at (-1,4) {$y_{4}$};
\node[scale=1.65,black](A) at (-1,5) {$y_{5}$};
\node[scale=1.65,black](A) at (-1,6) {$y_{6}$};
\node[scale=1.65,black](A) at (-1,7) {$y_{7}$};
\node[scale=1.65,black](A) at (-1,8) {$y_{8}$};
\draw[line width=1.5pt,green!70!black,dashed](1.60,2.60)rectangle(4.40,7.40);
\draw[line width=1.5pt,black!50] (1.50,0.50)--(3.50,8.50);
\node[scale=1.50] (A) at (2.00,1.70) {\contour{\cntcol}{\textcolor{black!50}{$\cut_{4}$}}};\draw[line width=1.5pt,black!50] (2.50,0.50)--(4.50,7.50);
\node[scale=1.50] (A) at (2.90,1.90) {\contour{\cntcol}{\textcolor{black!50}{$\cut_{5}$}}};\draw[line width=2pt,red] (0.50,2.50)--(6.50,4.50);
\node[scale=1.50] (A) at (5.30,4.10) {\contour{\cntcol}{\textcolor{red}{$\cut_{1}$}}};\draw[line width=2pt,red] (0.50,3.50)--(5.50,6.50);
\node[scale=1.50] (A) at (4.50,5.90) {\contour{\cntcol}{\textcolor{red}{$\cut_{2}$}}};\draw[line width=2pt,red] (0.50,4.50)--(5.50,5.50);
\node[scale=1.50] (A) at (1.00,4.60) {\contour{\cntcol}{\textcolor{red}{$\cut_{3}$}}};\end{tikzpicture}}
&\scalebox{0.50}{\begin{tikzpicture}[baseline=(ZUZU.base)]
\coordinate(ZUZU) at (1,1);\node[scale=1.1,red] (A) at (0,1){$\cut_1,\cut_2,\cut_3:\flip$};
\node[scale=3,black] (A) at (0,0){$\longrightarrow$};
\end{tikzpicture}}
&\scalebox{0.50}{\begin{tikzpicture}[baseline=(ZUZU.base)]
\coordinate(ZUZU) at (5,4);\drawgrid{-1}{-1}{6}{8}\draw[line width=1pt,->,black!50] (-0.6,-1)--(-0.6,8.50);
\draw[line width=1pt,->,black!50] (-1,-0.6)--(6.50,-0.6);
\draw[line width=2pt,draw=black!50] (6.50,0.50)--(5.50,0.50)--(4.50,0.50)--(3.50,0.50)--(2.50,0.50)--(2.50,1.50)--(2.50,2.50)--(1.50,2.50)--(1.50,3.50)--(1.50,4.50)--(1.50,5.50)--(0.50,5.50)--(0.50,6.50)--(0.50,7.50)--(0.50,8.50);

\draw[line width=2pt,draw=black!50] (6.50,0.50)--(6.50,1.50)--(6.50,2.50)--(6.50,3.50)--(6.50,4.50)--(6.50,5.50)--(6.50,6.50)--(6.50,7.50)--(5.50,7.50)--(5.50,8.50)--(4.50,8.50)--(3.50,8.50)--(2.50,8.50)--(1.50,8.50)--(0.50,8.50);

\node[scale=1.65,green!70!black](A) at (1,-1) {$x_{6}$};
\node[scale=1.65,green!70!black](A) at (2,-1) {$x_{5}$};
\node[scale=1.65,black](A) at (3,-1) {$x_{2}$};
\node[scale=1.65,black](A) at (4,-1) {$x_{3}$};
\node[scale=1.65,black](A) at (5,-1) {$x_{4}$};
\node[scale=1.65,green!70!black](A) at (6,-1) {$x_{1}$};
\node[scale=1.65,black](A) at (-1,1) {$y_{1}$};
\node[scale=1.65,black](A) at (-1,2) {$y_{2}$};
\node[scale=1.65,green!70!black](A) at (-1,3) {$y_{7}$};
\node[scale=1.65,green!70!black](A) at (-1,4) {$y_{6}$};
\node[scale=1.65,green!70!black](A) at (-1,5) {$y_{5}$};
\node[scale=1.65,green!70!black](A) at (-1,6) {$y_{4}$};
\node[scale=1.65,green!70!black](A) at (-1,7) {$y_{3}$};
\node[scale=1.65,black](A) at (-1,8) {$y_{8}$};
\draw[line width=1.5pt,green!70!black,dashed](2.60,2.60)rectangle(5.40,7.40);
\draw[line width=1.5pt,black!50] (2.50,0.50)--(4.50,8.50);
\node[scale=1.50] (A) at (3.00,1.70) {\contour{\cntcol}{\textcolor{black!50}{$\cut'_{4}$}}};\draw[line width=1.5pt,black!50] (3.50,0.50)--(5.50,7.50);
\node[scale=1.50] (A) at (3.90,1.90) {\contour{\cntcol}{\textcolor{black!50}{$\cut'_{5}$}}};\draw[line width=2pt,red] (0.50,5.50)--(6.50,7.50);
\node[scale=1.50] (A) at (1.88,5.96) {\contour{\cntcol}{\textcolor{red}{$\cut'_{1}$}}};\draw[line width=2pt,red] (1.50,3.50)--(6.50,6.50);
\node[scale=1.50] (A) at (5.50,5.90) {\contour{\cntcol}{\textcolor{red}{$\cut'_{2}$}}};\draw[line width=2pt,red] (1.50,4.50)--(6.50,5.50);
\node[scale=1.50] (A) at (2.00,4.60) {\contour{\cntcol}{\textcolor{red}{$\cut'_{3}$}}};\end{tikzpicture}}
\end{tabular}

}}
\usefig\FIGglobalh{FIGglobalh}
\def\FIGgov{\noindent\resizebox{0.3\textwidth}{!}{\scalebox{0.50}{\begin{tikzpicture}[baseline=(ZUZU.base)]
\coordinate(ZUZU) at (1,3.3);\drawgrid{1}{1}{9}{7}\draw[line width=2pt,draw=black!50] (9.50,0.50)--(8.50,0.50)--(7.50,0.50)--(6.50,0.50)--(5.50,0.50)--(4.50,0.50)--(3.50,0.50)--(3.50,1.50)--(3.50,2.50)--(2.50,2.50)--(1.50,2.50)--(1.50,3.50)--(0.50,3.50)--(0.50,4.50)--(0.50,5.50)--(0.50,6.50)--(0.50,7.50);

\draw[line width=2pt,draw=black!50] (9.50,0.50)--(9.50,1.50)--(9.50,2.50)--(9.50,3.50)--(8.50,3.50)--(8.50,4.50)--(8.50,5.50)--(8.50,6.50)--(7.50,6.50)--(6.50,6.50)--(5.50,6.50)--(4.50,6.50)--(3.50,6.50)--(3.50,7.50)--(2.50,7.50)--(1.50,7.50)--(0.50,7.50);

\draw[line width=1.5pt,blue,dashed,fill=blue,draw opacity=0.3,fill opacity=0.05] (1.6,2.6)--(1.6,3.6)--(0.6,3.6)--(0.6,6.4)--(1.6,6.4)--(1.6,7.4)--(2.4,7.4)--(2.4,6.4)--(5.4,6.4)--(5.4,3.6)--(2.4,3.6)--(2.4,2.6)--cycle;
\draw[line width=1.5pt,red,dashed,fill=red,draw opacity=0.3,fill opacity=0.05] (3.6,0.6)--(3.6,3.4)--(5.6,3.4)--(5.6,6.4)--(8.4,6.4)--(8.4,3.4)--(9.4,3.4)--(9.4,0.6)--cycle;
\draw[line width=2pt,blue] (0.50,3.50)--(3.50,6.50);
\draw[line width=2pt,blue] (0.50,4.50)--(4.50,6.50);
\draw[line width=2pt,blue] (1.50,3.50)--(5.50,6.50);
\draw[line width=2pt,blue] (1.50,2.50)--(2.50,7.50);
\draw[line width=2pt,red] (3.50,2.50)--(8.50,3.50);
\draw[line width=2pt,red] (3.50,0.50)--(9.50,3.50);
\draw[line width=2pt,red] (5.50,0.50)--(8.50,6.50);
\draw[line width=2pt,red] (6.50,0.50)--(8.50,4.50);
\draw[line width=2pt,red] (7.50,0.50)--(9.50,2.50);
\end{tikzpicture}}
}}
\usefig\FIGgov{FIGgov}
\def\FIGgovdva{\noindent\resizebox{0.6\textwidth}{!}{\setlength{\tabcolsep}{0pt}\begin{tabular}{ccc}
\scalebox{0.50}{\begin{tikzpicture}[baseline=(ZUZU.base)]
\coordinate(ZUZU) at (5,3);\drawgrid{-1}{-1}{6}{6}\draw[line width=1pt,->,black!50] (-0.6,-1)--(-0.6,6.50);
\draw[line width=1pt,->,black!50] (-1,-0.6)--(6.50,-0.6);
\draw[line width=2pt,draw=black!50] (6.50,0.50)--(5.50,0.50)--(4.50,0.50)--(3.50,0.50)--(3.50,1.50)--(2.50,1.50)--(2.50,2.50)--(1.50,2.50)--(1.50,3.50)--(0.50,3.50)--(0.50,4.50)--(0.50,5.50)--(0.50,6.50);

\draw[line width=2pt,draw=black!50] (6.50,0.50)--(6.50,1.50)--(6.50,2.50)--(6.50,3.50)--(5.50,3.50)--(4.50,3.50)--(4.50,4.50)--(4.50,5.50)--(4.50,6.50)--(3.50,6.50)--(2.50,6.50)--(1.50,6.50)--(0.50,6.50);

\node[scale=1.65,black](A) at (1,-1) {$x_{1}$};
\node[scale=1.65,black](A) at (2,-1) {$x_{2}$};
\node[scale=1.65,black](A) at (3,-1) {$x_{3}$};
\node[scale=1.65,black](A) at (4,-1) {$x_{4}$};
\node[scale=1.65,black](A) at (5,-1) {$x_{5}$};
\node[scale=1.65,black](A) at (6,-1) {$x_{6}$};
\node[scale=1.65,black](A) at (-1,1) {$y_{1}$};
\node[scale=1.65,black](A) at (-1,2) {$y_{2}$};
\node[scale=1.65,black](A) at (-1,3) {$y_{3}$};
\node[scale=1.65,black](A) at (-1,4) {$y_{4}$};
\node[scale=1.65,black](A) at (-1,5) {$y_{5}$};
\node[scale=1.65,black](A) at (-1,6) {$y_{6}$};
\draw[line width=1.5pt,black!50] (3.50,0.50)--(4.50,6.50);
\node[scale=1.50] (A) at (4.05,5.60) {\contour{\cntcol}{\textcolor{black!50}{$\cut_{4}$}}};\draw[line width=1.5pt,black!50] (1.50,2.50)--(6.50,3.50);
\node[scale=1.50] (A) at (2.00,2.85) {\contour{\cntcol}{\textcolor{black!50}{$\cut_{5}$}}};\draw[line width=2pt,green!70!black] (0.50,4.50)--(2.50,6.50);
\node[scale=1.50] (A) at (1.50,5.50) {\contour{\cntcol}{\textcolor{green!70!black}{$\cut_{1}$}}};\draw[line width=2pt,blue] (2.50,2.50)--(4.50,4.50);
\node[scale=1.50] (A) at (3.30,3.60) {\contour{\cntcol}{\textcolor{blue}{$\cut_{2}$}}};\draw[line width=2pt,red] (2.50,1.50)--(5.50,3.50);
\node[scale=1.50] (A) at (4.55,2.70) {\contour{\cntcol}{\textcolor{red}{$\cut_{3}$}}};\end{tikzpicture}}
&\scalebox{0.50}{\begin{tikzpicture}[baseline=(ZUZU.base)]
\coordinate(ZUZU) at (1,1);\node[scale=2,black] (A) at (0,0){$\longrightarrow$};
\end{tikzpicture}}
&\scalebox{0.50}{\begin{tikzpicture}[baseline=(ZUZU.base)]
\coordinate(ZUZU) at (5,3);\drawgrid{-1}{-1}{6}{6}\draw[line width=1pt,->,black!50] (-0.6,-1)--(-0.6,6.50);
\draw[line width=1pt,->,black!50] (-1,-0.6)--(6.50,-0.6);
\draw[line width=2pt,draw=black!50] (6.50,0.50)--(5.50,0.50)--(4.50,0.50)--(4.50,1.50)--(3.50,1.50)--(3.50,2.50)--(2.50,2.50)--(1.50,2.50)--(1.50,3.50)--(1.50,4.50)--(0.50,4.50)--(0.50,5.50)--(0.50,6.50);

\draw[line width=2pt,draw=black!50] (6.50,0.50)--(6.50,1.50)--(6.50,2.50)--(6.50,3.50)--(5.50,3.50)--(5.50,4.50)--(5.50,5.50)--(5.50,6.50)--(4.50,6.50)--(3.50,6.50)--(2.50,6.50)--(1.50,6.50)--(0.50,6.50);

\node[scale=1.65,black](A) at (1,-1) {$x_{1}$};
\node[scale=1.65,black](A) at (2,-1) {$x_{2}$};
\node[scale=1.65,green!70!black](A) at (3,-1) {$x_{5}$};
\node[scale=1.65,blue](A) at (4,-1) {$x_{3}$};
\node[scale=1.65,blue](A) at (5,-1) {$x_{4}$};
\node[scale=1.65,black](A) at (6,-1) {$x_{6}$};
\node[scale=1.65,black](A) at (-1,1) {$y_{1}$};
\node[scale=1.65,green!70!black](A) at (-1,2) {$y_{4}$};
\node[scale=1.65,green!70!black](A) at (-1,3) {$y_{3}$};
\node[scale=1.65,green!70!black](A) at (-1,4) {$y_{2}$};
\node[scale=1.65,black](A) at (-1,5) {$y_{4}$};
\node[scale=1.65,black](A) at (-1,6) {$y_{5}$};
\draw[line width=1.5pt,black!50] (4.50,0.50)--(5.50,6.50);
\node[scale=1.50] (A) at (5.05,5.60) {\contour{\cntcol}{\textcolor{black!50}{$\cut'_{4}$}}};\draw[line width=1.5pt,black!50] (1.50,2.50)--(6.50,3.50);
\node[scale=1.50] (A) at (2.00,2.85) {\contour{\cntcol}{\textcolor{black!50}{$\cut'_{5}$}}};\draw[line width=2pt,green!70!black] (0.50,4.50)--(2.50,6.50);
\node[scale=1.50] (A) at (1.50,5.50) {\contour{\cntcol}{\textcolor{green!70!black}{$\cut'_{1}$}}};\draw[line width=2pt,blue] (3.50,1.50)--(5.50,3.50);
\node[scale=1.50] (A) at (4.10,2.10) {\contour{\cntcol}{\textcolor{blue}{$\cut'_{2}$}}};\draw[line width=2pt,red] (2.50,2.50)--(5.50,4.50);
\node[scale=1.50] (A) at (4.25,3.90) {\contour{\cntcol}{\textcolor{red}{$\cut'_{3}$}}};\end{tikzpicture}}
\end{tabular}
}}
\usefig\FIGgovdva{FIGgovdva}
\def\FIGdoubleldur{\noindent\resizebox{0.35\textwidth}{!}{\scalebox{0.50}{\begin{tikzpicture}[baseline=(ZUZU.base)]
\coordinate(ZUZU) at (5,6);\drawgrid{-1}{-1}{9}{11}\draw[line width=1pt,->,black!50] (-0.6,-1)--(-0.6,11.50);
\draw[line width=1pt,->,black!50] (-1,-0.6)--(9.50,-0.6);
\draw[line width=2pt,draw=black!50] (9.50,0.50)--(8.50,0.50)--(7.50,0.50)--(6.50,0.50)--(5.50,0.50)--(4.50,0.50)--(4.50,1.50)--(3.50,1.50)--(3.50,2.50)--(3.50,3.50)--(3.50,4.50)--(2.50,4.50)--(2.50,5.50)--(2.50,6.50)--(2.50,7.50)--(1.50,7.50)--(1.50,8.50)--(1.50,9.50)--(0.50,9.50)--(0.50,10.50)--(0.50,11.50);

\draw[line width=2pt,draw=black!50] (9.50,0.50)--(9.50,1.50)--(9.50,2.50)--(9.50,3.50)--(9.50,4.50)--(9.50,5.50)--(8.50,5.50)--(7.50,5.50)--(7.50,6.50)--(7.50,7.50)--(7.50,8.50)--(7.50,9.50)--(7.50,10.50)--(7.50,11.50)--(6.50,11.50)--(5.50,11.50)--(4.50,11.50)--(3.50,11.50)--(2.50,11.50)--(1.50,11.50)--(0.50,11.50);

\node[scale=1.65,black](A) at (4,-1) {$x_{l}$};
\node[scale=1.65,black](A) at (7,-1) {$x_{r}$};
\node[scale=1.65,green!70!black](A) at (-1,9) {$y_{u}$};
\node[scale=1.65,green!70!black](A) at (-1,3) {$y_{d}$};
\node[scale=1.65,blue](A) at (-1,5) {$y_{d'}$};
\node[scale=1.65,blue](A) at (-1,6) {$y_{u'}$};
\begin{scope}[opacity=0.3]
\draw[line width=1.5pt,black!50] (4.50,0.50)--(6.50,11.50);
\draw[line width=1.5pt,black!50] (3.50,1.50)--(9.50,2.50);
\draw[line width=1.5pt,black!50] (6.50,0.50)--(9.50,1.50);
\draw[line width=1.5pt,black!50] (0.50,9.50)--(7.50,11.50);
\draw[line width=1.5pt,black!50] (1.50,9.50)--(3.50,11.50);
\draw[line width=2pt,blue] (2.50,4.50)--(9.50,5.50);
\draw[line width=2pt,blue] (2.50,5.50)--(7.50,6.50);
\draw[line width=2pt,red] (3.50,2.50)--(7.50,8.50);
\draw[line width=2pt,red] (3.50,3.50)--(7.50,6.50);
\draw[line width=2pt,red] (3.50,4.50)--(7.50,9.50);
\end{scope}
\draw[line width=2.5pt,dashed] (3.5,2.5)--(9.5,2.5);
\draw[line width=2.5pt,dashed] (2.5,4.5)--(9.5,4.5);
\draw[line width=2.5pt,dashed] (2.5,6.5)--(7.5,6.5);
\draw[line width=2.5pt,dashed] (1.5,9.5)--(7.5,9.5);
\draw[line width=2.5pt,dashed] (3.5,4.5)--(3.5,9.5);
\draw[line width=2.5pt,dashed] (7.5,2.5)--(7.5,6.5);
\fill[opacity=0.15] (7.5,2.5)--(7.5,4.5)--(9.5,4.5)--(9.5,2.5)--cycle;
\fill[opacity=0.15] (2.5,6.5)--(2.5,7.5)--(1.5,7.5)--(1.5,9.5)--(3.5,9.5)--(3.5,6.5)--cycle;
\node[scale=2.00] (A) at (4,10.5) {$U$};
\node[scale=2.00] (A) at (5.5,8) {$A$};
\node[scale=2.00] (A) at (5.5,5.5) {$A'\subsetneq A$};
\node[scale=2.00] (A) at (5.5,3.5) {$A$};
\node[scale=2.00] (A) at (6.5,1.5) {$D$};
\node[scale=2.00] (A) at (8.5,5) {$R$};
\node[scale=2.00] (A) at (3,5.5) {$L$};
\end{tikzpicture}}
}}
\usefig\FIGdoubleldur{FIGdoubleldur}
\def\FIGrel{\noindent\resizebox{0.3\textwidth}{!}{\scalebox{0.50}{\begin{tikzpicture}[baseline=(ZUZU.base)]
\coordinate(ZUZU) at (5,6);\drawgrid{1}{1}{14}{11}\draw[line width=2pt,draw=black!50] (14.50,0.50)--(13.50,0.50)--(12.50,0.50)--(11.50,0.50)--(10.50,0.50)--(9.50,0.50)--(8.50,0.50)--(7.50,0.50)--(6.50,0.50)--(5.50,0.50)--(4.50,0.50)--(4.50,1.50)--(4.50,2.50)--(3.50,2.50)--(3.50,3.50)--(2.50,3.50)--(1.50,3.50)--(0.50,3.50)--(0.50,4.50)--(0.50,5.50)--(0.50,6.50)--(0.50,7.50)--(0.50,8.50)--(0.50,9.50)--(0.50,10.50)--(0.50,11.50);

\node[scale=2.00,blue,anchor=east,inner sep=1pt](PP11) at (4.50,1.00) {$\ein{c_1}$};
\node[scale=2.00,red,anchor=north,inner sep=1pt](PP16) at (2.00,3.50) {$\ein{c_2}$};
\draw[line width=2pt,draw=black!50] (14.50,0.50)--(14.50,1.50)--(14.50,2.50)--(14.50,3.50)--(14.50,4.50)--(14.50,5.50)--(14.50,6.50)--(14.50,7.50)--(13.50,7.50)--(13.50,8.50)--(13.50,9.50)--(13.50,10.50)--(12.50,10.50)--(11.50,10.50)--(10.50,10.50)--(9.50,10.50)--(9.50,11.50)--(8.50,11.50)--(7.50,11.50)--(6.50,11.50)--(5.50,11.50)--(4.50,11.50)--(3.50,11.50)--(2.50,11.50)--(1.50,11.50)--(0.50,11.50);

\draw[line width=1.5pt,green!70!black,dashed](4.60,2.60)rectangle(13.40,10.40);
\draw[line width=2pt] (7.5,7.5) rectangle (8.5,8.5);
\node[scale=2.6] (Aij) at (10.9,9.8) {$(i,j)$};
\draw[line width=2pt,->] (Aij.west) to[out=180,in=90] (8,8);
\draw[blue,line width=3pt, rounded corners=12] (4.50,1.00)--(5.00,1.00)--(5.00,2.00)--(7.00,2.00)--(7.00,8.00)--(7.50,8.00);\draw[red,line width=3pt, rounded corners=12] (2.00,3.50)--(2.00,4.00)--(5.00,4.00)--(5.00,5.00)--(8.00,5.00)--(8.00,7.50);\draw[line width=3pt,green!70!black] (4.50,2.50)--(13.50,10.50);
\node[scale=3.00] (A) at (10.80,8.10) {\contour{\cntcol}{\textcolor{green!70!black}{$\cut_{1}$}}};\draw[line width=1.5pt,black!50] (4.50,1.50)--(14.50,6.50);
\node[scale=3.00] (A) at (12.50,5.50) {\contour{\cntcol}{\textcolor{black!50}{$\cut_{2}$}}};\draw[line width=1.5pt,black!50] (0.50,4.50)--(9.50,11.50);
\node[scale=3.00] (A) at (4.10,7.30) {\contour{\cntcol}{\textcolor{black!50}{$\cut_{3}$}}};\end{tikzpicture}}
}}
\usefig\FIGrel{FIGrel}

\def\FIGgrex{\resizebox{0.4\textwidth}{!}{\begin{tabular}{cccc}
&\scalebox{0.50}{\begin{tikzpicture}[baseline=(ZUZU.base)]
\coordinate(ZUZU) at (1,1);\drawgrid{1}{1}{2}{2}\draw[black,line width=2pt, rounded corners=12] (0.50,1.00)--(1.00,1.00)--(1.00,1.50);\draw[black,line width=2pt, rounded corners=12] (1.00,0.50)--(1.00,1.00)--(1.50,1.00);\draw[black,line width=2pt, rounded corners=12] (0.50,2.00)--(1.00,2.00)--(1.00,2.50);\draw[black,line width=2pt, rounded corners=12] (1.00,1.50)--(1.00,2.00)--(1.50,2.00);\draw[black,line width=2pt, rounded corners=12] (1.50,1.00)--(2.00,1.00)--(2.00,1.50);\draw[black,line width=2pt, rounded corners=12] (2.00,0.50)--(2.00,1.00)--(2.50,1.00);\draw[black,line width=2pt, rounded corners=12] (1.50,2.00)--(2.00,2.00)--(2.00,2.50);\draw[black,line width=2pt, rounded corners=12] (2.00,1.50)--(2.00,2.00)--(2.50,2.00);\end{tikzpicture}}
&\hspace{0.3in}&\scalebox{0.50}{\begin{tikzpicture}[baseline=(ZUZU.base)]
\coordinate(ZUZU) at (1,1);\drawgrid{1}{1}{2}{2}\draw[black,line width=2pt] (0.50,1.00)--(1.50,1.00);\draw[black,line width=2pt] (1.00,0.50)--(1.00,1.50);\draw[black,line width=2pt] (1.50,2.00)--(2.50,2.00);\draw[black,line width=2pt] (2.00,1.50)--(2.00,2.50);\draw[black,line width=2pt, rounded corners=12] (0.50,2.00)--(1.00,2.00)--(1.00,2.50);\draw[black,line width=2pt, rounded corners=12] (1.00,1.50)--(1.00,2.00)--(1.50,2.00);\draw[black,line width=2pt, rounded corners=12] (1.50,1.00)--(2.00,1.00)--(2.00,1.50);\draw[black,line width=2pt, rounded corners=12] (2.00,0.50)--(2.00,1.00)--(2.50,1.00);\end{tikzpicture}}
\\
Probability:&\scalebox{0.80}{$\left(1-1/q\right)^4$} &\hspace{0.3in}& \scalebox{0.80}{$1/q\cdot\left(1-1/q\right)^2$}\end{tabular}
}
}

\usefig\FIGgrex{FIGgrex}

\def\FIGwdomgen{

\def\sclbx{1}
\def\gridop{0.5}
\def\nodescl{0.7}
\def\wirelw{1pt}
\def\dotscl{1}
\def\roc{3}
\def\setscl{1}
\def\rbound{14}
\def\lbound{-1}
\def\bnscl{1}

\def\leop{0.5}
\def\strandlw{1pt}

\def\xscl{0.4}
\def\yscl{0.7}

\begin{tabular}{c|c}

\begin{tikzpicture}[xscale=0.38,yscale=0.5,baseline=(Z.base)]
\coordinate(Z) at (0,3.5);
\def\ledot##1(##2,##3){
\xxdot{##1}(##2,##3)
\node[scale=\dotscl,inner sep=0pt] (A) at (##2.65,##3+1.3) {$\sp_{##1}$};
}
\def\ex{.4}
\def\xxdot##1(##2,##3){
\fill[white] ($ (##2,##3) + (-0\ex,-0.1) $) rectangle ($ (##2,##3) + (1\ex,1.1) $);
\draw[line width=\wirelw,rounded corners=\roc] ($ (##2,##3) + (-0\ex,0) $) -- (##2,##3) --($ (##2,##3) + (1,1) $)--($ (##2,##3) + (1\ex,1) $);
\draw[line width=\wirelw,rounded corners=\roc] ($ (##2,##3) + (-0\ex,1) $) -- ($ (##2,##3) + (0,1) $) --($ (##2,##3) + (1,0) $)--($ (##2,##3) + (1\ex,0) $);
}

\foreach \i in {1,...,5}{
\draw[line width=\wirelw] (\lbound,\i) -- (\rbound,\i);
\node[scale=\bnscl,anchor=east] (A) at (\lbound,\i) {$\i$};
\node[scale=\bnscl,anchor=west] (A) at (\rbound,\i) {};
}
\ledot1(0,4)
\ledot2(2,2)
\ledot3(4,3)
\ledot4(6,2)
\ledot5(8,1)
\ledot6(10,2)
\ledot7(12,3)

\node[scale=0.8] (A) at (7,0.2) {$\bi=(4,2,3,2,1,2,3)$};
\end{tikzpicture}

&

\begin{tikzpicture}[xscale=0.38,yscale=0.5,baseline=(Z.base)]
\coordinate(Z) at (0,3.5);
\def\ledot##1(##2,##3){
\xxdot{##1}(##2,##3)
\node[scale=\dotscl,inner sep=0pt] (A) at (##2.65,##3+1.3) {$\sp_{##1}$};
}
\def\ex{.4}
\def\xxdot##1(##2,##3){
\fill[white] ($ (##2,##3) + (-0\ex,-0.1) $) rectangle ($ (##2,##3) + (1\ex,1.1) $);
\draw[line width=\wirelw,rounded corners=\roc] ($ (##2,##3) + (-0\ex,0) $) -- (##2,##3) --($ (##2,##3) + (1,1) $)--($ (##2,##3) + (1\ex,1) $);
\draw[line width=\wirelw,rounded corners=\roc] ($ (##2,##3) + (-0\ex,1) $) -- ($ (##2,##3) + (0,1) $) --($ (##2,##3) + (1,0) $)--($ (##2,##3) + (1\ex,0) $);
}

\def\yxdot##1(##2,##3){
\fill[white] ($ (##2,##3) + (-0\ex,-0.1) $) rectangle ($ (##2,##3) + (1\ex,1.1) $);
\draw[line width=\wirelw,rounded corners=\roc] ($ (##2,##3) + (-0\ex,0) $) -- (##2,##3)--($ (##2,##3) + (0.5,0.5) $) --($ (##2,##3) + (1,0) $)--($ (##2,##3) + (1\ex,0) $);
\draw[line width=\wirelw,rounded corners=\roc] ($ (##2,##3) + (-0\ex,1) $) -- ($ (##2,##3) + (0,1) $)  -- ($ (##2,##3) + (0.5,0.5) $) --($ (##2,##3) + (1,1) $)--($ (##2,##3) + (1\ex,1) $);
}

\foreach \i/\j/\k in {1/3/1,2/1/3,3/2/5,4/5/2,5/4/4}{
\draw[line width=\wirelw] (\lbound,\i) -- (\rbound,\i);
\node[scale=\bnscl,anchor=east] (A) at (\lbound,\i) {$\j$};
\node[scale=\bnscl,anchor=west] (A) at (\rbound,\i) {$\k$};
}
\yxdot1(0,4)
\xxdot2(2,2)
\yxdot3(4,3)
\xxdot4(6,2)
\xxdot5(8,1)
\yxdot6(10,2)
\xxdot7(12,3)

\node[scale=0.8] (A) at (6.7,0.2) {$(1-q\sp_1)\sp_2(1-\sp_3)q\sp_4q\sp_5(1-q\sp_6)\sp_7$};

\draw[line width=0.5pt,red,dashed] (-2.3,0.5) rectangle (-1.1,5.5);
\node[scale=1.1,anchor=east,red] (A) at (-0.1+\lbound,0) {$\sigma$};

\draw[line width=0.5pt,blue,dashed] (14.1,0.5) rectangle (15.3,5.5);
\node[scale=1.1,anchor=west,blue] (A) at (0.1+\rbound,0) {$\pi^{-1}$};
\end{tikzpicture}

\end{tabular}
}
\usefig{\FIGwdomgen}{FIGwdomgen}

\def\FIGwdomy{

\def\sclbx{1}
\def\gridop{0.5}
\def\nodescl{0.7}
\def\wirelw{1pt}
\def\dotscl{1}
\def\roc{3}
\def\setscl{1}
\def\rbound{14}
\def\lbound{-1}
\def\bnscl{1}

\def\leop{0.5}
\def\strandlw{1pt}

\def\xscl{0.4}
\def\yscl{0.7}
\def\captscl{1}
\begin{tabular}{c|c}

\begin{tikzpicture}[xscale=0.38,yscale=0.5,baseline=(Z.base)]
\coordinate(Z) at (0,3);
\def\ledot##1(##2,##3){
\xxdot{##1}(##2,##3)
\node[scale=\dotscl,inner sep=0pt] (A) at (##2.65,##3+1.3) {$\sp_{##1}$};
}
\def\ex{.4}
\def\xxdot##1(##2,##3){
\fill[white] ($ (##2,##3) + (-0\ex,-0.1) $) rectangle ($ (##2,##3) + (1\ex,1.1) $);
\draw[line width=\wirelw,rounded corners=\roc] ($ (##2,##3) + (-0\ex,0) $) -- (##2,##3) --($ (##2,##3) + (1,1) $)--($ (##2,##3) + (1\ex,1) $);
\draw[line width=\wirelw,rounded corners=\roc] ($ (##2,##3) + (-0\ex,1) $) -- ($ (##2,##3) + (0,1) $) --($ (##2,##3) + (1,0) $)--($ (##2,##3) + (1\ex,0) $);
}

\foreach \i/\j in {1/5,2/3,3/2,4/1,5/4}{
\draw[line width=\wirelw] (\lbound,\i) -- (\rbound,\i);
\node[scale=\bnscl,anchor=east] (A) at (\lbound,\i) {$\i$};
\node[scale=\bnscl,anchor=east] (A) at (-1+\lbound,\i) {$z_{\i}$};
\node[scale=\bnscl,anchor=west] (A) at (\rbound,\i) {$\j$};
}
\ledot{4,5}(0,4)
\ledot{2,3}(2,2)
\ledot{2,5}(4,3)
\ledot{3,5}(6,2)
\ledot{1,5}(8,1)
\ledot{1,3}(10,2)
\ledot{1,2}(12,3)

\node[scale=\captscl] (A) at (7,-0.2) {$\YB^w$ for $w^{-1}=(5,3,2,1,4)$};
\end{tikzpicture}

&

\def\rbound{12}

\begin{tikzpicture}[xscale=0.38,yscale=0.5,baseline=(Z.base)]
\coordinate(Z) at (0,3);
\def\ledot##1(##2,##3){
\xxdot{##1}(##2,##3)
\node[scale=\dotscl,inner sep=0pt] (A) at (##2.65,##3+1.3) {$\sp_{##1}$};
}
\def\ex{.4}
\def\xxdot##1(##2,##3){
\fill[white] ($ (##2,##3) + (-0\ex,-0.1) $) rectangle ($ (##2,##3) + (1\ex,1.1) $);
\draw[line width=\wirelw,rounded corners=\roc] ($ (##2,##3) + (-0\ex,0) $) -- (##2,##3) --($ (##2,##3) + (1,1) $)--($ (##2,##3) + (1\ex,1) $);
\draw[line width=\wirelw,rounded corners=\roc] ($ (##2,##3) + (-0\ex,1) $) -- ($ (##2,##3) + (0,1) $) --($ (##2,##3) + (1,0) $)--($ (##2,##3) + (1\ex,0) $);
}

\foreach \i/\j in {1/5,2/6,3/7,4/1,5/2,6/3,7/4}{
\draw[line width=\wirelw] (\lbound,\i) -- (\rbound,\i);
\node[scale=\bnscl,anchor=east] (A) at (\lbound,\i) {$\i$};
\node[scale=\bnscl,anchor=west] (A) at (\rbound,\i) {$\j$};
}

\foreach \i/\j in {1/4,2/3,3/2,4/1}{
\node[scale=\bnscl,anchor=east] (A) at (-1+\lbound,\i) {$x_{\j}=z_{\i}$};
}
\foreach \i/\j in {5/1,6/2,7/3}{
\node[scale=\bnscl,anchor=east] (A) at (-1+\lbound,\i) {$y_{\j}=z_{\i}$};
}
\node[yscale=2.45,xscale=0.8,anchor=east] (BR1) at (-5.6,2.5) {$\Big\{$};
\node[scale=1] (A) at (-7,2.5) {$M$};

\node[yscale=1.7,xscale=0.8,anchor=east] (BR2) at (-5.6,6) {$\Big\{$};
\node[scale=1] (A) at (-7,6) {$N$};

\xxdot{4,5}(0,4)
\xxdot{2,3}(2,3)
\xxdot{2,3}(2,5)
\xxdot{2,5}(4,2)
\xxdot{2,5}(4,4)
\xxdot{2,5}(4,6)
\xxdot{3,5}(6,1)
\xxdot{3,5}(6,3)
\xxdot{3,5}(6,5)
\xxdot{1,5}(8,2)
\xxdot{1,5}(8,4)
\xxdot{1,3}(10,3)

\node[scale=\captscl] (A) at (5,-0.2) {$\wMN$ for $M=4,N=3$};
\end{tikzpicture}

\end{tabular}
}
\usefig{\FIGwdomy}{FIGwdomy}

\def\FIGnondist{

\def\sclbx{1}
\def\gridop{0.5}
\def\nodescl{0.7}
\def\wirelw{1pt}
\def\dotscl{1}
\def\roc{3}
\def\setscl{1}
\def\rbound{14}
\def\lbound{-1}
\def\bnscl{1}

\def\leop{0.5}
\def\strandlw{1pt}

\def\xscl{0.4}
\def\yscl{0.7}

\begin{tikzpicture}[xscale=0.53,yscale=0.7,baseline=(Z.base)]
\coordinate(Z) at (0,3);
\def\ledot##1(##2,##3){
\xxdot{##1}(##2,##3)
\node[scale=\dotscl,inner sep=0pt] (A) at (##2.65,##3+1.3) {$\sp_{##1}$};
}
\def\ex{.4}
\def\xxdot##1(##2,##3){
\fill[white] ($ (##2,##3) + (-0\ex,-0.1) $) rectangle ($ (##2,##3) + (1\ex,1.1) $);
\draw[line width=\wirelw,rounded corners=\roc] ($ (##2,##3) + (-0\ex,0) $) -- (##2,##3) --($ (##2,##3) + (1,1) $)--($ (##2,##3) + (1\ex,1) $);
\draw[line width=\wirelw,rounded corners=\roc] ($ (##2,##3) + (-0\ex,1) $) -- ($ (##2,##3) + (0,1) $) --($ (##2,##3) + (1,0) $)--($ (##2,##3) + (1\ex,0) $);
}

\node[scale=\bnscl,anchor=east,blue] (A) at (0,0) {$i$};
\node[scale=\bnscl,anchor=east,green!70!black] (A) at (0,2) {$j$};

\draw[line width=\wirelw,rounded corners=\roc,blue] (0,0)--(2,0)--(4,2)--(5,2)--(6,1)--(8,1)--(10,3)--(11,3)--(12,4)--(13,4)--(15.5,1.5)--(16,2)--(18,2);

\draw[line width=\wirelw,rounded corners=\roc,green!70!black] (0,2)--(3,2)--(4,1)--(5,1)--(7,3)--(9,3)--(10,2)--(11,2)--(13,0)--(14,0)--(15.5,1.5)--(16,1)--(17,1)--(18,0);
\node[scale=2,draw=red,circle,line width=1pt] (CIRCLE) at (15.5,1.5) {};

\node[scale=1,red] (ND) at (13,-1) {Not distinguished!};
\draw[line width=1pt,red,->] (ND.20) to[out=30,in=-90] (CIRCLE);
\end{tikzpicture}
}
\usefig{\FIGnondist}{FIGnondist}

\def\FIGflipgen{

\def\sclbx{1}
\def\gridop{0.5}
\def\nodescl{0.7}
\def\wirelw{1pt}
\def\dotscl{1}
\def\roc{3}
\def\setscl{1}
\def\rbound{14}
\def\lbound{-1}
\def\bnscl{1}

\def\leop{0.5}
\def\strandlw{1pt}

\def\xscl{0.4}
\def\yscl{0.7}

\def\la{\lambda}
\def\nodescl{0.96}
\def\pscl{0.7}
\def\pdy{0.05}
\def\pdx{0.02}

  \resizebox{\textwidth}{!}{
\begin{tabular}{c|c}

\begin{tikzpicture}[scale=0.6]
\def\n{9}
\def\c{14}
\def\tw{11}
\def\lf{2}
\foreach \x in {1,...,\n}{
\draw[line width=\wirelw] (\lf,\x)--(\c,\x);
}
\draw[fill=white,draw=none] (5,2.5) rectangle (\tw,7.5);
\draw[fill=black!15,dashed,draw=black] (\lf+1,0.75) rectangle (5,5.25);
\draw[fill=green!15,dashed,draw=black] (\lf+1,5.75) rectangle (5,9.25);

\foreach\j in{1,...,\n}{
\node (L\j) at (\lf-0.5,\j) {$\j$};
\node (R\j) at (\c+0.5,\j) {$\j$};
}

\draw[line width=\wirelw,rounded corners=\roc] (5,6)--(6,6)--(9,3)--(\tw,3);
\draw[line width=\wirelw,rounded corners=\roc] (5,7)--(7,7)--(10,4)--(\tw,4);

\draw[line width=\wirelw,rounded corners=\roc] (5,5)--(6,5)--(8,7)--(\tw,7);
\draw[line width=\wirelw,rounded corners=\roc] (5,4)--(7,4)--(9,6)--(\tw,6);
\draw[line width=\wirelw,rounded corners=\roc] (5,3)--(8,3)--(10,5)--(\tw,5);

\draw[fill=blue!15,dashed,draw=black] (\tw,0.75) rectangle (2+\tw,4.25);

\draw[fill=black!15,dashed,draw=black] (\tw,4.75) rectangle (2+\tw,\n+0.25);

\node[scale=\nodescl] (A) at (4,7.5) {$Y_L$};
\node[scale=\nodescl] (A) at (4,3) {$Y_D$};

\node[scale=\nodescl] (A) at (12,7) {$Y_U$};
\node[scale=\nodescl] (A) at (12,2.5) {$Y_R$};

\draw[red,dashed,line width=1pt,fill=red,fill opacity=0.08] (5.5,5.5)--(7.5,7.5)--(10.5,4.5)--(8.5,2.5)--cycle;
\node[scale=\nodescl,red] (A) at (9,7.5) {\contour{white}{$\Yij$}};

\node[anchor=east,inner sep=-2pt] (A) at (L5.west) {$\ci=$};
\node[anchor=west,inner sep=-2pt] (A) at (R4.east) {$=\cj$};

\node[scale=\pscl] (A) at (6.5+\pdx,6+\pdy) {$\sp_{5,6}$};
\node[scale=\pscl] (A) at (7.5+\pdx,5+\pdy) {$\sp_{4,6}$};
\node[scale=\pscl] (A) at (8.5+\pdx,4+\pdy) {$\sp_{3,6}$};
\node[scale=\pscl] (A) at (7.5+\pdx,7+\pdy) {$\sp_{5,7}$};
\node[scale=\pscl] (A) at (8.5+\pdx,6+\pdy) {$\sp_{4,7}$};
\node[scale=\pscl] (A) at (9.5+\pdx,5+\pdy) {$\sp_{3,7}$};

\end{tikzpicture}
&

\begin{tikzpicture}[scale=0.6]
\def\n{9}
\def\c{14}
\def\tw{11}
\def\lf{2}
\foreach \x in {1,...,\n}{
\draw[line width=\wirelw] (\lf,\x)--(\c,\x);
}
\draw[fill=white,draw=none] (5,2.5) rectangle (\tw,7.5);
\draw[fill=black!15,dashed,draw=black] (\lf+1,0.75) rectangle (5,5.25);
\draw[fill=blue!15,dashed,draw=black] (\lf+1,5.75) rectangle (5,9.25);

\foreach\j in{1,...,\n}{
\node (L\j) at (\lf-0.5,\j) {$\j$};
\node (R\j) at (\c+0.5,\j) {$\j$};
}

\draw[line width=\wirelw,rounded corners=\roc] (5,6)--(6,6)--(9,3)--(\tw,3);
\draw[line width=\wirelw,rounded corners=\roc] (5,7)--(7,7)--(10,4)--(\tw,4);

\draw[line width=\wirelw,rounded corners=\roc] (5,5)--(6,5)--(8,7)--(\tw,7);
\draw[line width=\wirelw,rounded corners=\roc] (5,4)--(7,4)--(9,6)--(\tw,6);
\draw[line width=\wirelw,rounded corners=\roc] (5,3)--(8,3)--(10,5)--(\tw,5);

\draw[fill=green!15,dashed,draw=black] (\tw,0.75) rectangle (2+\tw,4.25);

\draw[fill=black!15,dashed,draw=black] (\tw,4.75) rectangle (2+\tw,\n+0.25);

\node[scale=\nodescl] (A) at (4,7.5) {$(Y_R)_{\ci,\cj}^\ast$};
\node[scale=\nodescl] (A) at (4,3) {$Y_D$};

\node[scale=\nodescl] (A) at (12,7) {$Y_U$};
\node[scale=\nodescl] (A) at (12,2.5) {$(Y_L)_{\ci,\cj}^\ast$};

\draw[red,dashed,line width=1pt,fill=red,fill opacity=0.08] (5.5,5.5)--(7.5,7.5)--(10.5,4.5)--(8.5,2.5)--cycle;
\node[scale=\nodescl,red] (A) at (9,7.5) {\contour{white}{$\Yijx$}};

\node[anchor=east,inner sep=-2pt] (A) at (L5.west) {$\ci=$};
\node[anchor=west,inner sep=-2pt] (A) at (R4.east) {$=\cj$};

\node[scale=\pscl] (A) at (6.5+\pdx,6+\pdy) {$\sp_{5,7}$};
\node[scale=\pscl] (A) at (7.5+\pdx,5+\pdy) {$\sp_{4,7}$};
\node[scale=\pscl] (A) at (8.5+\pdx,4+\pdy) {$\sp_{3,7}$};
\node[scale=\pscl] (A) at (7.5+\pdx,7+\pdy) {$\sp_{5,6}$};
\node[scale=\pscl] (A) at (8.5+\pdx,6+\pdy) {$\sp_{4,6}$};
\node[scale=\pscl] (A) at (9.5+\pdx,5+\pdy) {$\sp_{3,6}$};
\end{tikzpicture}
\end{tabular}
}
}
\usefig{\FIGflipgen}{FIGflipgen}

\def\FIGimp{

\def\circscl{0.3}
\def\duscl{0.8}
\def\lw{1.3pt}
\def\lwb{0.7pt}
\def\lwd{0.5pt}
\def\lwa{1pt}
\def\tikzscl{0.5}
\def\M{5}
\begin{tabular}{ccccc}

\begin{tikzpicture}[scale=\tikzscl]

\draw[line width=\lwb] (0,1)--(0,10);
\draw[line width=\lwb] (\M,1)--(\M,10);

\node[draw,fill,circle,scale=\circscl] (di) at (0,2.5) {};
\node[anchor=north west,scale=\duscl,inner sep=1pt](A) at (di) {$d_i$};
\node[draw,fill,circle,scale=\circscl] (dj) at (0,4) {};
\node[anchor=north west,scale=\duscl,inner sep=1pt](A) at (dj) {$d_j$};
\node[draw,fill,circle,scale=\circscl] (ui) at (\M,6) {};
\node[anchor=south east,scale=\duscl,inner sep=1pt](A) at (ui) {$u_i$};
\node[draw,fill,circle,scale=\circscl] (uj) at (\M,9) {};
\node[anchor=south east,scale=\duscl,inner sep=1pt](A) at (uj) {$u_j$};
\draw[line width=\lw] (di)--(ui);
\draw[line width=\lw] (dj)--(uj);
\draw[dashed,line width=\lwd] (0,4)--(-1.5,4);
\draw[dashed,line width=\lwd] (\M,6)--(-1.5,6);
\draw[<->,line width=\lwa] (-1,4)--(-1,6);
\node[anchor=east] (A) at (-1,5) {$\IMP_{i,j}$};
\end{tikzpicture}
&

\ \ 
&

\begin{tikzpicture}[scale=0.5]
\draw[line width=\lwb] (0,1)--(0,10);
\draw[line width=\lwb] (\M,1)--(\M,10);

\node[draw,fill,circle,scale=\circscl] (di) at (0,2.5) {};
\node[anchor=north west,scale=\duscl,inner sep=1pt](A) at (di) {$d_i$};
\node[draw,fill,circle,scale=\circscl] (dj) at (0,4) {};
\node[anchor=north west,scale=\duscl,inner sep=1pt](A) at (dj) {$d_j$};
\node[draw,fill,circle,scale=\circscl] (ui) at (\M,9) {};
\node[anchor=south east,scale=\duscl,inner sep=1pt](A) at (ui) {$u_i$};
\node[draw,fill,circle,scale=\circscl] (uj) at (\M,6) {};
\node[anchor=south east,scale=\duscl,inner sep=1pt](A) at (uj) {$u_j$};
\draw[line width=\lw] (di)--(ui);
\draw[line width=\lw] (dj)--(uj);
\draw[dashed,line width=\lwd] (0,4)--(-1.5,4);
\draw[dashed,line width=\lwd] (\M,6)--(-1.5,6);
\draw[<->,line width=\lwa] (-1,4)--(-1,6);
\node[anchor=east] (A) at (-1,5) {$\IMP_{i,j}$};
\end{tikzpicture}

&
\ \ 
&
\begin{tikzpicture}[scale=0.5]
\draw[line width=\lwb] (0,1)--(0,10);
\draw[line width=\lwb] (\M,1)--(\M,10);

\node[draw,fill,circle,scale=\circscl] (di) at (0,2.5) {};
\node[anchor=north west,scale=\duscl,inner sep=1pt](A) at (di) {$d_i$};
\node[draw,fill,circle,scale=\circscl] (dj) at (0,6) {};
\node[anchor=north west,scale=\duscl,inner sep=1pt](A) at (dj) {$d_j$};
\node[draw,fill,circle,scale=\circscl] (ui) at (\M,4) {};
\node[anchor=south east,scale=\duscl,inner sep=1pt](A) at (ui) {$u_i$};
\node[draw,fill,circle,scale=\circscl] (uj) at (\M,9) {};
\node[anchor=south east,scale=\duscl,inner sep=1pt](A) at (uj) {$u_j$};
\draw[line width=\lw] (di)--(ui);
\draw[line width=\lw] (dj)--(uj);
\node[anchor=east] (A) at (-0,5.5) {$\IMP_{i,j}=0$};
\node[anchor=east] (A) at (-0,4.2) {$\IM_{i,j}<0$};
\end{tikzpicture}

\end{tabular}
}
\usefig{\FIGimp}{FIGimp}

\begin{document}
\numberwithin{equation}{section}

\title{Symmetries of stochastic colored vertex models}
\author{Pavel Galashin}
\address{Department of Mathematics, University of California, Los Angeles, USA}
\email{galashin@math.ucla.edu}
\date{\today}

\subjclass[2010]{
Primary: 82C22. Secondary: 60K35, 14M15, 05E99.
}

\keywords{Six-vertex model, flip-invariance, Hecke algebra, last passage percolation, KPZ equation, Airy sheet, directed polymers, Kazhdan--Lusztig polynomials, positroid varieties}

\begin{abstract}
We discover a new property of the stochastic colored six-vertex model called \emph{flip-invariance}. We use it to show that for a given collection of observables of the model, any transformation that preserves the distribution of each individual observable also preserves their joint distribution. This generalizes recent shift-invariance results of Borodin--Gorin--Wheeler. As limiting cases, we obtain similar statements for the Brownian last passage percolation, the Kardar--Parisi--Zhang equation, the Airy sheet, and directed polymers. Our proof relies on an equivalence between the stochastic colored six-vertex model and the Yang--Baxter basis of the Hecke algebra. We conclude by discussing the relationship of the model with Kazhdan--Lusztig polynomials and positroid varieties in the Grassmannian.
\end{abstract}

\maketitle

\setcounter{tocdepth}{1}
\tableofcontents

\section{Introduction}
We study various symmetries of the \emph{stochastic colored six-vertex model}. It is a special case of a ``higher spin'' colored vertex model introduced in~\cite{KMMO}, related to the stochastic version of the $R$-matrix for the quantum affine algebra $U_q(\widehat{\mathfrak{sl}}_{n+1})$ studied earlier in~\cite{Bazhanov,FRT,Jimbo1,Jimbo2}. The model admits a very simple description: one fixes a subdomain of the square grid and considers a family of paths that enter it from the bottom left. Once two paths meet at a vertex of the grid, they either cross with probability $b$ or do not cross with probability $1-b$. The value of $b$ depends (in a certain \emph{integrable} way) on the coordinates of the vertex and on the parity of the number of times these two paths have crossed before, see \cref{fig:spec} and \cref{sec:intro-main}.

The main motivation for our work comes from the recent \emph{shift-invariance} results of~\cite{BGW}. In fact, the \emph{flip-invariance} property (\cref{thm:flip} and \cref{fig:flip}) was originally formulated as an attempt to give a simple explanation for the shift-invariance of~\cite{BGW}: one can realize their shift as a composition of two flips, see \cref{fig:doubleh} and \cref{rmk:gen_BGW}. We show in \cref{thm:main} that this happens more generally: whenever a transformation preserves the one-dimensional distributions of a given vector of \emph{height functions}, we prove that it can essentially be obtained as a composition of several flips, and therefore it preserves the joint distribution of this vector of height functions. Thus, in a sense, one can view flip-invariance as the \emph{fundamental} symmetry of the model inside the class of distribution-preserving transformations that we consider.

Our proofs rely on a connection between the stochastic colored six-vertex model and the \emph{Hecke algebra} $\Hecke$ of the symmetric group. Specifically, we make a simple observation (\cref{prop:YB_equals_Zbipi}) showing that the probabilities induced by the model coincide with the coefficients in the expansion of the \emph{Yang--Baxter basis}~\cite{LLT} of $\Hecke$ in the standard basis. This observation allows us to greatly simplify the proofs of our results. We also use it to give a simple proof (discovered independently by Bufetov~\cite{Buf20}) of the \emph{color-position symmetry}~\cite{BoBu} for the Asymmetric Simple Exclusion Process (ASEP).

\begin{figure}
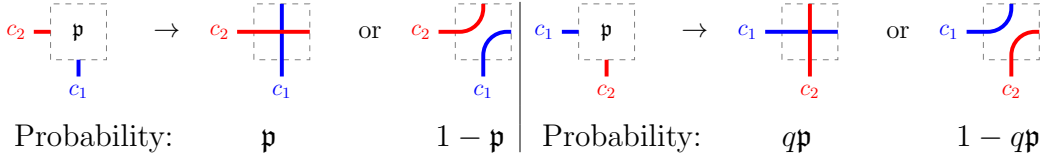

\FIGspec
  \caption{\label{fig:spec} When two paths of colors $c_1<c_2$ enter a cell $(i,j)$, they proceed in the up-right direction according to these probabilities. Here $\sp$ is equal to $\sp_{i,j}=\frac{y_j-x_i}{y_j-qx_i}$, see~\eqref{eq:sp}.}
\end{figure}

\begin{figure}
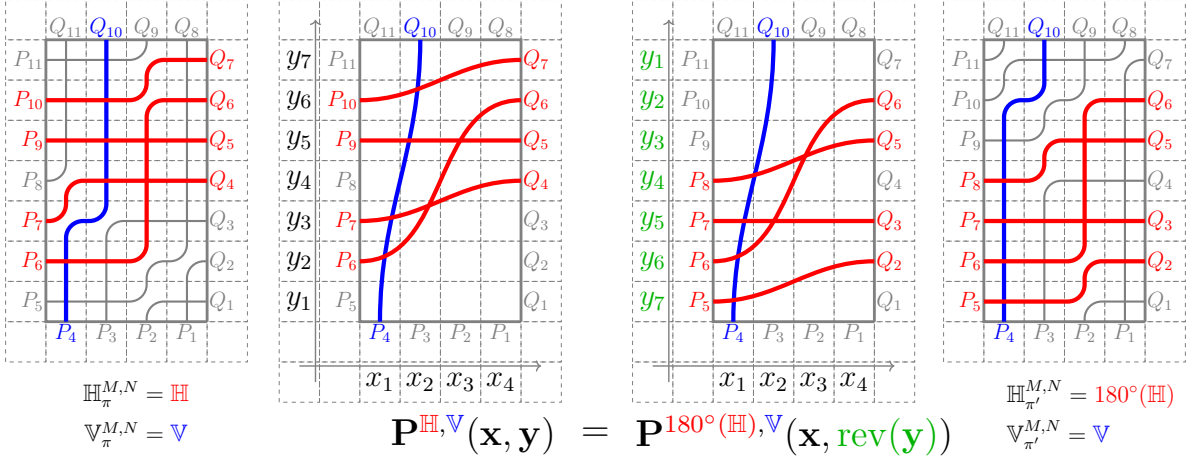

\FIGflip
  \caption{\label{fig:flip} The flip theorem states that the two partition functions shown in the middle are equal to each other. The picture on the far left (resp., far right) represents a configuration of the stochastic colored six-vertex model contributing to $\PF^{\H,\V}$ (resp., to $\PF^{\flip(\H),\V}$). See \cref{ex:flip}.}
\end{figure}

The stochastic colored six-vertex model specializes to many other probabilistic objects of interest. Following~\cite{BGW}, we obtain versions of our results for the Brownian last passage percolation and directed $(1+1)$d polymers in random media, as well as two universal objects, the Kardar--Parisi--Zhang (KPZ) equation and the Airy sheet, see \cref{sec:applications}. We refer the reader to \cite[Figure~2]{BGW} for the full chart of objects that can be obtained as limits of the stochastic colored six-vertex model (which includes many objects not considered here, for example, the colored $q$-PushTASEP, TASEP, ASEP, and the Bernoulli-Exponential last passage percolation).

Another consequence of the Hecke algebra approach is a family of unexpected connections between the stochastic colored six-vertex model and well studied algebraic objects such as \emph{Kazhdan--Lusztig polynomials}~\cite{KL1,KL2} and \emph{positroid varieties}~\cite{Pos,BGY,KLS} reviewed in \cref{sec:KL_GR}. In particular, we give an interpretation of the flip-invariance property in the language of Grassmannians in \cref{prop:Gr}. We hope to further explore the connections between the above objects in future work.

\subsection{Stochastic colored six-vertex model}\label{sec:intro-main}
An \emph{up-left path} $\Ppath$ is a lattice path in the positive quadrant with unit steps that go either up or left. A \emph{\skew domain} is a pair $\SKD$ of up-left paths with common start and end points such that $\Ppath$ is weakly below $\Qpath$. We denote by $n:=|\Ppath|=|\Qpath|$ the number of steps in $P$ and $Q$, and we label their steps by $\ein 1,\ein 2,\dots,\ein n$ and $\eout 1,\eout 2,\dots,\eout n$, respectively. 
The area between $\Ppath$ and $\Qpath$ is subdivided into unit squares called \emph{cells} whose Cartesian coordinates are pairs of positive integers, see \cref{fig:htpq}(left). 

Fix a \skew domain $\SKD$, two sets of formal variables $\bx=(x_1,x_2,\dots)$ and $\by=(y_1,y_2,\dots)$, called, respectively, \emph{column} and \emph{row rapidities}, and a parameter $0<q<1$.  To describe the model, we consider colored paths, where a \emph{color} is just an integer assigned to each path that determines its ``priority'' in the model dynamics. Suppose that for each $c=1,2,\dots,n$, a path of color $c$ enters the domain through the edge $\ein c$. The $n$ paths then propagate in the up-right direction according to the following rule: once two paths enter a cell $(i,j)$ from the left and from the bottom, they proceed in the up-right direction according to the probabilities given in \cref{fig:spec}. The probabilities depend on a parameter $\sp$, which for a given cell $(i,j)$ equals\footnote{This choice of parameters ensures that the model is \emph{integrable}, i.e., satisfies the \emph{Yang--Baxter equation}.}
\begin{equation}\label{eq:sp}
\sp_{i,j}:=\frac{y_j-x_i}{y_j-qx_i}.
\end{equation}
We apply this sampling procedure to all cells between $\Ppath$ and $\Qpath$, proceeding in the up-right direction, see \cref{fig:htpq}(right) for an example. This procedure gives rise to a random \emph{color permutation} $\pi=(\pi(1),\pi(2),\dots,\pi(n))$, where for $c=1,2,\dots,n$, the path of color $c$ exits the domain through the edge $\fout{\pi(c)}$. We let $\Sn$ be the group of permutations of $\{1,2,\dots,n\}$, and for each $\pi\in\Sn$, we denote by $\Zpi$ the total probability\footnote{This is a probability distribution in the sense that $\Zpi\geq 0$ and $\sum_{\pi\in\Sn}\Zpi=1$ when $0<q<1$ and $\bx,\by$ are specialized so that $0<x_i<y_j$ for all $i,j$.} of observing $\pi$ as the color permutation.

\subsection{Flips in rectangular domains}\label{sec:intro-rect}
We formulate our first result which we call the \emph{flip theorem}. The proof of the main result (\cref{thm:main}) will essentially consist of repeated applications of this fundamental hidden symmetry of the model.

For positive integers $M,N$, an \emph{$M\times N$-rectangular domain} is a \skew domain $\SKD$ such that $\Ppath$ goes $M$ steps left and then $N$ steps up while $\Qpath$ goes $N$ steps up and then $M$ steps left. The lengths of $\Ppath$ and $\Qpath$ are given by $n:=M+N$. We let $\bx=(x_1,\dots,x_{M})$ and $\by=(y_{1},\dots,y_{N})$.

We will be interested in the probability of observing a color permutation satisfying given \emph{horizontal and vertical boundary conditions}. For a permutation $\pi\in\Sn$, we let 
\begin{equation*}
  \Hpi:=\{(i,\pi(i))\mid i> M\text{ and }\pi(i)\leq N\},\quad    \Vpi:=\{(i,\pi(i))\mid i\leq M\text{ and }\pi(i)> N\}.
\end{equation*}

\noindent For two sets $\H=\{(\HL_1,\HR_1),\dots,(\HL_\h,\HR_\h)\}$ and $\V=\{(\VD_1,\VU_1),\dots,(\VD_\v,\VU_\v)\}$ of pairs, we denote
\begin{equation*}
\ZHV:=\sum_{\pi\in\Sn:\ \text{$\Hpi=\H$ and $\Vpi=\V$}} \Zpi.
\end{equation*}
Thus $\ZHV$ is the probability of observing a color permutation with specified endpoints of \emph{each} path that connects the opposite boundaries of the rectangle.

For a set of pairs $\H=\{(\HL_1,\HR_1),\dots,(\HL_\h,\HR_\h)\}$, we define its \emph{180-degree rotation} by 
\begin{equation*}
\flip(\H):=\{(n+1-\HR_1,n+1-\HL_1),\dots,(n+1-\HR_\h,n+1-\HL_\h)\}.
\end{equation*}
Additionally, denote $\rev(\by):=(y_{N},y_{N-1},\dots,y_{1})$. We are ready to state the flip theorem, which is a special case of \cref{thm:main} below. 
\begin{theorem}\label{thm:flip}
  For an $M\times N$-rectangular domain $\SKD$ and any $\H$ and $\V$, we have
\begin{equation}\label{eq:flip}
\ZHV=\ZHFV.
\end{equation}
\end{theorem}
\begin{example}\label{ex:flip}
\noindent \Cref{fig:flip} illustrates \cref{thm:flip} in the case $M=4$, $N=7$, and
\begin{equation*}
  \H=\{(6,6),(7,4),(9,5),(10,7)\},\quad \flip(\H)=\{(5,2),(6,6),(7,3),(8,5)\},\quad \V=\{(4,10)\}.  
\end{equation*}
\end{example}

\begin{figure}
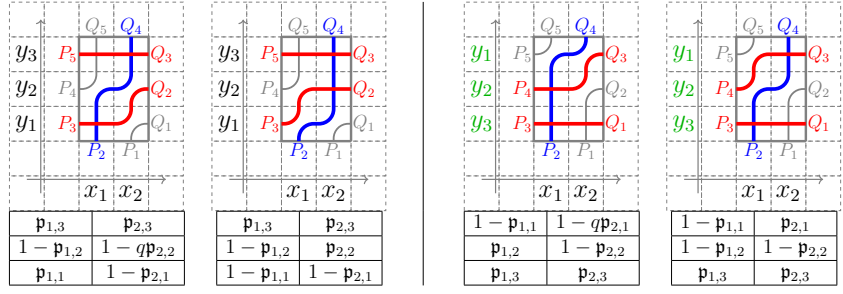

\FIGexdt
  \caption{\label{fig:exdt} Applying the flip theorem to a $2\times 3$ rectangle leads to a non-trivial identity, see \cref{ex:23}. For each of the four configurations, its probability equals the product of the entries in the corresponding table shown below it.}
\end{figure}

\begin{example}\label{ex:23}
Let $M=2$, $N=3$, $\H=\{(3,2),(5,3)\}$, $\V=\{(2,4)\}$. The left (resp., right) hand side of~\eqref{eq:flip} is the sum of probabilities of the two configurations shown in \cref{fig:exdt} on the left (resp., right). Thus we get
\begin{align*}
  \ZHV&= \sp_{1,3}\sp_{2,3} (1-\sp_{1,2})(1-\sp_{2,1}) \left(\sp_{1,1}(1-q\sp_{2,2})+\sp_{2,2}(1-\sp_{1,1})\right) \\
  \ZHFV&= \sp_{1,3}\sp_{2,3} (1-\sp_{1,1})(1-\sp_{2,2}) \left(\sp_{1,2}(1-q\sp_{2,1})+\sp_{2,1}(1-\sp_{1,2})\right).
\end{align*}
A quick calculation verifies that the right hand sides coincide as rational functions in $(q,\bx,\by)$.
\end{example}
\begin{remark}
In the above example, the \emph{number} of configurations contributing to each side of~\eqref{eq:flip} was the same (equal to $2$). This is true in general, see \cref{sec:counting-pipe-dreams}.
\end{remark}

\begin{remark}\label{rmk:flip_colors_ordered}
\Cref{thm:flip} is stated in the case of \emph{ordered} incoming colors: for each $c$, the color of the path entering through the edge $\ein c$ equals to $c$. As we explain in \cref{rmk:bottom_colors}, \cref{thm:flip} holds more generally when the colors entering from the left are ordered and are larger than the colors entering from the bottom (which may be ordered arbitrarily).
\end{remark}

Our proof of \cref{thm:flip} is given in \cref{sec:flip_proof}. It is quite short and makes repeated use of the Yang--Baxter equation~\eqref{eq:YB} combined with the machinery of Hecke algebras that we develop in \cref{sec:Hecke}. In contrast, deducing \cref{thm:main} from \cref{thm:flip} requires a long technical argument which is presented in \cref{sec:flip_conseq,sec:main_proof}.

\begin{figure}
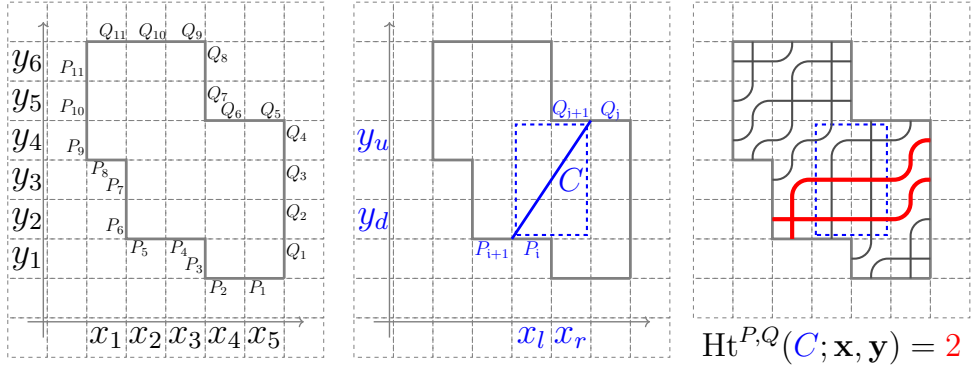

\FIGhtpq
  \caption{\label{fig:htpq} Left: a \skew domain $\SKD$ shown together with row/column rapidities. Middle: a $\SKD$-cut $\cut=(l,d,u,r)$. Right: a configuration of the model for which the height function $\HT(\cut)$ is equal to $2$ since there are two paths that connect the left and right boundaries of the dashed rectangle.}
\end{figure}

\subsection{Main result: height functions}\label{sec:intro:main}
 Following~\cite{BGW}, we study \emph{height functions} defined as follows. Given a pair of integers $1\leq \ci,\cj\leq n$ and a permutation $\pi\in\Sn$, the value of the associated height function is simply given by $\Hij:=\#\{c>\ci\mid \pi(c)\leq \cj\}$. Let us reformulate this slightly, taking the geometry of the domain into account.

Let $\ZP:=\{1,2,\dots\}$ and denote by $\SKDZ\subset\ZP\times\ZP$ the set of cells inside the domain $\SKD$. A \emph{$\SKD$-cut} is a quadruple $\cut=(\cutL,\cutD,\cutU,\cutR)$ of positive integers satisfying  $\cutL\leq\cutR$, $\cutD\leq\cutU$, and such that $(\cutL,\cutD),(\cutR,\cutU)\in\SKDZ$ while $(\cutL-1,\cutD-1),(\cutR+1,\cutU+1)\notin\SKDZ$. Define $1\leq \ci,\cj\leq n$ so that the bottom left corner of the cell $(\cutL,\cutD)$ belongs to both $\ein{\ci}$ and $\ein{\ci+1}$ while the top right corner of the cell $(\cutR,\cutU)$ belongs to both $\eout{\cj}$ and $\eout{\cj+1}$, see \cref{fig:htpq}(middle).  Given a permutation $\pi\in \Sn$, we set 
\begin{equation}\label{eq:HTpi}
 \HTpi_{\pi}(\cut):=\Hij=\#\{c>\ci\mid \pi(c)\leq \cj\}.
\end{equation} 
In other words, $\HTpi_\pi(\cut)$ counts the number of colored paths that connect the left and right boundaries of the sub-rectangle $\{\cutL,\cutL+1,\dots,\cutR\}\times\{\cutD,\cutD+1,\dots,\cutU\}$ of $\SKDZ$, see \cref{fig:htpq}(right). We let $\HT(\cut)$ denote the associated random variable.

Let $\suppH(\cut):=\{x_\cutL,x_{\cutL+1},\dots,x_{\cutR}\}$ and $\suppV(\cut):=\{y_\cutD,y_{\cutD+1},\dots,y_{\cutU}\}$ denote the \emph{unordered} sets of column and row rapidities covered by $\cut$. Let us say that $\bx'=(x'_1,x'_2,\dots)$ is a \emph{permutation of the variables in $\bx$} if there exists a bijection $\phi:\ZP\to\ZP$ such that $x'_i=x_{\phi(i)}$ for all $i\in\ZP$. 
 Our main result is the following equality between joint distributions of two vectors of height functions, conjectured by Borodin--Gorin--Wheeler. 

\begin{samepage}
\begin{theorem}\label{thm:main}
Suppose that we are given the following data:
\begin{itemize}
\item two \skew domains $\SKD$ and $\SKDp$;
\item a permutation $\bx'$ of the variables in $\bx$ and a permutation $\by'$ of the variables in $\by$;
\item a tuple $(\cut_1,\cut_2,\dots,\cut_m)$ of $\SKD$-cuts and a tuple $(\cut'_1,\cut'_2,\dots,\cut'_m)$ of $\SKDp$-cuts.
\end{itemize}
Assume that for each $i=1,2,\dots,m$, we have 
\begin{equation}\label{eq:main:supp=supp}
  \suppH(\cut_i)=\suppHp(\cut'_i)\quad\text{and}\quad\suppV(\cut_i)=\suppVp(\cut'_i).
\end{equation} 
Then the  distributions of the following two vectors of height functions agree:
\[\Big(\HT(\cut_1),\dots,\HT(\cut_m)\Big)\eqd \Big(\HTp(\cut'_1),\dots,\HTp(\cut'_m)\Big).\]
\end{theorem}
\end{samepage}
\noindent An illustration of \cref{thm:main} is shown in \cref{fig:doubleh}. For example, we check that $\suppV(\cut_3)=\{y_4,y_5,y_6\}=\suppVp(\cut'_3)$, in agreement with~\eqref{eq:main:supp=supp}.

\begin{remark}\label{rmk:supp_necessary}
The distribution of an individual height function $\HT(\cut_i)$ depends\footnote{In fact, it follows as a simple consequence of the Yang--Baxter equation that the distribution of $\HT(\cut_i)$ is symmetric in the variables $\suppH(\cut_i)$ and (separately) $\suppV(\cut_i)$, see \cref{sec:flip_proof}.} on the variables in $\suppH(\cut_i)$ and $\suppV(\cut_i)$, and it is straightforward to check that the condition~\eqref{eq:main:supp=supp} is necessary in order for a distributional identity
\[\HT(\cut_i)\eqd\HTp(\cut'_i)\]
to hold.  The content of \cref{thm:main} is that this condition is also sufficient even when one considers joint distributions of multiple height functions.
\end{remark}

\begin{figure}
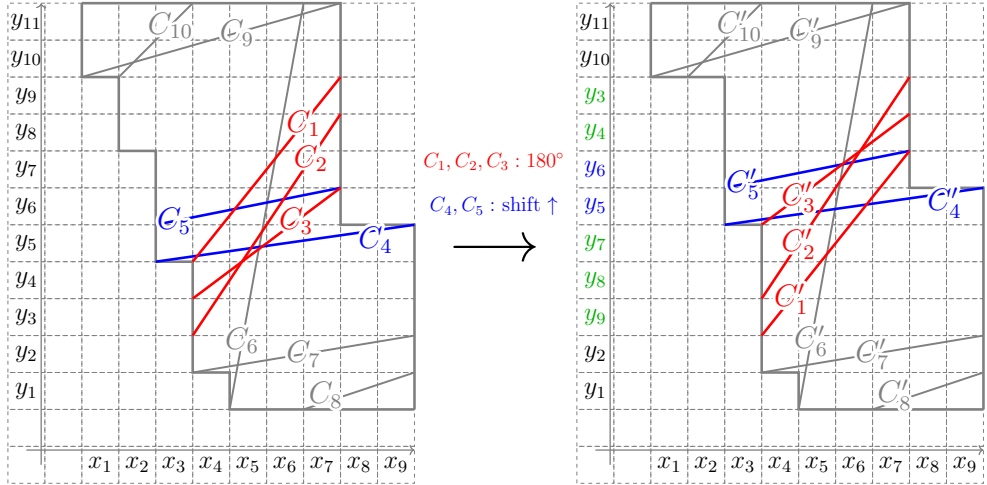

\FIGdoubleh
  \caption{\label{fig:doubleh} An application of \cref{thm:main}: the joint distributions of the two vectors of height functions are the same. This transformation is called a \emph{double \Hd-flip}, see \cref{lemma:double_H_flip}.}
\end{figure}

Examples of various transformations satisfying the assumptions of \cref{thm:main} are given in \cref{fig:doubleh} and \cref{sec:examples} (cf. \cref{fig:global_H,fig:local_H}). These examples include the shift-invariance of~\cite[Theorems~1.2 and~4.13]{BGW}, see \cref{rmk:gen_BGW}.

\subsection{Applications}\label{sec:applications}
Following~\cite{BGW}, we describe some probabilistic models and universal objects which can be obtained 
as limiting cases of the stochastic colored six-vertex model and state analogs of \cref{thm:main} for them. In particular, the generalized shift-invariance property \cite[Conjecture~1.5]{BGW} of the KPZ equation is given below in \cref{thm:KPZ}. We closely follow the notation and exposition of~\cite[Section~1]{BGW}.

For simplicity, we restrict to the case where the \skew domain is a vertical strip of some fixed width $M$. We start by introducing the notion of an \emph{intersection matrix} which, in view of our results below, plays a role similar to the covariance matrix of a multivariate Gaussian distribution.

\begin{figure}
\FIGimp
  \caption{\label{fig:IMP} The geometric meaning of the intersection matrix $\IMP(\bd,\bu)$.}
\end{figure}

\begin{definition}\label{dfn:IMP}
Given vectors  $\bd=(d_1,d_2,\dots,d_m)$ and $\bu=(u_1,u_2,\dots,u_m)$ in $\R^m$, we introduce two $m\times m$ symmetric matrices $\IM(\bd,\bu),\IMP(\bd,\bu)$ whose entries are given by 
\begin{equation}\label{eq:IM_dfn}
  \IM_{i,j}:=\min(u_i,u_j)-\max(d_i,d_j),\quad \IMP_{i,j}=\max(\IM_{i,j},0) \quad\text{for all $i,j=1,2,\dots,m$.}
\end{equation}
We refer to $\IMP(\bd,\bu)$ as the \emph{intersection matrix of $(\bd,\bu)$} since its $(i,j)$-th entry equals the length of the intersection of line segments $[d_i,u_i]\cap [d_j,u_j]$, see \cref{fig:IMP}.
\end{definition}

\subsection*{Brownian last passage percolation}

Fix a collection $\{B_n(t)\}_{n\in\Z}$ of independent standard Brownian motions on the real line. Given $l,r\in\Z$ and $d,u\in\R$ satisfying $l\leq r$ and $d\leq u$, the \emph{last passage time} $\zux ldru$ is defined by
\begin{equation}\label{eq:BLPP}
  \zux ldru:=\max_{d=t_{l}<t_{l+1}<\dots<t_{r+1}=u} \left[\sum_{i=l}^{r}(B_{i}(t_{i+1})-B_i(t_i)) \right].
\end{equation}

The following result is a special case of \cref{thm:BLPP2}.
\begin{theorem}\label{thm:BLPP}
Fix $M\in\ZP$ and consider vectors $\bd,\bd',\bu,\bu'\in\R^m$ satisfying $d_i\leq u_i$ and $d'_i\leq u'_i$ for all $i=1,\dots,m$. Suppose that the intersection matrices of $(\bd,\bu)$ and $(\bd',\bu')$ coincide: $\IMP(\bd,\bu)=\IMP(\bd',\bu')$. Then we have a distributional identity
\begin{equation}\label{eq:zu_dist}
  \left(\zux 0{d_1}M{u_1},\dots,\zux 0{d_m}M{u_m}\right)\eqd  \left(\zux 0{d'_1}M{u'_1},\dots,\zux 0{d'_m}M{u'_m}\right).
\end{equation}
\end{theorem}
\noindent Similarly to \cref{rmk:supp_necessary}, we expect that the condition $\IMP(\bd,\bu)=\IMP(\bd',\bu')$ is not only sufficient but also necessary in order for~\eqref{eq:zu_dist} to hold.

\subsection*{KPZ equation}
Consider a two-dimensional white Gaussian noise $\eta(x,t)$. For $y\in\R$, define a random function $\KZxxx ytx$ as a solution to the following \emph{stochastic heat equation with multiplicative white noise}:
\begin{equation*}
  \KZx y_t=\frac12\KZx y_{xx}+\eta\KZx y,\quad t\in\R_{\geq0},x\in\R;\quad \KZx y(0,x)=\delta(x-y),
\end{equation*}
where the initial condition is given by the delta function at $y$. We will consider the random variables $\KZxxx ytx$ for fixed $t\in\R_{\geq0}$ and different pairs $(x,y)$, and we assume that the white noise $\eta$ is the same for different values of $y$. The formal logarithm $\Hcal:=-\ln(\KZx y)$ satisfies the celebrated Kardar--Parisi--Zhang (KPZ) equation~\cite{KPZ}:
\begin{equation*}%
  \Hcal_t=\frac12\Hcal_{xx}-\frac12(\Hcal_{x})^2-\eta.
\end{equation*}
The KPZ universality class has been a subject of intense interest throughout the past two decades, see~\cite{Corwin,QS} for reviews.

For the following result, we consider vectors $\bx,\by\in\R^m$ that do not necessarily satisfy $x_i\leq y_i$. In fact, we will swap $\bx$ and $\by$ and use the matrix $\IM(\by,\bx)$ from \cref{dfn:IMP} rather than the intersection matrix $\IMP(\by,\bx)$. The geometric meaning of the entries of $\IM(\by,\bx)$ is that for all large enough $L$, the length of the intersection $[y_i,x_i+L]\cap[y_j,x_j+L]$ is given by $\IM_{i,j}+L$.
\begin{theorem}\label{thm:KPZ}
Let $t\in\R_{\geq0}$ and $\bx,\bx',\by,\by'\in\R^m$. If $\IM(\by,\bx)=\IM(\by',\bx')$ then 
\begin{equation*}%
  \Big(\KZxxx{y_1}t{x_1},\dots,\KZxxx{y_m}t{x_m}\Big)\eqd  \Big(\KZxxx{y'_1}t{x'_1},\dots,\KZxxx{y'_m}t{x'_m}\Big).
\end{equation*}
\end{theorem}
\noindent In particular, this result implies~\cite[Conjecture~1.5]{BGW}.

\subsection*{Airy sheet}
It is believed~\cite{CQR} that the large time limit of $\KZxxx ytx$ (as well as the universal limit of various directed polymers and last-passage percolation models) is described by the \emph{Airy sheet} $\Acal(x,y)$. We define it in~\eqref{eq:Airy} building on recent results of~\cite{DOV}. For $(x,y)\in\R^2$, the random variable $\Acal(x,y)$ has the Tracy--Widom distribution~\cite{TW1,TW2}.

\begin{theorem}\label{thm:Airy}
  Let $t\in\R_{\geq0}$ and $\bx,\bx',\by,\by'\in\R^m$. If $\IM(\bx,\by)=\IM(\bx',\by')$ then 
\begin{equation*}%
  \left(\Acal(x_1,y_1),\dots,\Acal(x_m,y_m)\right)\eqd    \left(\Acal(x'_1,y'_1),\dots,\Acal(x'_m,y'_m)\right).
\end{equation*}
\end{theorem}

\subsection*{Directed polymers}
The limit transition from the stochastic colored six-vertex model to the objects described above is obtained through a sequence of intermediate steps. First, one passes to the \emph{fused} vertex model~\cite{BW,Kuan1,KMMO} which then gets degenerated further to the \emph{Beta polymer}~\cite{BC}, the \emph{Gamma polymer}~\cite{CSS,OCO}, and the \emph{O’Connell--Yor polymer}~\cite{OCY}. We will define these objects in \cref{sec:applications-body}. 

Note that our proof of \cref{thm:main} has the following advantage. The general Beta polymer depends on a choice of parameters $(\sigma_i)_{i\in\ZNN}$ and $(\rho_j)_{j\in\ZP}$. The shift-invariance of~\cite{BGW} is stated for the Beta polymer that is \emph{homogeneous} in the vertical direction, i.e., when all parameters $\rho_j$ are equal to each other. The result is expected \cite[Remark~7.4]{BGW} to hold more generally for the Beta polymer that is inhomogeneous in both directions. Our proof allows to confirm this: we show that the analog of \cref{thm:main} holds for Beta polymers with arbitrary parameters $(\sigma_i)_{i\in\ZNN}$ and $(\rho_j)_{j\in\ZP}$, see \cref{thm:Beta}.

\subsection{Related work}
In a very recent preprint~\cite{Duncan}, the author relies on the \emph{geometric RSK correspondence}~\cite{Kirillov,NY} to describe several transformations similar to the ones we construct. The models considered in~\cite{Duncan} are mostly disjoint from the models studied here: we focus on the stochastic six-vertex model and Beta polymers, while~\cite{Duncan} works with the last passage percolation with geometric or exponential weights and log-gamma polymers. There is a certain overlap among the limiting Gaussian objects, such as the Brownian last passage percolation or the KPZ equation. For example, \cite[Conjecture~1.5]{BGW} that we obtain as a consequence of \cref{thm:KPZ} can also be deduced from \cite[Theorem~1.1]{Duncan}.  

The central transformation of this paper, \cref{thm:flip}, shares some similarities with the notion of a \emph{decoupled polymer model}~\cite[Definition~2.2]{Duncan} that captures the class of models to which the RSK correspondence applies, see~\cite[Theorem~2.3]{Duncan}. It would be interesting to understand whether our flip theorem leads to an analog of the geometric RSK correspondence for the stochastic colored six-vertex model. We thank Vadim Gorin for bringing the paper~\cite{Duncan} to our attention.

\subsection{Outline} 
We give background on Hecke algebras in \cref{sec:Hecke} and use them to prove \cref{thm:flip} in \cref{sec:flip}. The next two sections are devoted to the proof of \cref{thm:main}: in \cref{sec:flip_conseq}, we describe a family of transformations and use \cref{thm:flip} to show that they preserve joint distributions of height functions. In \cref{sec:main_proof}, we show that any transformation satisfying the assumptions of \cref{thm:main} can be obtained as a composition of the transformations constructed in \cref{sec:flip_conseq}. In \cref{sec:applications-body}, we describe the limiting transitions of~\cite{BGW} and use them to prove the results stated in \cref{sec:applications}. Finally, we discuss the relationship of the stochastic colored six-vertex model with Kazhdan--Lusztig polynomials and positroid varieties in \cref{sec:KL_GR}. We include some examples and a conjecture relating \cref{thm:main} to more general wiring diagram domains in \cref{sec:arbitr-perm}.

\subsection*{Acknowledgments}
I am deeply grateful to Alexei Borodin for sparking my interest in this problem and for his guidance throughout the various stages of the project. I am also indebted to Vadim Gorin for the numerous consultations and explanations. Additionally, I would like to thank Thomas Lam and Pavlo Pylyavskyy with whom I discussed some questions and objects from \cref{sec:KL_GR}. Finally, I am grateful to the anonymous referees for their extremely careful reading of this manuscript and many suggested improvements. This work was partially supported by the National Science Foundation under Grant No.~DMS-1954121.

\section{Hecke algebra and the Yang--Baxter basis}\label{sec:Hecke}
One of our main tools is a certain direct relationship between the stochastic colored six-vertex model (which we from now on abbreviate as the \emph{\SCSV model}) and the \emph{Yang--Baxter basis} of the Hecke algebra of $\Sn$ introduced in~\cite{LLT}. A simple proof of the color-position symmetry of~\cite{BoBu} is given in \cref{sec:CPS}. Another non-trivial property of the \SCSV model that becomes obvious in the language of Hecke algebras is known as the \emph{non-local relations} of Borodin--Wheeler~\cite{BW}, see~\eqref{eq:R_k*T_w} below.

Fix $n\geq1$. For integers $i\leq j$, we denote  $[i,j]:=\{i,i+1,\dots,j\}$, and for $i>j$, we set $[i,j]:=\emptyset$. For $i\in[1,n-1]$, denote by $s_i\in \Sn$ the transposition $(i,i+1)$. For indeterminates $q$ and $\bzz:=(\zz_1,\zz_2,\dots,\zz_n)$, we consider the \emph{Hecke algebra} $\Hecke$, which is an associative algebra over $\C(q;\bzz):=\C(q,z_1,\dots,z_n)$ with 
basis $\{T_w\}_{w\in \Sn}$ and relations 
\begin{equation}\label{eq:Hecke_T}
T_{u}T_w=T_{uw}\quad\text{if $\ell(uw)=\ell(u)+\ell(w)$},\quad\text{and}\quad (T_i+q)(T_i-1)=0 \quad\text{for $i\in[1,n-1]$}, 
\end{equation}
where $T_i:=T_{s_i}$ and $\ell(w)$ denotes \emph{length} of $w\in \Sn$, i.e., the number of inversions of $w$. For the identity permutation $\id\in\Sn$, we have $T_{\id}=1\in\Hecke$.

For $j\in[1,n]$ and $u,w\in \Sn$, we use the convention that $(uw)(j):=w(u(j))$. For $k\in[1,n-1]$ and $\spx\in \C(q;\bzz)$, we write 
\[R_k(\spx):=\spx T_k+(1-\spx)\in\Hecke.\] 
Given an arbitrary permutation $\pi\in\Sn$ and $k\in[1,n-1]$, the above relations imply the following rules for multiplying $T_\pi$ by $R_k(\spx)$:
\begin{align}
\label{eq:T_w*R_k}
T_\pi\cdot R_k(\spx)&=
  \begin{cases}
    \spx T_{\pi s_k}+(1-\spx)T_\pi, &\text{if $\ell(\pi s_k)=\ell(\pi)+1$,}\\
    q\spx T_{\pi s_k}+(1-q\spx)T_\pi, &\text{if $\ell(\pi s_k)=\ell(\pi)-1$;}\\
  \end{cases}\\
\label{eq:R_k*T_w}
 R_k(\spx)\cdot T_\pi&=
  \begin{cases}
    \spx T_{s_k \pi}+(1-\spx)T_\pi, &\text{if $\ell(s_k\pi)=\ell(\pi)+1$,}\\
    q\spx T_{s_k \pi}+(1-q\spx)T_\pi, &\text{if $\ell(s_k\pi)=\ell(\pi)-1$.}\\
  \end{cases}
\end{align}
We note the formal similarity between~\eqref{eq:T_w*R_k} and \cref{fig:spec}. We will make this precise in \cref{prop:YB_equals_Zbipi} by interpreting the \SCSV model in terms of products of elements of the form $R_k(\spx)$. With this interpretation, \eqref{eq:R_k*T_w} turns into the \emph{non-local relations} studied in~\cite[Theorem~5.3.1]{BW}. In the language of \cref{sec:intro:main}, applying~\eqref{eq:T_w*R_k}  (resp.,~\eqref{eq:R_k*T_w}) corresponds to adding a single square to the top right (resp., bottom left) boundary of a \skew domain $\SKD$. Then~\eqref{eq:T_w*R_k} can be proved using a simple bijection on configurations that only changes them locally inside that square. The analogous bijection proving~\eqref{eq:R_k*T_w} is much more complicated and involves changing the configurations globally. See the discussion after~\cite[Theorem~5.3.1]{BW} for further details.

\subsection{Wiring diagram domains}\label{sec:wiring_domains}
We would like to consider the \SCSV model described in \cref{sec:intro-main} associated to more general domains that we call \emph{wiring diagram domains}. 

Let $\bi=(i_1,i_2,\dots,i_r)$ be an arbitrary sequence of elements of $[1,n-1]$, and choose an arbitrary family $\bro=(\sp_1,\sp_2,\dots,\sp_r)$ of elements of $\C(q;\bzz)$. The sequence $\bi$ gives rise to a \emph{wiring diagram} as in \cref{ex:wdom_gen} and \cref{fig:wdom_gen}: there are $n$ paths called \emph{wires}, and each wire moves horizontally left-to-right. The wires start on the left at heights $1,2,\dots,n$. For each $j\in[1,r]$, the wires at heights $i_j$ and $i_j+1$ \emph{cross} at the point $\left(j,i_j+\frac12\right)$ and then proceed to the right at heights $i_j+1$ and $i_j$, respectively.

Let us also  fix a permutation $\cper\in\Sn$ called the \emph{incoming color permutation}.

\begin{definition}\label{dfn:wiring_domain}
The \emph{\SCSV model inside $(\bi,\bro)$ with incoming colors $\cper$} is a probability distribution on $\Sn$ denoted  $(\ZBiber)_{\pi\in\Sn}$, defined as follows. Suppose that for each $c\in[1,n]$, a path of color $\cper(c)$ enters the wiring diagram of $\bi$ on the left at height $c$. We say that \emph{the incoming colors are ordered} if $\cper=\id$. For each $j\in[1,r]$, the two paths entering the crossing point $\left(j,i_j+\frac12\right)$ from the left proceed to the right according to the probabilities in \cref{fig:spec}, where the parameter $\sp$ is equal to $\sp_j$. Once all paths reach the right boundary of the wiring diagram of $\bi$, they give rise to a random \emph{outgoing color permutation} $\pi$ defined so that for $c\in[1,n]$, the path of color $c$ exits at height $\pi(c)$. We let $\ZBiber$ denote the total probability of observing $\pi$ as the outgoing color permutation.
\end{definition}

\begin{figure}
  \FIGwdomgen
  \caption{\label{fig:wdom_gen} Left: a wiring diagram domain from \cref{dfn:wiring_domain}. Right: a configuration of the \SCSV model inside $(\bi,\bro)$, see \cref{ex:wdom_gen}.}
\end{figure}

\begin{example}\label{ex:wdom_gen}
Let $n=5$, $\bi=(4,2,3,2,1,2,3)$, and $\bro:=(\sp_1,\sp_2,\dots,\sp_7)$. The corresponding wiring diagram domain is shown in \cref{fig:wdom_gen}(left). An example of a configuration of the \SCSV model inside $(\bi,\bro)$ with incoming color permutation $\sigma=(3,1,2,5,4)$ and outgoing color permutation $\pi=(1,4,2,5,3)$ is shown in \cref{fig:wdom_gen}(right) together with its probability. 
\end{example}

Our first goal is to interpret the probabilities $\ZBiber$ in terms of elements of $\Hecke$. Let us write
\begin{equation}\label{eq:YBiber_dfn}
\YBiber:=T_\cper R_{i_1}(\sp_1) R_{i_2}(\sp_2)\cdots R_{i_r}(\sp_r)=\sum_{\pi\in\Sn} \YBCxx{\bi,\bro,\cper}\pi T_\pi.
\end{equation}

\begin{proposition}\label{prop:YB_equals_Zbipi}
For all $\bi$, $\bro$, $\cper$, and $\pi$, we have
\begin{equation}\label{eq:Biber:Y_equals_Z}
\ZBiber=\YBCxx{\bi,\bro,\cper}\pi.
\end{equation}
\end{proposition}
\begin{proof}
We prove this by induction on $r$. The case $r=0$ is trivial. Suppose now that~\eqref{eq:Biber:Y_equals_Z} holds for some $\bi=(i_1,i_2,\dots,i_r)$, $\bro=(\sp_1,\sp_2,\dots,\sp_r)$, $\cper\in\Sn$, and all $\pi\in\Sn$. Choose $k\in[1,n-1]$ and $\sp\in\C(q;\bzz)$ and let  $\bi':=(i_1,i_2,\dots,i_r,k)$, $\bro':=(\sp_1,\sp_2,\dots,\sp_r,\sp)$. Consider a permutation $\pi\in\Sn$ such that $\ell(\pi s_k)=\ell(\pi)+1$.   

The following transition formulas are easily deduced respectively from \cref{dfn:wiring_domain} and from Equations~\eqref{eq:T_w*R_k}, \eqref{eq:YBiber_dfn}:
\[\SmallMatrix{
\ZBibx{\bi'}{\bro'}{\pi}\\
\ZBibx{\bi'}{\bro'}{\pi s_k}
}=\SmallMatrix{
1-\sp & q\sp\\
\sp & 1-q\sp
}\SmallMatrix{
\ZBibx{\bi}{\bro}{\pi}\\
\ZBibx{\bi}{\bro}{\pi s_k}
},\quad
\SmallMatrix{
\YBCxx{\bi',\bro',\cper}\pi\\
\YBCxx{\bi',\bro',\cper}{\pi s_k}
}=\SmallMatrix{
1-\sp & q\sp\\
\sp & 1-q\sp
}\SmallMatrix{
\YBCxx{\bi,\bro,\cper}\pi\\
\YBCxx{\bi,\bro,\cper}{\pi s_k}
}. \]
By the induction hypothesis, we have $\ZBibx{\bi}{\bro}{\pi}=\YBCxx{\bi,\bro,\cper}\pi$ and $\ZBibx{\bi}{\bro}{\pi s_k}=\YBCxx{\bi,\bro,\cper}{\pi s_k}$, therefore we get that the left hand sides are equal, completing the induction step.
\end{proof}

\subsection{Color-position symmetry}\label{sec:CPS}
Color-position symmetry in interacting particle systems has been studied in e.g.~\cite{AAV,AHR,BoBu,BW,Kuan2}. 
As a warm up, we explain the recent results of~\cite{BoBu} using the machinery of Hecke algebras. Let $\bi=(i_1,i_2,\dots,i_r)$ and $\bro=(\sp_1,\sp_2,\dots,\sp_r)$ be arbitrary. First, following~\cite{BoBu}, we would like to consider the case where the incoming colors are ordered, i.e., $\cper=\id$. In this case, we usually omit $\cper$ from the notation and write $\YB^{\bi,\bro}$ and $\PF^{\bi,\bro}_\pi(\bzz)$ instead of $\YB^{\bi,\bro,\id}$ and $\PF^{\bi,\bro,\id}(\bzz)$.

For $\bi=(i_1,i_2,\dots,i_r)$ and $\bro=(\sp_1,\sp_2,\dots,\sp_r)$, let us denote $\revv(\bi):=(i_r,\dots,i_2,i_1)$, $\revv(\bro)=(\sp_r,\dots,\sp_2,\sp_1)$. The color-position symmetry of~\cite{BoBu} amounts to the following statement.

\begin{theorem}[{\cite[Theorem~2.2]{BoBu}}]\label{thm:CPS}
  For all $\bi$, $\bro$, and $\pi\in\Sn$, the coefficient of $T_\pi$ in $\YB^{\bi,\bro}$ is equal to the coefficient of $T_{\pi^{-1}}$ in $\YB^{\revv(\bi),\revv(\bro)}$:
\[\PF^{\bi,\bro}_\pi(\bzz)=\PF^{\revv(\bi),\revv(\bro)}_{\pi^{-1}}(\bzz).\]
\end{theorem}
\begin{proof}
Consider the anti-automorphism $\Daut:\Hecke\to\Hecke$ sending $T_\pi$ to $T_{\pi^{-1}}$ for each $\pi\in\Sn$. It reverses the order of multiplication, sending $\YBarb\bi\bro=R_{i_1}(\sp_1)R_{i_2}(\sp_2)\cdots R_{i_r}(\sp_r)$ to $\YBarb{\revv(\bi)}{\revv(\bro)}=R_{i_r}(\sp_r)\cdots R_{i_2}(\sp_2)R_{i_1}(\sp_1)$. The result follows.
\end{proof}

\begin{remark}\label{rmk:Bprep}
The above proof was also independently found by Bufetov~\cite{Buf20}. Connections between Hecke algebras and the ASEP have been discovered in the literature a number of times throughout the years, see e.g.~\cite{ADHR,Cantini,CdGW,CMW}. On the other hand, the equivalence between the \SCSV model and the Yang--Baxter basis of~\cite{LLT} discussed in \cref{sec:yang-baxter-basis} appears to have not been pointed out before.
\end{remark}

Next, we would like to reduce the case of an arbitrary incoming color permutation $\cper$ to the above case $\cper=\id$. Given $\cper\in\Sn$, a \emph{reduced word} for $\cper$ is a sequence $\bj=(j_1,j_2,\dots,j_l)$ such that $\cper=s_{j_1}s_{j_2}\cdots s_{j_l}$ and $l=\ell(\cper)$. For two sequences $\bi=(i_1,\dots,i_r)$, $\bj=(j_1,\dots,j_l)$, let $\concat\bi\bj=(i_1,\dots,i_r,j_1,\dots,j_l)$ denote their concatenation. For $l\geq1$, denote by $1^l:=(1,1,\dots,1)$ the sequence that consists of $l$ ones.
\begin{lemma}\label{lemma:inc_colors}
Let $\bi=(i_1,i_2,\dots,i_r)$, $\bro=(\sp_1,\sp_2,\dots,\sp_r)$, and $\cper\in\Sn$ be arbitrary, and choose a reduced word $\bj=(j_1,j_2,\dots,j_l)$ for $\cper$, where $l=\ell(\cper)$. Let $\bi':=\concat\bj\bi$ and $\bro':=\concat{1^l} \bro$. Then 
\[\YBiber=\YB^{\bi',\bro',\id}.\]
\end{lemma}
\begin{proof}
By definition, $\YB^{\bi',\bro',\id}=R_{j_1}(1)\cdots R_{j_l}(1)\YB^{\bi,\bro,\id}$. Since $R_k(1)=T_k$, we get 
$\YB^{\bi',\bro',\id}=T_\cper\YB^{\bi,\bro,\id}$, which is by definition equal to  $\YBiber$.
\end{proof}

\begin{remark}
\Cref{lemma:inc_colors} yields a generalization of \cref{thm:CPS} to the case of arbitrary incoming colors: for the \SCSV model inside $(\bi,\bro)$ with incoming colors $\sigma$, the probability of observing $\pi$ as the outgoing color permutation is equal to the probability of observing $\pi^{-1}$ as the outgoing color permutation for the \SCSV model inside $(\bi',\bro')$ with ordered incoming colors $(\sigma'=\id)$. Here $\bi'=\concat{\revv(\bi)}{\revv(\bj)}$, $\bj$ is a reduced word for $\cper$, and $\bro'=\concat\bro{1^{\ell(\cper)}}$.
\end{remark}

\begin{figure}
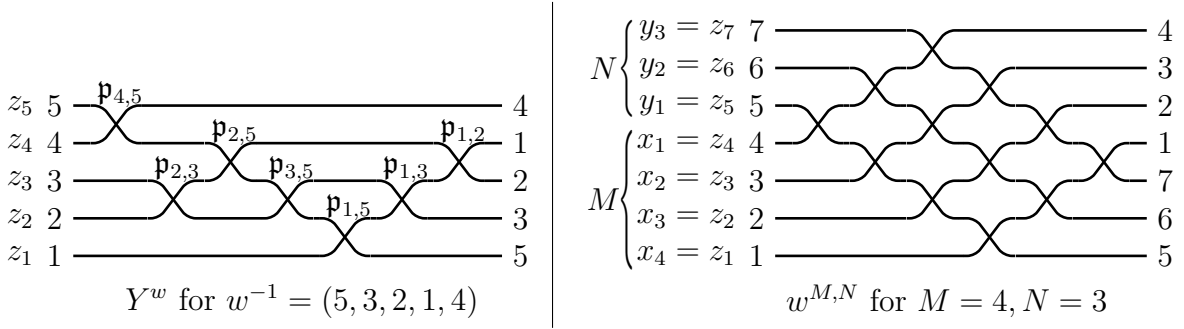

  \FIGwdomy
  \caption{\label{fig:wdom_Y} Left: the wiring diagram domain associated with a Yang--Baxter basis element $\YB^w$. Here $\sp_{i,j}=\frac{\zz_j-\zz_i}{\zz_j-q\zz_i}$, see~\eqref{eq:YB_rec}. Right: The case of a Grassmannian permutation $\wMN$ from \cref{ex:wMN}.}
\end{figure}
\subsection{Yang--Baxter basis}\label{sec:yang-baxter-basis}
The Hecke algebra $\Hecke$ has a basis called the \emph{Yang--Baxter basis} $\{\YB^w\}_{w\in\Sn}$, introduced in~\cite{LLT}. It is defined via the following recurrence relation. For the identity permutation $\id\in \Sn$, we set $\YB^\id=T_\id=1\in\Hecke$. For any $w\in \Sn$ and $k\in[1,n-1]$ such that $\ell(ws_k)=\ell(w)+1$ (equivalently, such that $w^{-1}(k)<w^{-1}(k+1)$), we have
\begin{equation}\label{eq:YB_rec}
\YB^{ws_k}=\YB^w\cdot R_k(\sp_{w^{-1}(k),w^{-1}(k+1)}),\quad\text{where}\quad
\sp_{i,j}:=\frac{\zz_j-\zz_i}{\zz_j-q\zz_i}.
\end{equation}
This is a special case of the construction from the previous subsection, see \cref{fig:wdom_Y}(left): we have $\YB^w=\YB^{\bi,\bro,\id}$ for a reduced word $\bi$ of $w$ and a particular choice of $\bro$. The element $\YB^w$ can be computed via this recurrence in various ways (which correspond to the various reduced words for $w$), but the result is uniquely determined because the elements $R_k(\spx)$ satisfy the following \emph{Yang--Baxter relation}: for any $a<b<c\in[1,n]$ and $k\in [1,n-2]$, we have
\begin{equation}\label{eq:YB}
R_k(\sp_{a,b})R_{k+1}(\sp_{a,c})R_{k}(\sp_{b,c})=R_{k+1}(\sp_{b,c})R_{k}(\sp_{a,c})R_{k+1}(\sp_{a,b}).
\end{equation}
Consider the entries $\PF^{w}_\pi(\bz)$ of the transition matrix between the bases $\{\YB^w\}$ and $\{T_\pi\}$:
\begin{equation}\label{eq:YBC_dfn}
  \YB^w=\sum_{\pi\in \Sn} \PF^{w}_\pi(\bz) T_\pi.
\end{equation}
By \cref{prop:YB_equals_Zbipi}, each of these coefficients equals the probability of observing $\pi$ as an outgoing color permutation for the \SCSV model associated with a reduced word for $w$.

\subsection{From \skew domains to wiring diagram domains}\label{sec:skew_to_wiring}
Let $\SKD$ be a \skew domain with $n=|\Ppath|=|\Qpath|$. Recall that the column and row rapidities are given by $\bx=(x_1,x_2,\dots)$ and $\by=(y_1,y_2,\dots)$. The \SCSV model inside $\SKD$ gives rise to a probability distribution $(\Zpi)_{\pi\in\Sn}$ on permutations. We claim that there exists a permutation $\wPQ\in \Sn$ and the values $\bzzPQ=(\zzPQ_1,\zzPQ_2,\dots,\zzPQ_n)$ such that for all $\pi\in\Sn$, we have $\PF^{\wPQ}_{\pi}(\bzzPQ)=\Zpi$. Indeed, recall that the steps of $\Ppath$ and $\Qpath$ are given respectively by $\ein 1,\ein 2,\dots,\ein n$ and $\eout 1,\eout 2,\dots,\eout n$. Let $i\in[1,n]$. Suppose that $\ein i$ is a vertical step located in row $r\in\ZP$.  Then there exists a unique vertical step $\eout j$ of $\Qpath$ in the same row, and we set $\zzPQ_i:=y_{r}$ and $\wPQ(i):=j$. Similarly, suppose that $\ein i$ is a horizontal step located in column $c\in\ZP$. Then there exists a unique horizontal step $\eout j$ of $\Qpath$ in the same column, and we set $\zzPQ_i:=x_{c}$ and $\wPQ(i):=j$. Comparing the descriptions of the \SCSV model in \cref{sec:intro-main,sec:wiring_domains}, we find the following result.
\begin{proposition}\label{prop:PQ_vs_w}
For any \skew domain $\SKD$ and any $\pi\in\Sn$, we have
\[\PF^{\wPQ}_{\pi}(\bzzPQ)=\Zpi.\]
\end{proposition}
\noindent A similar statement holds for the case of an arbitrary incoming color permutation $\cper$.

\begin{example}\label{ex:wMN}
In the setting of \cref{sec:intro-rect}, consider the \SCSV model inside an $M\times N$-rectangular domain $\SKD$. We have $n=M+N$, $\bzzPQ=(x_M,\dots,x_1,y_1,\dots,y_N)$, and the permutation $\wMN:=\wPQ$ is defined by $\wMN(i)=i+N$ modulo $n$ (thus $\wMN$ is a \emph{Grassmannian permutation} of length $MN$). We can write it as a product 
\begin{equation}\label{eq:Gr_wred}
\wMN=(s_Ms_{M+1}\cdots s_{M+N-1})\cdot (s_{M-1}s_M\cdots s_{M+N-2})\cdots (s_1s_2\cdots s_{N})
\end{equation}
of $MN$ simple transpositions. They correspond naturally to the cells in $[1,M]\times [1,N]=\SKDZ$. The case $M=4$, $N=3$ is shown in \cref{fig:wdom_Y}(right).
\end{example}

\section{Flip symmetry}\label{sec:flip}
Our goal is to prove \cref{thm:flip} and state a more general version that will be used in the proof of \cref{thm:main}. Throughout the first two subsections, we fix an $M\times N$-rectangular domain $\SKD$, a vertical boundary condition $\V=\{(\VD_1,\VU_1),\dots,(\VD_\v,\VU_\v)\}$, and an integer $\h$.

\subsection{Boundary conditions and the Hecke algebra}
Let $\H=\{(\HL_1,\HR_1),\dots,(\HL_\h,\HR_\h)\}$ be a set of pairs. Let $\SATHV:=\{\pi\in\Sn\mid\Hpi=\H\text{ and }\Vpi=\V\}$ denote the set of all permutations $\pi\in\Sn$ satisfying given horizontal and vertical boundary conditions $(\H,\V)$. For an arbitrary element 
\begin{equation}\label{eq:HVpart}
  Y=\sum_{\pi\in\Sn}\YBC_\pi T_\pi \in\Hecke,\quad\text{let}\quad\HVpart{Y}:=\sum_{\pi\in\SATHV} \YBC_\pi\in\C(q;\bzz).
\end{equation}
\def\sper{\tau}
For $\H=\{(\HL_1,\HR_1),\dots,(\HL_\h,\HR_\h)\}$ and a permutation $\sper\in\Sn$, we denote 
\[\sper\cdot \H:=\{(\sper(\HL_1),\HR_1),\dots,(\sper(\HL_\h),\HR_\h)\}\quad\text{and}\quad\H\cdot \sper:=\{(\HL_1,\sper(\HR_1)),\dots,(\HL_\h,\sper(\HR_\h))\}.\]

Let $r\in [1,N-1]$. We write $\H<\H\cdot s_r$ if for all $\pi\in\SATHV$, we have $\ell(\pi)<\ell(\pi s_r)$. Similarly, we write $\H>\H\cdot s_r$ if for all $\pi\in\SATHV$, we have $\ell(\pi)>\ell(\pi s_r)$. It is easy to check that if neither $\H>\H\cdot s_r$ nor $\H<\H\cdot s_r$ is satisfied then we must have $\H=\H\cdot s_r$ as sets of pairs. Indeed, for each $\pi\in\SATHV$, the condition that $\pi$ satisfies $\H$ implies that either $\pi^{-1}(r)=\HL_j$ for some $j\in[1,\h]$, or $\pi^{-1}(r)\in[1,M]$, and similarly for $\pi^{-1}(r+1)$. If both $\pi^{-1}(r),\pi^{-1}(r+1)$ belong to $[1,M]$ then $\H=\H\cdot s_r$. Otherwise, $\H$ determines whether $\pi^{-1}(r)<\pi^{-1}(r+1)$ or $\pi^{-1}(r)>\pi^{-1}(r+1)$, which determines respectively whether $\H\cdot s_r>\H$ or $\H\cdot s_r<\H$. Similarly, for $l\in[M+1,n-1]$, we write $\H<s_l\cdot \H$ (resp., $\H>s_l\cdot \H$) if for all $\pi\in\SATHV$,  we have $\ell(\pi)<\ell(s_l\pi)$ (resp., $\ell(\pi)>\ell(s_l\pi)$).

For an element $Y\in\Hecke$ and $r\in[1,N-1],l\in[M+1,n-1]$, \eqref{eq:T_w*R_k}--\eqref{eq:R_k*T_w} give
\begin{align}
\label{eq:HV_T_w*R_k}
\HVpart{(Y\cdot R_r(\spx))}&=
  \begin{cases}
    \HVpart{Y}, &\text{if $\H\cdot s_r=\H$,}\\
    \spx \HVpx{Y}{\H\cdot s_r}+(1-q\spx)\HVpx{Y}{\H}, &\text{if $\H\cdot s_r<\H$,}\\
    q\spx \HVpx{Y}{\H\cdot s_r}+(1-\spx)\HVpx{Y}{\H}, &\text{if $\H\cdot s_r>\H$;}
  \end{cases}\\
\label{eq:HV_R_k*T_w}
 \HVpart{(R_{l}(\spx)\cdot Y)}&=
  \begin{cases}
    \HVpart{Y}, &\text{if $s_l\cdot \H=\H$,}\\
    \spx \HVpx Y{s_l\cdot \H}+(1-q\spx)\HVpx{Y}{\H}, &\text{if $s_l\cdot \H<\H$,}\\
    q\spx \HVpx Y{s_l\cdot \H}+(1-\spx)\HVpx{Y}{\H}, &\text{if $s_l\cdot \H>\H$.}
  \end{cases}
\end{align}
\def\YBCP{\YBC}
Let us prove one of these identities, the other cases being completely analogous. Suppose that $\H\cdot s_r>\H$ and let $\PA:=\SATHV$. Then we may assume $Y=\sum_{\pi\in \PA} \YBC_\pi T_\pi+\sum_{\pi\in \PA} \YBCP_{\pi s_r} T_{\pi s_r}$, so $\HVpart{Y}=\sum_{\pi\in \PA} \YBC_\pi$ and $\HVpx{Y}{\H\cdot s_r}=\sum_{\pi\in \PA} \YBCP_{\pi s_r}$. Applying~\eqref{eq:T_w*R_k}, we find 
\begin{align*}
Y\cdot R_r(\spx)&=\sum_{\pi\in \PA} \YBC_\pi \left(\spx T_{\pi s_r}+(1-\spx) T_\pi\right)+\sum_{\pi\in \PA} \YBCP_{\pi s_r} \left(q\spx T_{\pi}+(1-q\spx) T_{\pi s_r}\right)\\
               &=\sum_{\pi\in \PA}\left(\YBC_\pi(1-\spx)+\YBCP_{\pi s_r}q\spx  \right) T_\pi 
+\sum_{\pi\in \PA}\left( \YBCP_{\pi s_r}(1-q\spx)+\YBC_{\pi}\spx\right) T_{\pi s_r}.
\end{align*}
Thus $\HVpart{(Y\cdot R_r(\spx))}=\sum_{\pi\in \PA}\left(\YBC_\pi(1-\spx)+\YBCP_{\pi s_r}q\spx  \right)=(1-\spx)\HVpart{Y}+q\spx \HVpx{Y}{\H\cdot s_r}$, which completes one of the cases in~\eqref{eq:HV_T_w*R_k}.

\subsection{Proof of \cref{thm:flip}}\label{sec:flip_proof}
We would like to show that for any horizontal boundary condition $\H=\{(\HL_1,\HR_1),\dots,(\HL_\h,\HR_\h)\}$, we have $\ZHV=\ZHFV$. We will do this by induction, making extensive use of Equations~\eqref{eq:HV_T_w*R_k}--\eqref{eq:HV_R_k*T_w}. We will induct on the set $\HRS:=\{\HR_1,\HR_2,\dots,\HR_\h\}$ using a certain partial order defined below.

Let $w:=\wMN$ be the Grassmannian permutation from \cref{ex:wMN}, and let $\YB^w\in\Hecke$ be the corresponding element of the Yang--Baxter basis. Equation~\eqref{eq:YB} implies
\begin{equation}\label{eq:master_flip}
  \YB^w\cdot R_r(\sp_{l,l+1})=R_l(\sp_{l,l+1})\cdot \sz_l(\YB^w)\quad\text{for $r\in [1,N-1]$ and $l:=r+M$,}
\end{equation}
where $\sz_l$ is the automorphism of $\Hecke$ that swaps $z_l$ and $z_{l+1}$.

Let $\H_0:=\{(M+1,1),(M+2,2),\dots,(M+\h,\h)\}$ be the horizontal boundary condition where the corresponding $\h$ colored paths are fully packed at the bottom of the rectangle.  In this case, we have $\flip(\H_0)=\{(n,N),(n-1,N-1),\dots,(n-\h+1,N-\h+1)\}$, and the corresponding $\h$ paths are fully packed at the top of the rectangle. 

Since $\V,P,Q,\bx$ are fixed, we denote $\ZHY[\H,\by]:=\ZHV$. The idea of the proof is to first show the result for $\H_0$ and then use~\eqref{eq:master_flip} to express $\ZHY[\H,\by]$ inductively for any $\H$ in terms of $\ZHY[\H_0,\by]$. Because of the $\flip$-symmetry of~\eqref{eq:HV_T_w*R_k}--\eqref{eq:master_flip}, 
 it will follow that $\ZHY[\flip(\H),\rev(\by)]$ can be expressed in terms of $\ZHY[\flip(\H_0),\rev(\by)]$ in a symmetric way, which will imply $\ZHY[\H,\by]=\ZHY[\flip(\H),\rev(\by)]$.

\begin{figure}
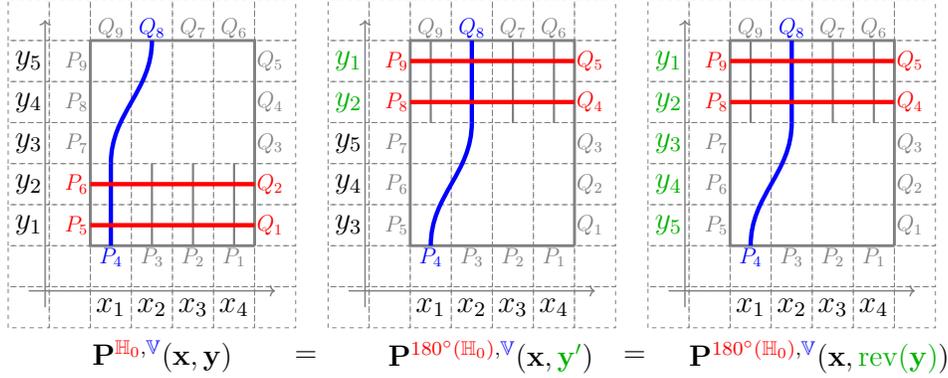

  \FIGflipbase
  \caption{\label{fig:flip_base} The base of the induction in the proof of \cref{thm:flip}. The first equality is true by definition, and the second equality follows by applying a sequence of Yang--Baxter moves to flip the order of $y_3,y_4,y_5$.}
\end{figure}

We start with the base case $\ZHY[\H_0,\by]=\ZHY[\flip(\H_0),\rev(\by)]$, illustrated in \cref{fig:flip_base}. Recall from \cref{fig:wdom_Y}(right) and \cref{ex:wMN} that the variables $\bx,\by,\bz$ are related as $\bz=(x_M,x_{M-1},\dots,x_1,y_1,y_2,\dots,y_N)$. Let $\by':=(y_{\h+1},y_{\h+2},\dots,y_N,y_\h,y_{\h-1},\dots,y_1)$ and recall that $\rev(\by)=(y_N,y_{N-1},\dots,y_1)$. It follows from the definition of the model that $\ZHY[\H_0,\by]=\ZHY[\flip(\H_0),\by']$. To go from $\by'$ to $\rev(\by)$, we use the standard application of the Yang--Baxter equation, see e.g.~\cite[Eq.~(5.18)]{BGW}: it is clear that $\H_0':=\flip(\H_0)$ satisfies $\H_0'\cdot s_r=s_{r+M}\cdot \H_0'=\H_0$ for all $r\in[1,N-\h-1]$. Thus~\eqref{eq:master_flip} combined with~\eqref{eq:HV_T_w*R_k}--\eqref{eq:HV_R_k*T_w} shows that $\ZHY[\flip(\H_0),\by']$ is symmetric in the variables $y_{\h+1},y_{\h+2},\dots,y_N$ and therefore equals $\ZHY[\flip(\H_0),\rev(\by)]$. This shows $\ZHY[\H_0,\by]=\ZHY[\flip(\H_0),\rev(\by)]$, finishing the induction base.

Recall that for $\H=\{(\HL_1,\HR_1),\dots,(\HL_\h,\HR_\h)\}$, we denote  $\HRS:=\{\HR_1,\HR_2,\dots,\HR_\h\}\subset[1,N]$. For two $\h$-element subsets $I=\{i_1<i_2<\dots<i_\h\}$ and $J=\{j_1<j_2<\dots<j_\h\}$ of $[1,N]$, we write $I\leqI J$ if $i_1\leq j_1,i_2\leq j_2,\dots,i_\h\leq j_\h$. We also write $I\lessI J$ if $I\leqI J$ but $I\neq J$. Note that if $\HRS=[1,\h]$ is minimal in this order then either $\H=\H_0$ or $\ZHY[\H,\by]=0$.

\begin{figure}
  \FIGhsr
  \caption{\label{fig:Hsr} The induction step in the proof of \cref{thm:flip}: one can express $\ZHY[\H,\by]$ recursively in terms of $\ZHY[\H',\by]$ and $\ZHY[\H'',\by]$.}
\end{figure}

Let $\H=\{(\HL_1,\HR_1),\dots,(\HL_\h,\HR_\h)\}$. Suppose that we have shown $\ZHY[\H',\by]=\ZHY[\flip(\H'),\rev(\by)]$ for all $\H'$ satisfying $\HRx(\H')\lessI \HRS$. If $\H= \H_0$ then we are done, otherwise $\HRS\neq [1,\h]$, so let $r\in[1,N-1]$ be an index such that $r\notin \HRS$ but $r+1\in\HRS$. Let $\H':=\H\cdot s_r>\H$ and denote $l:=r+M$.  Applying the map $Y\mapsto \HVpx{Y}{\H'}$ to both sides of~\eqref{eq:master_flip} and using the relations~\eqref{eq:HV_T_w*R_k}--\eqref{eq:HV_R_k*T_w}, we get
\begin{equation}\label{eq:YH_ind_step}
\YBC_1\HVpx{(\YB^w)}{\H}+\YBC_1'\HVpx{(\YB^w)}{\H'}=\YBC_2' \HVpx{\sz_l(\YB^w)}{\H'}+\YBC_2''\HVpx{\sz_l(\YB^w)}{\H''},
\end{equation}
for some coefficients $\YBC_1$, $\YBC_1'$, $\YBC_2'$, and $\YBC_2''$, where $\H'':=s_l\cdot \H'=s_l\cdot \H\cdot s_r$. See \cref{fig:Hsr}.

Since $\H'\cdot s_r=\H<\H'$, we get $\YBC_1=\sp_{l,l+1}$ and $\YBC_1'=(1-q\sp_{l,l+1})$. The coefficients $\YBC_2'$ and $\YBC_2''$ depend on whether $\H'<\H''$, $\H'=\H''$, or $\H'>\H''$. For example, \cref{fig:Hsr} illustrates the case $\H'>\H''$. Recall that $\ZHY[\H,\by]=\HVpx{(\YB^w)}{\H}$ by \cref{prop:PQ_vs_w}. 
 Thus~\eqref{eq:YH_ind_step} allows us to express $\ZHY[\H,\by]$ as a linear combination of $\ZHY[\H',\by]$, $\sz_l \ZHY[\H',\by]$, and $\sz_l \ZHY[\H'',\by]$.

Let $\lbar:=n-r$, $\rbar:=n-l$. Observe that $\H<\H\cdot s_r$ implies $\flip(\H)<\flip(\H\cdot s_r)$, where $\flip(\H\cdot s_r)=s_{\lbar}\cdot \flip(\H)$. Applying the map $Y\mapsto \HVpx{Y}{\flip(\H')}$ to both sides of~\eqref{eq:master_flip} yields
\begin{equation}\label{eq:YH_ind_step'}
\bar\YBC_2' \HVpx{(\YB^w)}{\flip(\H')}+\bar\YBC_2''\HVpx{(\YB^w)}{\flip(\H'')}=\bar\YBC_1\HVpx{\sz_{\lbar}(\YB^w)}{\flip(\H)}+\bar\YBC_1'\HVpx{\sz_{\lbar}(\YB^w)}{\flip(\H')}
\end{equation}
for some coefficients $\bar\YBC_1$, $\bar\YBC_1'$, $\bar\YBC_2'$, and $\bar\YBC_2''$. Because of the symmetry between Equations~\eqref{eq:HV_T_w*R_k} and~\eqref{eq:HV_R_k*T_w},  we find that the tuple $(\bar\YBC_1,\bar\YBC_1',\bar\YBC_2',\bar\YBC_2'')$ is obtained from $(\YBC_1,\YBC_1',\YBC_2',\YBC_2'')$ by replacing $\sp_{l,l+1}$ with $\sp_{\lbar,\lbar+1}$. For instance, we have $\bar\YBC_1=\sp_{\lbar,\lbar+1}$ and $\bar\YBC_1'=(1-q\sp_{\lbar,\lbar+1})$. 

Since $\sp_{l,l+1}=\left(\sz_{\lbar}(\sp_{\lbar,\lbar+1})\right)\mid_{\by\mapsto\rev(\by)}$, the transformation $\bar\YBC\mapsto \left(\sz_{\lbar}(\bar\YBC)\right)\mid_{\by\mapsto\rev(\by)}$ takes $(\bar\YBC_1,\bar\YBC_1',\bar\YBC_2',\bar\YBC_2'')$ to $(\YBC_1,\YBC_1',\YBC_2',\YBC_2'')$. Applying $\sz_{\lbar}$ to both sides of~\eqref{eq:YH_ind_step'} and substituting $\by\mapsto\rev(\by)$ therefore gives
\begin{equation}\label{eq:YBC_ind_step3}
\begin{split}
\YBC_2' \sz_{l}\ZHY[\flip(\H'),\rev(\by)]+\YBC_2''\sz_{l}&\ZHY[\flip(\H''),\rev(\by)]\\
&=\YBC_1\ZHY[\flip(\H),\rev(\by)]+\YBC_1'\ZHY[\flip(\H'),\rev(\by)].
\end{split}
\end{equation}
Finally, notice that $\HRx(\H')=\HRx(\H'')\prec \HRS$. Applying the induction hypothesis, we get 
\begin{equation}\label{eq:YBC_ind_step4}
\ZHY[\H',\by]=\ZHY[\flip(\H'),\rev(\by)]\quad\text{and}\quad 
\ZHY[\H'',\by]=\ZHY[\flip(\H''),\rev(\by)].
\end{equation}
Combining~\eqref{eq:YH_ind_step}, \eqref{eq:YBC_ind_step3}, and \eqref{eq:YBC_ind_step4} with the fact that $\YBC_1=\sp_{l,l+1}\neq0$,  we find $\ZHY[\H,\by]=\ZHY[\flip(\H),\rev(\by)]$, completing the induction step. \qed

\subsection{Generalized flip theorem}\label{sec:gen_flip}
In order to pass from rectangles to arbitrary \skew domains, we will need to state a certain generalization of the flip theorem to the case where the incoming colors are not necessarily ordered.

Fix two elements $\ci,\cj\in[1,n-1]$. For simplicity, let us first assume that $\ci+\cj=n$ so that the map $k\mapsto \ci+\cj-k$ is a bijection $[1,n-1]\to[1,n-1]$. (We explain how to lift this assumption in \cref{rmk:gen_flip_general_i_j}.) 
Consider an involutive anti-automorphism $Y\mapsto (Y)_{\ci,\cj}^\ast$ of $\Hecke$ sending $T_{k}\mapsto T_{\ci+\cj-k}$ for all $k\in[1,n-1]$, thus we have 
\begin{equation*}
  (T_{k_1}T_{k_2}\cdots T_{k_r})_{\ci,\cj}^\ast:=T_{\ci+\cj-k_r}\cdots T_{\ci+\cj-k_2}T_{\ci+\cj-k_1}.
\end{equation*}

Fix $M\in[1,\ci]$ and $N\in[1,\cj]$ satisfying $N+\ci=M+\cj$.   Consider a permutation $w\in \Sn$ defined for all $r\in[1,n]$ by
\[w(r):=
  \begin{cases}
    r+M+\cj-\ci, &\text{if $r\in[\ci-M+1,\ci]$,}\\
    r-N+\cj-\ci, &\text{if $r\in[\ci+1,\ci+N]$,}\\
    r, &\text{otherwise.}\\
  \end{cases} \]
We let $\Yij:=\YB^w$ be the corresponding Yang--Baxter basis element, see \cref{fig:flip_gen}(left). We also denote  $\Yijx:=\Yij\mid_{\bz\mapsto \rv_{\ci+1,\ci+N}(\bz)}$, where the substitution $\bz\mapsto \rv_{\ci+1,\ci+N}(\bz)$ is a shorthand for $z_{\ci+1}\mapsto z_{\ci+N}$, $z_{\ci+2}\mapsto z_{\ci+N-1}$, \dots, $z_{\ci+N}\mapsto z_{\ci+1}$, see \cref{fig:flip_gen}(right).

Analogously to \cref{sec:intro-rect}, given a permutation $\pi\in\Sn$, we let 
\begin{equation}\label{eq:Hcicj}
\H^{\ci,\cj}_\pi:=\{(i,\pi(i))\mid i> \ci\text{ and }\pi(i)\leq \cj\},\quad    \V^{\ci,\cj}_\pi:=\{(i,\pi(i))\mid i\leq \ci\text{ and }\pi(i)> \cj\}.
\end{equation}
For two sets $\H=\{(\HL_1,\HR_1),\dots,(\HL_\h,\HR_\h)\}$ and $\V=\{(\VD_1,\VU_1),\dots,(\VD_\v,\VU_\v)\}$ of pairs, we denote $\SATop_{\ci,\cj}(\H,\V):=\{\pi\in\Sn\mid \H^{\ci,\cj}_\pi=\H\text{ and }\V^{\ci,\cj}_\pi=\V\}$ and for $Y=\sum_{\pi\in\Sn}\YBC_\pi T_\pi\in\Hecke$, we let $(Y)^{\H,\V}_{\ci,\cj}:=\sum_{\pi\in\SATop_{\ci,\cj}(\H,\V)}\YBC_\pi$ as in~\eqref{eq:HVpart}. Finally, we define
\begin{equation}\label{eq:flip_ci_cj}
\flip_{\ci,\cj}(\H):=\{(\ci+\cj+1-\HR_1,\ci+\cj+1-\HL_1),\dots,(\ci+\cj+1-\HR_\h,\ci+\cj+1-\HL_\h)\}.
\end{equation}

For $1\leq a\leq b\leq n$, we let $\Spar ab$ denote the (parabolic) subgroup of $\Sn$ generated by $\{s_k\mid k\in[a,b-1]\}$, and we let $\Hpar ab$ be the subalgebra of $\Hecke$ generated by $\{T_\pi\mid \pi\in\Spar ab\}$.

\begin{figure}
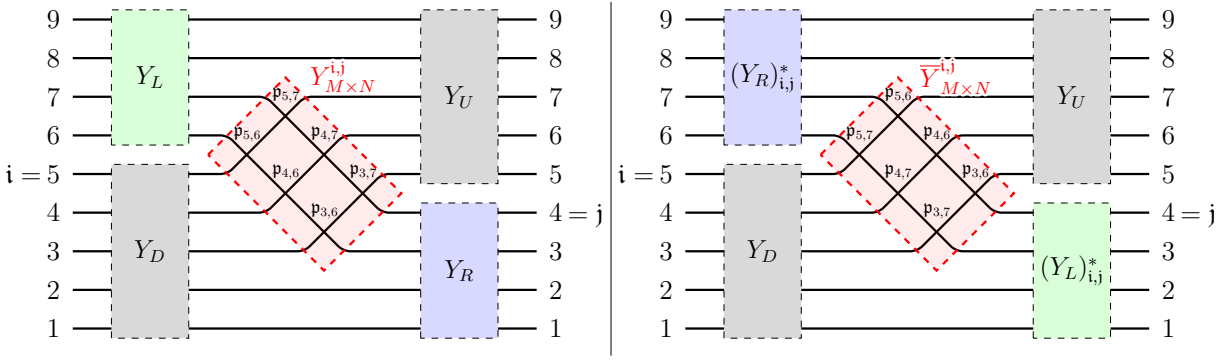

  \FIGflipgen
  \caption{\label{fig:flip_gen}An application of the generalized flip theorem (\cref{thm:flip_gen}).}
\end{figure}

\begin{theorem}\label{thm:flip_gen}
Fix $n,\ci,\cj,M,N$ as above and let 
\[Y_L\in \Hpar{\ci+1}{n},\quad Y_D\in \Hpar{1}{\ci},\quad Y_U\in \Hpar{\cj+1}{n},\quad Y_R\in \Hpar{1}{\cj}\] 
be arbitrary elements. Then for all $\H$ and $\V$, we have
\begin{equation}\label{eq:flip_gen}
\Big(Y_L\cdot Y_D\cdot \Yij\cdot Y_U\cdot Y_R\Big)^{\H,\V}_{\ci,\cj}=\Big((Y_R)_{\ci,\cj}^\ast\cdot Y_D\cdot \Yijx\cdot Y_U\cdot (Y_L)_{\ci,\cj}^\ast\Big)^{\flip_{\ci,\cj}(\H),\V}_{\ci,\cj}.
\end{equation}
\end{theorem}
\noindent See \cref{fig:flip_gen} for an illustration when $n=9$, $\ci=5$, $\cj=4$, $M=3$, and $N=2$.
\begin{remark}
The substitution $\bz\mapsto \rv_{\ci+1,\ci+N}(\bz)$ is performed only for the element $\Yij$, the parameters appearing in the other four elements remain unchanged. Thus if the left hand side $Y_L\cdot Y_D\cdot \Yij\cdot Y_U\cdot Y_R$ belongs to the Yang--Baxter basis, the right hand side $(Y_R)_{\ci,\cj}^\ast\cdot Y_D\cdot \Yijx\cdot Y_U\cdot (Y_L)_{\ci,\cj}^\ast$ in general does not belong to the Yang--Baxter basis.
\end{remark}
\begin{proof}
By the $\C(q;\bzz)$-linearity of the maps $Y\mapsto (Y)^{\H,\V}_{\ci,\cj}$ and $Y\mapsto (Y)_{\ci,\cj}^\ast$, it suffices to show the result for the case $Y_L=T_{w_L}$, $Y_D=T_{w_D}$, $Y_U=T_{w_U}$, $Y_R=T_{w_R}$ for some $w_L\in \Spar{\ci+1}{n}$, $w_D\in \Spar{1}{\ci}$, $w_U\in \Spar{\cj+1}{n}$, $w_R\in \Spar{1}{\cj}$. We do this by induction on the total length of these permutations.

The base case $\left(\Yij\right)^{\H,\V}_{\ci,\cj}=\left(\Yijx\right)^{\flip_{\ci,\cj}(\H),\V}_{\ci,\cj}$ is the content of \cref{thm:flip}. For the induction step, observe that $Y_L$ commutes with $Y_D$ and $Y_U$ commutes with $Y_R$. Let us for example consider the case $w_L=s_kw'_L$ with $\ell(w_L)=\ell(w'_L)+1$. Then $Y_L=T_k Y'_L$, where $Y'_L:=T_{w'_L}$, and we also have $(Y_L)_{\ci,\cj}^\ast=(Y'_L)_{\ci,\cj}^\ast T_{\ci+\cj-k}$. Recall that $T_k=R_k(1)$, thus the rules for how both sides of~\eqref{eq:flip_gen} change under multiplication by $T_k$ are obtained from~\eqref{eq:HV_T_w*R_k}--\eqref{eq:HV_R_k*T_w} by specializing $\spx=1$. The induction step follows. The other three cases are handled similarly.
\end{proof}

\begin{remark}\label{rmk:gen_flip_general_i_j}
We have assumed above that $\ci+\cj=n$. For general $\ci$ and $\cj$, let us introduce $\kmin:=\max(1,\ci+\cj+1-n)$ and $\kmax:=\min(n,\ci+\cj)$. Then the map $k\mapsto \ci+\cj-k$ is a bijection sending $[\kmin,\kmax-1]$ to itself. In the statement of \cref{thm:flip_gen} we then additionally assume that the elements $Y_L$, $\Yij$, and $Y_R$ all belong to $\Hpar{\kmin}{\kmax}$. The proof of \cref{thm:flip_gen} remains the same in this more general case.
\end{remark}

\begin{remark}\label{rmk:bottom_colors}
Substituting $Y_L=Y_U=Y_R=1$ into~\eqref{eq:flip_gen}, we find that the flip theorem (\cref{thm:flip}) holds more generally as described in \cref{rmk:flip_colors_ordered}. Thus the two assumptions required for \cref{thm:flip} to hold are that the incoming colors entering the rectangle from the left are increasing bottom-to-top and are all larger than the colors entering the rectangle from the bottom. 
 \Cref{thm:flip} fails in general if one of these two assumptions is not satisfied.
\end{remark}

\section{Consequences of the flip theorem}\label{sec:flip_conseq}
Our proof of \cref{thm:main} will consist of two parts. In this section, we use the flip theorem to construct a ``zoo'' of transformations on tuples of $\SKD$-cuts that preserve the joint distribution of the associated vector of height functions. In the next section, we will show that, after restricting to \emph{connected components}, any transformation satisfying the conditions of \cref{thm:main} can be obtained as a composition of transformations constructed in this section. We start by introducing some notation related to $\SKD$-cuts.

\begin{definition}\label{dfn:cut_cross}
Let $\SKD$ be a \skew domain and consider a $\SKD$-cut $\cut=(l,d,u,r)$. Recall from \cref{sec:intro:main} and \cref{fig:htpq}(middle) that we associate the numbers $1\leq \ci,\cj\leq n$ to $\cut$ so that $\ci$ (resp., $\cj$) is the number of steps in $\Ppath$ (resp., of $\Qpath$) between its start and the bottom left (resp., top right) corner of the cell $(l,d)$ (resp., $(r,u)$). We call the numbers $(\ci,\cj)$ the \emph{color cutoff levels} of $\cut$.  Given another $\SKD$-cut $\cut'=(l',d',u',r')\neq\cut$, we say that $\cut$ and $\cut'$ \emph{cross} if the closed line segments $[(l,d),(r,u)]$ and $[(l',d'),(r',u')]$ intersect in the plane. Equivalently, $\cut$ and $\cut'$ \emph{cross} if $[l,r]\subset[l',r']$ and $[d,u]\supset[d',u']$ or vice versa. 
\end{definition}

For the rest of this section, we fix a \skew domain $\SKD$ and a tuple $\CUT=(\cut_1,\cut_2,\dots,\cut_m)$ of $\SKD$-cuts, where $\cut_i=(l_i,d_i,u_i,r_i)$ for $i\in[1,m]$. We let 
\begin{equation*}
  \HT(\CUT):=\left(\HT(\cut_1),\HT(\cut_2),\dots,\HT(\cut_m)\right) 
\end{equation*}
denote the associated vector of height functions.

\subsection{Admissible transformations: definitions}

Given two integers $\al\leq \ar\in\ZP$, we let $\rv_{\al,\ar}:\ZP\to\ZP$ be the bijection defined by
\begin{equation*}
  \rv_{\al,\ar}(k)=
  \begin{cases}
    \al+\ar-k, &\text{if $k\in[\al,\ar]$,}\\
    k, &\text{otherwise,}
  \end{cases}\qquad\text{for all $k\in\ZP$.}
\end{equation*}

\begin{definition}\label{dfn:cut_adm}
A \emph{transformation} is a pair $\PHI:=(\pih,\piv)$ of bijections $\pih,\piv:\ZP\to\ZP$. We say that a transformation $\PHI$ is \emph{$\CUT$-admissible} if there exists another \skew domain $\SKDp$ and a tuple $\CUT'=(\cut'_1,\cut'_2,\dots,\cut'_m)$ of $\SKDp$-cuts with $\cut'_i=(l'_i,d'_i,u'_i,r'_i)$ such that for all $i\in[1,m]$, we have $\pih([l_i,r_i])=[l'_i,r'_i]$ and $\piv([d_i,u_i])=[d'_i,u'_i]$. In this case, we say that $\PHI$ is \emph{strongly $\CUT$-admissible} if $\HT(\CUT)\eqd \HTp(\CUT')$, where $\bx':=\pih^{-1}(\bx)=(x_{\pih^{-1}(1)},x_{\pih^{-1}(2)},\dots)$ and $\by':=\piv^{-1}(\by)=(y_{\piv^{-1}(1)},y_{\piv^{-1}(2)},\dots)$.
\end{definition}

\begin{remark}\label{rmk:inverse}
Having $\pih([l_i,r_i])=[l'_i,r'_i]$ is equivalent to having $\suppH(\cut_i)=\suppHp(\cut'_i)$ for $\bx'=\pih^{-1}(\bx)$: we have $\suppHp(\cut'_i)=\{x'_s\}_{s\in \lrp i}=\{x'_{\pih(t)}\}_{t\in\lr i}$, and this equals to $\suppH(\cut_i)=\{x_{t}\}_{t\in\lr i}$ when we take $x'_{\pih(t)}:=x_t$.
\end{remark}

\begin{remark}\label{rmk:depend_on_P_Q}
Given a $\CUT$-admissible transformation $\PHI$, the above tuple $\CUT'$ is uniquely determined by the conditions $\pih([l_i,r_i])=[l'_i,r'_i]$, $\piv([d_i,u_i])=[d'_i,u'_i]$, thus we denote $\CUT'=\PHI(\CUT)$.
 However, there are in general many choices for domains $\SKD$ and $\SKDp$ such that $\CUT$ is a tuple of $\SKD$-cuts and $\CUT'$ is a tuple of $\SKDp$-cuts. So it may appear that the notion of strong $\CUT$-admissibility depends on the choices of \skew domains $\SKD$ and $\SKDp$, but actually this is not the case as we show in \cref{cor:domain_indep}. 
Thus proving \cref{thm:main} amounts to showing that each $\CUT$-admissible transformation is strongly $\CUT$-admissible.  We now give a list of strongly $\CUT$-admissible transformations.
\end{remark}

\subsection{Admissible transformations: examples}\label{sec:examples}
In this subsection, we state that several transformations are strongly $\CUT$-admissible. We will prove these statements later in \cref{sec:admiss-transf-proofs}. Fix some integers $\Minf,\Ninf\in\ZP$ satisfying $\SKDZ\subset[1,\Minf]\times[1,\Ninf]$.

\subsection*{Color-position symmetry}
Our first transformation comes from \cref{thm:CPS}. In fact, this is the only transformation that is not a consequence of the flip theorem (although it is a trivial transformation from the point of view of Hecke algebras as explained in \cref{sec:CPS}). Recall that we have fixed a \skew domain $\SKD$ and a tuple $\CUT=(\cut_1,\cut_2,\dots,\cut_m)$ of $\SKD$-cuts.

\begin{lemma}[Color-position symmetry]\label{lemma:CPS}
$\PHI:=(\rv_{1,\Minf},\rv_{1,\Ninf})$ is a strongly $\CUT$-admissible transformation.
\end{lemma}
\noindent Under this transformation, the domain $\SKD$ and all $\SKD$-cuts rotate by $180$ degrees.

\begin{figure}
  \FIGglobalh
  \caption{\label{fig:global_H} A global \Hd-flip (\cref{lemma:global_H_flip}).}
\vspace{\floatsep}
  \FIGlocalh
  \caption{\label{fig:local_H} A local \Hd-flip (\cref{lemma:local_H_flip}).}
\end{figure}
\subsection*{Global flips}
Our next transformation is a direct application of \cref{thm:flip_gen}. In fact, we describe two transformations, a \emph{global \Hd-flip} and a \emph{global \Vd-flip}. They are obtained from each other by interchanging the horizontal and vertical directions, i.e., reflecting along the line $y=x$. We thus describe only one of the two transformations.
\begin{lemma}[Global \Hd-flip]\label{lemma:global_H_flip}
Suppose that $\cut=(l,d,u,r)$ is a $\SKD$-cut that crosses $\cut_i$ for all $i\in[1,m]$. Then $\PHI:=(\rv_{1,\Minf}\circ\rv_{l,r},\rv_{d,u})$ is a strongly $\CUT$-admissible transformation.
\end{lemma}
\noindent See \cref{fig:global_H} for an example. The dashed rectangle represents the cut $\cut=(l,d,u,r)$.

\subsection*{Local flips}
Similarly to global flips, we introduce \emph{local \Hd-flips} and \emph{local \Vd-flips}. For ${d\leq u\in\ZP}$, we say that $\ZP\times [d,u]$ is a \emph{horizontal $\CUT$-strip} if there does not exist $i\in[1,m]$ such that either $d_i<d\leq u_i<u$ or $d<d_i\leq u<u_i$. In other words, for all $i\in[1,m]$, we have either $u_i<d$, or $u<d_i$, or $[d,u]\subset [d_i,u_i]$, or $[d_i,u_i]\subset[d,u]$.

\begin{lemma}[Local \Hd-flip]\label{lemma:local_H_flip}
Suppose that $\ZP\times[d,u]$ is a horizontal $\CUT$-strip. Let 
\begin{equation}\label{eq:local_H_IJ}
  J:=\{j\in[1,m]\mid [d_j,u_j]\subsetneq [d,u]\},\quad   K:=\{k\in[1,m]\mid [d,u]\subset [d_k,u_k]\}.
\end{equation}
Suppose that there exist integers $l\leq r\in\ZP$ satisfying $[l,r]=[l_j,r_j]$ and $[l_k,r_k]\subset [l,r]$ for all $j\in J$ and $k\in K$. Then $\PHI:=(\id,\rv_{d,u})$ is a strongly $\CUT$-admissible transformation.
\end{lemma}
\noindent See \cref{fig:local_H} for an example with $J=\{1,2,3\}$ and $K=\{4\}$. The dashed rectangle represents the cut $(l,d,u,r)$.

\subsection*{Double flips and shifts}
The last pair of local transformations that we consider are \emph{double \Hd-flips} and \emph{double \Vd-flips}. 

\begin{lemma}[Double \Hd-flip]\label{lemma:double_H_flip}
Let $\ZP\times [d',u']$ and $\ZP\times[d,u]$ be two horizontal $\CUT$-strips with $[d',u']\subsetneq[d,u]$. Denote
\begin{align}\label{eq:double_H_I}
I&:=\{i\in[1,m]\mid [d_i,u_i]\subset[d',u']\},\\
\label{eq:double_H_J}
  J&:=\{j\in[1,m]\mid [d_j,u_j]\subsetneq [d,u]\}\setminus I,\\ 
\label{eq:double_H_K}
K&:=\{k\in[1,m]\mid [d,u]\subset[d_k,u_k]\}.
\end{align}
Suppose that there exist integers $l\leq r\in\ZP$ satisfying $[l,r]\subset[l_i,r_i]$, $[l,r]=[l_j,r_j]$, and $[l_k,r_k]\subset [l,r]$ for all $i\in I$, $j\in J$, and $k\in K$. Then $\PHI:=(\id,\rv_{d,u}\circ\rv_{d',u'})$ is a strongly $\CUT$-admissible transformation.
\end{lemma}
\noindent See \cref{fig:doubleh} for an example with $d=3$, $u=9$, $d'=5$, $u'=6$, $l=4$, $r=7$, $I=\{1,2,3\}$, $J=\{4,5\}$, and $K=\{6\}$. 

\begin{remark}\label{rmk:H_shift_BGW}
Under the double \Hd-flip, the tuple $(\cut_j)_{j\in J}$ rotates by $180$ degrees while the tuple $(\cut_i)_{i\in I}$ shifts up by $u+d-u'-d'$. The transformation makes sense when either $I$ or $J$ is empty. If $I$ is empty then it is essentially a special case of a local \Hd-flip (or rather a composition of two local \Hd-flips one of which does not change the tuple of cuts). When $J$ is empty, we call this transformation an \emph{\Hd-shift}. 
\end{remark}
\begin{remark} \label{rmk:gen_BGW}
When $J=\{j\}$ consists of one element and $d'+u'=d+u$, the double \Hd-flip can be considered as an operation of shifting $\cut_j$ relative to $(\cut_i)_{i\in [1,m]\setminus J}$. In this case, \cref{lemma:double_H_flip} generalizes the shift-invariance results of~\cite{BGW}. Specifically, \cite[Theorems~1.2 and~4.13]{BGW} correspond to the special case of \cref{lemma:double_H_flip} where $|J|=1$, $(d',u')=(d+1,u-1)$, and $I\cup J\cup K=[1,m]$.
\end{remark}

\subsection{Equivalence classes of pipe dreams}\label{sec:equiv-class-pipe}
In order to prove that the above transformations are strongly $\CUT$-admissible, we need to introduce a certain equivalence relation on the set of \SCSV model configurations. Following the literature~\cite{BeBi,FK,KM} on Schubert polynomials, we say that a \emph{$\SKD$-pipe dream} (equivalently, a configuration of the \SCSV model) is a diagram  $\pip$ obtained by replacing each cell in $\SKDZ$ with either a crossing \scalebox{0.4}{\begin{tikzpicture}[baseline=(ZUZU.base)]
\coordinate(ZUZU) at (0,-0.5);\drawgrid{0}{0}{0}{0}\draw[line width=3pt, rounded corners=12] (0.00,-0.50)--(0.00,0.50);\draw[line width=3pt, rounded corners=12] (-0.50,0.00)--(0.50,0.00);
\end{tikzpicture}} or an elbow \scalebox{0.4}{\begin{tikzpicture}[baseline=(ZUZU.base)]
\coordinate(ZUZU) at (0,-0.5);\drawgrid{0}{0}{0}{0}\draw[line width=3pt, rounded corners=12] (0.00,-0.50)--(0.00,0.00)--(0.50,0.00);\draw[line width=3pt, rounded corners=12] (-0.50,0.00)--(0.00,0.00)--(0.00,0.50);
\end{tikzpicture}}. Thus the total number of $\SKD$-pipe dreams equals $2^{\#\SKDZ}$, and each $\SKD$-pipe dream $\pip$ is naturally a union of $n$ colored paths, so we denote the corresponding color permutation by $\cperm_\pip$.  For a $\SKD$-pipe dream $\pip$, denote $\HTpi_\pip(\CUT):=\left(\HTpi_{\cperm_{\pip}}(\cut_1),\dots,\HTpi_{\cperm_\pip}(\cut_m)\right)$.
When the \skew domain $\SKD$ is fixed, we usually omit the dependence on it from the notation.

 The probability $\Prob(\pip)$ of a given pipe dream $\pip$ is defined by
\begin{equation*}
  \Prob(\pip)=\prod_{(i,j)\in\SKDZ}\wtij{i,j}{\pip},\quad\text{where}\quad \wtij{i,j}{\pip}\in\{\sp_{i,j},q\sp_{i,j},1-\sp_{i,j},1-q\sp_{i,j}\}
\end{equation*}
according to the four possibilities in \cref{fig:spec}.  For all $\pi\in\Sn$, we have $\Zpi=\sum_{\pip:\cperm_\pip=\pi} \Prob(\pip)$.

Recall that we have fixed a tuple $\CUT=(\cut_1,\cut_2,\dots,\cut_m)$ of $\SKD$-cuts. 
 For each $k\in[1,m]$,  let $\ci_k,\cj_k\in[1,n]$ be the color cutoff levels of $\cut_k$ (cf. \cref{dfn:cut_cross}).

\begin{figure}
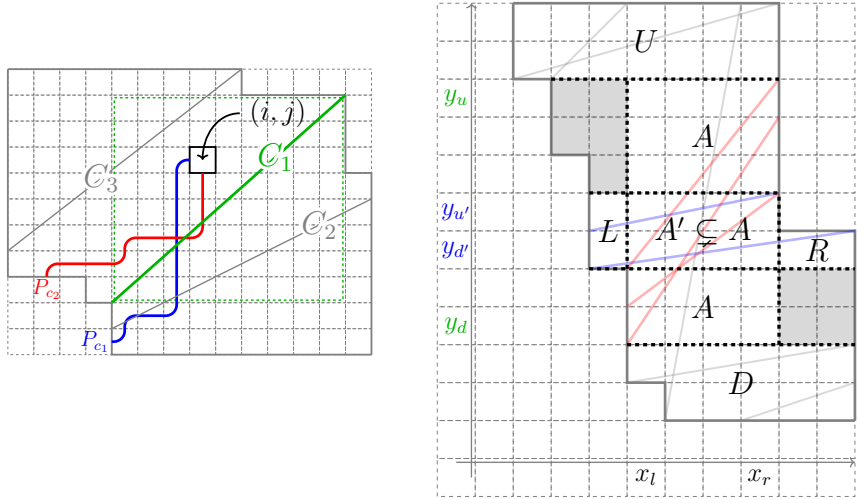

\begin{tabular}{ccc}
  \FIGrel 
& \hspace{1in}&
             \FIGdoubleldur
\end{tabular}
  \caption{\label{fig:rel} Left: The cut $\cut_1$ involves $(i,j)$ in $\pip$ while $\cut_2$ and $\cut_3$ do not, see \cref{dfn:rel}. Right: The partition of the double \Hd-flip \skew domain from \cref{fig:doubleh} as described in the proof of  \cref{lemma:double_H_flip}. The shaded areas contain no cells from $\suppcut$ (and therefore no $(\pip,\CUT)$-relevant cells for any $\pip$).}
\end{figure}
\begin{definition}\label{dfn:rel}
 Given a pipe dream $\pip$ and a cell $(i,j)\in\SKDZ$, let $c_1<c_2\in[1,n]$ be the colors of the two paths entering this cell. For $k\in[1,m]$, we say that $\cut_k$ \emph{involves $(i,j)$ in $\pip$} if $(i,j)\in[l_k,r_k]\times[d_k,u_k]$ and $\ci_k\in[c_1,c_2-1]$, see \cref{fig:rel}(left). We say that $(i,j)$ is \emph{$(\pip,\CUT)$-relevant} if there exists $k\in[1,m]$ such that $\cut_k$ involves $(i,j)$ in $\pip$, otherwise we say that $(i,j)$ is \emph{$(\pip,\CUT)$-irrelevant}. Let $\SKDZ=\rel(\pip,\CUT)\sqcup\irrel(\pip,\CUT)$ denote the sets of all $(\pip,\CUT)$-relevant and $(\pip,\CUT)$-irrelevant cells, respectively. 
\end{definition}

\begin{lemma}\label{lemma:pipe_irrel}
Let $\pip$ and $\pip'$ be pipe dreams such that $\pip'$ is obtained from $\pip$ by changing its values in some $(\pip,\CUT)$-irrelevant cells. Then
\begin{align*}
\rel(\pip,\CUT)&=\rel(\pip',\CUT),\quad \irrel(\pip,\CUT)=\irrel(\pip',\CUT),\quad \HTpi_\pip(\CUT)=\HTpi_{\pip'}(\CUT),\quad\text{and}\\
\wtij{i,j}{\pip}&=\wtij{i,j}{\pip'}
\quad\text{for all $(i,j)\in\rel(\pip,\CUT).$} 
\end{align*}
\end{lemma}
\begin{proof}
Without loss of generality we may assume that $\pip'$ is obtained from $\pip$ by changing the value in a single cell $(i,j)\in\irrel(\pip,\CUT)$ from an elbow to a crossing or vice versa. Let $p_1$ and $p_2$ be the paths in $\pip$ that pass through the cell $(i,j)$  and let $c_1<c_2\in[1,n]$ be their respective colors.

 Let $(i',j')\in\SKDZ$ be some cell such that either $(i',j')\in\rel(\pip,\CUT)\cap\irrel(\pip',\CUT)$ or $(i',j')\in\irrel(\pip,\CUT)\cap\rel(\pip',\CUT)$. Then at least one of the paths $p_1,p_2$ has to enter the cell $(i',j')$ after leaving $(i,j)$, which implies $i\leq i'$ and $j\leq j'$.

 Moreover, there must exist $k\in[1,m]$ such that $(i',j')\in[l_k,r_k]\times[d_k,u_k]$ and $\ci_k\in[c_1,c_2-1]$: the interval $[c_1,c_2-1]$ has to contain at least one such $\ci_k$ in order for the status of $(i',j')$ to change between irrelevant and relevant. Because $\ci_k\in[c_1,c_2-1]$, the cell $(l_k,d_k)$ has to be weakly below and to the left from $(i,j)$, so $(i,j)\in[l_k,r_k]\times[d_k,u_k]$. It follows that $\cut_k$ involves $(i,j)$ in $\pip$, so $(i,j)$ is $(\pip,\CUT)$-relevant, a contradiction. This shows $\irrel(\pip,\CUT)=\irrel(\pip',\CUT)$ and $\rel(\pip,\CUT)=\rel(\pip',\CUT)$. It is also clear that $\HTpi_\pip(\CUT)=\HTpi_{\pip'}(\CUT)$ because $\HTpi_{\pip}(\cut_k)\neq \HTpi_{\pip'}(\cut_k)$ can only happen when $\cut_k$ involves $(i,j)$ in $\pip$.

Suppose now that $(i',j')\in\rel(\pip,\CUT)=\rel(\pip',\CUT)$  is such that $\wtij{i',j'}{\pip}\neq\wtij{i',j'}{\pip'}$. As before, this implies that $i\leq i'$ and $j\leq j'$. Let $k\in[1,m]$ be such that $\cut_k$ involves $(i',j')$ in $\pip$, and let $c'_1<c'_2$ be the colors of the paths entering $(i',j')$ in $\pip$, thus $\ci_k\in[c'_1,c'_2-1]$. 
 In order for $\wtij{i',j'}{\pip}$ to change as we swap $c_1$ with $c_2$, it is necessary to have $\{c'_1,c'_2\}\cap\{c_1,c_2\}\neq\emptyset$ and $[c'_1,c'_2-1]\subset[c_1,c_2-1]$. But then $\cut_k$ involves $(i,j)$ in $\pip$, contradicting the assumption that $(i,j)\in\irrel(\pip,\CUT)$.
\end{proof}

Let us say that pipe dreams $\pip,\pip'$ are \emph{$\CUT$-equivalent} if they are obtained from each other by changing the values of $(\pip,\CUT)$-irrelevant cells. We denote by $\cl[\pip]$ the $\CUT$-equivalence class of $\pip$. It consists of $2^{\#\irrel(\pip,\CUT)}$ elements which all have the same height vector $\HTpi_{\pip}(\CUT)$. By \cref{lemma:pipe_irrel}, we have
\begin{equation}\label{eq:prob_cl_pip}
\Prob(\cl[\pip]):=\sum_{\pip'\in\cl[\pip]} \Prob(\pip')=\prod_{(i,j)\in\rel(\pip,\CUT)} \wtij{i,j}{\pip},
\end{equation}
where for $(i,j)\in\rel(\pip,\CUT)$, $\wtij{i,j}{\pip}$ does not depend on the choice of $\pip\in\cl[\pip]$. Note that 
\begin{equation}\label{eq:suppcut}
\rel(\pip,\CUT)\subset \suppcut:=\bigcup_{i=1}^m \left([l_i,r_i]\times[d_i,u_i]\right),
\end{equation}
where $\suppcut$ does not depend on either $\pip$ or $\SKD$. Thus, modulo $\CUT$-equivalence, we may restrict each $\SKD$-pipe dream only to the cells in $\suppcut$ (or more generally only to the cells in some set $A$ containing $\suppcut$). One immediate application of this approach is that the distribution of the height vector $\HT(\CUT)$ does not depend on the choice of $\SKD$.
\begin{corollary}\label{cor:domain_indep}
Let $\SKD$ and $\SKDp$ be two \skew domains, and let $\CUT=(\cut_1,\cut_2,\dots,\cut_m)$ be such that for all $i\in[1,m]$, $\cut_i$ is both a $\SKD$-cut and a $\SKDp$-cut. Then 
\begin{equation*}
  \HT(\CUT)\eqd \HTop^{\Ppath',\Qpath'}(\CUT;\bx,\by).
\end{equation*}
\end{corollary}
\begin{proof}
The set of relevant cells is contained inside $\suppcut\subset\SKDZ\cap\SKDpZ$, so $\SKD$ and $\SKDp$ give rise to the same set of $\CUT$-equivalence classes of pipe dreams.
\end{proof}

\subsection{Admissible transformations: proofs}\label{sec:admiss-transf-proofs}
We are ready to prove Lemmas~\ref{lemma:CPS}--\ref{lemma:double_H_flip}. We start with a tuple $\CUT=(\cut_1,\cut_2,\dots,\cut_m)$ of $\SKD$-cuts and then each lemma gives a transformation $\Phi=(\pih,\piv)$. It is straightforward to check in each case that this transformation is $\CUT$-admissible, and following \cref{rmk:depend_on_P_Q} we have  another $m$-tuple ${\CUT'=(\cut'_1,\cut'_2,\dots,\cut'_m):=\Phi(\CUT)}$ of $\SKDp$-cuts for some other \skew domain $\SKDp$. As usual, for $k\in[1,m]$, we denote $\cut_k=(l_k,d_k,u_k,r_k)$, $\cut'_k=(l'_k,d'_k,u'_k,r'_k)$, and we let  $(\ci_k,\cj_k)$ and $(\ci'_k,\cj'_k)$ be the color cutoff levels of $\cut_k$ and $\cut'_k$. Recall also that we have $\SKDZ\subset[1,\Minf]\times[1,\Ninf]$.  We let $\bx':=\pih^{-1}(\bx)$ and $\by':=\piv^{-1}(\by)$, cf. \cref{rmk:inverse}.

\begin{proof}[Proof of \cref{lemma:CPS}]
In this case, $\SKDp$ is just a $\flip$-rotation of $\SKD$. Let $w_0\in\Sn$ be the permutation of maximal length, defined by $w_0(i)=n+1-i$ for all $i\in[1,n]$. We claim that the probability $\PF^{\Ppath,\Qpath}_\pi$ equals the probability $\PF^{\Ppath',\Qpath'}_{\pi'}$ for $\pi':=w_0\pi^{-1}w_0$. This follows from the analog of \cref{thm:CPS} where instead of the anti-automorphism $T_\pi\mapsto T_{\pi^{-1}}$, one uses the anti-automorphism $T_\pi\mapsto T_{w_0\pi^{-1}w_0}$. 
 Observe also that $\cj'_k=n-\ci_k$ and $\ci'_k=n-\cj_k$ for all $k\in[1,m]$. We are done since by definition, $\HTop_\pi^{\ci_k,\cj_k}=\HTop_{\pi'}^{n-\cj_k,n-\ci_k}$.
\end{proof}

For the remaining results, we need to discuss the further relationship between height functions, horizontal/vertical boundary conditions, \skew domains, and Hecke algebra elements.

We may assume that the paths $\Ppath$ and $\Qpath$ connect the bottom right and the top left vertices of the rectangle $[1,\Minf]\times[1,\Ninf]$, in which case $n=\Minf+\Ninf$. For a cell $(i,j)\in\SKDZ$, let $\cont ij:=\Minf+j-i\in[1,n-1]$ denote its \emph{content}. For a subset $A\subset\SKDZ$, denote 
\begin{equation}\label{eq:Y_A}
Y_A:=\prod_{(i,j)\in A} R_{\cont ij}(\sp_{i,j})\in\Hecke,
\end{equation}
where the (non-commutative) product is taken in the ``up-right reading order'', so that the term corresponding to $(i,j)$ appears before the term corresponding to $(i',j')$ for all $i'\geq i$ and $j'\geq j$. The element $Y_A=\sum_{\pi\in\Sn}\YBC^A_\pi T_\pi$ defines a probability distribution on $\Sn$. For $k\in[1,m]$ and $\pi\in\Sn$, we set $\HTop_\pi(\cut_k):=\HTop_\pi^{\ci_k,\cj_k}$, and this gives rise to a random variable $\HTop^{Y_A}(\cut_k)$: for $h\in\ZNN$, the probability that $\HTop^{Y_A}(\cut_k)=h$ equals $\sum_{\pi\in\Sn:\HTop_\pi(\cut_k)=h} \YBC^A_\pi$. We let $\HTop^{Y_A}(\CUT)=\left(\HTop^{Y_A}(\cut_1),\dots,\HTop^{Y_A}(\cut_m)\right)$. By \cref{lemma:pipe_irrel}, we have
\begin{equation}\label{eq:Y_A=CUT}
\HTop^{Y_A}(\CUT)\eqd \HT(\CUT)\quad \text{if $\suppcut\subset A\subset\SKDZ$.}
\end{equation}

For the next result, we borrow some notation from \cref{sec:gen_flip}. We assume that all indices appearing in the statement belong to the set $[\kmin,\kmax]$ from \cref{rmk:gen_flip_general_i_j}.
\begin{lemma}\label{lemma:Ht_H_V_flip}
Let $\ci,\cj\in[1,n]$. Let $\pi,\pi'\in\Sn$ be permutations satisfying
\[\H^{\ci,\cj}_{\pi'}=\flip_{\ci,\cj} \left(\H^{\ci,\cj}_{\pi}\right) \quad\text{and}\quad  \V^{\ci,\cj}_{\pi'}=\V^{\ci,\cj}_{\pi}.\]
Then 
\begin{alignat}{2}\label{eq:ci_zero_one}
  \HTop^{\ci_\zero,\cj_\one}_\pi&=\HTop^{\ci_\zero,\cj_\one}_{\pi'}&\qquad\text{for all $\ci_0\leq \ci$ and $\cj_1\geq\cj$;}\\
\label{eq:ci_one_zero}
 \HTop^{\ci_\one,\cj_\zero}_\pi&=\HTop^{\ci+\cj-\cj_\zero,\ci+\cj-\ci_\one}_{\pi'}&\qquad\text{for all $\ci_1\geq \ci$ and $\cj_0\leq\cj$.}
\end{alignat}
\end{lemma}
\begin{proof}
The values $\HTop^{\ci_\zero,\cj_\one}_\pi$ and $\HTop^{\ci_\one,\cj_\zero}_\pi$ can be expressed in terms of $\H^{\ci,\cj}_{\pi}$ and $\V^{\ci,\cj}_{\pi}$: 
\begin{align*}
\HTop^{\ci_\zero,\cj_\one}_\pi&=\#\{(d,u)\in\V^{\ci,\cj}_{\pi}\mid d\leq \ci_\zero\text{ and }u>\cj_\one\}+\cj_\one-\ci_\zero, \\
\HTop^{\ci_\one,\cj_\zero}_\pi&=\#\{(l,r)\in\H^{\ci,\cj}_{\pi}\mid l>\ci_\one\text{ and }r\leq \cj_\zero\}. 
\end{align*}
The first equation immediately implies~\eqref{eq:ci_zero_one}, and comparing the second equation with the definition~\eqref{eq:flip_ci_cj} of $\flip_{\ci,\cj}$, we see that~\eqref{eq:ci_one_zero} also follows.
\end{proof}

\begin{remark}
Given two $\SKD$-cuts $\cut,\cut'$ with color cutoff levels $(\ci,\cj)$ and $(\ci',\cj')$, note that $\cut$ and $\cut'$ cross (in the sense of \cref{dfn:cut_cross}) if and only if either $\ci'\geq \ci$ and $\cj'\leq \cj$ or $\ci'\leq \ci$ and $\cj'\geq \cj$, which is precisely when Equations~\eqref{eq:ci_zero_one} and~\eqref{eq:ci_one_zero} apply.
\end{remark}

\begin{proof}[Proof of \cref{lemma:global_H_flip}]
  Let $C=(l,d,u,r)$ be as in \cref{lemma:global_H_flip}, and let $\ci,\cj$ be its color cutoff levels. Let $L,D,A,U,R\subset \SKDZ$ be the intersections of $\SKDZ$ with $[1,l-1]\times[d,u]$, $[l,r]\times[1,d-1]$, $[l,r]\times[d,u]$, $[l,r]\times[u+1,\Ninf]$, and $[r+1,\Minf]\times[d,u]$, respectively. The assumptions of \cref{lemma:global_H_flip} imply that $\suppcut\subset L\cup D\cup A\cup U\cup R$.  

After possibly changing $\Minf$ and shifting $\SKDZ$ inside $[1,\Minf]\times[1,\Ninf]$, we may assume that:
\begin{itemize}
\item $l+r=1+\Minf$,
\item $\SKDpZ\cap \left([l,r]\times\ZP\right)=\SKDZ\cap \left([l,r]\times\ZP\right)$, and
\item the sets $L':=\SKDpZ\cap \left([1,l-1]\times[d,u]\right)$ and $R':=\SKDpZ\cap \left([r+1,\Minf]\times[d,u]\right)$ are the images of, respectively, $R$ and $L$ under the map $(i,j)\mapsto(l+r-i,u+d-j)$.
\end{itemize}
 We find that 
\begin{equation*}%
  (Y_L)^\ast_{\ci,\cj}=Y_{R'}\mid_{(\bx,\by)\mapsto(\bx',\by')} \quad\text{and}\quad (Y_R)^\ast_{\ci,\cj}=Y_{L'}\mid_{(\bx,\by)\mapsto(\bx',\by')}.
\end{equation*}
Denote $\Yov_A:=Y_A\mid_{(\bx,\by)\mapsto(\bx',\by')}$ and let 
\begin{equation*}
  Y:=Y_L\cdot Y_D\cdot Y_A\cdot Y_U\cdot Y_R,\quad\text{and}\quad \Yov:=(Y_R)_{\ci,\cj}^\ast\cdot Y_D\cdot \Yov_A\cdot Y_U\cdot (Y_L)_{\ci,\cj}^\ast.
\end{equation*}
By \cref{thm:flip_gen}, for all $\H,\V$, we have $(Y)^{\H,\V}_{\ci,\cj}=(\Yov)^{\flip_{\ci,\cj}(\H),\V}_{\ci,\cj}$.  By~\eqref{eq:Y_A=CUT}, we get $\HTop^{Y}(\CUT)\eqd \HT(\CUT)$ and $\HTop^{\Yov}(\CUT')\eqd \HTp(\CUT')$. Finally, we have assumed that the cut $\cut$ crosses all elements of $\CUT$, and therefore all elements of $\CUT'$. Applying \cref{lemma:Ht_H_V_flip}, we find $\HTop^{Y}(\CUT)\eqd \HTop^{\Yov}(\CUT')$.
\end{proof}

\begin{proof}[Proof of \cref{lemma:local_H_flip}]
Let $l,d,u,r,J,K$ be as in \cref{lemma:local_H_flip}. Let $D,A,U\subset \SKDZ$ be the intersections of $\SKDZ$ with $\ZP\times[1,d-1]$, $\ZP\times[d,u]$, and $\ZP\times[u+1,\Ninf]$, respectively. We may assume that $A=[l,r]\times[d,u]$ is a rectangle, in which case $\cut:=(l,d,u,r)$ is a $\SKD$-cut and we denote by $\ci,\cj$ its color cutoff levels.

Given a $\SKD$-pipe dream $\pip$, let us denote by
\begin{equation}\label{eq:Vapip}
\V^A_\pip:=\{(d^{A,\pip}_1,u^{A,\pip}_1),\dots,(d^{A,\pip}_\v,u^{A,\pip}_\v)\}
\end{equation}
the set of all pairs $d^{A,\pip}_i,u^{A,\pip}_i\in[l,r]$ such that the path in $\pip$ that enters $A$ from below in column $d^{A,\pip}_i$ exits $A$ from above in column $u^{A,\pip}_i$. 

Let $\pip$ and $\pip'$ be two $\SKD$-pipe dreams. We say that $\pip$ is \emph{$(\CUT,\cut)$-equivalent} to $\pip'$ if their restrictions to $D\cup (\rel(\pip,\CUT)\cap U)$ coincide and in addition $\V^A_\pip=\V^A_{\pip'}$. Our first goal is to show that $\rel(\pip,\CUT)\cap U=\rel(\pip',\CUT)\cap U$, and thus $\pip'$ is $(\CUT,\cut)$-equivalent to $\pip$. Along the way, we will also see that for each cell $(i,j)\in \rel(\pip,\CUT)\cap  U$, we have $\wtij{i,j}{\pip}=\wtij{i,j}{\pip'}$.

\def\DOWN#1#2#3{\mathcal{D}_{#1,#2}^{\cut_{#3}}(\pip)}
\def\LEFT#1#2#3{\mathcal{L}_{#1,#2}^{\cut_{#3}}(\pip)}
\def\DOWNp#1#2#3{\mathcal{D}_{#1,#2}^{\cut_{#3}}(\pip')}
\def\LEFTp#1#2#3{\mathcal{L}_{#1,#2}^{\cut_{#3}}(\pip')}

For a $\SKD$-pipe dream $\pip$, $k\in[1,m]$, and $(i,j)\in[l_k,r_k]\times [d_k,u_k]$, let $\DOWN ijk\in\{1,0\}$ (resp., $\LEFT ijk\in\{1,0\}$) be equal to $1$ if the path in $\pip$ entering $(i,j)$ from below (resp., from the left) enters the rectangle $[l_k,r_k]\times[d_k,u_k]$ from the left, and to $0$ if it enters the rectangle $[l_k,r_k]\times[d_k,u_k]$ from below. Thus $\cut_k$ involves $(i,j)$ in $\pip$ iff $\DOWN ijk\neq\LEFT ijk$. 

Suppose that $\pip$ is $(\CUT,\cut)$-equivalent to $\pip'$. Our goal is to show that in this case, for all $k\in[1,m]$ and all $(i,j)\in U\cap([l_k,r_k]\times [d_k,u_k])$, we have $\DOWN ijk=\DOWNp ijk$ and $\LEFT ijk=\LEFTp ijk$. Otherwise, choose $(i,j)\in U\cap([l_k,r_k]\times [d_k,u_k])$ with the minimal possible value of $i+j$ such that either $\DOWN ijk\neq\DOWNp ijk$ or $\LEFT ijk\neq\LEFTp ijk$. Suppose first that $\DOWN ijk\neq\DOWNp ijk$. If $(i,j-1)\in  U\cap([l_k,r_k]\times [d_k,u_k])$ then the restrictions of $\pip$ and $\pip'$ to $(i,j-1)$ must be different, so $(i,j-1)\notin \rel(\pip,\CUT)$. Therefore $\cut_k$ does not involve $(i,j-1)$ in $\pip$, so $\DOWN i{j-1}k=\LEFT i{j-1}k=\DOWN ijk$. By the minimality of $i+j$, we must have $\DOWNp i{j-1}k=\DOWN i{j-1}k$ and $\LEFTp i{j-1}k=\LEFT i{j-1}k$, which implies $\DOWN ijk=\DOWNp ijk$, a contradiction. Thus we must have $(i,j-1)\notin  U\cap([l_k,r_k]\times [d_k,u_k])$. If $(i,j-1)\notin [l_k,r_k]\times [d_k,u_k]$ then by definition we have $\DOWN ijk=\DOWNp ijk=0$. The only other option is that $(i,j-1)\in ([l_k,r_k]\times [d_k,u_k])\setminus U$, which implies that $(i,j-1)\in A$ and $k\in K$. We know that $\V^A_\pip=\V^A_{\pip'}$ and that the restrictions of $\pip$ and $\pip'$ to $D$ coincide, so $\DOWN ijk=\DOWNp ijk$. Suppose now that $\LEFT ijk\neq\LEFTp ijk$. An argument similar to the one above shows that we cannot have $(i-1,j)\in U\cap([l_k,r_k]\times [d_k,u_k])$. But then $(i-1,j)\notin [l_k,r_k]\times [d_k,u_k]$, so $\LEFT ijk=\LEFTp ijk=1$. 

We have shown that $\DOWN ijk=\DOWNp ijk$ and $\LEFT ijk=\LEFTp ijk$ for all $(i,j)\in U\cap([l_k,r_k]\times [d_k,u_k])$. This implies both of our desired statements: that $\rel(\pip,\CUT)\cap U=\rel(\pip',\CUT)\cap U$ and that $\wtij{i,j}{\pip}=\wtij{i,j}{\pip'}$ for all $(i,j)\in \rel(\pip,\CUT)\cap  U$.

\def\CCL#1{[#1]_{\CUT,\cut}}
\def\CCLO{\CCL{\pipo}}
\def\pipo{\pip_0}

Fix some $\SKD$-pipe dream $\pipo$ and denote by $\CCLO$ the $(\CUT,\cut)$-equivalence class of $\pipo$. Denote by $\cper_D\in S_{[1,\ci]}$ the color permutation induced by the restriction of $\pipo$ to $D$ on the top right boundary of $D$. 
 Applying \cref{thm:flip_gen} (cf. \cref{rmk:bottom_colors}) to the element $T_{\cper_D}\cdot  Y_A$, we find that for any set $\H$ of pairs, we have 
\begin{equation*}%
  \sum_{\pip\in\CCLO:\ \H^{\ci,\cj}_{\pi_\pip}=\H } \Prob(\pip)=  \sum_{\pip\in\CCLO:\ \H^{\ci,\cj}_{\pi_\pip}=\flip_{\ci,\cj}(\H) } \Prob(\pip)\midd_{(\bx,\by)\mapsto (\bx',\by')}.
\end{equation*}
Summing over all possible classes $\CCLO$ and applying \cref{lemma:Ht_H_V_flip}, the result follows.
\end{proof}

\begin{proof}[Proof of \cref{lemma:double_H_flip}]
We proceed similarly to the proof of \cref{lemma:local_H_flip}, except that now we want to apply two \Hd-flips. Let $L,D,A, A',U,R\subset\SKDZ$ be the intersections of $\SKDZ$ with $[1,l-1]\times [d',u']$, $[l,\Minf]\times[1,d-1]$, $[l,r]\times[d,u]$, $[l,r]\times [d',u']$, $[1,r]\times[u+1,\Ninf]$, and $[r+1,\Minf]\times [d',u']$ respectively, see \cref{fig:rel}(right). Note that the shaded areas $[1,l-1]\times[u'+1,u]$ and $[r+1,\Minf]\times[d,d'-1]$ contain no cells in $\suppcut$, so $\suppcut\subset L\cup D\cup A\cup U\cup R$. We may assume that $A'\subsetneq A$ are rectangles: $A'=[l,r]\times [d',u']$ and $A=[l,r]\times[d,u]$.

\def\pipo{\pip_0}

\def\CCLPQ#1{[#1]^{\Ppath,\Qpath}_{\CUT,\cut}}
\def\CCLPQp#1{[#1]^{\Ppath',\Qpath'}_{\CUT,\cut}}
\def\CCLOPQ{\CCLPQ{\pipo}}
\def\CCLOPQp{\CCLPQp{\pipo}}

\def\HTcond#1{\HTop(#1;\bx,\by)|_{\CCLOPQ}}
\def\HTcondp#1{\HTop(#1;\bx',\by')|_{\CCLOPQp}}
Introduce $\SKD$-cuts $\cut:=(l,d,u,r)$ and $\cut':=(l,d',u',r)$ with color cutoff levels $\ci,\cj$ and $\ci',\cj'$, respectively. As above, fix a $\SKD$-pipe dream $\pipo$. The domain $\SKD$ changes into $\SKDp$, however, the subdomains $D,A,U$ stay invariant. We denote by $\CCLOPQ$ the set of $\SKD$-pipe dreams that are $(\CUT,\cut)$-equivalent to $\pipo$. We also denote by $\CCLOPQp$ the set of all $\SKDp$-pipe dreams $\pip'$ such that $\V^A_{\pipo}=\V^A_{\pip'}$ and the restrictions of $\pipo$ and $\pip'$ to $D\cup (\rel(\pipo,\CUT)\cap U)$ coincide. Recall that our probability space consists of $2^{\#\SKDZ}$ $\SKD$-pipe dreams. The class $\CCLOPQ$ is considered an \emph{event}, and for $i\in[1,m]$, we let $\HTcond{\cut_i}$ denote the random variable $\HT(\cut_i)$ conditioned on this event. Let $\HTcond{\CUT}:=\left(\HTcond{\cut_i}\right)_{i\in [1,m]}$.

Let $\cper_D\in S_{[1,\ci]}$ be the color permutation induced by the restriction of $\pipo$ to $D$ on the top right boundary of $D$, and denote $Y:=T_{\cper_D}\cdot Y_L\cdot Y_A\cdot Y_R$. Consider subsets $L',R'\subset\SKDpZ$ obtained respectively from $L$ and $R$ via the map sending $(i,j)\mapsto (i,j+u+d-u'-d')$. We may assume that $L'$ and $R'$ are the intersections of $\SKDpZ$ with $[1,l-1]\times[d'',u'']$ and $[r+1,\Minf]\times[d'',u'']$, where $u'':=u+d-d'$, $d'':=u+d-u'$. Thus $\suppcutp\subset L'\cup D\cup A\cup U\cup R'$. Let $Y':=\left(T_{\cper_D}\cdot Y_{L'}\cdot Y_A\cdot Y_{R'}\right)\mid_{(\bx,\by)\mapsto(\bx',\by')}$. We claim that $Y'$ is obtained from $Y$ via two applications of \cref{thm:flip_gen}, first at $(\ci',\cj')$ and then at $(\ci,\cj)$. To see that, let us write $A=A_D\sqcup A'\sqcup A_U$, where $A_D$ and $A_U$ are the intersections of $A$ with $[l,r]\times [d,d'-1]$ and $[l,r]\times [u'+1,u]$, respectively. Thus $Y_A=Y_{A_D}\cdot Y_{A'}\cdot Y_{A_U}$ and $T_{\cper_D},Y_{A_D}$ commute with $Y_L$ while $Y_{A_U}$ commutes with $Y_R$. Applying \cref{thm:flip_gen} at $(\ci',\cj')$ sends
\begin{equation}\label{eq:flipgen_1}
 Y_L\cdot  (T_{\cper_D} Y_{A_D})\cdot Y_{A'}\cdot Y_{A_U}\cdot Y_R \mapsto   (Y_R)^*_{\ci',\cj'}\cdot (T_{\cper_D} Y_{A_D})\cdot  \overline{Y_{A'}}\cdot Y_{A_U}\cdot (Y_L)^*_{\ci',\cj'}.
\end{equation}
(In~\eqref{eq:flipgen_1} and~\eqref{eq:flipgen_2}, the five terms on each side, separated by the $\cdot$ symbol, appear in the same order as the corresponding five terms on each side of~\eqref{eq:flip_gen}.) We may now apply \cref{thm:flip_gen} at $(\ci,\cj)$, sending
\begin{equation}\label{eq:flipgen_2}
(Y_R)^*_{\ci',\cj'}\cdot  T_{\cper_D}\cdot  \left( Y_{A_D} \overline{Y_{A'}} Y_{A_U}\right)\cdot 1\cdot  (Y_L)^*_{\ci',\cj'}\mapsto   ((Y_L)^\ast_{\ci',\cj'})^\ast_{\ci,\cj}\cdot T_{\cper_D}\cdot  \left( \overline{Y_{A_D} \overline{Y_{A'}} Y_{A_U}}\right)\cdot 1\cdot  ((Y_R)^\ast_{\ci',\cj'})^\ast_{\ci,\cj}.
\end{equation}
The terms $((Y_L)^\ast_{\ci',\cj'})^\ast_{\ci,\cj}$, $\overline{Y_{A_D} \overline{Y_{A'}} Y_{A_U}}$, and $((Y_R)^\ast_{\ci',\cj'})^\ast_{\ci,\cj}$ are obtained respectively from $Y_{L'}$, $Y_{A}$, and $Y_{R'}$ by substituting $(\bx,\by)\mapsto(\bx',\by')$. 
 Each of the cuts $\cut'$ and $\cut$ (at which we applied \cref{thm:flip_gen}) crosses $\cut_k$ all $k\in I\cup J\cup K$.  By \cref{lemma:Ht_H_V_flip}, we get 
\begin{equation}\label{eq:HTop_conditional_dist2}
\HTcond{\CUT}\eqd \HTcondp{\CUT'}.
\end{equation}
Summing over all possible classes $\CCLOPQ$, the result follows.
\end{proof}

\section{Proof of \texorpdfstring{\cref{thm:main}}{the main result}}\label{sec:main_proof}
Recall from \cref{rmk:depend_on_P_Q} that our goal is to show that each $\CUT$-admissible transformation is strongly $\CUT$-admissible. We showed above that several $\CUT$-admissible transformations are strongly $\CUT$-admissible. The purpose of this section is to show that, after passing to \emph{connected components} defined below, any $\CUT$-admissible transformation can be represented as a composition of transformations introduced in \cref{sec:examples}.

Throughout, we fix the following data:
\begin{itemize}
\item two \skew domains $\SKD$ and $\SKDp$;
\item a tuple $\CUT:=(\cut_1,\cut_2,\dots,\cut_m)$ of $\SKD$-cuts with $\cut_i:=(l_i,d_i,u_i,r_i)$ for $i\in[1,m]$;
\item a tuple $\CUT':=(\cut'_1,\cut'_2,\dots,\cut'_m)$ of $\SKDp$-cuts with $\cut'_i:=(l'_i,d'_i,u'_i,r'_i)$ for $i\in[1,m]$;
\item a $\CUT$-admissible transformation $\PHI:=(\pih,\piv)$ satisfying $\CUT'=\PHI(\CUT)$ (cf. \cref{rmk:depend_on_P_Q}). 
\end{itemize}
We assume that $\cut_i\neq\cut_j$ for $i\neq j$.

\subsection{Connected components}
Clearly, if $\cut_i,\cut_j$ satisfy $[l_i,r_i]\cap[l_j,r_j]=[d_i,u_i]\cap [d_j,u_j]=\emptyset$ then the random variables $\HT(\cut_i)$ and $\HT(\cut_j)$ are independent. One can check that this remains true when only one of the two intersections is empty. A generalization of this to joint distributions is given in \cref{prop:conn_compts} below.
\begin{definition}
We say that two cuts $\cut_i$ and $\cut_j$ \emph{have overlapping rectangles} if $[l_i,r_i]\cap[l_j,r_j]\neq\emptyset$ and $[d_i,u_i]\cap [d_j,u_j]\neq\emptyset$. We denote by $\Gov(\CUT)$ the \emph{overlap graph} of $\CUT$: the vertex set of $\Gov(\CUT)$ is $[1,m]$, and $i$ and $j$ are connected by an edge if and only if $\cut_i$ and $\cut_j$ have overlapping rectangles.
\end{definition}

\begin{remark}
Since $\PHI$ is $\CUT$-admissible, the graphs $\Gov(\CUT)$ and $\Gov(\CUT')$ coincide.
\end{remark}

\begin{proposition}\label{prop:conn_compts}
  Suppose that the set $[1,m]$ is partitioned into two nonempty subsets $[1,m]=I\sqcup J$ such that $\Gov(\CUT)$ contains no edges connecting a vertex in $I$ to a vertex in $J$. Then the height vectors $\left(\HT(\cut_i)\right)_{i\in I}$ and $\left(\HT(\cut_j)\right)_{j\in J}$ are independent (as random variables).
\end{proposition}
\noindent See \cref{fig:GOV}(left) for an example.
\begin{figure}
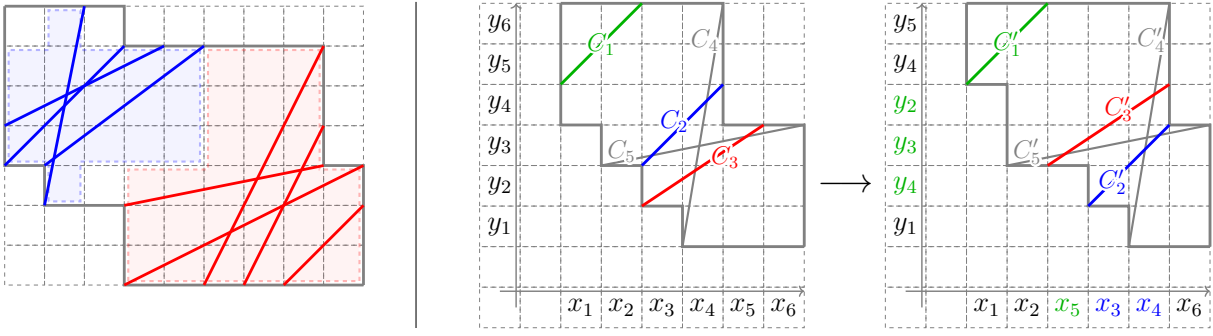

\begin{tabular}{cc|cc}
  \FIGgov
&\hspace{0.2in}&\hspace{0.2in}
\FIGgovdva
\end{tabular}
  \caption{\label{fig:GOV}Left: a tuple $\CUT$ of cuts with disconnected overlap graph $\Gov(\CUT)$. The sets $\supp(\CUT_I)$ and $\supp(\CUT_J)$ are shaded. Right: a $\CUT$-admissible transformation that gives a counterexample to \cref{lemma:rigid_trans} when the assumption that $\Gov(\CUT)$ is connected is not satisfied.}
\end{figure}
\begin{proof}
Denote $\CUT_I:=(\cut_i)_{i\in I}$ and $\CUT_J:=(\cut_j)_{j\in J}$, and let $\supp(\CUT_I),\supp(\CUT_J)\subset\SKDZ$ be the corresponding supports defined in~\eqref{eq:suppcut}. Consider a pipe dream $\pip$ and a cell $(a,b)\in\rel(\pip,\CUT)$. By~\eqref{eq:suppcut}, we must have either $(a,b)\in \rel(\pip,\CUT_I)\subset\supp(\CUT_I)$ or $(a,b)\in\rel(\pip,\CUT_J)\subset\supp(\CUT_J)$. Suppose that, say, $(a,b)\in \rel(\pip,\CUT_I)$ and let $i\in I$ be such that $\cut_i$ involves $(a,b)$ in $\pip$. Let $p_1,p_2$ be the two paths of colors $c_1<c_2$ entering $(a,b)$ in $\pip$. Then $p_1$ must enter the rectangle $[l_i,r_i]\times[d_i,u_i]$ from the bottom while $p_2$ must enter the same rectangle from the left. Thus changing the values of $\pip$ inside the cells of $\supp(\CUT_J)$ preserves $\rel(\pip,\CUT_I)$, $\wtij{a,b}{\pip}$ for $(a,b)\in\rel(\pip,\CUT_I)$, and $\HTpi_\pip(\CUT_I)$. Similarly, changing the values of $\pip$ inside the cells of $\supp(\CUT_I)$ preserves $\rel(\pip,\CUT_J)$, $\wtij{a,b}{\pip}$ for $(a,b)\in\rel(\pip,\CUT_J)$, and $\HTpi_\pip(\CUT_J)$. We are done by~\eqref{eq:prob_cl_pip}.
\end{proof}

In view of \cref{prop:conn_compts}, \textbf{we assume from now on that the graphs $\Gov(\CUT)=\Gov(\CUT')$ are connected}.

\subsection{Non-crossing cut poset}
\begin{definition}
We write $\cut_i\leq \cut_j$ if $\cut_i$ and $\cut_j$ do not cross (cf. \cref{dfn:cut_cross}) and in addition $\cut_j$ is \emph{weakly up-left} from $\cut_i$, that is, $l_i\geq l_j$, $r_i\geq r_j$, $d_i\leq d_j$, and $u_i\leq u_j$. We write $\cut_i<\cut_j$ if $\cut_i\leq\cut_j$ and $\cut_i\neq\cut_j$. We denote by $\PC=(\CUT,\leq)$ the associated poset (partially ordered set).
\end{definition}

\def\lkd{$\PHI$-linked\xspace}

As follows from our examples in \cref{sec:examples}, the posets $\PC$ and $\PCP$ in general need not coincide. For example, the color-position symmetry (\cref{lemma:CPS}) reverses all relations in $\PC$. However, for certain \emph{pairs of relations} in $\PC$, we can show that their ``relative orientation'' is preserved.
\begin{definition}
We say that the relations $\cut_{i_1}<\cut_{j_1}$ and $\cut_{i_2}<\cut_{j_2}$ in $\PC$ are \emph{\lkd} if either $\cut'_{i_1}<\cut'_{j_1}$ and $\cut'_{i_2}<\cut'_{j_2}$ or $\cut'_{i_1}>\cut'_{j_1}$ and $\cut'_{i_2}>\cut'_{j_2}$.
\end{definition}
\noindent Thus when $\PHI$ is the color-position symmetry of \cref{lemma:CPS}, any two relations  $\cut_{i_1}<\cut_{j_1}$ and $\cut_{i_2}<\cut_{j_2}$ are \lkd. 
The following simple observation gives rise to a family of pairs of relations that are \lkd regardless of the choice of $\PHI$.
\begin{lemma}\label{lemma:rigid_overlap}
Suppose that $i,j,k\in[1,m]$ are such that $\cut_i<\cut_j$ and $\cut_i<\cut_k$. Suppose in addition that $\cut_j$  and $\cut_k$ have overlapping rectangles. Then the relations $\cut_i<\cut_j$ and $\cut_i<\cut_k$ are \lkd.
\end{lemma}
\begin{remark}\label{rmk:rigid_cross}
Recall that if $\cut_j$ crosses $\cut_k$ then they have overlapping rectangles, so the above lemma applies in this important special case.
\end{remark}
\begin{proof}
Suppose first that $\cut_j$ crosses $\cut_k$. Then $\cut'_j$ and $\cut'_k$ cross but none of them crosses $\cut'_i$, which immediately implies the result. Suppose now that $\cut_j$ does not cross $\cut_k$. Since they have overlapping rectangles,  the $\SKD$-cuts $\cut_{j,k}:=(l_j,d_j,u_k,r_k)$ and $\cut_{k,j}:=(l_k,d_k,u_j,r_j)$ cross and satisfy $\cut_i<\cut_{j,k}$ and $\cut_i<\cut_{k,j}$. The overlapping rectangles condition also implies that both the intersection and the union of $[l_j,r_j]$ and $[l_k,r_k]$ is a single interval that is sent by $\pih$ to some other interval. This shows that $\PHI$ sends $\cut_{j,k},\cut_{k,j}$ to $\SKDp$-cuts $\cut'_{j,k}:=(l'_j,d'_j,u'_k,r'_k)$ and $\cut'_{k,j}:=(l'_k,d'_k,u'_j,r'_j)$ that cross each other but do not cross $\cut'_i$. Since each of them crosses both $\cut'_j$ and $\cut'_k$, the result follows.
\end{proof}

Our next goal is to describe more \lkd pairs of relations. We will repeatedly make use of the \emph{transitivity} property of \lkd relations: if $\cut_{i_1}<\cut_{j_1}$ and $\cut_{i_2}<\cut_{j_2}$ are \lkd and $\cut_{i_2}<\cut_{j_2}$ and $\cut_{i_3}<\cut_{j_3}$ are \lkd then $\cut_{i_1}<\cut_{j_1}$ and $\cut_{i_3}<\cut_{j_3}$ are \lkd.

\subsection{Rigid relations and indecomposable components}

As it is apparent from local and double \Hd- and \Vd-flips (\cref{lemma:local_H_flip,lemma:double_H_flip}), one needs to take special care of relations between $\SKD$-cuts whose horizontal or vertical projections coincide.
\begin{definition}
We say that a relation $\cut_i<\cut_j$ in $\PC$ is \emph{rigid} if $[l_i,r_i]\neq [l_j,r_j]$ \emph{and} $[d_i,u_i]\neq[d_j,u_j]$. In this case, we write $\cut_i\ler \cut_j$.
\end{definition}
\noindent One immediate property of rigid relations is that having $\cut_i\leq\cut_j\ler\cut_k$ or $\cut_i\ler\cut_j\leq\cut_k$ implies $\cut_i\ler \cut_k$. The following lemma is a special case of \cref{prop:rigid_indec} below.

\begin{lemma}\label{lemma:rigid_trans}
Suppose that $i,j,k\in [1,m]$ satisfy $\cut_i\leq\cut_j\ler\cut_k$. Then the rigid relations $\cut_i\ler\cut_k$ and $\cut_j\ler\cut_k$ are \lkd.
\end{lemma}
\begin{proof}
We will use the following observation: 
\begin{equation}\label{eq:observation_supp_cut}
\parbox{0.85\textwidth}{
if $\cut'_s,\cut'_t$ do not cross and $a\in\lrp s\setminus\lrp t$, then $\cut'_s<\cut'_t$ (resp., $\cut'_s>\cut'_t$) \\
if and only if $a$ is larger (resp., smaller) than all elements of $\lrp t$.}
\end{equation}
\noindent  Our assumptions on $i,j,k$ imply that $\cut'_k$ does not cross  $\cut'_i$ and $\cut'_j$. 

Let us assume first that $\lr i\cap\lr j\neq \emptyset$, which in view of $\cut_i\leq \cut_j$ is equivalent to $l_i\leq r_j$. We consider three cases: $\lr i\cap\lr j\not\subset\lr k$, $\lr k\not\subset\lr i\cup\lr j$ and $\lr i\cap \lr j\subset\lr k\subset\lr i\cup\lr j$. In the first case, let  $a\in \lr i\cap \lr j$ be such that $a\notin \lr k$. Then $\pih(a)$ is either strictly to the left or strictly to the right of $\pih(\lr k)=\lrp k$, therefore by~\eqref{eq:observation_supp_cut}, we have either $\cut'_i,\cut'_j>\cut'_k$ or $\cut'_i,\cut'_j<\cut'_k$, respectively. Thus the relations $\cut_i\ler\cut_k$ and $\cut_j\ler\cut_k$ are \lkd. Similarly, in the second case, let $b\notin \lr i\cup \lr j$ be such that $b\in \lr k$. Because the intersection $\lr i\cap\lr j$ is nonempty, the image of the union $\lr i\cup \lr j$ under $\pih$ is a single interval $[l',r']=\lrp i\cup \lrp j$. Then $\pih(b)\notin [l',r']$ is either to the left or to the right of this interval, so we again find that the relations $\cut_i\ler\cut_k$ and $\cut_j\ler\cut_k$ are \lkd. In the third case, we have
\begin{equation*}
  \left( \lr i\cap \lr j\right)\subset\lr k\subset \left(\lr i\cup \lr j\right) \quad\Longleftrightarrow\quad [l_i,r_j]\subset\lr k\subset[l_j,r_i].
\end{equation*}
 This yields $l_j\leq l_k\leq l_i\leq r_j\leq r_k\leq r_i$. On the other hand, $\cut_j<\cut_k$ implies $l_k\leq l_j$ and $r_k\leq r_j$, so we get $l_j=l_k$ and $r_j=r_k$. This leads to a contradiction since $\cut_j\ler \cut_k$ requires $\lr j\neq \lr k$. We have completed the proof in the case $\lr i\cap\lr j\neq \emptyset$.

A similar argument finishes the proof in the case $\du i\cap \du j\neq\emptyset$. Suppose now that $\lr i\cap\lr j=\du i\cap \du j=\emptyset$. In this case, we finish the proof by induction on the graph distance $\dist(i,j)$ between $i$ and $j$ in $\Gov(\CUT)$, which we have assumed to be connected. The conditions $\lr i\cap\lr j=\du i\cap \du j=\emptyset$ imply that $\dist(i,j)\geq 2$. Consider the shortest path from $i$ to $j$ in $\Gov(\CUT)$, and let $i'\neq i,j$ be the first vertex after $i$ on this path. By definition, $\cut_i$ and $\cut_{i'}$ have overlapping rectangles. On the other hand, it is easy to see that $\cut_{i'}$ does not cross $\cut_j$ because of the condition $\lr i\cap\lr j=\du i\cap \du j=\emptyset$. Thus $\cut_{i'}<\cut_j\ler \cut_k$ and $\dist(i',j)<\dist(i,j)$, so we apply the induction hypothesis to conclude that the relations $\cut_{i'}\ler\cut_k$ and $\cut_j\ler\cut_k$ are \lkd. Since $\cut_i$ and $\cut_{i'}$ have overlapping rectangles, \cref{lemma:rigid_overlap} shows that the relations $\cut_{i'}\ler\cut_k$ and $\cut_i\ler\cut_k$ are \lkd, so by transitivity, $\cut_i\ler\cut_k$ and $\cut_j\ler\cut_k$ are \lkd.
\end{proof}
\begin{example}
\Cref{lemma:rigid_trans} does not hold without our running assumption that the graph $\Gov(\CUT)$ is connected, see \cref{fig:GOV}(right).
\end{example}

The following notion is motivated by global \Hd- and \Vd-flips described in \cref{lemma:global_H_flip}.
\begin{definition}\label{dfn:indec}
We say that $\CUT$ is \emph{indecomposable} if there does not exist a partition $[1,m]=I\sqcup J$ into nonempty subsets such that $\cut_i$ crosses $\cut_j$ for all $i\in I$ and $j\in J$.
\end{definition}

\begin{proposition}\label{prop:rigid_indec}
Suppose that $\CUT$ is indecomposable. Then any two rigid relations $\cut_{i_1}\ler\cut_{j_1}$ and $\cut_{i_2}\ler\cut_{j_2}$ are \lkd.
\end{proposition}
\begin{proof}
Assume that $\cut_{i_1}\leq \cut_{i_2}$. By \cref{lemma:rigid_trans}, the relations $\cut_{i_1}\ler\cut_{j_2}$ and $\cut_{i_2}\ler\cut_{j_2}$ are \lkd. If $\cut_{j_1}$ and $\cut_{j_2}$ cross then by \cref{rmk:rigid_cross}, the relations $\cut_{i_1}<\cut_{j_2}$ and $\cut_{i_1}<\cut_{j_1}$ are \lkd, and by transitivity this implies the result. If $\cut_{j_1}$ and $\cut_{j_2}$ do not cross then since the relations $\cut_{i_1}\ler\cut_{j_2}$ and $\cut_{i_1}\ler\cut_{j_1}$ are both rigid, \cref{lemma:rigid_trans} shows that they are \lkd, so we are done by transitivity. We have shown the result in the case $\cut_{i_1}\leq \cut_{i_2}$. 

The cases $\cut_{i_2}\leq \cut_{i_1}$, $\cut_{j_1}\leq \cut_{j_2}$, and $\cut_{j_2}\leq \cut_{j_1}$ are completely similar. Thus we may assume that $\cut_{i_1}$ crosses $\cut_{i_2}$ and $\cut_{j_1}$ crosses $\cut_{j_2}$.

Since $\cut_{i_1}$ crosses $\cut_{i_2}$ and $\cut_{i_1}<\cut_{j_1}$, we see that either $\cut_{i_2}<\cut_{j_1}$ or $\cut_{i_2}$ crosses $\cut_{j_1}$. If $\cut_{i_2}<\cut_{j_1}$ then by \cref{rmk:rigid_cross}, the relations $\cut_{i_2}<\cut_{j_1}$ and $\cut_{i_1}\ler \cut_{j_1}$ are \lkd, and by the same remark, the relations $\cut_{i_2}<\cut_{j_1}$ and $\cut_{i_2}\ler \cut_{j_2}$ are \lkd, so the result follows. The case $\cut_{i_1}<\cut_{j_2}$ is handled similarly, thus we may assume that each of $\cut_{i_1},\cut_{j_1}$ crosses each of $\cut_{i_2},\cut_{j_2}$.

Since $\CUT$ is indecomposable, there exists a sequence $k_0,k_1,\dots,k_t,k_{t+1}$ satisfying $\{k_0,k_1\}=\{i_1,j_1\}$, $\{k_t,k_{t+1}\}=\{i_2,j_2\}$, and such that for all $s\in[0,t]$ we have either $\cut_{k_s}<\cut_{k_{s+1}}$ or $\cut_{k_s}>\cut_{k_{s+1}}$. Out of all such sequences, choose the one with minimal possible $t$. Because of the above assumptions, we must have $t\geq 3$. 

Since $t$ is minimal possible, we find that for any $s\in[1,t-2]$, $\cut_{k_s}$ must cross $\cut_{k_{s+2}}$. 
 Thus by \cref{rmk:rigid_cross}, the relations between $\cut_{k_s},\cut_{k_{s+1}}$ and $\cut_{k_{s+1}},\cut_{k_{s+2}}$ are \lkd for all $s\in[1,t-2]$. By transitivity, the relations between $\cut_{k_1},\cut_{k_2}$ and $\cut_{k_{t-1}},\cut_{k_t}$ are \lkd. 

Our next goal is to show that the relations between $\cut_{k_0},\cut_{k_1}$ and $\cut_{k_1},\cut_{k_2}$ are \lkd. Note that $\cut_{k_3}$ must cross both $\cut_{i_1}$ and $\cut_{j_1}$, otherwise we could remove $k_2$ from the sequence (after possibly swapping $k_0$ and $k_1$) thus decreasing $t$. Because $\cut_{k_2}$ does not cross $\cut_{k_3}$, we cannot have $\cut_{i_1}\leq \cut_{k_2}\leq \cut_{j_1}$. Recall also that $\cut_{k_2}$ does not cross $\cut_{k_1}$. Thus either $\cut_{k_2}$ crosses $\cut_{k_0}$ (in which case the relations between $\cut_{k_0},\cut_{k_1}$ and $\cut_{k_1},\cut_{k_2}$ are \lkd by \cref{rmk:rigid_cross}), or $\cut_{k_2}<\cut_{i_1}\ler\cut_{j_1}$, or $\cut_{i_1}\ler\cut_{j_1}<\cut_{k_2}$. In the latter two cases, after possibly swapping $k_0$ and $k_1$, we see that the relations between $\cut_{k_0},\cut_{k_1}$ and $\cut_{k_1},\cut_{k_2}$ are \lkd by \cref{lemma:rigid_trans}. Similarly, we show that the relations between $\cut_{k_{t-1}},\cut_{k_t}$ and $\cut_{k_t},\cut_{k_{t+1}}$ are \lkd. The result follows by transitivity.
\end{proof}

\subsection{Finishing the proof}
So far our main focus has been on structural properties of $\CUT$-admissible transformations. Next, we describe how each such transformation can be represented as a composition of strongly $\CUT$-admissible transformations constructed in \cref{sec:examples}.

\begin{definition}
Two $\CUT$-admissible transformations are called \emph{flip-equivalent} if they can be obtained from each other by composing with the transformations described in Lemmas~\ref{lemma:CPS}--\ref{lemma:double_H_flip}.
\end{definition}

Recall that we have fixed a particular $\CUT$-admissible transformation $\PHI$.
\begin{lemma}\label{lemma:rigid_preserved}
$\PHI$ is flip-equivalent to a transformation that preserves all rigid relations. More specifically, after possibly applying the color-position symmetry (\cref{lemma:CPS}) and several global \Hd-flips (\cref{lemma:global_H_flip}), we may assume that for all $\cut_i\ler\cut_j$ we have $\cut'_i<\cut'_j$.
\end{lemma}
\begin{proof}
Suppose that $\cut_i\ler\cut_j$ is such that $\cut'_i>\cut'_j$. If $\CUT$ is indecomposable then after applying the color-position symmetry (\cref{lemma:CPS}), we find $\cut'_i<\cut'_j$, and for any other rigid relation $\cut_a\ler\cut_b$, we get $\cut'_a<\cut'_b$ by \cref{prop:rigid_indec}. If $\CUT$ is not indecomposable then we can split it into indecomposable components: $[1,m]=I_1\sqcup I_2\sqcup \dots\sqcup I_s$. We assume that these  indecomposable components are ordered ``by slope'' so that for any $i,j\in[1,s]$, $i<j$, and any $a\in I_i$, $b\in I_j$, we have $\lr a\supset \lr b$ and $\du a\subset \du b$. It is then clear that for each $p\in[1,s]$, one choose a cut $\cut=(l,d,u,r)$ crossing all cuts in $\CUT$ in such a way that $\lr a\supset [l,r]\supset \lr b$ and $\du a\subset [d,u]\subset \du b$ for all $a\in[1,p]$ and $b\in[p+1,s]$. Applying a global \Hd-flip at $\cut$ will flip the first $p$ components while leaving the rest unchanged. Similarly, applying a second global \Hd-flip, we can flip the first $p-1$ components back. Thus each individual component $I_p$ can be flipped using a composition of two \Hd-flips.
\end{proof}

\begin{definition}
We say that $\PHI$ is \emph{orientation-preserving} if for any relation $\cut_i<\cut_j$, we have $\cut'_i<\cut'_j$.
\end{definition}

\begin{proposition}\label{prop:or_pres}
$\PHI$ is flip-equivalent to an orientation-preserving transformation.
\end{proposition}
\begin{proof}
By \cref{lemma:rigid_preserved}, we may assume that for all $\cut_i\ler\cut_j$ we have $\cut'_i<\cut'_j$. Let us say that $(a,b)$ is a \emph{bad pair} if $\cut_a<\cut_b$ and $\cut'_a>\cut'_b$. For each bad pair $(a,b)$, we must have either $\lr a=\lr b$ or $\du a=\du b$, so assume that $l\leq r$ are such that $\lr a=\lr b=[l,r]$ for some bad pair $(a,b)$.

Let $\bplr$ denote the set of all bad pairs $(a,b)$ satisfying $\lr a=\lr b=[l,r]$. Choose $(a,b)\in\bplr$ to be the pair which maximizes the value $u_b-d_a$, and let $d:=d_a$, $u:=u_b$. We claim that $\ZP\times [d,u]$ is a horizontal $\CUT$-strip. Indeed, otherwise there would exist $i\in [1,m]$ such that either $d_i<d\leq u_i<u$ or $d<d_i\leq u<u_i$. Suppose for example that $d_i<d\leq u_i<u$. These conditions imply that $C_i<C_b$. Observe that $d\in [d_i,u_i]\setminus[d_b,u_b]$ and also $d\in[d_a,u_a]\setminus [d_b,u_b]$. Since $\cut'_a>\cut'_b$, \eqref{eq:observation_supp_cut} shows that $\cut'_i>\cut'_b$ so $(i,b)$ is a bad pair, and because $\du i\neq \du b$, we must have $\lr i=\lr b=[l,r]$. This implies $(i,b)\in\bplr$, and since we have $u_b-d_i>u_b-d_a$, we get a contradiction.  The case $d<d_i\leq u<u_i$ is handled similarly, thus we have shown that $\ZP\times [d,u]$ is a horizontal $\CUT$-strip.

Our goal is to apply either a local \Hd-flip or a double \Hd-flip inside $\ZP\times [d,u]$ so that $(a,b)$ would stop being a bad pair. We note that this operation may introduce new bad pairs $(a',b')$ but for each of them, the quantity $u_{b'}-d_{a'}$ will be strictly smaller than $u-d=u_b-d_a$. Indeed, our flip only changes the orientation of the pairs $(a',b')$ such that $[d_{a'},u_{b'}]\subset[d,u]$, and for any pair satisfying $\lr{a'}=\lr{b'}=[l,r]$, $d_{a'}=d$, $u_{b'}=u$, and $\cut_{a'}<\cut_{b'}$, we see that $\cut_{a'}$ crosses $\cut_a$, $\cut_{b'}$ crosses $\cut_b$, $\cut_{a'}<\cut_b$, and $\cut_a<\cut_{b'}$. By \cref{rmk:rigid_cross}, we get that $\cut'_{a'}>\cut'_{b'}$, so $(a',b')$ is already a bad pair and will stop being a bad pair after we apply the flip. We have shown that it is enough to apply either \cref{lemma:local_H_flip} or \cref{lemma:double_H_flip} inside $\ZP\times [d,u]$ so that, in the notation of~\eqref{eq:local_H_IJ}--\eqref{eq:double_H_I}, we would have $a,b\in J$.

\def\Jp{J'}
\def\Ip{I'}
 Following~\eqref{eq:local_H_IJ}, let 
\begin{equation*}
  \Jp:=\{j\in[1,m]\mid \du j\subsetneq [d,u]\},\quad K:=\{k\in[1,m]\mid [d,u]\subset\du k\}.
\end{equation*}
Since $d_a<d_b$ and $u_a<u_b$, we find that for all $k\in K$, $\cut_k$ crosses both $\cut_a$ and $\cut_b$ and therefore satisfies $\lr k\subset [l,r]$. Thus if $\lr j=[l,r]$ for all $j\in \Jp$ then we can apply a local \Hd-flip (\cref{lemma:local_H_flip}) and finish the proof. Otherwise, let $\Ip:=\{i\in \Jp\mid \lr i\neq [l,r]\}$. We need to apply a double \Hd-flip for a careful choice of $[d',u']\subsetneq[d,u]$. 

First, we would like to show that for all $i\in \Ip$, $\cut_i$ crosses both $\cut_a$ and $\cut_b$. To see that, recall that $\du i\subsetneq [d,u]$ and $\lr i\neq [l,r]$. Thus either $d\notin \du i$ or $u\notin \du i$ (or both), so assume $d\notin \du i$. Then we cannot have $\cut_i<\cut_a$. If $\cut_i$ crosses both $\cut_a$ and $\cut_b$ then we are done. If $\cut_i$ crosses  $\cut_a$ but not $\cut_b$ then $\du i\neq \du b$ so we must have $\cut_i\ler \cut_b$. This contradicts \cref{rmk:rigid_cross} since $\cut'_i<\cut'_b$ but $\cut'_a>\cut'_b$. Thus we may assume that $\cut_i$ does not cross $\cut_a$, and since $d\notin \du i$, we get $\cut_a\ler \cut_i$. We cannot have $\cut_a<\cut_i<\cut_b$ because then we would have $\lr i=[l,r]$. Thus $\cut_i\not<\cut_b$, so there exists $c\in(\du b\cap \du i)\setminus \du a$, in which case we arrive at a contradiction via~\eqref{eq:observation_supp_cut}. The case $u\notin \du i$ is handled similarly. We have shown that for all $i\in\Ip$, $\cut_i$ crosses both $\cut_a$ and $\cut_b$.

We now proceed as in the proof of \cref{prop:rigid_indec}. An \emph{alternating \ab i-path} is a sequence $k_0,k_1,\dots,k_t$ of elements of $\Jp$ such that $\{k_0,k_1\}=\{a,b\}$, and for all $s\in[0,t-1]$, we have either $\cut_{k_s}<\cut_{k_{s+1}}$ or $\cut_{k_s}>\cut_{k_{s+1}}$. We say that such a path \emph{ends at $k_t$}. We let 
\begin{equation}\label{eq:J_alt_path}
  J:=\{j\in \Jp\mid\text{there exists an alternating \ab j-path ending at $j$}\}.
\end{equation}
\def\jp{j}
\def\ip{i}
Clearly this set contains both $a$ and $b$. We claim that $J\cap \Ip=\emptyset$, in other words, that for all $\jp\in J$, we have $\lr \jp=[l,r]$. Otherwise, consider an alternating \ab \jp-path of smallest possible length $t$ such that $k_t=\jp\in J\cap \Ip$. We showed above that $\cut_\jp$ crosses both $\cut_a$ and $\cut_b$, so $t\geq 3$. Since $t$ is minimal possible, we see that $\cutk_{s}$ crosses $\cutk_{s+2}$ for all $s\in[1,t-2]$, and in addition $\cutk_3$ crosses both $\cutk_0$ and $\cutk_1$. By construction, $\cutk_2$ crosses neither $\cutk_1$ nor $\cutk_3$. If it also does not cross $\cutk_0$ then we must have  $\cut_a<\cutk_2<\cut_b$ (because $\du{k_2}\subset[d,u]$ and $\lr{k_2}=[l,r]$). This contradicts the fact that $\cutk_3$ crosses both $\cutk_0,\cutk_1$ but not $\cutk_2$. Thus $\cutk_{s}$ crosses $\cutk_{s+2}$ for all $s\in[0,t-2]$. In particular, for all $s\in[0,t-2]$, the relations between $\cutk_s,\cutk_{s+1}$ and $\cutk_{s+1},\cutk_{s+2}$ are \lkd. Since $t$ is minimal and $k_t\in \Ip$, we find $\lr{k_{t-1}}=[l,r]\neq\lr{k_t}$. We showed above that $\cut_{k_t}$ must cross both $\cut_a$ and $\cut_b$. Since $\du{k_t}\subsetneq [d,u]$, we get $[l,r]\subsetneq \lr{k_t}$. By construction, $\cutk_{t-1}$ does not cross $\cutk_t$. Together with $[l,r]\subsetneq \lr{k_t}$, this implies $\du{k_{t-1}}\neq\du{k_t}$, so the relation between $\cutk_{t-1}$ and $\cutk_t$ must be rigid. We arrive at a contradiction since the relation between $\cutk_0,\cutk_1$ (i.e., between $\cut_a$ and $\cut_b$) was reversed by $\PHI$ while the \lkd relation between $\cutk_{t-1},\cutk_t$ was preserved by $\PHI$ due to being rigid. We have shown that $J\cap \Ip=\emptyset$.

Consider a graph $G$ with vertex set $\Jp$ and edges $\{i,j\}$ for all $i,j\in\Jp$ such that $\cut_i$ and $\cut_j$ do not cross. Then it is easy to check from~\eqref{eq:J_alt_path} that $J$ is actually the connected component of $G$ that contains the edge $\{a,b\}$. Thus letting $I:=\{\Jp\setminus J\}$, we find that for all $i\in I$ and $j\in J$, $\cut_i$ crosses $\cut_j$. Set $d':=\min_{i\in I} d_i$ and $u':=\max_{i\in I} u_i$. We see that all conditions of \cref{lemma:double_H_flip} are satisfied, and that moreover $a,b\in J$. Applying the double \Hd-flip, we proceed by induction on $u_b-d_a$ as described above.
\end{proof}

\begin{proof}[Proof of \cref{thm:main}]
By \cref{prop:or_pres}, we may assume that $\PHI$ is orientation-preserving. We will show that $\PHI$ is a composition of \Hd- and \Vd-shifts as described in \cref{rmk:H_shift_BGW}. Specifically, we would like to apply \cref{lemma:double_H_flip} when $d'=d+1$ and $u'=u$, in which case~\eqref{eq:double_H_J} gives $J=\emptyset$ and the transformation $(\id,\rv_{d,u}\circ\rv_{d+1,u})$ swaps $y_d$ and $(y_d+1,\dots,y_u)$, preserving the order of the latter. 

\def\r{a}
\def\rr{\r'}

Let $\r\in[1,\Ninf-1]$ be such that $\piv(\r)>\piv(\r+1)$. If there exist indices $i,j\in[1,m]$ such that $u_i=\r$ and $d_j=\r+1$ then $\cut_i<\cut_j$ and $\cut'_i>\cut'_j$ by~\eqref{eq:observation_supp_cut}, which contradicts the assumption that $\PHI$ is orientation-preserving. Thus let us assume that there does not exist $i\in[1,m]$ such that $u_i=\r$. We set $d:=\r$.

We construct $u\geq d+1$ by the following algorithm. First, set $u:=d+1$. If we have found $i\in[1,m]$ such that $d<d_i\leq u<u_i$, we increase $u$ by setting $u:=u_i$, and repeat this procedure until there is no $i$ satisfying this assumption.  It is straightforward to see that for each $\rr\in[d+1,u]$ we must have $\piv(\rr)<\piv(d)$.

We claim that $\ZP\times[d,u]$ and $\ZP\times[d+1,u]$ are both horizontal $\CUT$-strips. By construction, there is no $i\in[1,m]$ satisfying $d<d_i\leq u<u_i$. Suppose that we have found $j\in[1,m]$ satisfying $d_j<d\leq u_j<u$. We cannot have $u_j=d$, thus $d_j<d+1\leq u_j<u$. Because of the way we constructed the interval $[d+1,u]$, there must exist some $i\in I$ such that $d_j<d_i\leq u_j<u_i$. This implies $\cut_i>\cut_j$, and since $d\in \du j\setminus \du i$, we get a contradiction by~\eqref{eq:observation_supp_cut}. We have shown that $\ZP\times[d,u]$ and $\ZP\times[d+1,u]$ are horizontal $\CUT$-strips. 

Let $I:=\{i\in[1,m]\mid \du i\subset[d+1,u]\}$ and  $K:=\{k\in[1,m]\mid d\in \du k\}$. By our assumption, we have $d+1\in \du k$ for all $k\in K$.  Because $\ZP\times[d,u]$ and $\ZP\times[d+1,u]$ are horizontal $\CUT$-strips, we get that $[d,u]\subset \du k$ for all $k\in K$. We claim that $\cut_i$ crosses $\cut_k$ for all $i\in I$ and $k\in K$. Indeed, if otherwise $\cut_i$ does not cross $\cut_k$ for some $i\in I$ and $k\in K$, we must have $\cut_i>\cut_k$ by~\eqref{eq:observation_supp_cut}. But we have already shown that $\piv(\rr)<\piv(d)$ for $\rr\in[d+1,u]$, which by~\eqref{eq:observation_supp_cut} contradicts the fact that $\Phi$ is orientation-preserving. Thus indeed $\cut_i$ crosses $\cut_k$ for all $i\in I$ and $k\in K$, so there exist $l\leq r$ such that $\lr k \subset [l,r]\subset \lr i$ for all $i\in I$ and $k\in K$. We may now apply the corresponding \Hd-shift of \cref{lemma:double_H_flip} with $d'=d+1$, $u'=u$, and $J=\emptyset$. We have already observed that $\piv(\rr)<\piv(d)$ for all $\rr\in[d+1,u]$. Therefore this \Hd-shift reduces the number of inversions of $\piv$.

We have assumed above that $u_i\neq \r$ for all $i\in[1,m]$. If this is not the case then we must have $d_i\neq \r+1$ for all $i\in[1,m]$. This case is treated similarly -- we apply an \Hd-shift with $d'=d$ (for some choice of $d\leq \r$), $u'=\r$, and $u=\r+1$. Applying these \Hd-shifts repeatedly, we arrive at the case $\piv=\id$, and then we apply a sequence of \Vd-shifts until $\pih=\id$.
\end{proof}

\section{Applications}\label{sec:applications-body}
We degenerate the \SCSV model proceeding step-by-step in the order specified in~\cite[Sections~6 and~7]{BGW}. The nature of the limiting procedure in~\cite{BGW} requires us to only consider tuples of $\SKD$-cuts all of whose left endpoints belong to the same vertical line.

\begin{definition}\label{dfn:left_aligned}
We say that $\CUT=(\cut_1,\dots,\cut_m)$ with $\cut_i=(l_i,d_i,u_i,r_i)$ is a \emph{left-aligned tuple of cuts} if $l_i=1$ for all $i\in[1,m]$ and there exists an up-left path $\Qpath$ that contains the top right corners of the cells $(r_i,u_i)$ for all $i\in[1,m]$. Given a bijection $\piv:\ZP\to\ZP$ and another left-aligned tuple of cuts $\CUT'=(\cut'_1,\dots,\cut'_m)$, we write $\CUT'=\piv(\CUT)$ if $l'_i=l_i=1$, $r'_i=r_i$, and $\piv([d_i,u_i])=[d'_i,u'_i]$ for all $i\in[1,m]$.
\end{definition}
\begin{remark}
In the language of \cref{dfn:cut_adm}, $(\id,\piv)$ is a $\CUT$-admissible transformation. It is thus a composition of double \Hd-flips, \Hd-shifts, and local \Hd-flips described in \cref{sec:examples}.
\end{remark}

In all cases discussed below, the convergence in finite-dimensional distributions was shown in~\cite{BGW} in order to prove the shift-invariance property, and we rely on the same convergence result to take the limit of \cref{thm:main}.

\def\brho{\pmb{\rho}}
\def\bro{\brho}
\def\bsig{\pmb{\sigma}}
\subsection{Continuous model}\label{sec:continuous_model}
Recall that the \emph{Beta distribution} $B(a,b)$ with parameters $a,b\in\R_{>0}$ is supported on a line segment $(0,1)\subset\R$ and has density
\begin{equation*}%
  \frac{\Gamma(a)\Gamma(b)}{\Gamma(a+b)} x^{a-1}(1-x)^{b-1},\quad 0<x<1.
\end{equation*}
Fix two families $\bsig=(\sigma_i)_{i\in\ZNN}$ and $\brho=(\rho_j)_{j\in\ZP}$ of real numbers satisfying $0<\rho_j<\sigma_i$ for all $(i,j)\in\ZNN\times\ZP$. The following description\footnote{The description in~\cite{BGW} is stated for the case when all $\rho_j$ are equal to a single value $\rho$, but their limiting results, specifically, \cite[Corollary~6.21]{BGW} are valid for an arbitrary choice of $(\rho_j)_{j\in\ZP}$.} of the \emph{continuous \SCSV model} can be found in~\cite[Section~6.4]{BGW}.
\begin{itemize}[leftmargin=0.2in]
\item For each $(i,j)\in\ZNN\times\ZP$, we sample (independently) a Beta-distributed random variable $\eta_{i,j}\sim B(\sigma_i-\rho_j,\rho_j)$.
\item We consider the vertices in $\ZP\times\ZP$. To each of the four edges adjacent to a vertex of $\ZP\times\ZP$ we will assign a (random) vector $(m_1,m_2,\dots)$ of nonnegative real numbers. We refer to $m_c$ as the \emph{mass} of color $c$ passing through this edge.
\item For $i\in\ZP$, the edge entering the vertex $(i,1)$ from the bottom has all masses equal to $0$.
\item For $j\in\ZP$, the masses $(m_1,m_2,\dots)$ assigned to the edge entering the vertex $(1,j)$ from the left are sampled as follows. We set $m_c=0$ for all $c\neq j$ and let $m_j:=-\ln(\eta_{0,j})$.
\item Suppose that $(i,j)\in\ZP\times\ZP$ has incoming masses $(\alpha_1,\alpha_2,\dots)$ and $(\beta_1,\beta_2,\dots)$ assigned to the bottom and left edges, respectively. The outgoing masses $(\gamma_1,\gamma_2,\dots)$ and $(\delta_1,\delta_2,\dots)$ assigned respectively to the top and right outgoing edges are defined by
\begin{equation*}%
  \exp(-\delta_{\geq c})=\exp(-\alpha_{\geq c})+(1-\exp(-\alpha_{\geq c}))\eta_{i,j},\quad c=1,2,\dots,
\end{equation*}
where $\delta_{\geq c}:=\sum_{c'=c}^{\infty} \delta_{c'}$ and $\alpha_{\geq c}:=\sum_{c'=c}^{\infty} \alpha_{c'}$. We define the remaining masses via the mass conservation law: for $c=1,2,\dots$, we set $\gamma_c:=\alpha_c+\beta_c-\delta_c$.
\end{itemize}

Given a left-aligned tuple $\CUT=(\cut_1,\dots,\cut_m)$ of cuts with $\cut_i=(1,d_i,u_i,r_i)$, we will be interested in the random vector $\left(\HTop(\cut_1;\bsig,\brho),\dots,\HTop(\cut_m;\bsig,\brho)\right)$ of height functions, where each individual height function $\HTop(\cut_i;\bsig,\brho)$ is defined as the total mass of color $\geq d_i$ passing through the right outgoing edges of the vertices $(r_i,d_i),(r_i,d_i+1),\dots,(r_i,u_i)$. For a bijection $\piv:\ZP\to\ZP$, we denote $\piv^{-1}(\brho):=(\rho_{\piv^{-1}(1)},\rho_{\piv^{-1}(2)},\dots)$ (cf. \cref{rmk:inverse}).
\begin{theorem}\label{thm:continuous_model}
In the above setting of the continuous \SCSV model, consider two left-aligned tuples $\CUT=(\cut_1,\dots,\cut_m)$ and $\CUT'=(\cut'_1,\dots,\cut'_m)$ of cuts. Suppose that we have a bijection $\piv:\ZP\to\ZP$ such that $\piv(\CUT)=\CUT'$. Then
\begin{equation*}%
\left(\HTop(\cut_1;\bsig,\brho),\dots,\HTop(\cut_m;\bsig,\brho)\right)\eqd
\left(\HTop(\cut'_1;\bsig,\brho'),\dots,\HTop(\cut'_m;\bsig,\brho')\right),\quad\text{where $\brho':=\piv^{-1}(\brho)$}.
\end{equation*}
\end{theorem}
\begin{proof}
\def\ov#1{\overline{#1}}
The continuous \SCSV model is obtained as the following limit. First, consider the original \SCSV model. Choose $M_0,M_1,\dots\in\ZP$ and $N_1,N_2,\dots\in\ZP$ and subdivide the positive quadrant into rectangles of sizes $M_i\times N_j$ for all $(i,j)\in\ZNN\times \ZP$. One then combines each rectangle into a single \emph{fused vertex} by specializing the corresponding row and column rapidities to geometric progressions in $q$ of lengths $M_i$ and $N_j$. The fact that the height functions of the resulting \emph{fused \SCSV model} are specializations of the height functions of the original \SCSV model is shown in~\cite[Theorem~6.2]{BGW}. (It is stated for the case when all $N_j$ are equal but translates verbatim to the case of arbitrary $N_j$.) For each cut $\cut_i=(1,d_i,u_i,r_i)$, we can consider the \emph{unmerged cut} $\ov\cut_i=(1,\ov d_i,\ov u_i,\ov r_i)$ given by $\ov d_i=N_1+N_2+\dots+N_{d_i-1}+1$, $\ov u_i=N_1+N_2+\dots+N_{u_i}$, and $\ov r_i=M_1+M_2+\dots+M_{r_i}$. The unmerged versions $\ov\piv$ of $\piv$ and $\ov\cut'_i$ of $\cut'_i$ are defined similarly. 
 The two resulting tuples of unmerged cuts satisfy the assumptions of \cref{thm:main} and therefore the corresponding vectors of height functions have the same distribution modulo replacing the row rapidities $\by=(y_1,y_2,\dots)$ with $(\ov\piv)^{-1}(\by)$.  Thus after specializing the rapidities to geometric progressions and merging each $M_i\times N_j$ rectangle into a single fused vertex, we see that the two height vectors have the same distribution in the setting of the fused \SCSV model. A similar argument can be found in the proof of \cite[Theorem~6.3]{BGW}. The result now follows from \cite[Corollary~6.21]{BGW} which states that vertex weights of the continuous \SCSV model are obtained as particular limits of the fused \SCSV model vertex weights.
\end{proof}

\def\del{{\operatorname{delayed}}}

\subsection{Beta polymer}
Our next object is the \emph{directed Beta polymer}~\cite{BC} which is essentially just a restatement of the continuous \SCSV model from the previous subsection. We again fix two families $\bsig=(\sigma_i)_{i\in\ZNN}$ and $\brho=(\rho_j)_{j\in\ZP}$ as above. We first describe the model following \cite[Section~7.1]{BGW}.
\begin{itemize}[leftmargin=0.2in]
\item For each $(i,j)\in\ZNN\times\ZP$, we sample (independently) a Beta-distributed random variable $\eta_{i,j}\sim B(\sigma_i-\rho_j,\rho_j)$.
\item We consider the vertices of the grid $\ZNN\times\ZNN$ with diagonal and vertical edges, that is, for each $(i,j)\in\ZP\times\ZP$, we introduce diagonal edges connecting $(i-1,j-1)$ to $(i,j)$ and vertical edges connecting $(i,j-1)$ to $(i,j)$. For each $j\in\ZP$, we also have a vertical edge connecting $(0,j-1)$ to $(0,j)$.
\item To each edge $e$ we assign a \emph{weight} $\wt(e)$: if $e$ is a 
vertical edge connecting $(i,j-1)$ to $(i,j)$ then we set $\wt(e):=\eta_{i,j}$, and if $e$ is a diagonal edge connecting $(i-1,j-1)$ to $(i,j)$ then we set $\wt(e):=1-\eta_{i,j}$. 
\item For $l,d,u,r\in\ZNN$ such that $l\leq r$ and $r-l\leq u-d$, the \emph{delayed Beta polymer partition function} $\zux ldru^{B;\bsig,\brho}$ is defined by
\begin{equation*}%
 \zux ldru^{B;\bsig,\brho}=\sum_{(l,d)=\pi_d\to\pi_{d+1}\to\dots\to\pi_{u}=(r,u)} \ \prod_{k=f(\pi)}^{u}\wt(\pi_{k-1}\to\pi_k), 
\end{equation*}
where $\pi_k-\pi_{k-1}\in\{(1,1),(0,1)\}$ for all $k$ and $f(\pi):=\min\{k\mid \pi_k-\pi_{k-1}=(0,1)\}$. In other words, we take the product of weights of all edges of our lattice path that occur weakly after its first vertical edge. 
\end{itemize}

\begin{definition}\label{dfn:Cdur}
For vectors $\bd,\bu,\br\in(\ZNN)^m$, we denote $\CUT(\bd,\bu,\br):=(\cut_1,\dots,\cut_m)$, where  $\cut_i:=(1,d_i+1,u_i,r_i)$ for all $i\in[1,m]$.
\end{definition}
\begin{remark}
We put $d_i+1$ above in order to account for the following discrepancy: for $\cut=(1,d,u,r)$, $\HTop(\cut;\bsig,\brho)$ depends on $\sigma_0,\dots,\sigma_r$ and $\rho_d,\dots,\rho_u$, while $\zux 0dru^{B;\bsig,\brho}$ depends on $\sigma_0,\dots,\sigma_r$ and $\rho_{d+1},\dots,\rho_u$. It is thus natural to identify $\zux 0dru^{B;\bsig,\brho}$ with the cut $(1,d+1,u,r)$.
\end{remark} 

The following result confirms the prediction of \cite[Remark~7.4]{BGW}.
\begin{theorem}\label{thm:Beta}
Consider vectors $\bd,\bu,\bd',\bu',\br\in(\ZNN)^m$ such that both $\CUT(\bd,\bu,\br)$ and $\CUT(\bd',\bu',\br)$ are left-aligned tuples of cuts. 
 Suppose that we have a bijection $\piv:\ZP\to\ZP$ such that $\piv(\CUT(\bd,\bu,\br))=\CUT(\bd',\bu',\br)$ and set $\brho':=\piv^{-1}(\brho)$. Then
\begin{equation}\label{eq:Beta}
  \left(\zux 0{d_1}{r_1}{u_1}^{B;\bsig,\brho},\dots,\zux0{d_m}{r_m}{u_m}^{B;\bsig,\brho}\right)\eqd  \left(\zux 0{d'_1}{r_1}{u'_1}^{B;\bsig,\brho'},\dots,\zux0{d'_m}{r_m}{u'_m}^{B;\bsig,\brho'}\right).
\end{equation}
\end{theorem}
\begin{proof}
By \cite[Proposition~7.2]{BGW}, the joint distribution of the left hand side of~\eqref{eq:Beta} coincides with the distribution of the vector $\left(\exp \left(-\HTop(\cut_1;\bsig,\brho)\right),\dots,\exp \left(-\HTop(\cut_m;\bsig,\brho)\right)\right)$ of exponentiated height functions $\HTop(\cut_i;\bsig,\brho)$ from \cref{sec:continuous_model}. Here $(\cut_1,\dots,\cut_m)=\CUT(\bd,\bu,\br)$ as in \cref{dfn:Cdur}. The result now follows from \cref{thm:continuous_model}.
\end{proof}

\subsection{Intersection matrices}
When taking the limits in the next subsections, we will specialize to the \emph{homogeneous case} where all $\sigma_i$ are equal to a single value $\sigma$ and all $\rho_j$ are equal to a single value $\rho$ satisfying $0<\rho<\sigma$. In view of this, let us discuss intersection matrices introduced in \cref{dfn:IMP}.

\begin{proposition}\label{prop:left_aligned}
Consider vectors $\bd,\bu,\bd',\bu',\br\in(\ZNN)^m$ and let $\CUT(\bd,\bu,\br)=(\cut_1,\dots,\cut_m)$ and $\CUT(\bd',\bu',\br)=(\cut'_1,\dots,\cut'_m)$ be as in \cref{dfn:Cdur}. Assume that both $\CUT(\bd,\bu,\br)$ and $\CUT(\bd',\bu',\br)$ are left-aligned tuples of cuts. Then there exists a bijection $\piv:\ZP\to\ZP$ satisfying $\piv(\CUT(\bd,\bu,\br))=\CUT(\bd',\bu',\br)$ if and only if $\IMP(\bd,\bu)=\IMP(\bd',\bu')$.
\end{proposition}
\begin{proof}
In order to distinguish between subsets of $\Z$ and subsets of $\R$, let us denote $[d,u]_\Z:=\{j\in\Z\mid d\leq j\leq u\}$ and $[d,u]_\R:=\{y\in\R\mid d\leq y\leq u\}$. Recall that $\cut_i=(1,d_i+1,u_i,r_i)$ and $\cut'_i=(1,d'_i+1,u'_i,r'_i)$ for all $i\in[1,m]$. Recall also that the entries of $\IMP(\bd,\bu)$ are given by $\IMP_{i,j}=\max(0,\min(u_i,u_j)-\max(d_i,d_j))$, which is the cardinality of $[d_i+1,u_i]_\Z\cap[d_j+1,u_j]_\Z$. 

If $\piv$ exists then it sends $[d_i+1,u_i]_\Z\cap[d_j+1,u_j]_\Z$ bijectively to $[d'_i+1,u'_i]_\Z\cap[d'_j+1,u'_j]_\Z$, thus $\IMP(\bd,\bu)=\IMP(\bd',\bu')$. Conversely, it suffices to show that for any $I\subset[1,m]$, the intersections $\bigcap_{i\in I}[d_i+1,u_i]_\Z$ and $\bigcap_{i\in I}[d'_i+1,u'_i]_\Z$ have the same cardinality. Indeed, we have
\begin{equation*}\label{eq:IMPZ}
  \# \bigcap_{i\in I}[d_i+1,u_i]_\Z=\max\left(0,\min\{u_i\mid i\in I\}-\max\{d_j\mid j\in I \}\right)=\min \left\{\IMP_{i,j}\middle|i,j\in I\right\}.
\end{equation*}
Thus the assumption $\IMP(\bd,\bu)=\IMP(\bd',\bu')$ implies the result for all $I\subset[1,m]$.
\end{proof}

\subsection{Gamma polymer}

Recall that the Gamma distribution with parameter $\kappa>0$ is supported on $\R_{>0}$ with density
\begin{equation*}%
  \frac1{\Gamma(\kappa)} x^{\kappa-1}\exp(-x),\quad x>0.
\end{equation*}
Let us now fix some $\kappa>0$ and describe the Gamma polymer~\cite{CSS,OCO}.
\begin{itemize}[leftmargin=0.2in]
\item For each $(i,j)\in\ZNN\times\ZP$, we sample an independent random variable $\eta_{i,j}$ which is Gamma-distributed with parameter  $\kappa$.
\item We again deal with the grid $\ZNN\times\ZNN$ with vertical and diagonal edges. The weight of each vertical edge connecting $(i,j-1)$ to $(i,j)$ is equal to $\eta_{i,j}$, and the weights of all diagonal edges are equal to $1$.
\item For $l,d,u,r\in\ZNN$ such that $l\leq r$ and $r-l\leq u-d$, the \emph{Gamma polymer partition function} $\zux ldru^{\Gamma}$ is defined by
\begin{equation*}%
 \zux ldru^{\Gamma}=\sum_{(l,d)=\pi_d\to\pi_{d+1}\to\dots\to\pi_{u}=(r,u)} \ \prod_{k=d+1}^{u}\wt(\pi_{k-1}\to\pi_k), 
\end{equation*}
where $\pi_k-\pi_{k-1}\in\{(1,1),(0,1)\}$ for all $k$. Note that since the diagonal edges have weight $1$, only the vertical edges contribute to the product.
\end{itemize}

\begin{theorem}\label{thm:Gamma}
  Consider vectors $\bd,\bu,\bd',\bu',\br\in(\ZNN)^m$ such that both $\CUT(\bd,\bu,\br)$ and $\CUT(\bd',\bu',\br)$ are left-aligned tuples of cuts. If $\IMP(\bd,\bu)=\IMP(\bd',\bu')$ then 
\begin{equation*}
  \left(\zux 0{d_1}{r_1}{u_1}^{\Gamma},\dots,\zux 0{d_m}{r_m}{u_m}^{\Gamma}\right)\eqd    \left(\zux 0{d'_1}{r_1}{u'_1}^{\Gamma},\dots,\zux 0{d'_m}{r_m}{u'_m}^{\Gamma}\right).
\end{equation*}
\end{theorem}
\begin{proof}
  By \cref{prop:left_aligned}, there exists a bijection $\piv:\ZP\to\ZP$ sending $\CUT(\bd,\bu,\br)$ to $\CUT(\bd',\bu',\br)$. The result now follows from \cref{thm:Beta} by substituting $\rho_j=\varepsilon^{-1}$, $\sigma_i=\varepsilon^{-1}+\kappa$ and taking a limit as $\varepsilon\to0$, see \cite[Section~7.2]{BGW}.
\end{proof}

\subsection{O’Connell--Yor polymer}
Our next limiting transition leads to the following model.
\begin{itemize}[leftmargin=0.2in]
\item For each $n\in\ZNN$, let $B_n(t)$, $t\geq0$, be an independent standard Brownian motion.
\item For each $l, r\in\ZNN$ and $d,u\in\R_{\geq0}$ satisfying $l\leq r$ and $d\leq u$, define
\begin{equation*}%
  \zux ldru^{OY}=\int_{d=t_l<t_{l+1}<\dots<t_{r+1}=u} \exp \left[\sum_{i=l}^r B_i(t_{i+1})-B_i(t_i)\right]\,dt_{l+1}\dots dt_{r}.
\end{equation*}
\end{itemize}

We extend the definition of a cut $\cut=(l,d,u,r)$ to the case where $d\leq u$ are real numbers. In this case, we say that  $\CUT=(\cut_1,\dots,\cut_m)$ is a \emph{left-aligned tuple of cuts} if $l_i=1$ for all $i\in[1,m]$ and for any $i,j\in[1,m]$, we have either $r_i\leq r_j$ and $u_i\geq u_j$ or $r_i\geq r_j$ and $u_i\leq u_j$.
 
\begin{theorem}\label{thm:OY}
  Consider vectors $\bd,\bu,\bd',\bu'\in(\R_{\geq0})^m$ and $\br\in(\ZNN)^m$ such that both $\CUT(\bd,\bu,\br)$ and $\CUT(\bd',\bu',\br)$ are left-aligned tuples of cuts. If $\IMP(\bd,\bu)=\IMP(\bd',\bu')$ then 
\begin{equation*}
  \Big(\zux 0{d_1}{r_1}{u_1}^{OY},\dots,\zux 0{d_m}{r_m}{u_m}^{OY}\Big)\eqd    \Big(\zux 0{d'_1}{r_1}{u'_1}^{OY},\dots,\zux 0{d'_m}{r_m}{u'_m}^{OY}\Big).
\end{equation*}
\end{theorem}
\begin{proof}
As explained in \cite[Section~7.3]{BGW}, one obtains $\zux ldru^{OY}$ as a properly rescaled limit of 
$\zux l{ Ld}r{ Lu}^{\Gamma}$ as $L\to\infty$. It is easy to check that there exists a sequence $\bd_L,\bu_L,\bd'_L,\bu'_L\in(\ZNN)^m$ for $L=1,2,\dots$ such that for all $L$,
\begin{itemize}[leftmargin=0.2in]
\item $\CUT(\bd_L,\bu_L,\br)$ and $\CUT(\bd'_L,\bu'_L,\br)$ are left-aligned tuples of cuts,
\item  $\IMP(\bd_L,\bu_L)=\IMP(\bd'_L,\bu'_L)$, 
\item the limit of $\frac1L(\bd_L,\bu_L,\bd'_L,\bu'_L)$ as $L\to\infty$ equals $(\bd,\bu,\bd',\bu')$, and
\item the limit of $\frac1L\IMP(\bd_L,\bu_L)$ as $L\to\infty$ equals $\IMP(\bd,\bu)$.
\end{itemize}
The result now follows by applying \cref{thm:Gamma} and taking the limit as $L\to\infty$.
\end{proof}

We are ready to consider the models described in \cref{sec:applications}.
\subsection{Brownian last passage percolation}
Recall that $\zux ldru$ has been defined in~\eqref{eq:BLPP}. Let us give a generalization of \cref{thm:BLPP}.
\begin{theorem}\label{thm:BLPP2}
  Consider vectors $\bd,\bu,\bd',\bu'\in(\R_{\geq0})^m$ and $\br\in(\ZNN)^m$ such that both $\CUT(\bd,\bu,\br)$ and $\CUT(\bd',\bu',\br)$ are left-aligned tuples of cuts. If $\IMP(\bd,\bu)=\IMP(\bd',\bu')$ then 
\begin{equation*}
  \left(\zux 0{d_1}{r_1}{u_1},\dots,\zux 0{d_m}{r_m}{u_m}\right)\eqd    \left(\zux 0{d'_1}{r_1}{u'_1},\dots,\zux 0{d'_m}{r_m}{u'_m}\right).
\end{equation*}  
\end{theorem}
\begin{proof}
As explained in \cite[Section~7.4]{BGW}, $\zux ldru$ is obtained as a scaling limit of the logarithm of $\zux l{dL}r{uL}^{OY}$ as $L\to\infty$. The result thus follows from \cref{thm:OY}.
\end{proof}

\def\sqL{\sqrt{L}}
\subsection{KPZ equation}
\begin{proof}[Proof of \cref{thm:KPZ}]
The argument in \cite[Section~7.5]{BGW} shows that one can obtain $\KZxxx ytx$ as a scaling limit of $\zux 0{y}{tL}{t\sqrt L+x}^{OY}$ as $L\to\infty$. Let $L$ be such that $tL$ is an integer and set $\br=(tL,tL,\dots,tL)\in(\ZNN)^m$, $\br'=(t\sqL,t\sqL,\dots,t\sqL)\in(\R_{\geq0})^m$. We introduce $\bd,\bd',\bu,\bu'\in(\R_{\geq0})^m$ by $\bd:=\by$, $\bd':=\by'$, $\bu:=\br'+\bx$, and $\bu':=\br'+\bx'$. Let $\onebb$ denote the $m\times m$ matrix with all entries equal to $1$. If $L$ is large enough so that $t\sqL>d_i$ for all $i\in[1,m]$ then we have 
\begin{equation}\label{eq:KPZ_IMP}
\IMP(\bd,\bu)=t\sqL\onebb+\IM(\by,\bx)=  t\sqL\onebb+\IM(\by',\bx')=\IMP(\bd',\bu').
\end{equation}
Therefore \cref{thm:OY} applies and we obtain the result by sending $L\to\infty$.
\end{proof}

\subsection{Airy sheet}
In view of the recent results~\cite{DOV}, one can define the Airy sheet $\Acal(x,y)$ as a limit of the Brownian last passage percolation, cf. \cite[Equation~(7.14)]{BGW}:
\begin{equation}\label{eq:Airy}
 \frac{\zux 0{2x n^{2/3}}{n}{n+2 y n^{2/3}} -2n -2n^{2/3}(y-x)+(x-y)^2n^{1/3}}{  n^{1/3}} \to \mathcal A(x,y)
\end{equation}
for $x,y\in\R$ as $n\to\infty$. 
\begin{proof}[Proof of \cref{thm:Airy}]
We proceed as in the proof of \cref{thm:KPZ}: for an integer $n\in\ZP$, set $\br=(n,n,\dots,n)\in(\ZNN)^m$, $\bd:=2n^{2/3}\bx$, $\bd':=2n^{2/3}\bx'$, $\bu:=\br+2n^{2/3}\by$, and $\bu':=\br+2n^{2/3}\by'$. As in~\eqref{eq:KPZ_IMP}, we find $\IMP(\bd,\bu)=\IMP(\bd',\bu')$ and then deduce the result from \cref{thm:BLPP2} by sending $n\to\infty$.
\end{proof}

\section{Arbitrary permutations, Kazhdan--Lusztig polynomials, and~positroid~varieties}\label{sec:KL_GR}

We discuss a surprising connection between the \SCSV model and a family of algebraic and combinatorial objects: Kazhdan--Lusztig $R$-polynomials (see~\cite{KL1} or~\cite[Chapter~5]{BjoBre}), Deodhar's \emph{distinguished subexpressions}~\cite{Deodhar,MR}, and the positroid decomposition~\cite{Pos,BGY,KLS} of the Grassmannian. As we speculate in \cref{sec:common_gen}, these observations suggest that there could be a common generalization of the \SCSV model and the combinatorics of positroid varieties.

In the last subsection, we give a counterexample to a natural extension of \cref{thm:main} to arbitrary wiring diagram domains but give a conjectural generalization of the shift invariance of~\cite{BGW} to arbitrary domains.

First, we concentrate on the limit of the \SCSV model as $y_1,y_2,\dots \to 0$, in which case all parameters $\sp_{i,j}$ defined in~\eqref{eq:sp} become equal to $1/q$. More generally, for an arbitrary Yang--Baxter element $\YB^w$ defined for $w\in\Sn$ in~\eqref{eq:YB_rec}, we consider a limiting regime where $z_1\gg z_2\gg\dots\gg z_n$, i.e., we assume that $z_j/z_i\to0$ for all $1\leq i<j\leq n$. In the limit, we get $\sp_{i,j}=1/q$ for all $1\leq i<j\leq n$. We denote the resulting element by 
\[\YB^w|_{z_1\gg\dots\gg z_n}=R_{i_1}(1/q)R_{i_2}(1/q)\cdots R_{i_r}(1/q),\]
where $w=s_{i_1}s_{i_2}\cdots s_{i_r}$ is a reduced word. 

\subsection{Kazhdan--Lusztig $R$-polynomials}
For all pairs $\pi,w\in\Sn$ such that $\pi\leq w$ in the Bruhat order, the associated \emph{Kazhdan--Lusztig $R$-polynomials} $\Rvw(q)$ are defined uniquely by the condition that for all $w\in \Sn$, we have
\begin{equation}\label{eq:R_v^w_dfn}
  q^{\ell(w)}(T_{w^{-1}})^{-1}=\sum_{\pi\leq w} \Rvw(q) T_\pi.
\end{equation}

The following result was shown in~\cite{BuNa} by induction on $\ell(w)$. We will prove it bijectively in \cref{prop:dist_SCSV} below.
\begin{lemma}
\label{lemma:YB_Rvw}
For all $w\in\Sn$, we have
\begin{equation*}%
  q^{\ell(w)}\YB^w|_{z_1\gg\dots\gg z_n}=\sum_{\pi\leq w} \Rvw(q) T_\pi.
\end{equation*}
\end{lemma}
\begin{proof}
Indeed, we have $R_k(1/q)=1/q T_k+(1-1/q)$ which by~\eqref{eq:Hecke_T} is equal to $T_k^{-1}$. Thus $\YB^w|_{z_1\gg\dots\gg z_n}=T_{i_1}^{-1}\cdots T_{i_r}^{-1}=(T_{w^{-1}})^{-1}$. The result follows from~\eqref{eq:R_v^w_dfn}.
\end{proof}
\noindent By \cref{prop:YB_equals_Zbipi}, this gives an interpretation of each Kazhdan--Lusztig $R$-polynomial in terms of the \SCSV model: $q^{-\ell(w)}\Rvw(q)$ equals the probability of observing $\pi$ as the color permutation of the \SCSV model inside a wiring diagram for $w$ with all $\sp$-parameters equal to $1/q$. Note that one of the vertex weights (namely, $1-q\sp$) in \cref{fig:spec} becomes equal to zero, however, the resulting ``five-vertex model'' is very different from the standard five-vertex model studied e.g. in~\cite{BBB,BSW,dGKW}. Our goal is to show that after sending all $\sp$-parameters to $1/q$, the configurations of the \SCSV model with nonzero weight admit a weight-preserving bijection with well-studied combinatorial objects called \emph{Deodhar's distinguished subexpressions}~\cite{Deodhar}.

\def\bv{{\pmb{\pi}}}

\begin{figure}
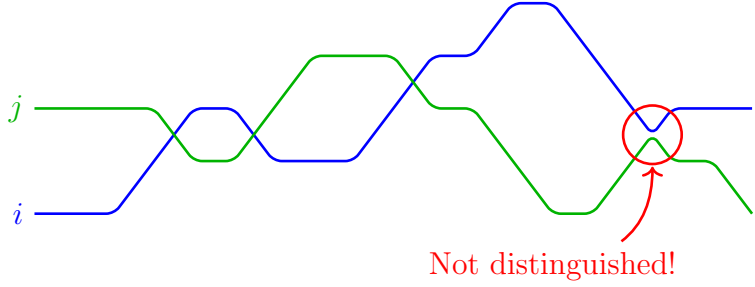

  \FIGnondist
  \caption{\label{fig:non_dist} The distinguished condition (\cref{dfn:dist}) states that two paths cannot touch after they have crossed an odd number of times. This is precisely the condition describing \SCSV configurations that remain after we send the weight $1-q\sp$ to $0$.}
\end{figure}

\begin{definition}\label{dfn:dist}
A \emph{reduced expression} for $w\in\Sn$ is a sequence $\bw=(s_{i_1},s_{i_2},\dots,s_{i_r})$ such that $w=s_{i_1}s_{i_2}\cdots s_{i_r}$ is a reduced word. A \emph{subexpression} for $\pi\in\Sn$ inside $\bw$ is a sequence $\bv=(\sb_{i_1},\sb_{i_2},\dots,\sb_{i_r})$ such that $\pi=\sb_{i_1}\sb_{i_2}\cdots \sb_{i_r}$ and each $\sb_{i_j}$ is equal to either $s_{i_j}$ or $\id$. We introduce partial products $\vp_j:=\sb_{i_1}\sb_{i_2}\cdots \sb_{i_j}$. A subexpression $\bv$ is called \emph{distinguished} (see \cref{fig:non_dist}) if whenever $\ell(\vp_{j-1}s_{i_j})=\ell(\vp_{j-1})-1$, we have $\sb_{i_j}=s_{i_j}$. We denote $\Jo:=\{j\in[1,r]\mid \sb_{i_j}=\id\}$, $\Jm:=\{j\in[1,r]\mid \ell(\vp_{j-1}\sb_{i_j})=\ell(\vp_{j-1})-1\}$ and define the \emph{weight} of $\bv$ by $(q-1)^{|\Jo|}q^{|\Jm|}$.
\end{definition} 
When $q\geq2$ is a prime power,\footnote{Note that for $q\geq 1$ and $\sp:=1/q$, the probabilities $\sp$, $1-\sp$, $q\sp$, and $1-q\sp$ all belong to $[0,1]$.} the above weights have the following geometric meaning: the $R$-polynomial $\Rvw(q)$ equals the number of points over $\Fbb_q$ inside an \emph{(open) Richardson variety} $\Rich$. Richardson varieties give a natural stratification of the \emph{flag variety} $G/B$, where $G=\SL_n(\C)$ and $B$ is the subgroup of $G$ consisting of upper triangular matrices. (Thus $G/B$ is the space of \emph{complete flags} $(V_0\subset V_1\subset\dots\subset V_n=\C^n)$ of linear subspaces of $\C^n$ with $\dim V_i=i$ for all $i$.) We have $G/B=\bigsqcup_{\pi\leq w}\Rich$. Deodhar~\cite{Deodhar} constructed a decomposition of each Richardson variety $\Rich$ into finer pieces corresponding to distinguished subexpressions of $\pi$ inside a reduced expression $\bw$ for $w$. Each piece is isomorphic to $(\C^\ast)^{|\Jo|}\times \C^{|\Jm|}$. The same decomposition works over finite fields, thus the number of $\Fbb_q$-points inside the corresponding piece equals $(q-1)^{|\Jo|}q^{|\Jm|}$. In particular, we have
\begin{equation*}%
  \Rvw(q)=\sum_{\substack{\text{distinguished subexpressions }\bv\\\text{for $\pi$ inside $\bw$}}} (q-1)^{|\Jo|}q^{|\Jm|}.
\end{equation*}
The following observation gives a bijective proof of \cref{lemma:YB_Rvw}.
\begin{proposition}\label{prop:dist_SCSV}
Let $\pi,w\in\Sn$ and choose a reduced expression $\bw$ for $w$. Then distinguished subexpressions for $\pi$ inside $\bw$ are in a natural bijection with configurations of the \SCSV model inside a wiring diagram corresponding to $\bw$ with all $\sp$-parameters set to $1/q$. For a given subexpression $\bv$, the \SCSV-weight of the corresponding configuration equals $q^{-\ell(w)}\cdot (q-1)^{|\Jo|}q^{|\Jm|}$.
\end{proposition}
\begin{proof}
This is obvious from the definitions, see \cref{fig:non_dist}.
\end{proof}

\subsection{Grassmannian interpretation of the flip theorem}
The flip theorem (\cref{thm:flip}) gives a non-trivial relation between partition functions of the \SCSV model associated with a wiring diagram of a ``rectangular'' permutation $w:=\wMN$ from \cref{ex:wMN}. Therefore specializing to $\by=0$, we obtain a relation between $R$-polynomials $\Rvw(q)$ for various $\pi$ satisfying horizontal and vertical boundary conditions. 

Recall that we set $n:=M+N$.  When $w=\wMN$ is a rectangular permutation and $\pi\leq w$, the Richardson variety $\Rich$ projects isomorphically onto a \emph{positroid variety} inside the \emph{Grassmannian} $\GRMN$ of complex $M$-dimensional linear subspaces of $\C^{n}$.  The positroid decomposition of $\GRMN$ was constructed by Knutson--Lam--Speyer~\cite{KLS} building on the work of Postnikov~\cite{Pos}. \Cref{thm:flip} implies that the number of points over $\Fbb_q$ in one subvariety $\Pi^{\H,\V}$ of $\GRMN$ equals the number of points in another subvariety $\Pi^{\flip(\H),\V}$ of $\GRMN$. Our goal is to describe the subvarieties $\Pi^{\H,\V}$ and $\Pi^{\flip(\H),\V}$ and give a simple bijection between their points over any field. This gives a ``lift'' of the $\by=0$ specialization of \cref{thm:flip} to $\GRMN$. One might wonder whether the whole \cref{thm:flip}, as well as other properties and objects related to the \SCSV model, can be lifted to the level of $\GRMN$. We discuss this further in \cref{sec:common_gen}.

Let $\Bound$ denote the set of \emph{bounded affine permutations} which are bijections $f:\Z\to\Z$ such that 
\begin{itemize}
\item $f(i+n)=f(i)+n$ for all $i\in\Z$, 
\item $\frac1n\sum_{i=1}^n(f(i)-i)=M$, and
\item $i\leq f(i)\leq i+n$ for all $i\in\Z$.
\end{itemize}
An element $X\in\GRMN$ can be viewed as a row span of a full rank complex $M\times \MPN$ matrix $A$. Denote the columns of $A$ by $A_1,A_2,\dots,A_{\MPN}\in\C^M$. We extend this uniquely to all $i\in\Z$ by requiring that $A_{i+n}=A_i$ for all $i\in\Z$. We define $f_A:\Z\to\Z$ by
\begin{equation}\label{eq:f_A}
  f_A(i)=\min \left\{j\geq i\mid A_i\in\Span \left(A_{i+1},A_{i+2},\dots,A_j\right)\right\}\quad\text{for $i\in\Z$.}
\end{equation}
It is not hard to see that $f_A\in\Bound$ and that it only depends on the row span $X$ of $A$, thus we denote $f_X:=f_A$ and consider a map $\GRMN\to\Bound$ sending $X\mapsto f_X$. Denoting $\Pio_f:=\{X\in\GRMN\mid f_X=f\}$, we obtain the \emph{positroid decomposition} of $\GRMN$ given by $\GRMN=\bigsqcup_{f\in\Bound}\Pio_f$.

A permutation $u\in\Sn$ is called \emph{$(M,N)$-Grassmannian} if $u(1)<u(2)<\dots<u(N)$ and $u(N+1)<u(N+2)<\dots<u(n)$. Let  $\tom\in\Bound$ denote the map sending $i\mapsto i+n$ for $i\in[1,M]$ and $i\mapsto i$ for $i\in[M+1,\MPN]$ (this defines $\tom(i)$ uniquely for all other $i\in\Z$). Extend each permutation $\pi\in \Sn$ to a map $\pi:\Z\mapsto\Z$ satisfying $\pi(i+n)=\pi(i)+n$. Then it is well known~\cite[Proposition~3.15]{KLS} that the map $(\pi,w)\mapsto f_{\pi,w}:=\pi\tom w^{-1}$ (where multiplication is given by composition) gives a bijection between pairs $(\pi,w)$ such that $\pi\leq w$ and $w$ is $(M,N)$-Grassmannian and the set $\Bound$. See~\cite[Figure~2]{GKL3} for a pictorial representation of this correspondence. By~\cite[Proposition~5.4]{KLS}, the Richardson variety $\Rich$ is isomorphic to the positroid variety $\Pio_{f_{\pi,w}}$, thus the number of points in $\Pio_{f_{\pi,w}}$ over $\Fbb_q$ is counted by the $R$-polynomial $\Rvw(q)$. When $w=\wMN$, we have $f_{\pi,w}=\pi\id_M$, where $\id_M:\Z\to\Z$ sends $i\mapsto i+M$ for all $i\in\Z$. 

\begin{example}
Let $w=\wMN$ and $\pi=\id$. In this case, $f_{\pi,w}=\id_M$ corresponds to the \emph{top-dimensional positroid variety} $\Pio_{\id_M}$. The row span of a $k\times n$ matrix $A$ belongs to $\Pio_{\id_M}$ if and only if all of its \emph{Pl\"ucker coordinates} corresponding to cyclic intervals are nonzero:
\begin{equation*}%
  \Delta_{1,2,\dots,M}(A)\neq0,\quad \Delta_{2,3,\dots,M+1}(A)\neq0,\quad\dots,\quad \Delta_{n,1,2,\dots,M-1}(A)\neq0.
\end{equation*}
 In other words, the row span of $A$ belongs to $\Pio_{\id_M}$ if and only if the vectors $A_{i+1},A_{i+2},\dots,A_{i+M}$ form a basis of $\C^M$ for all $i\in\Z$, as clearly follows from~\eqref{eq:f_A}. Let us now consider the case $M=N=2$ and count the number of points inside $\Pio_{\id_M}$ over $\Fbb_q$ in two different ways. Having $\pi=\id$ corresponds to the following two \SCSV model configurations:
\begin{center}
\FIGgrex
\end{center}
By \cref{prop:dist_SCSV}, multiplying the probabilities by $q^{\ell(w)}$, we get $\Rvw(q)=(q-1)^4+q(q-1)^2$. On the other hand, applying row operations, we find that $\Pio_{\id_M}$ consists of (row spans of) matrices $\begin{pmatrix}
1 & 0 & a & b\\
0 & 1 & c & d
\end{pmatrix}$ such that $a\neq 0$, $d\neq 0$, and $ad\neq bc$. The number of such matrices over $\Fbb_q$ equals $(q-1)^4+q(q-1)^2$, as expected.
\end{example}

Recall from \cref{sec:intro-rect} that we have defined the quantities $\H:=\Hpi$ and $\V:=\Vpi$ associated to $\pi\in\Sn$. We see that if $\H=\{(\HL_1,\HR_1),\dots,(\HL_\h,\HR_\h)\}$ and $\V=\{(\VD_1,\VU_1),\dots,(\VD_\v,\VU_\v)\}$ then $f:=f_{\pi,w}=\pi\id_M$ satisfies the following conditions:
\begin{align}
\label{eq:Haff}
\{(i,f(i))\mid i,f(i)\in[1,N]\}&=\{(\HL_1-M,\HR_1),\dots,(\HL_\h-M,\HR_\h)\};\\
\label{eq:Vaff}
\{(i,f(i))\mid i\in[1-M,0]\text{ and }f(i)\in[N+1,n]\}&=\{(\VD_1-M,\VU_1),\dots,(\VD_\v-M,\VU_\v)\}.
\end{align}
Denote the left and right hand sides of~\eqref{eq:Haff} (resp., \eqref{eq:Vaff}) by $\Haff_f$ and $\Haff$ (resp., $\Vaff_f$ and $\Vaff$). 

Let $X\in\GRMN$ be the row span of a matrix $A$ and consider the pair $(\pi_X,w_X)$ such that $f_X=f_{\pi_X,w_X}$. It is not hard to see that $w_X=\wMN$ if and only if the vectors $A_{N+1},A_{N+2},\dots,A_{n}$ form a basis of $\C^M$. We denote by $\Omega_{[N+1,n]}\subset\GRMN$ the set of (row spans of) such matrices $A$, which is usually called an \emph{opposite Schubert cell}. We let
\begin{equation*}%
  \Pi^{\H,\V}:=\{X\in\Omega_{[N+1,n]}\mid \Haff_{f_X}=\Haff\text{ and }\Vaff_{f_X}=\Vaff\}.
\end{equation*}
Finally, define the map $\fop:\GRMN\to\GRMN$ sending (the row span of) a matrix with columns $A_1,A_2,\dots,A_n$ to (the row span of) the matrix with columns
\begin{equation*}%
  A_N,A_{N-1},\dots,A_1,A_{N+1},A_{N+2},\dots,A_n.
\end{equation*}
The $\by\to0$ limit of \cref{thm:flip} states that the varieties $\Pi^{\H,\V}$ and $\Pi^{\flip(\H),\V}$ contain the same number of $\Fbb_q$-points. In fact, a much stronger statement holds.
\begin{proposition}\label{prop:Gr}
For all $\H,\V$, the map $\fop$ gives an isomorphism of varieties:
\begin{equation*}%
  \fop:\Pi^{\H,\V}\xrasim \Pi^{\flip(\H),\V}.
\end{equation*}
\end{proposition}
\begin{proof}
  Let $X\in\Pi^{\H,\V}$ be the row span of $A$, $f:=f_X$, $Y:=\fop(X)$, $g:=f_Y$. Clearly, $Y\in\Omega_{[N+1,n]}$. By~\eqref{eq:f_A}, we see that if $f(i)=j$ then $A_j$ appears with a nonzero coefficient in the expansion of $A_i$ as a linear combination of $A_{i+1},\dots,A_j$. In particular, $\max\{i'\leq j\mid A_j\in\Span \left(A_{j-1},A_{j-2},\dots,A_{i'}\right)\}$ is equal to $i$. This shows that $\Haff_g=\widetilde{\flip(\H)}$. Since $\Span(A_1,\dots,A_N)=\Span(A_N,\dots,A_1)$, we get $\Vaff_g=\Vaff$, thus indeed $Y\in\Pi^{\flip(\H),\V}$.
\end{proof}

\subsection{Common generalizations?}\label{sec:common_gen}
The main message suggested by the above observations is the following. The $\by=0$ specialization of the \SCSV model  recovers well studied objects such as $R$-polynomials, which are ``shadows'' of geometric objects  such as positroid varieties. The latter can be parametrized by other combinatorial objects such as planar bipartite graphs of~\cite{Pos}. 
\begin{question}
Does there exist a common generalization of the stochastic colored six-vertex model in a rectangle (with arbitrary $\bx,\by$) and positroid varieties in the Grassmannian?
\end{question}

For example, it would be interesting to give an interpretation of the \SCSV model in terms of \emph{planar bipartite graphs} of~\cite{Pos}. Surprisingly, the recurrence~\eqref{eq:HV_T_w*R_k}--\eqref{eq:master_flip} used in the proof of \cref{thm:flip} has already appeared in the literature precisely in the language of such graphs.  We thank Thomas Lam for pointing out the following remark to us.
\begin{remark}\label{rmk:MuSp}
The Grassmannian lift of the $\by=0$ specialization of~\eqref{eq:HV_T_w*R_k}--\eqref{eq:master_flip} was used in~\cite[Section~4]{MuSp} to show local acyclicity of cluster algebras associated to Postnikov's planar bipartite graphs. Applying either recurrence to $R$-polynomials and distinguished subexpressions leads to unexpected phenomena related to rational Catalan combinatorics which will be explored in future work with Thomas Lam.
\end{remark}
The above local acyclicity property was used in~\cite{GL} to show that positroid varieties are cluster varieties, which allows one to study their cohomology via the associated \emph{mixed Hodge tables}, see~\cite{LS}. The number of points over $\Fbb_q$ can be calculated from the mixed Hodge table, thus one possible direction is to understand the relation between the \SCSV model and the cohomology of positroid varieties.

Kazhdan--Lusztig $R$-polynomials arise in several other contexts. For example, they are used to express \emph{Kazhdan--Lusztig $P$-polynomials}~\cite{KL1,KL2}, which have nonnegative integer coefficients and are of great interest in representation theory and algebraic geometry. One can define $P$-polynomials through an involution on $\Hecke$ that is very similar (but different) from the one used in the proof of \cref{thm:CPS}. We thank Pavlo Pylyavskyy for discussions related to the following question.
\begin{question}
Determine the relation between Kazhdan--Lusztig $P$-polynomials and the \SCSV model. For example, is there a probabilistic interpretation of their nonnegative integer coefficients?
\end{question}

By \cref{prop:YB_equals_Zbipi}, the \SCSV model is literally equivalent to the expansion of the Yang--Baxter basis $\{\YB^w\}$ in the $\{T_\pi\}$ basis. The Yang--Baxter basis has several interesting properties. For example, the \emph{orthogonality relations}, discovered already in~\cite[Theorem~5.1]{LLT}, are yet to be understood from the \SCSV model point of view.

One other direction that we think is worth exploring is to understand the appearance of pipe dreams both in the \SCSV model dynamics and in Schubert calculus. There are several classes of pipe dreams arising in Schubert calculus, for example, ordinary pipe dreams~\cite{BeBi,FK,KM} or \emph{bumpless} pipe dreams recently introduced in~\cite{LLS}. One may consider \emph{reduced} pipe dreams (where two paths can intersect at most once) or \emph{non-reduced} pipe dreams which appear e.g. in the study of Grothendieck polynomials~\cite{FK2,LRS,Lascoux2}. 
\begin{question}
Determine which of the above classes of pipe dreams are related to the \SCSV model or its specializations.
\end{question}
An example of a relation like this between \emph{bumpless pipe dreams}~\cite{LLS} and the six-vertex model was recently pointed out by Anna Weigandt~\cite{Weigandt}. These objects are closely related to alternating sign matrices~\cite{Lascoux2,KupASM}. See also~\cite{BBB,BSW} for related work.

\def\PD{\mathcal{PD}}
\subsection{Counting pipe dreams}\label{sec:counting-pipe-dreams}
We give a curious enumerative specialization of \cref{thm:flip}. The number of configurations of the \SCSV model inside an $M\times N$-rectangular domain is equal to $2^{MN}$, and for each configuration $\pip$, we denote the associated color permutation by $\pi_\pip$ (as in \cref{sec:equiv-class-pipe}). Given $\H,\V$, we let $\PD^{\H,\V}:=\{\pip\mid  \text{$\H^{M,N}_{\pi_\pip}=\H$ and $\V^{M,N}_{\pi_\pip}=\V$}\}.$
\begin{corollary}
For any $\H,\V$, we have 
\begin{equation}\label{eq:PD_cnt}
  \#\PD^{\H,\V}=\#\PD^{\flip(\H),\V}.
\end{equation}
\end{corollary}
\begin{proof}
For a variable $t$, denote $y:=\frac{1-qt}{1-t}x$ and substitute $y_j:=y$ and $x_i:=x$ for all $i\in[1,M]$ and $j\in[1,N]$. Then each parameter $\sp_{i,j}$ given by~\eqref{eq:sp} becomes equal to $t$. For a pipe dream $\pip$, let $\xing(\pip)$ denote the number of cells of $\pip$ that contain a crossing, thus $0\leq \xing(\pip)\leq MN$. Applying \cref{thm:flip} and sending $q\to1$, we get
\begin{equation}\label{eq:PD_cnt_q}
  \sum_{\pip\in\PD^{\H,\V} } t^{\xing(\pip)}(1-t)^{MN-\xing(\pip)}=
\sum_{\pip\in\PD^{\flip(\H),\V} } t^{\xing(\pip)}(1-t)^{MN-\xing(\pip)}.
\end{equation}
Substituting $t=\frac12$, the result follows.
\end{proof}
\begin{remark}
Dividing both sides of~\eqref{eq:PD_cnt_q} by $(1-t)^{MN}$ and replacing $\frac{t}{1-t}$ with $q$ yields a $q$-analog of~\eqref{eq:PD_cnt}:
\begin{equation}\label{eq:PD_cnt_t}
  \sum_{\pip\in\PD^{\H,\V}} q^{\xing(\pip)}=\sum_{\pip\in\PD^{\flip(\H),\V}} q^{\xing(\pip)}.
\end{equation}
\end{remark}
\begin{problem}
Give a non-recursive bijective proof of~\eqref{eq:PD_cnt} or~\eqref{eq:PD_cnt_t}.
\end{problem}

\def\Ht{\HTop}
\subsection{Arbitrary permutations}\label{sec:arbitr-perm}
We finish with a discussion of more general wiring diagram domains associated to arbitrary Yang--Baxter basis elements.

The statement of \cref{thm:main} extends perfectly well to arbitrary $w\in\Sn$: one may consider two families of height functions $(\Ht(\ci_1,\cj_1),\dots,\Ht(\ci_m,\cj_m))$, $(\Ht(\ci'_1,\cj'_1),\dots,\Ht(\ci'_m,\cj'_m))$ defined in~\eqref{eq:HTpi}. They are considered as random vectors with respect to the probability distribution arising from the \SCSV model associated with $\YB^w$. Suppose that a permutation $\bz'=(z'_1,\dots,z'_n)$ of the variables in $\bz$ is such that all marginal distributions are preserved: 
\begin{equation*}%
  \Ht(\ci_k,\cj_k)\eqd\Ht(\ci'_k,\cj'_k)\mid_{\bz\mapsto\bz'}\quad\text{for all $k\in[1,m]$}.
\end{equation*}
Does this imply that 
\begin{equation*}%
 (\Ht(\ci_1,\cj_1),\dots,\Ht(\ci_m,\cj_m))\eqd\left(\Ht(\ci'_1,\cj'_1),\dots,\Ht(\ci'_m,\cj'_m)\right)\mid_{\bz\mapsto\bz'}?
\end{equation*}
When $w$ is associated with a \skew domain as in \cref{sec:skew_to_wiring} (such permutations are called \emph{fully commutative}), the answer is positive since this is precisely the subject of \cref{thm:main}. For general $w$, we give a counterexample.
\begin{example}
Let $n=4$, $w=s_3s_2s_3=s_2s_3s_2$ and consider height functions $\Ht(2,2)$ and $\Ht(3,3)$. Suppose $\pi\in\Sn$ has nonzero probability (i.e., $\pi\leq w$ in the Bruhat order). Then
\begin{equation*}%
  \Ht_\pi(2,2)=
  \begin{cases}
    0, &\text{if $\pi(2)=2$,}\\
    1, &\text{otherwise;}\\
  \end{cases}\quad\Ht_\pi(3,3)=
  \begin{cases}
    0, &\text{if $\pi(4)=4$,}\\
    1, &\text{otherwise.}\\
  \end{cases}
\end{equation*}
Considering the wiring diagram associated with the reduced word $w=s_3s_2s_3$, we see that the probability that $\Ht(2,2)=0$ is equal to $1-\sp_{2,4}$. If we instead use the Yang--Baxter relation and consider the reduced word $w=s_2s_3s_2$, we similarly find that the probability that $\Ht(3,3)=0$ is also equal to $1-\sp_{2,4}$. In particular, the distributions of $\Ht(2,2)$ and $\Ht(3,3)$ are the same and are invariant under swapping $z_1$ and $z_3$. However, their joint distribution is not invariant under this substitution: for example, the only permutation $\pi$ satisfying $\Ht_\pi(2,2)=0$ and $\Ht_\pi(3,3)=1$ is the simple transposition $\pi=s_3$ which appears with probability $(1-\sp_{2,3})(1-\sp_{3,4})\sp_{2,4}$. This quantity depends on $z_3$ but not on $z_1$ so it is not invariant under swapping these two variables. Therefore \cref{thm:main} does not extend to arbitrary permutations $w$.
\end{example}

\def\HO{\H}
\def\VO{\V}

\def\Ho{\accentset{\circ}{\H}^{\alpha,\delta}}
\def\Vo{\accentset{\circ}{\V}^{\alpha,\delta}}
\def\SATHVO{\accentset{\circ}{\SATop}_{\alpha,\delta}(\HO,\VO)}
\def\ZHV{\accentset{\circ}{\PF}^{\HO,\VO}}
Despite this example, we can still give a conjectural generalization of the shift-invariance results of~\cite{BGW} to arbitrary $w$. Fix $\alpha,\delta\in[1,n]$, and for each $\pi\in\Sn$, define
\begin{equation*}%
  \Ho_\pi:=\{(i,\pi(i))\mid i>\alpha \text{ and }\pi(i)<\delta \},\quad    \Vo_\pi:=\{(i,\pi(i))\mid i< \alpha\text{ and }\pi(i)> \delta\}.
\end{equation*}
This differs from $\H^{\alpha,\delta}_\pi$ and $\V^{\alpha,\delta}_\pi$ defined in~\eqref{eq:Hcicj} by passing from weak to strong inequalities.

Fix two sets $\HO,\VO$ of pairs and denote $\SATHVO:=\{\pi\in\Sn\mid \Ho_\pi=\HO,\Vo_\pi=\VO\}$. Also let us fix $w\in\Sn$ and consider the probability distribution associated with $\YB^w$. We let $\ZHV_{\XX}$ denote the probability that $\pi\in\SATHVO$. We further introduce probabilities $\ZHV_\MU,\ZHV_\MMa,\ZHV_\MD$ that $\pi$ belongs to $\SATHVO$ and satisfies, respectively, $\pi(\alpha)>\delta$, $\pi(\alpha)=\delta$, or $\pi(\alpha)<\delta$. Similarly, we let $\ZHV_\UM,\ZHV_\MMd,\ZHV_\DM$ denote the probabilities that $\pi$ belongs to $\SATHVO$ and satisfies, respectively, $\pi^{-1}(\delta)>\alpha$, $\pi^{-1}(\delta)=\alpha$, or $\pi^{-1}(\delta)<\alpha$. The probabilities $\ZHV_\MMa$ and $\ZHV_\MMd$ are equal and denoted by $\ZHV_\MM$. We have
\begin{equation*}%
  \ZHV_\MU+\ZHV_\MM+\ZHV_\MD=\ZHV_\UM+\ZHV_\MM+\ZHV_\DM=\ZHV_\XX.
\end{equation*}

Let $\beta:=w^{-1}(\delta)$ and $\gamma:=w(\alpha)$. We say that $\{\alpha,\beta\}$ is an \emph{inversion of $w$} if either $\alpha<\beta$ and $\gamma>\delta$ or $\alpha>\beta$ and $\gamma<\delta$. (In particular, $\{\alpha,\beta\}$ is not an inversion when $\alpha=\beta$.) The following conjecture has been verified for $n\leq 6$.

\def\middd{\Bigg|}
\begin{conjecture}\label{conj:cond_shift}
Let $\alpha,\delta\in[1,n]$, $w\in\Sn$, and $\beta:=w^{-1}(\delta)$ be such that $\{\alpha,\beta\}$ is not an inversion of $w$. Then for any $\H,\V$, we have
\begin{equation}\label{eq:cond_shift}
  \frac{\ZHV_\MU}{\ZHV_\XX}=\frac{\ZHV_\DM}{\ZHV_\XX}\middd_{z_\alpha\leftrightarrow z_\beta},\qquad \frac{\ZHV_\MM}{\ZHV_\XX}=\frac{\ZHV_\MM}{\ZHV_\XX}\middd_{z_\alpha\leftrightarrow z_\beta},\quad\text{and}\quad
\frac{\ZHV_\MD}{\ZHV_\XX}=\frac{\ZHV_\UM}{\ZHV_\XX}\middd_{z_\alpha\leftrightarrow z_\beta}.
\end{equation}
\end{conjecture}

\begin{remark}
Both sides of~\eqref{eq:cond_shift} can be interpreted as \emph{conditional} probabilities and can be restated in the language of height functions. If the denominator $\ZHV_\XX$ is symmetric in $z_\alpha,z_\beta$ then the equalities in~\eqref{eq:cond_shift} hold for just the numerators. This is the case for \skew domains, and therefore it is straightforward to check that \cref{conj:cond_shift} implies~\cite[Theorems~1.2 and~4.13]{BGW}. However, $\ZHV_\XX$ is not symmetric in $z_\alpha,z_\beta$ for other choices of $w$ as \cref{ex:cond} demonstrates.
\end{remark}
\begin{remark}
When $\{\alpha,\beta\}$ is an inversion of $w$, \cref{conj:cond_shift} does not hold: in the limit regime where $z_1\gg\dots\gg z_n$, the $\sp$-parameters become equal to $1/q$, and then sending $q\to 1$, we obtain $T_w$ as the limit of $\YB^w$. In this case, $w$ appears as the color permutation with probability $1$. If, say, $\alpha>\beta$ then we get $\ZHV_\MD=\ZHV_\DM=1$ while $\ZHV_\MU=\ZHV_\UM=0$, violating the first and the third equalities in~\eqref{eq:cond_shift}.
\end{remark}

\def\lw{1pt}
\def\tikzscl{0.7}
\def\sclbx{0.7}
\def\roc{3}

\def\ZHVr{\accentset{\circ}{Z}^{\HO,\textcolor{red}{\VO}}}

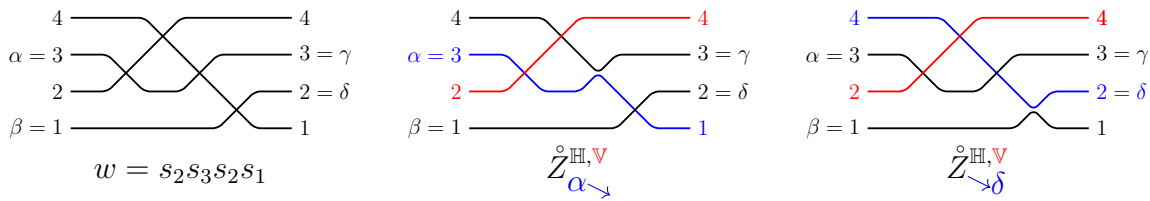
\begin{figure}

\begin{tabular}{ccc}
\scalebox{\sclbx}{
\begin{tikzpicture}[yscale=-\tikzscl, xscale=\tikzscl,baseline=(ZUZU.base)]
\coordinate(ZUZU) at (0,0);
\node[anchor=east] (A) at (0,0) {$4$};
\node[anchor=east] (A) at (0,1) {$\alpha=3$};
\node[anchor=east] (A) at (0,2) {$2$};
\node[anchor=east] (A) at (0,3) {$\beta=1$};

\draw[rounded corners=\roc,line width=\lw] (0,0)--(2,0)--(5,3)--(6,3);
\draw[rounded corners=\roc,line width=\lw] (0,1)--(1,1)--(2,2)--(3,2)--(4,1)--(6,1);
\draw[rounded corners=\roc,line width=\lw] (0,2)--(1,2)--(3,0)--(6,0);
\draw[rounded corners=\roc,line width=\lw] (0,3)--(4,3)--(5,2)--(6,2);

\node[anchor=west] (A) at (6,0) {$4$};
\node[anchor=west] (A) at (6,1) {$3=\gamma$};
\node[anchor=west] (A) at (6,2) {$2=\delta$};
\node[anchor=west] (A) at (6,3) {$1$};
\end{tikzpicture}
}

&

\scalebox{\sclbx}{
\begin{tikzpicture}[yscale=-\tikzscl, xscale=\tikzscl,baseline=(ZUZU.base)]
\coordinate(ZUZU) at (0,0);
\node[anchor=east] (A) at (0,0) {$4$};
\node[anchor=east,blue] (A) at (0,1) {$\alpha=3$};
\node[anchor=east,red] (A) at (0,2) {$2$};
\node[anchor=east] (A) at (0,3) {$\beta=1$};

\draw[rounded corners=\roc,line width=\lw] (0,0)--(2,0)--(3.5,1.5)--(4,1)--(6,1);
\draw[rounded corners=\roc,line width=\lw,blue] (0,1)--(1,1)--(2,2)--(3,2)--(3.5,1.5)--(5,3)--(6,3);
\draw[rounded corners=\roc,line width=\lw,red] (0,2)--(1,2)--(3,0)--(6,0);
\draw[rounded corners=\roc,line width=\lw] (0,3)--(4,3)--(5,2)--(6,2);

\node[anchor=west,blue] (A) at (6,3) {$1$};
\node[anchor=west] (A) at (6,1) {$3=\gamma$};
\node[anchor=west] (A) at (6,2) {$2=\delta$};
\node[anchor=west,red] (A) at (6,0) {$4$};
\end{tikzpicture}
}

&

\scalebox{\sclbx}{
\begin{tikzpicture}[yscale=-\tikzscl, xscale=\tikzscl,baseline=(ZUZU.base)]
\coordinate(ZUZU) at (0,0);
\node[anchor=east,blue] (A) at (0,0) {$4$};
\node[anchor=east] (A) at (0,1) {$\alpha=3$};
\node[anchor=east,red] (A) at (0,2) {$2$};
\node[anchor=east] (A) at (0,3) {$\beta=1$};

\draw[rounded corners=\roc,line width=\lw,blue] (0,0)--(2,0)--(4.5,2.5)--(5,2)--(6,2);
\draw[rounded corners=\roc,line width=\lw] (0,1)--(1,1)--(2,2)--(3,2)--(4,1)--(6,1);
\draw[rounded corners=\roc,line width=\lw,red] (0,2)--(1,2)--(3,0)--(6,0);
\draw[rounded corners=\roc,line width=\lw] (0,3)--(4,3)--(4.5,2.5)--(5,3)--(6,3);

\node[anchor=west] (A) at (6,3) {$1$};
\node[anchor=west] (A) at (6,1) {$3=\gamma$};
\node[anchor=west,red] (A) at (6,0) {$4$};
\node[anchor=west,blue] (A) at (6,2) {$2=\delta$};
\node[anchor=west,red] (A) at (6,0) {$4$};
\end{tikzpicture}
}

\\
$w=s_2s_3s_2s_1$
 & $\ZHVr_\MDb$ & $\ZHVr_\UMb$\\

\end{tabular}

  \caption{\label{fig:cond_shift} In this case, \cref{conj:cond_shift} does not hold without the denominators, see \cref{ex:cond}.}
\end{figure}

\begin{example}\label{ex:cond}
Let $n=4$, $w=s_2s_3s_2s_1$, $\alpha=3$, $\beta=1$, $\gamma=3$, $\delta=2$, $\H=\emptyset$, and $\V=\{(2,4)\}$, see \cref{fig:cond_shift}. For such $\H$ and $\V$, we find $\ZHV_{\XX}=\sp_{23}\sp_{24}(1-\sp_{34}\sp_{14})$. 
The permutations contributing to $\ZHV_\MD$ and $\ZHV_\UM$ are shown in \cref{fig:cond_shift}, so~\eqref{eq:cond_shift} reads $\frac{\sp_{23}\sp_{24}(1-\sp_{34})\sp_{14}}{\sp_{23}\sp_{24}(1-\sp_{34}\sp_{14})}=\frac{\sp_{23}\sp_{24}\sp_{34}(1-\sp_{14})}{\sp_{23}\sp_{24}(1-\sp_{34}\sp_{14})}\mid_{z_1\leftrightarrow z_3}.$ 
 This identity is true, but would be false without the denominators.
\end{example}

\newcommand{\arxiv}[1]{\href{https://arxiv.org/abs/#1}{\textup{\texttt{arXiv:#1}}}}

\bibliographystyle{alpha}
\bibliography{biblio}

\end{document}